\documentclass[a4paper,reqno]{amsart}

\usepackage{amsmath}
\usepackage{amsthm}
\usepackage{amssymb}
\usepackage{color}
\usepackage{amscd} 

\usepackage{bbm} 
\usepackage{mathtools,mathrsfs} 

\usepackage{times} 

\usepackage{verbatim} 

\usepackage{enumitem}

\newtheoremstyle{newremark}
  {5pt}
  {5pt}
  {\rmfamily}
  {}
  {\rmfamily\bf}
  {.}
  {.5em}
  {}

\newtheorem{theorem}{Theorem}
\newtheorem{lemma}[theorem]{Lemma}
\newtheorem{corollary}[theorem]{Corollary}
\newtheorem{proposition}[theorem]{Proposition}

\theoremstyle{newremark}
\newtheorem{remark}[theorem]{Remark}
\newtheorem{definition}[theorem]{Definition}

\newtheorem*{definition*}{Definition} 
\newtheorem*{notations*}{Notations}
\newtheorem*{convention*}{Convention}

\numberwithin{theorem}{section}
\numberwithin{equation}{section}

\newcommand{\N}{\mathbb{N}} 
\newcommand{\R}{\mathbb{R}} 

\def\XXint#1#2#3{{%
\setbox0=\hbox{$#1{#2#3}{\int}$}
\vcenter{\hbox{$#2#3$}}\kern-.5\wd0}}

\renewcommand{\leq}{\leqslant}
\renewcommand{\geq}{\geqslant}
\renewcommand{\subset}{\subseteq}

\newcommand{\LL}{\mathop{\hbox{\vrule height 6pt width .5pt depth 0pt
\vrule height .5pt width 3pt depth 0pt}}\nolimits}

\newcommand{\res}{\mathop{\hbox{\vrule height 7pt width .5pt depth 0pt
\vrule height .5pt width 6pt depth 0pt}}\nolimits}

\newcommand{\Om}{\Omega}
\newcommand{\eps}{\varepsilon}

\newcommand{\e}{{\rm e}}
\newcommand{\de}{{\rm d}}

\begin{document}


\title[\bf Fractional multi-phase transitions]{Fractional multi-phase transitions\\
 and nonlocal minimal partitions}

\author{Thomas Gabard}
\address{Thomas Gabard\\ LAMA, Universit\'e Paris Est Cr\'eteil, France} 
\email{thomas.gabard@u-pec.fr}
\author{Vincent Millot}
\address{Vincent Millot\\ LAMA, Universit\'e Paris Est Cr\'eteil, France}
\email{vincent.millot@u-pec.fr}





\maketitle


\tableofcontents


\section{Introduction}

The most popular phase-field model for phase transitions is probably the van der Waals-Cahn-Hilliard theory for mixtures of immiscible fluids. In this theory,  
multi-phase systems are described through energetic flows of Allen-Cahn or Ginzburburg-Landau functionals of the form 
\begin{equation}\label{EN0}
\int_{\Om} \eps|\nabla u|^2+\frac{1}{\eps}W(u) \,\de x\quad\text{ with }\,\eps\in(0,1)\,,
\end{equation}
where the {\sl phase indicator function} $u:\Om\subset\R^n\to\R^d$ represents a vector valued density distribution of the different phases  within the container $\Omega$, and each scalar component of $u$ usually describes the density of a single ingredient of the mixture. The potential $W:\R^d\to[0,\infty)$ vanishes exactly at finitely many global minima ${\bf a}_1,\ldots,{\bf a}_m\in\R^d$ with $m\geq 2$, see e.g. \cite{Bal,BrRe,FonTar,Gurt,Ster,Ster2} and the references therein.  
Here the container $\Omega$ is assumed to be a smooth and bounded open set in dimension $n\geq 2$. Equilibrium states of the model corresponds to critical points of the energy, and they 
satisfy the vectorial Allen-Cahn  equation
\begin{equation}\label{AC0}
-\Delta u_\eps+\frac{1}{\eps^{2}}\nabla W(u_\eps)=0\quad\text{in $\Om$}\,. 
\end{equation}
For small values of the parameter $\eps$, a uniform upper bound on the energy  implies  that $W(u_\eps)\simeq 0$ away from a thin region of characteristic thickness $\eps$. At the formal level, the transition layer between two different minima  has a width of order~$\eps$, and the container $\Omega$ is gets divided into subdomains where $u_\varepsilon$ assumes values close to one of the minima. In the {\sl sharp interface limit} $\varepsilon\to 0$, the union of the boundaries of the resulting partition of $\Omega$ 
is expected to be a critical point of some inhomogeneous area functional where the area of an interface separating two subdomains is weighted by a surface tension coefficient depending on the two minima of $W$ on both sides of it.   

In the scalar case $d=1$ and for binary systems $m=2$, this description has been first proved for energy minimizing solutions by means of $\Gamma$-convergence \cite{Mod,ModMort}.  Still concerning minimizers, the vectorial case $d\geq 2$ for binary systems has been addressed in \cite{FonTar,Ster,Ster2}. The $\Gamma$-convergence as $\varepsilon\to 0$ of the functional \eqref{EN0} in the general case has been established in \cite{Bal}.  
It shows that sequences $\{u_\varepsilon\}$ with uniformly bounded energy subconverges in $L^1(\Omega)$ as $\varepsilon\to 0$ to some map $u_0\in BV(\Om;\R^d)$ with values in 
$\{W=0\}=\{{\bf a}_1,\ldots,{\bf a}_m\}$. The limiting map $u_0$ being of bounded variation with values in the finite set $\{{\bf a}_1,\ldots,{\bf a}_m\}$ means that $u_0$ rewrites as  $u_0=\sum_j\chi_{E_j}{\bf a}_j$ (with $\chi_{E_j}$  the characteristic function of $E_j$)  for a family $(E_1,\ldots,E_m)$ of essentially disjoint subsets of $\Omega$ of finite distributional perimeter (see e.g. \cite{AFP}) realizing a partition of $\Omega$ (in the sense that $|\Omega\setminus(\cup_jE_j)|=0$ and $|E_i\cap E_j|=0$ for $i\not=j$).  In terms of a partition $(E_1,\ldots,E_m)$, the $\Gamma$-limit functional obtained in \cite{Bal} reads
\begin{equation}\label{locfunctintro}
\mathscr{P}_1^{\bf S}\big((E_1,\ldots,E_m),\Omega\big):=\frac{1}{2}\sum_{i,j=1}^m{\bf s}_{ij}\, \mathcal{H}^{n-1}(\partial^*E_i\cap\partial^*E_j\cap\Omega) \,,
\end{equation}
where $\mathcal{H}^{n-1}$ denotes the $(n-1)$-dimensional Hausdorff measure, $\partial^*E_i$ is the reduced boundary of $E_i$ (see e.g. \cite{AFP}), and the symmetric $m\times m$-matrix 
${\bf S}=({\bf s}_{ij})$  
of surface tension coefficients  is determined by ${\bf s}_{ii}=0$, and for $i\not= j$, 
\begin{equation}\label{formcoeffsij}
{\bf s}_{ij}:=\inf\Big\{2\int_0^1W^{1/2}\big(\gamma(t)\big)|\gamma^\prime(t)|\,\de t : \gamma\in C^1([0,1];\R^d)\,,\;\gamma(0)={\bf a}_i\,,\;\gamma(1)={\bf a}_j\Big\}>0\,.
\end{equation} 
We point out that the ${\bf s}_{ij}$'s satisfy a {\sl triangle inequality}, that is 
\begin{equation}\label{triangineqintro}
{\bf s}_{ij}\leq {\bf s}_{ik}+{\bf s}_{kj}
\end{equation}
for every triplet of indices $\{i,j,k\}\subset\{1,\ldots,m\}$. This can be either checked directly from formula~\eqref{formcoeffsij}, or one may argue that, $\mathscr{P}_1^{\bf S}$ being a $\Gamma$-limit, it must be lower semicontinuous (with respect to the $L^1$-topology), so that the results in \cite{AmBr,Morg,W} ensure that this  triangle inequality must hold. 
\vskip3pt

Concerning general critical points \eqref{AC0}, the scalar case $d=1$ for binary systems $m=2$ has been  treated  in \cite{HT}, and it reveals a more intricate analysis. Under a uniform energy bound, \cite{HT} shows that  the energy density converges in the sense of measures as $\eps\to0$ to a stationary integral $(n-1)$-varifold, i.e., a measure theoretic minimal hypersurface with integer multiplicity.
The results from \cite{HT}  have been recently improved in the case of stable or finite Morse index solutions, and used to produce minimal hypersurfaces in Riemannian manifolds, see e.g. \cite{ChoMan,GasGua,Gua,TW}. In the vectorial case $d\geq 2$ and arbitrary $m\geq 2$, the study of critical points was essentially out of reach until very recently as many technical tools from the scalar case are not available (the most important feature is probably the lack of a useful monotonicity formula). In planar domains $n=2$, the analogue of the result in \cite{HT} has been obtained in the vectorial case with an arbitrary number $m$ of wells in \cite{Bet0,Bet}. Unfortunately, the arguments are truly based on two dimensional quantities (such as the Hopf differential), and the higher dimensional case $n\geq 3$ remains open. 
\vskip5pt

In this article, we are interested in a nonlocal or fractional version of equation \eqref{AC0} and energy \eqref{EN0}. Nonlocal models for phase transitions appear in many physical models such as  
stochastic Ising models from statistical mechanics \cite{KUH,LebPen}, or models for dislocations in crystals from continuum mechanics \cite{GM,Im,ImSoug}. Here, we shall consider one of the simplest nonlocal version of equation \eqref{AC0} where the Laplace operator is replaced by the fractional Laplacian $(-\Delta)^s$ with $s\in(0,1)$, i.e., the Fourier multiplier  of symbol $(2\pi|\xi|)^{2s}$. In this direction, many mathematical studies are already available in the scalar case $d=1$ with $m=2$, starting with  \cite{ALBBEL,ABScr,ABS,CS1,Gonz,MilSirW,SV1} (and the references therein), and more recently in 
\cite{CFSS,CDSV} with applications to minimal hypersurfaces and Yau type conjectures in the spirit of the standard local case. The analogue of \cite{HT} for the asymptotic analysis of critical points has been addressed in \cite{MilSirW} for $s\in(0,1/2)$. It represents the starting point of this article, and our aim is to generalize \cite{MilSirW}  to the vectorial case $d\geq 1$ with $m\geq 2$ arbitrary. Before going further, let us now describe in details the problem we plan to study. 
\vskip3pt

Given a fixed exponent $s\in(0,1/2)$, we are interested in the asymptotic behavior as $\varepsilon\to 0$ of weak solutions $u_\varepsilon:\R^n\to\R^d$ to the vectorial fractional Allen-Cahn equation  
\begin{equation}\label{GLIntro}
 (-\Delta)^{s} u_\varepsilon+\frac{1}{\varepsilon^{2s}}\nabla W(u_\varepsilon) =0\quad \text{in $\Omega$}\,,
 \end{equation}
 subject to a Dirichlet condition of the form 
\begin{equation}\label{Dircondintro}
u_\eps= g_\eps \qquad\text{on $\R^n\setminus\Omega$}\,,
\end{equation}
where $g_\eps:\R^n\to\R^d$ is a given smooth and bounded function. 
The operator $ (-\Delta)^{s} $ acts on  smooth bounded functions through the formula 
\begin{equation}\label{formpvfraclap}
 (-\Delta)^{s}  u(x):={\rm p.v.} \left(\gamma_{n,s}\int_{\mathbb{R}^n}\frac{u(x)-u(y)}{|x-y|^{n+2s}}\,\de y\right) \quad \text{with } \gamma_{n,s}:=s2^{2s}\pi^{-\frac{n}{2}}\frac{\Gamma\big(\frac{n+2s}{2}\big)}{\Gamma(1-s)} \,,
 \end{equation}
where the notation ${\rm p.v.}$ means that the integral  is taken in the {\sl Cauchy principal value} sense. A variational (distributional) formulation holds whenever $u\in L^2_{\rm loc}(\R^n;\R^d)$ satisfies
\begin{multline}\label{defenergE}
\mathcal{E}_s(u,\Omega):=\frac{\gamma_{n,s}}{4}  \iint_{\Omega\times\Omega}\frac{|u(x)-u(y)|^2}{|x-y|^{n+2s}}\,\de x\de y\\
+\frac{\gamma_{n,s}}{2} \iint_{\Omega\times(\R^n\setminus\Omega)} \frac{|u(x)-u(y)|^2}{|x-y|^{n+2s}}\,\de x\de y<\infty\,,
\end{multline}
and the functional $\mathcal{E}_s(\cdot,\Omega)$  corresponds to the  {\sl fractional Dirichlet energy} in~$\Omega$ associated to $ (-\Delta)^{s} $.
 Integrating  the potential in \eqref{GLIntro}, we obtain the {\sl fractional Allen-Cahn energy} in $\Omega$ associated to \eqref{GLIntro}, 
\begin{equation}\label{defFGLenerg}
\mathcal{E}_{s,\varepsilon}(u,\Omega):=\mathcal{E}_s(u,\Omega)+ \frac{1}{\varepsilon^{2s}}\int_\Omega W(u)\,\de x\,.
\end{equation}
To include the Dirichlet condition \eqref{Dircondintro}, we shall simply restrict the energy to the affine space of admissible functions $H^{s}_{g_\eps}(\Omega;\R^d):=g_\eps+H_{00}^{s}(\Omega;\R^d)$  (see Section \ref{secHs} for details about these spaces). 
In this way, weak solutions of \eqref{GLIntro}-\eqref{Dircondintro} are defined as critical points  of  $\mathcal{E}_{s,\varepsilon}(\cdot,\Omega)$  over $H^{s}_{g_\eps}(\Omega;\R^d)$. 
\vskip3pt

In the scalar case $d=1$ with $\{W=0\}=\{\pm 1\}$, the asymptotic behavior of minimizers of  $\mathcal{E}_{s,\varepsilon}(\cdot,\Omega)$ has been addressed in  \cite{SV1} through a $\Gamma$-convergence analysis.  In fact,   \cite{SV1} handles the two cases $s<1/2$ and $s\geq 1/2$, and proves a strong dichotomy. While the case $s\geq 1/2$ reveals a behavior in a sense similar to the one of the local energy \eqref{EN0} (see also \cite{ALBBEL,ABScr,ABS,Gonz}),  the case $s<1/2$ is of a quite different nature since $H^s$-regularity does not exclude (all) characteristic functions. 
For $s\in(0,1/2)$,  assuming that $g_\eps\to g$  a.e. in $\R^n\setminus\Omega$ for some function $g$ with values in $\{\pm 1\}$, the functionals $\mathcal{E}_{s,\varepsilon}(\cdot,\Om)$ (restricted to $H^{s}_{g_\eps}(\Omega)$) converge as $\eps\to 0$ both in the variational and pointwise sense to the functional defined by $\mathcal{E}_{s,0}(u,\Om)=\mathcal{E}_{s}(u,\Om)$ if $u\in H^{s}_g(\Om;\{\pm1\})$, and $\mathcal{E}_{s,0}(u,\Om)=+\infty$ otherwise.  
We now point out that 
\begin{equation}\label{idDirPersharpint}
 \mathcal{E}_s(u,\Omega)  = 2\gamma_{n,s} P_{2s}\big(\{u=1\},\Omega\big)\quad \forall u\in H^{s}_g(\Om;\{\pm1\})\,,
 \end{equation}
where $P_{2s}(E,\Omega)$ is the so-called {\sl fractional $2s$-perimeter} in $\Omega$ of a set $E\subset\R^n$, i.e., 
\begin{multline}\label{P2s}
P_{2s}(E,\Omega):=\int_{E\cap\Omega}\int_{E^c\cap\Omega}\frac{\de x\de y}{|x-y|^{n+2s}}+\int_{E\cap\Omega}\int_{E^c\setminus\Omega}\frac{\de x\de y}{|x-y|^{n+2s}}\\
+\int_{E\setminus\Omega}\int_{E^c\cap\Omega}\frac{\de x\de y}{|x-y|^{n+2s}}\,,
\end{multline}
introduced in \cite{CRS}. Here and below, $E^c$ denotes the complement in $\R^n$ of a set $E$. By $\Gamma$-convergence routine, minimizers go to minimizers, and limits of minimizers are of the form $\chi_{E_*}-\chi_{E_*^c}$ for a limiting set $E_*\subset \R^n$ which is minimizing its $2s$-perimeter in $\Om$, i.e., 
\begin{equation}\label{defminimizingnonlocminsurf}
P_{2s}(E_*,\Om)\leq P_{2s}(F,\Om) \qquad \forall F\subset\R^n\,,\; F\setminus\Omega=E_*\setminus\Om\,.
\end{equation}
The boundary $\partial E_*\cap \Omega$ is referred to as minimizing {\sl nonlocal minimal surfaces} in~$\Omega$ by analogy with the classical perimeter, 
a notion introduced and first studied concerning regularity  in   \cite{CRS}. The Euler-Lagrange equation induced by  minimality condition \eqref{defminimizingnonlocminsurf} can be obtained considering 
the first inner variation in $\Omega$ of the functional $P_{2s}$ at $E_*$, i.e., 
\begin{equation}\label{defstatnonlocminsurf}
\delta P_{2s}(E_*,\Om)[X]:=\left[\frac{\de}{\de t}  P_{2s}\big(\phi_t(E_*) , \Omega\big)\right]_{t=0}=0 
\end{equation}
 for every vector field $X\in C^1(\R^n;\R^n)$ compactly supported in $\Omega$, where $\{\phi_t\}_{t\in\R}$ denotes the integral flow generated by $X$ (see \eqref{defintegflow}). 
 For an arbitrary set $E\subset \R^n$ with smooth boundary (e.g. of class $C^2$), the value of $\delta P_{2s}$ at $E$ is given by (see e.g. \cite[Section 6]{FFMMM})
  \begin{equation}\label{smoothfirstcarper}
 \delta P_{2s}(E,\Omega)[X]= \int_{\partial E \cap \Omega} \mathrm{H}^{(2s)}_{\partial E}(x) \,X\cdot \nu_E \,\de\mathcal{H}^{n-1}\,,
 \end{equation}
 where $\nu_E$ denotes the unit exterior normal field on $\partial E$, and $\mathrm{H}^{(2s)}_{\partial E}$ is the so-called {\sl nonlocal} (or fractional) {\sl mean curvature} of $\partial E$, defined by
\begin{equation}\label{nonlocmeancurvdef}
\mathrm{H}^{(2s)}_{\partial E}(x):={\rm p.v.}\left(\int_{\mathbb{R}^n}\frac{\chi_{E^c}(y)-\chi_E(y)}{|x-y|^{n+2s}}\,\de y\right)\,,\quad x\in\partial E\,.
\end{equation}
Consequently, a set $E_*$ whose boundary is a minimizing nonlocal  minimal surface in $\Omega$ must satisfy relation \eqref{defstatnonlocminsurf} which can be interpreted as the weak formulation of the {\sl nonlocal minimal surface equation}
\begin{equation}\label{zeromeancurv}
\mathrm{H}^{(2s)}_{\partial E_*} = 0\quad \text{on $\partial E_*\cap\Omega$}\,.
\end{equation}
In \cite{MilSirW}, it has been proved, still in the scalar case $d=1$ and $\{W=0\}=\{\pm1\}$, that the behavior of critical points of $\mathcal{E}_{s,\varepsilon}$ is somehow 
similar to the behavior of minimizers in the sense that critical points converge as $\varepsilon\to0$ to functions of the form $\chi_{E_*}-\chi_{E_*^c}$ for a set $E_*\subset \R^n$ 
which is critical or more precisely {\sl stationary} in $\Omega$ for the geometric functional $P_{2s}$, i.e., relation \eqref{defstatnonlocminsurf} holds. In other words, $\partial E_*$ is a weak solution of the nonlocal minimal surface equation  \eqref{zeromeancurv} in the domain $\Omega$. 
\vskip5pt

The first main objective of this article is to extend the results of \cite{MilSirW,SV1} to the vectorial fractional   Allen-Cahn equation \eqref{GLIntro} with $d\geq 1$ and $m\geq 2$ arbitrary for both non minimizing and minimizing solutions and $s\in(0,1/2)$, i.e., in the regime of nonlocal minimal surfaces. 
In our results, we shall require the following set of structural assumptions on the multiple-well potential $W:\R^d\to [0,\infty)$. 
\begin{enumerate}
\item[(H1)] $W\in C^2\big(\R^d;[0,\infty)\big)$.
\vskip3pt

\item[(H2)]  There exist $m\geq 2$ distinct points ${\bf a}_1,\ldots,{\bf a}_m\in\R^d$ such that 
$$\mathcal{Z}:=\{W=0\}=\{{\bf a}_1,\ldots,{\bf a}_m\}\,,$$ 
and $D^2W({\bf a}_j)>0$ in the sense of quadratic forms for each $j\in\{1,\ldots,m\}$. 
\vskip3pt

\item[(H3)]  There exist an exponent $p\in (1,\infty)$  and  three constants  $\boldsymbol{c}^1_W,\boldsymbol{c}^2_W,\boldsymbol{c}^3_W>0$  such that for all directions $\nu\in\mathbb{S}^{d-1}$ and all $t>0$, 
$$\boldsymbol{c}^1_W\big(t^{p-1}-\boldsymbol{c}^2_W\big) \leq \frac{\partial W}{\partial\nu}(t\nu)\leq \boldsymbol{c}^3_W\big(t^{p-1}+1\big)\,.$$   
\end{enumerate}
Those assumptions are  satisfied  by the prototypical potential $\displaystyle W(z)=\Pi_j|z-{\bf a}_j|^2$ with  $p=2m$. Notice that assumption (H3) implies that $W$ has a $p$-growth at infinity so that finite energy solutions of  \eqref{GLIntro} belong to $L^p(\Omega)$. Assuming that (H1)-(H2)-(H3) hold, we will prove that any weak solution of \eqref{GLIntro}-\eqref{Dircondintro} actually belongs to $C^{1,\alpha}_{\rm loc}(\Omega)\cap C^0(\R^n)$ for some $\alpha\in(0,1)$. 
\vskip3pt

Before stating our first result in details, let us explain what can be expected from the analogy with~\cite{MilSirW,SV1} and take the advantage to introduce some useful notations. 
Assuming that the exterior boundary condition in \eqref{Dircondintro} converges a.e. in $\R^n\setminus\Omega$  as $\varepsilon\to 0$ to some map with values in $\mathcal{Z}$, solutions to \eqref{GLIntro} with uniformly bounded energy should converge (up to subsequences) to a $\mathcal{Z}$-valued map (defined over~$\R^n$) with finite $\mathcal{E}_{s}$-energy in $\Omega$. A (measurable) map $u$ defined over $\R^n$ taking values in $\mathcal{Z}$ can be written in the form $u=\sum_j\chi_{E_j}{\bf a}_j$ where the sets $E_1,\ldots,E_m\subset \R^n$ realize a  partition of~$\R^n$. Then the energy rewrites for such a $\mathcal{Z}$-valued map $u$, 
\begin{equation}\label{correspEPintro}
\mathcal{E}_s(u,\Omega)=\frac{\gamma_{n,s}}{2}\mathscr{P}^{\boldsymbol{\sigma}}_{2s}(\mathfrak{E},\Omega)\,, 
\end{equation}
where we have set $\mathfrak{E}:=(E_1,\ldots,E_m)$ and 
\begin{equation}\label{defpersigintro}
\mathscr{P}^{\boldsymbol{\sigma}}_{2s}(\mathfrak{E},\Omega):=\frac{1}{2}\sum^m_{i,j=1}\sigma_{ij}\Big[\mathcal{I}_{2s}(E_i\cap\Omega,E_j\cap\Omega)+\mathcal{I}_{2s}(E_i\cap\Omega,E_j\cap\Omega^c)+\mathcal{I}_{2s}(E_i\cap\Omega^c,E_j\cap\Omega)\Big] \,,
\end{equation}
with 
\begin{equation}\label{definterintro}
\mathcal{I}_{2s}(A,B):=\iint_{A\times B}\frac{1}{|x-y|^{n+2s}}\,\de x\de y \,,
\end{equation}
and  a symmetric $m\times m$ matrix $\boldsymbol{\sigma}=(\sigma_{ij})$ given by
\begin{equation}\label{condlimitWintro}
\boldsymbol{\sigma}=\big(|{\bf a}_i-{\bf a}_j|^2\big)_{1\leq i,j\leq m} \,.
\end{equation}
The family of  partitions of $\R^n$ for which \eqref{correspEPintro} is finite will be denoted by
\begin{multline*}
\mathscr{A}_m(\Omega):=\Big\{\mathfrak{E}=(E_1,\ldots,E_m): E_j\subset\R^n\text{ measurable}\,,\\
\sum_{j=1}^m\chi_{E_j}=1\text{ a.e. in }\R^n\,,\; \mathscr{P}^{\boldsymbol{\sigma}}_{2s}(\mathfrak{E},\Omega)<\infty\Big\}\,.
\end{multline*}
(Note that the class $\mathscr{A}_m(\Omega)$ does not depends on $\boldsymbol{\sigma}$ but only on the fact that $\sigma_{ij}>0$ for $i\not=j$.) 
Drawing now the analogy with \cite{MilSirW}, we may  expect that limits of critical points in $\Omega$ of $\mathcal{E}_{s,\varepsilon}$ provide critical points in $\Omega$ of $\mathscr{P}^{\boldsymbol{\sigma}}_{2s}$, i.e., partitions of $\R^n$ which are stationary in $\Omega$ for  $\mathscr{P}^{\boldsymbol{\sigma}}_{2s}$ that is partitions for which the first inner variation $\delta\mathscr{P}^{\boldsymbol{\sigma}}_{2s}$ in $\Omega$ vanishes, a {\sl nonlocal minimal partition} in $\Omega$ to borrow the vocabulary from~\cite{CRS}. The first variation $\delta\mathscr{P}^{\boldsymbol{\sigma}}_{2s}$ at a partition $\mathfrak{E}\in \mathscr{A}_m(\Omega)$ is defined similarly to \eqref{defstatnonlocminsurf} by 
\begin{equation}\label{deffirstvarparti}
\delta\mathscr{P}^{\boldsymbol{\sigma}}_{2s}(\mathfrak{E},\Omega)[X]:=\left[\frac{\de}{\de t}  \mathscr{P}^{\boldsymbol{\sigma}}_{2s}\big(\phi_t(\mathfrak{E}), \Omega\big)\right]_{t=0}
\end{equation}
 for a  vector field $X\in C^1(\R^n;\R^n)$ compactly supported in $\Omega$, where $\{\phi_t\}_{t\in\R}$ still denotes the integral flow generated by $X$ (see \eqref{defintegflow}), and 
$\phi_t(\mathfrak{E}):=(\phi_t({E}_1),\ldots,\phi_t({E}_m))$ (which is still an essential partition of $\R^n$ since $\phi_t$ is a $C^1$-diffeomorphism). 
\vskip5pt

We are now ready to state our first main result concerning arbitrary solutions to the vectorial fractional Allen-Cahn equation.

 \begin{theorem}\label{main1new}
Assume that $s\in(0,1/2)$ and that {\rm (H1)-(H2)-(H3)} hold. Let $\Omega\subset \R^n$ be a smooth and bounded open set. Consider a sequence $\eps_k\to  0$ and maps $\{g_k\}_{k\in\mathbb{N}}\subset C^{0,1}_{\rm loc}(\R^n;\R^d)$  such that $\sup_k\|g_k\|_{L^\infty(\R^n\setminus\Omega)}<\infty$ and $g_k(x)\to g(x)$ with $g(x)\in\mathcal{Z}$ 
for a.e. $x\in\R^n\setminus\Omega\,$ as $k\to\infty$. For each $k\in\mathbb{N}$, let $u_k\in H_{g_k}^s(\Omega;\R^d)\cap L^p(\Omega)$ be a weak solution of 
 \begin{equation}\label{eqthm}
 \begin{cases}
 \displaystyle  (-\Delta)^{s} u_k+\frac{1}{\varepsilon_k^{2s}}\nabla W(u_k) =0 & \text{in $\Omega$}\,,\\
  u_k=g_k & \text{in $\R^n\setminus \Omega$}\,.
\end{cases}
\end{equation}
 If $\sup_k \mathcal{E}_{s,\varepsilon_k}(u_k,\Omega)<\infty$, then there exist a (not relabeled) subsequence, a map $u_*\in H_{g}^s(\Omega;\mathcal{Z})$ and a partition 
 $\mathfrak{E}^*=(E_1^*,\ldots,E^*_m)\in \mathscr{A}_m(\Omega)$   such that 
  \begin{itemize}[leftmargin=25pt]
\item[ \rm  (i)] $u_k\to u_*=\sum_{j=1}^m\chi_{E_j^*}{\bf a}_j$ strongly in $H^{s^\prime}_{\rm loc}(\Omega)\cap L^2_{\rm loc}(\R^n)$ for every $s^\prime<\min(2s,1/2)$;  
\vskip5pt

\item[\rm (ii)] each set $E_j^*\cap \Omega$ is open and $P_{2s}(E^*_j,\Omega)<\infty$; 
\vskip5pt

\item[(iii)]  $\mathfrak{E}^*$ is a nonlocal minimal partition in $\Omega$ for the functional $\mathscr{P}^{\boldsymbol{\sigma}}_{2s}$ with ${\boldsymbol{\sigma}}=\big(|{\bf a}_i-{\bf a}_j|^2\big)_{i,j}\,$, 
i.e.,  
$$\delta\mathscr{P}^{\boldsymbol{\sigma}}_{2s}(\mathfrak{E}^*,\Omega)[X]=0$$
 for every  vector field $X\in C^1(\R^n;\R^n)$ compactly supported in $\Omega$.
\end{itemize}
 In addition, for every smooth open set $\Omega'\subset\Omega$ such that $\overline{\Omega'}\subset\Omega$, 
 \begin{itemize}[leftmargin=25pt]
\item[ \rm  (iv)] $ \mathcal{E}_s(u_k,\Omega^\prime) \to \mathcal{E}_{s}(u_*,\Omega^\prime)= \frac{\gamma_{n,s}}{2}\mathscr{P}^{\boldsymbol{\sigma}}_{2s}(\mathfrak{E}^*,\Omega^\prime)$; 
\vskip5pt

\item[ \rm  (v)] $\int_{\Omega^\prime}W(u_k)\,\de x=O(\varepsilon_k^{\min(4s,\alpha)})$ for every $\alpha\in(0,1)$; 
\vskip5pt

\item[\rm  (vi)] $\displaystyle \frac{1}{\eps_k^{2s}}\nabla W(u_k)\to - \sum_{j=1}^m V_{E^*_j}{\bf a}_j$ strongly in  $H^{-s}(\Omega^\prime)$ and weakly in $L^{\bar p}(\Omega^\prime)$ for every $\bar p<1/2s$, where  each potential $V_{E^*_j}$  given by
$$V_{E^*_j}(x):=\Big(\gamma_{n,s}\int_{\R^n}\frac{|\chi_{E^*_j}(x)-\chi_{E^*_j}(y)|^2}{|x-y|^{n+2s}}\,\de y\Big) \big(2\chi_{E^*_j}(x)-1\big)$$

\noindent and belongs to  $L^{\bar p}(\Omega^\prime)$ for every $\bar p<1/2s\,$.
\vskip8pt 

\noindent Finally, setting $\partial \mathfrak{E}^*:=\bigcup_{j=1}^m\partial E^*_j$, 
\vskip3pt

\item[\rm  (vii)] $u_k\to u_*$ in $C^{1,\alpha}_{\rm loc}(\Omega\setminus \partial \mathfrak{E}^*)$ for some $\alpha=\alpha(n,s)\in(0,1)$; 
\vskip5pt 

\item[\rm  (viii)] $\|u_k-u_*\|_{L^\infty(K)}\leq C_K \eps_k^{2s}$  for every compact set $K\subset \Omega\setminus \partial \mathfrak{E}^*$; 
\vskip5pt 

\item[\rm  (ix)] there exists $t_W>0$ (depending only on $W$) such that the transition set $L_k^t:=\{{\rm dist}(u_k,\mathcal{Z})\geq t\}$ converges locally uniformly  in $\Omega$ to $\partial \mathfrak{E}^*$ for each $t\in(0,t_W)$, i.e., for every compact set $K\subset \Omega$ and every $r>0$,
$$L^t_k\cap K\subset \mathscr{T}_r(\partial \mathfrak{E}^*\cap\Omega) \quad \text{and}\quad \partial \mathfrak{E}^*\cap K\subset \mathscr{T}_r(L_k^t\cap \Omega)$$
whenever $k$ is large enough. Here, $\mathscr{T}_r(A)$ represents the open tubular neighborhood of radius $r$ of a set $A$. 
\end{itemize}
 \end{theorem}

Our second result is specific to minimizing solutions and can actually be  seen as a corollary of the previous theorem. A slight improvement lies in the fact that we do not require in the minimizing case  
an energy bound since it is a direct consequence of minimality.  

 \begin{theorem}\label{main2new}
 Assume that $s\in(0,1/2)$ and that {\rm (H1)-(H2)-(H3)} hold. Let $\Omega\subset \R^n$ be a smooth and bounded open set. Given sequence $\eps_k\downarrow  0$, let $\{g_k\}_{k\in\mathbb{N}}$ 
 be as in Theorem~\ref{main1new}.   If $u_k$ is a solution of the minimization problem 
 $$\min\Big\{ \mathcal{E}_{s,\varepsilon_k}(u,\Omega) : u\in H^{s}_{g_k} (\Omega;\R^d)\cap L^p(\Omega) \Big\}\,, $$
 then $\sup_k \mathcal{E}_{s,\varepsilon_k}(u_k,\Omega)<\infty$ and Theorem \ref{main1part1} applies. In addition,  the partition $\mathfrak{E}^*=(E_1^*,\ldots,E^*_m)$ 
 solves the minimization problem
 \begin{equation}\label{limitminpbintro}
 \min\Big\{\mathscr{P}^{\boldsymbol{\sigma}}_{2s}(\mathfrak{E},\Omega) : \mathfrak{E}=(E_1,\ldots,E_m)\in\mathscr{A}_m(\Omega)\,,\;  \sum_{j=1}^m\chi_{E_j}{\bf a}_j=g\text{ a.e. in $\R^n\setminus\Omega$} \Big\}\,.
 \end{equation}
 Moreover, $u_k\to u_*$ strongly in $H^s(\Omega)$, $\mathcal{E}_{s}(u_k,\Omega) \to \mathcal{E}_{s}(u_*,\Omega)$, and $\eps_k^{-2s}W(u_k)\to 0$ in $L^1(\Omega)$. 
 \end{theorem}

The proofs of Theorems \ref{main1new} \& \ref{main2new} follow the strategy initiated in \cite{MilSirW}, and the results  concerning the identification of the limit and the various convergence estimates are completely  analogous to the scalar case. These results represent a major difference with the classical local case where such an analysis is still unavailable in its full generality (except for minimizers). Another striking difference with the local case lies in the fact that the {\sl generalized surface tension coefficients} $\sigma_{ij}$ do not depend on the potential $W$ through some geodesic problem but only on the relative Euclidean distance between the ${\bf a}_j$'s, compare with formula~\eqref{formcoeffsij}. We mention the coefficients $\sigma_{ij}$ to be generalized surface tension coefficients because  the interaction terms $\mathcal{I}_{2s}(E_i\cap\Omega,E_j\cap\Omega)$ in the definition of $\mathscr{P}^{\boldsymbol{\sigma}}_{2s}(\mathfrak{E},\Omega)$ in \eqref{defpersigintro}-\eqref{definterintro} can be interpreted as {\sl nonlocal interfacial energy} in $\Omega$  between a chamber $E_i$ and a chamber~$E_j$ (the two other terms being essentially of lower order). Pursuing the main differences with the local case, we may now comment on the various convergence estimates, the first one being the convergence in higher regularity spaces in item (i) of Theorem \ref{main1new}, and already spotted in \cite{MilSirW}. Concerning the speed of convergence  in item (viii), we must mention that  it is in fact optimal (compare with the known exponential speed of convergence in the local case). This optimality can be seen from item (vi), and the formal argument (which can be made rigourous) is as follows. Given a compact set $K\subset  \Omega\cap E^*_{j_0}$ for some index $j_0$,  we may write $u_k(x)={\bf a}_{j_0}+\varepsilon_k^{2s}w_k(x)$ for $x\in K$. Then, expanding $\nabla W$ near ${\bf a}_{j_0}$, we obtain $\nabla W(u_k(x))= \varepsilon^{2s}_{k}D^2W({\bf a}_{j_0})w_k(x)+o(\varepsilon_k^{2s})$ for $x\in K$.  Hence, item (vi) indicates that  
$$w_k\to -\sum_{j=1}^m V_{E^*_j}[D^2W({\bf a}_{j_0})]^{-1}{\bf a}_j\quad\text{on $K$}\,,$$
this limit being not zero in general. In turn, this optimal speed of convergence explains the unusual estimate on the vanishing rate of the potential term in item (v). More precisely, the estimate in item (viii) being optimal, the rate of vanishing of $W(u_k)$ away from $\partial \mathfrak{E}^*$ cannot exceed the order $\varepsilon_k^{4s}$, simply by Taylor expansion. On the other hand, 
a classical scaling analysis show that the transition between two wells across an $(n-1)$-dimensional interface should take place in layer of width $\varepsilon_k$, so that the rate of 
vanishing of $\int W(u_k)$ can neither exceed the order $\varepsilon_k$, whence the competition between this two rates. Note that in item (v), the value $\alpha=1$ is not allowed, but we strongly believe that the optimal rate is of order $\varepsilon_k^{\min(4s,1)}$. This would require a much deeper analysis, and we leave it as an open question. 
\vskip5pt

Now we would like to draw  attention on the fact that the asymptotic analysis of the vectorial fractional Allen-Cahn equation gives raise to (relatively) new geometrical objects, the nonlocal minimal partitions. 
As described above, a partition $\mathfrak{E}^*=(E^*_1,\ldots,E^*_m)\in\mathscr{A}_m(\Omega)$ is a nonlocal minimal partition in $\Omega$ for a functional $\mathscr{P}^{\boldsymbol{\sigma}}_{2s}$ if
$\delta\mathscr{P}^{\boldsymbol{\sigma}}_{2s}(\mathfrak{E}^*,\Omega)=0$. This includes in particular {\sl minimizing nonlocal partitions} in $\Omega$, i.e., those partitions for which 
$$\mathscr{P}^{\boldsymbol{\sigma}}_{2s}(\mathfrak{E}^*,\Omega) \leq \mathscr{P}^{\boldsymbol{\sigma}}_{2s}(\mathfrak{F},\Omega)  $$
for every  $\mathfrak{F}=(F_1,\ldots,F_m)\in\mathscr{A}_m(\Omega)$ such that $\overline{F_j\triangle E^*_j}\subset \Omega$ for each $j\in\{1,\ldots,m\}$. Solutions to \eqref{limitminpbintro} provide examples of such minimizing nonlocal partitions. Now, a main observation is that the geometrical functional $\mathscr{P}^{\boldsymbol{\sigma}}_{2s}$ is actually well defined for {\sl any} symmetric matrix $\boldsymbol{\sigma}$ in the set
$$\mathscr{S}_m:=\big\{\boldsymbol{\sigma}=(\sigma_{ij})\in \mathscr{M}_{m\times m}(\R): \sigma_{ii}=0\,,\;\sigma_{ij}=\sigma_{ji}>0\text{ for }i\not=j\big\}\,. $$
Even more interesting, the functional  $\mathscr{P}^{\boldsymbol{\sigma}}_{2s}$ is lower semicontinuous with respect to $L^1_{\rm loc}(\R^n)$-convergence for any $\boldsymbol{\sigma}\in\mathscr{S}_m$, simply by Fatou's lemma. In particular, no triangle inequality is required compare to the local functional $\mathscr{P}^{\bf S}_1$ (see \eqref{triangineqintro}), and the minimization problem  \eqref{limitminpbintro} always 
admits a solution for any choice of $\boldsymbol{\sigma}\in\mathscr{S}_m$ by the direct method of calculus of variations. In the standard local case, the triangle inequality does not only enter into play for existence, but also concerning the regularity theory for the boundaries of minimal partitions, at least for minimizers \cite{Leo,W}. In sharp contrast, a general regularity theory in our nonlocal case 
should not rely on any other assumptions than $\boldsymbol{\sigma}\in\mathscr{S}_m$. 
\vskip3pt

Our second main objective is to initiate such a regularity theory for nonlocal minimal partitions both in the stationary and the minimizing case, and to illustrate with some specific examples the discussion above. We shall actually not consider the general case of an arbitrary matrix in $\mathscr{S}_m$, but rather the case where the functional $\mathscr{P}^{\boldsymbol{\sigma}}_{2s}$ arises as a limit of  Allen-Cahn functionals $\mathcal{E}_{s,\varepsilon}$ as $\varepsilon\to 0$ for a suitable choice of the potential $W$. In view of Theorem \ref{main1new}, the shape of $W$ does not play a role, and we just need to require that $\boldsymbol{\sigma}$ is of the form \eqref{condlimitWintro}. This condition can be rephrased as a the classical condition of $\ell^2$-embeddedness from graph theory and optimisation \cite{DL}, i.e., we shall require that the matrix $(\sqrt{\sigma_{ij}})$ is $\ell^2$-embeddable: {\sl there exists a dimension $d\in\mathbb{N}\setminus\{0\}$ and $m$ points ${\bf a}_1,\ldots,{\bf a}_m\in\R^d$  such that $\sqrt{\sigma_{ij}}=|{\bf a}_i-{\bf a}_j|$ for every $i,j\in\{1,\ldots,m\}$}. We shall denote this class of matrices by 
$$\mathscr{S}_m^2:=\Big\{\boldsymbol{\sigma}=(\sigma_{ij})\in \mathscr{S}_m: \text{$(\sqrt{\sigma_{ij}})$ is $\ell^2$-embeddable}  \Big\} \,.$$
Note that $\mathscr{S}_m^2$ is a strict subspace of $\mathscr{S}_m$, but it still contains matrices which do not satisfy the triangle inequality \eqref{triangineqintro}. 
\vskip3pt

Our first regularity result concerns  {\sl stationary} nonlocal minimal partitions and represents the analogue of \cite[Theorem 1.4]{MilSirW} about stationary nonlocal minimal surfaces (see Section \ref{complrestprescrNMMC} for a more complete set of results).   

\begin{theorem}\label{mainthm2}
Assume that $s\in(0,1/2)$ and $\boldsymbol{\sigma}\in \mathscr{S}_m^2$. Let $\Omega\subset \R^n$ be a smooth and bounded open set. If  a partition $\mathfrak{E}^*=(E^*_1,\ldots,E^*_m)\in\mathscr{A}_m(\Omega)$ satisfies 
$\delta\mathscr{P}^{\boldsymbol{\sigma}}_{2s}(\mathfrak{E}^*,\Omega)=0$, 
then
\begin{itemize}[leftmargin=30pt]
\item[ \rm  (i)] each set $E_j^*\cap \Omega$ is (essentially) open; 
\vskip5pt

\item[ \rm  (ii)]  if $\partial\mathfrak{E}^*\cap\Omega=\bigcup_{j=1}^m\partial E_j^*\cap \Omega$ is not empty, it has a Minkowski codimension equal to 1;  
\vskip5pt

\item[ \rm  (iii)] $P_{2s^\prime}(E_*,\Omega^\prime)<\infty$ for every $s^\prime\in(0,1/2)$ and every open set $\Omega^\prime$ such that $\overline{\Omega^\prime}\subset \Omega$. 
\end{itemize}
\end{theorem}

As  already discussed in \cite[Remark 1.5]{MilSirW}, the theorem above provides a quite weak regularity statement, but item (iii) still gives some interesting indication as $s^\prime$ gets close to $1/2$. Indeed, assuming that $(1-2s^\prime)P_{2s^\prime}(E_j^*,\Omega^\prime)=O(1)$ as $s^\prime\uparrow 1/2$, then it would imply that $E_j^*$ has finite perimeter in~$\Omega^\prime$ since $(1-2s^\prime)P_{2s}(\cdot,\Omega^\prime)$ converges to the usual perimeter functional as $s^\prime\to 1/2$, see \cite{ADPM,CafVal,Dav}. Hence, such an estimate would imply that $\partial\mathfrak{E}^*\cap\Omega$ is locally $(n-1)$-rectifiable in $\Omega$. Unfortunately, our proof does not guarantee such an estimate, for the exact same reason that the exponent $\alpha=1$ in item (v) from Theorem \ref{main1new} is excluded. However, we believe that $\partial\mathfrak{E}^*\cap\Omega$ is in fact rectifiable. In the context of nonlocal minimal surfaces ($m=2$), eventually arising as limits of (scalar) Allen-Cahn critical points, this has been proved under a stability or finite Morse index assumption in \cite{CFSS,CDSV}. 
\vskip5pt

Concerning {\sl minimizing} nonlocal minimal partitions, some results are already available in some specific cases. The first result \cite{ColMag} concerns the {\sl homogeneous case} where $\sigma_{ij}=1$ for every $i\not =j$ (i.e., constant coefficients), and it is proven that $\partial\mathfrak{E}^*\cap\Omega$ is locally an hypersurface of class $C^{1,\alpha}$ away from a singular relatively closed subset $\Sigma_{\rm sing}\subset \partial\mathfrak{E}^*\cap\Omega$ of Hausdorff dimension at most $n-2$. The strategy of proof follows an argument from the founding work \cite{Alm} on minimizing partitions for the standard local functional $\mathscr{P}^{\bf S}_{1}$, dealing also with the homogeneous case ${\bf s}_{ij}=1$ for $i\not=j$. The general idea is to reduce the problem to the standard regularity theory for minimizing hypersurfaces in case of  $\mathscr{P}^{\bf S}_{1}$ (see e.g.~\cite{Maggi}), and to the regularity of nonlocal minimal surfaces  of \cite{CRS} in the nonlocal case (see also \cite{CapGui,DVV}).  The key ingredient to operate such a reduction is the so-called {\sl ``Non Infiltration Property''} (introduced in~\cite{Alm}). This property asserts that if a chamber of the partition 
has a sufficiently small proportion in volume in a ball, then the chamber does not intersect the ball of half radius. For balls centered at points of  $(\partial\mathfrak{E}^*\cap\Omega)\setminus\Sigma_{\rm sing}$ with sufficiently small radius, the smallness assumption is satisfied by $(m-2)$ chambers, so that only two chambers are present and one may apply the usual regularity theory in those balls. In the nonlocal case, the non infiltration property has been first obtained in \cite{ColMag} in the homogeneous case. The work \cite{CesNov} deals with the functional $\mathscr{P}^{\boldsymbol{\sigma}}_{2s}$ in the {\sl additive case} where $\sigma_{ij}=\alpha_i+\alpha_j$ for $i\not=j$ and some coefficients $\alpha_i,\alpha_j>0$. A non infiltration property \cite[Theorem 2.2]{CesNov} is claimed to hold with the very same argument of~\cite{ColMag}, hence leading to a partial regularity result similar to \cite{ColMag}. However, one can realize that such a non infiltration property {\it does not} hold for $m\geq 4$ in general (even in the additive case). The same pathology appears in the local case for $\mathscr{P}^{\bf S}_{1}$ where the original argument of \cite{Alm} does not apply. Under the {\sl strict} triangle inequality assumption (i.e., \eqref{triangineqintro} with strict inequality), a non infiltration property for {\sl unions} of $(m-2)$ chambers have been proved (and not for a single chamber), first in \cite{W}, and then  in \cite{Leo} in a stronger form.  For  3-partitions (i.e., $m=3$), it does not make a difference, and actually most of \cite{CesNov} concerns the case $m=3$. In particular, conical 3-partitions in $\R^2$ are fully classified in \cite{CesNov}  as ``triple junctions'', i.e., the total boundary is made three half lines meeting at the origin with angles determined through an implicit equation involving the coefficients $\alpha_i$. This shows in particular that the estimate on the Hausdorff dimension of the singular set $\Sigma_{\rm sing}$ is sharp. In addition, the local regularity of $(\partial\mathfrak{E}^*\cap\Omega)\setminus\Sigma_{\rm sing}$ has been improved from $C^{1,\alpha}$ to $C^\infty$ in \cite{CesNov} (still in the additive case), exploiting a higher regularity result from \cite{BFV} and the Euler-Lagrange equation satisfied around regular points. In our general context (and for $\boldsymbol{\sigma}\in\mathscr{S}_m$ arbitrary), we shall see that if $\partial\mathfrak{E}^*\cap D_{r}(x_0)$ is a smooth hypersurface, then there must be a pair $(h,k)$ of distinct indices such that $E^*_j\cap D_{r}(x_0)=\emptyset$ whenever $j\not\in\{h,k\}$ (so that $\partial\mathfrak{E}^*\cap D_{r}(x_0)=\partial E^*_h\cap  D_{r}(x_0)= \partial E^*_k\cap  D_{r}(x_0)$), and the stationarity condition $\delta\mathscr{P}^{\boldsymbol{\sigma}}_{2s}(\mathfrak{E}^*,\Omega)=0$ implies that  
\begin{equation}\label{ELeqpartintro}
\mathrm{H}^{(2s)}_{\partial E_h^*}(x) = f_{hk}(x)\quad \text{for $x\in \partial E_h^*\cap D_r(x_0)$}\,,
\end{equation} 
where $ \mathrm{H}^{(2s)}_{\partial E_h^*}$ is nonlocal mean curvature defined in \eqref{nonlocmeancurvdef}, and 
$$f_{h,k}(x):= \sum_{j\not\in\{h,k\}} \frac{\sigma_{hk}+\sigma_{kj}-\sigma_{hj}}{\sigma_{hk}} \int_{E^*_j}\frac{1}{|x-y|^{n+2s}}\,\de y\,,$$
see Remark \ref{RemELeqpartit}. 
\vskip3pt

Our regularity results on the minimizing nonlocal partitions always require that the matrix $\boldsymbol{\sigma}$ belongs to $\mathscr{S}_m^2$, an assumption which is 
trivially satisfied in the homogeneous case of \cite{ColMag}, or in the additive case of \cite{CesNov}. Our first theorem deals with the case where $\boldsymbol{\sigma}$ is {\sl nearly homogeneous}, i.e., the coefficients of  $\boldsymbol{\sigma}$ do not deviate too much from a constant value (with an explicit bound). In this situation, a non infiltration property for a single chamber can be proved following the strategy of \cite{ColMag}. 

\begin{theorem}\label{mainthm3}
Assume that $s\in(0,1/2)$ and that $\boldsymbol{\sigma}\in \mathscr{S}_m^2$ satisfies 
$$\frac{\max_{i\not= j}\sigma_{ij}}{\min_{i\not= j}\sigma_{ij}}<\frac{m-1}{m-2}\,.$$
Let $\Omega\subset \R^n$ be a smooth and bounded open set. If $\mathfrak{E}^*=(E^*_1,\ldots,E^*_m)\in\mathscr{A}_m(\Omega)$ satisfies 
\begin{equation}\label{minpartintrocond}
\mathscr{P}^{\boldsymbol{\sigma}}_{2s}(\mathfrak{E}^*,\Omega) \leq \mathscr{P}^{\boldsymbol{\sigma}}_{2s}(\mathfrak{F},\Omega)
\end{equation}
for every  $\mathfrak{F}=(F_1,\ldots,F_m)\in\mathscr{A}_m(\Omega)$ such that $\overline{F_j\triangle E^*_j}\subset \Omega$ for each $j\in\{1,\ldots,m\}$, then there 
exists a relatively closed subset $\Sigma_{\rm sing}\subset \partial\mathfrak{E}^*\cap\Omega$ of Hausdorff dimension at most $n-2$, and locally finite in $\Omega$ if $n=2$, such that 
$(\partial\mathfrak{E}^*\cap\Omega)\setminus \Sigma_{\rm sing}$ is locally a $C^\infty$-hypersurface. 
\end{theorem}

In the remaining cases, we focus on the tri-partition problem $m=3$, and we shall distinguish whether or not a strict triangle inequality holds. If the strict inequality triangle holds, it is easy to see that $\boldsymbol{\sigma}\in \mathscr{S}^2_3$ since this case is additive (see Lemma \ref{lemmacoeffm=3}), and we have the following result recovering \cite{CesNov}. 

\begin{theorem}\label{mainthm4}
Assume that $s\in(0,1/2)$ and that $\boldsymbol{\sigma}\in \mathscr{S}_3$ satisfies 
$$\sigma_{ij}<\sigma_{ik}+\sigma_{kj} \quad \forall i,j,k\in\{1,2,3\}\,. $$
Let $\Omega\subset \R^n$ be a smooth and bounded open set. If $\mathfrak{E}^*\in\mathscr{A}_3(\Omega)$ satisfies \eqref{minpartintrocond}, then the conclusion of Theorem \ref{mainthm3} holds. 
\end{theorem}

Most surprisingly, it appears that we were able to prove the same regularity in case the triangle inequality fails, but it must fail strictly, at least for $s$ close enough to $1/2$. 

\begin{theorem}\label{mainthm5}
Assume that $s\in(0,1/2)$ and that $\boldsymbol{\sigma}\in \mathscr{S}^2_3$ satisfies 
\begin{equation}\label{invstrictineqintro}
\sigma_{i_0j_0}>\sigma_{i_0k_0}+\sigma_{k_0j_0}\quad\text{for some $\{i_0,j_0,k_0\}=\{1,2,3\}$}\,.
\end{equation}
Let $\Omega\subset \R^n$ be a smooth and bounded open set. There exists $s_*=s_*(\boldsymbol{\sigma},n)\in(0,1/2)$ such that for $s\in(s_*,1/2)$, 
if $\mathfrak{E}^*=(E^*_1,E^*_2,E^*_3)\in\mathscr{A}_3(\Omega)$ satisfies \eqref{minpartintrocond}, then the conclusion of Theorem \ref{mainthm3} holds. In addition, $\partial E^*_{i_0}\cap \partial E^*_{j_0}\cap\Omega\subset \Sigma_{\rm sing}$. In particular, $\partial E^*_{i_0}\cap \partial E^*_{j_0}\cap\Omega$ has Hausdorff dimension at most $n-2$, and it is locally finite in $\Omega$ for $n=2$. 
\end{theorem}

The main feature of condition \eqref{invstrictineqintro} is that the non infiltration property does not hold for all the chambers of the partition. More precisely, it fails for the chamber $E^*_{k_0}$, but holds for the remaining ones.  This turns out to be sufficient to obtain partial regularity combining a stratification argument of the singular set  together with a classification of tangent cones showing that an hyperplane separating a chamber of index $i_0$ from a chamber $j_0$ is not minimizing, at least for $s$ close enough to~$1/2$.  This latter fact rests on the known asymptotic analysis 
of the (normalized) fractional perimeter $(1-2s)P_{2s}$ which converges as $s\to 1/2$ to a multiple of the standard perimeter functional \cite{ADPM,CafVal,Dav}. For $\mathscr{P}^{\boldsymbol{\sigma}}_{1}$, it is elementary to prove that an hyperplane separating a chamber $i_0$ from a chamber $j_0$ is not minimizing if \eqref{invstrictineqintro} holds, and it crucially needs the inequality to be strict. As far as we know, there is no regularity theory for minimizers of $\mathscr{P}^{\boldsymbol{\sigma}}_{1}$ for a matrix $\boldsymbol{\sigma}$ which does not satisfy the triangle inequality, probably because existence is not even ensured through the direct method.  The existing regularity theory for minimizers of $\mathscr{P}^{\boldsymbol{\sigma}}_{1}$ does not even cover the case where the triangle inequality is allowed to be large (in full generality).  In this direction, the only contribution we are aware of  is the recent work \cite{Nov} in dimension $n=2$ in the additive case where the total boundary can present singular points with the geometry of a cusp. This suggests that the strict inequality in \eqref{invstrictineqintro} is truly necessary, at least for the method of proof. Finally, we would like to mention that the Hausdorff  dimension estimate on $\partial E^*_{i_0}\cap \partial E^*_{j_0}$ in dimension $n=2$ implies that there are no ``triple junctions"  under \eqref{invstrictineqintro} in comparison to \cite{CesNov}. Still in dimension $n=2$, we actually believe that  $\partial E^*_{i_0}\cap \partial E^*_{j_0}\cap\Omega=\emptyset$, which would imply that $\Sigma_{\rm sing}=\emptyset$ by the full regularity of planar  (minimizing) nonlocal minimal surfaces proved in  \cite{SV1bis}. This remains an open problem. 

\vskip5pt

As many works on equations involving the fractional Laplacion $(-\Delta)^s$, our approach relies on the extension of functions defined on $\R^n$ to the open upper half space $\R^{n+1}_+:=\R^n\times (0,\infty)$ known as {\sl ``Caffarelli-Silvestre extension''}~\cite{CaffSil} (see also \cite{MolOs}). By means of this extension, $(-\Delta)^s$ can be realized as the Dirichlet-to-Neumann operator induced by  the degenerate elliptic operator $L_s:=-{\rm div}(z^{1-2s}\nabla\cdot)$  on $\R^{n+1}_+$, where $z\in(0,\infty)$ denotes the extension variable. In this way,  one can rephrases the fractional Allen-Cahn equation or the nonlocal minimal partition equation as 
$L_s$-harmonic functions in $\R^{n+1}_+$ satisfying a nonlinear boundary condition. Exactly as in \cite{MilSirW}, this extension provides a fundamental monotonicity formula. All the preliminary material concerning functional spaces, the operator  $(-\Delta)^s$ and the Caffarelli-Silvestre extension are presented in Section \ref{prelim}.  
In Section \ref{FractAC}, we first provide some general regularity estimates on solutions to the fractional Allen-Cahn equation and $L_s$-harmonic functions with Allen-Cahn degenerate boundary reaction, and then prove some of  the main ingredients necessary to the asymptotic analysis $\varepsilon\to 0$ such as a fondamental {\sl clearing out} lemma.  A first part of the asymptotic analysis as $\eps\to0$ is performed in Section \ref{FGLasymp}. 
Section \ref{GHSNMC} is devoted to the analysis of nonlocal minimal partitions where we prove Theorem \ref{mainthm2}.  Based on the different results from Section \ref{GHSNMC}, we complete  in Section \ref{imprsect} the proof of Theorem \ref{main1new} (and Theorem \ref{main2new}). 
Finally, the partial regularity for minimizing nonlocal partitions is addressed in Section \ref{regminpartsect}.

\subsection*{{Notation}}
We write $\R^{n+1}_+:=\R^n\times(0,\infty)$. For simplicity, $\R^n$ is identified with $\partial  \mathbb{R}^{n+1}_+=\R^n\times\{0\}$. More generally, sets $A\subset\mathbb{R}^n$ are identified with $A\times\{0\}\subset\partial  \mathbb{R}^{n+1}_+$. Points in $\mathbb{R}^{n+1}$ are then written $\mathbf{x}=(x,z)$ with $x\in\mathbb{R}^n$ and $z\in\mathbb{R}$, and we shall often identify $(x,0)\in\R^{n+1}$ with $x\in\R^n$.  
To avoid confusion, we denote by $B_r(\mathbf{x})$ the open ball in $\mathbb{R}^{n+1}$ of radius $r$ centered at $\mathbf{x}=(x,z)$, while $D_r(x):= B_r(\mathbf{x})\cap\mathbb{R}^{n}$ is the open ball (or disc) in $\R^n$ centered at ${\bf x}=(x,0)$. For an arbitrary set $G\subset  \mathbb{R}^{n+1}$, we set  
$$G^+:=G\cap \mathbb{R}^{n+1}_+\quad\text{ and }\quad\partial^+ G:=\partial G\cap \mathbb{R}^{n+1}_+\,.$$
For a bounded open set $G\subset\R^{n+1}_+$, we shall say that $G$ is  {\bf admissible} whenever 
\begin{itemize}
\item $\partial G$ is Lipschitz regular;  
\vskip2pt
\item the (relative) open set $\partial^0G\subset\R^n$ defined by 
$$\partial^0G:=\Big\{\mathbf{x}\in\partial G\cap\partial\R^{n+1}_+ : B^+_{r}(\mathbf{x})\subset G \text{ for some $r>0$}\Big \}\,,$$
is non empty and has a Lipschitz regular boundary; 
\vskip2pt

\item $\partial G=\partial^+ G\cup\overline{\partial^0G}\,$.
\end{itemize}

Finally, we often denote by~$C$ a generic positive constant which may only depend on the dimension~$n$, and possibly changing from line to line. If a constant depends on additional given parameters, we may write those parameters using the subscript notation.

   								 
\section{Preliminaries} \label{prelim}   
								 

\subsection{$H^{s}$-spaces and the fractional Laplacian}\label{secHs}
 
For an open set $\Omega\subset \mathbb{R}^n$,  the fractional Sobolev space $H^{s}(\Omega;\R^d)$ is made of functions $u\in L^2(\Omega;\R^d)$ such that\footnote{The normalization constant $\gamma_{n,s}$ is chosen in such a way that $\displaystyle [v]^2_{H^{s}(\R^n)}=\int_{\R^n}(2\pi|\xi|)^{2s}|\widehat v|^2\,\de\xi\,$. }  
$$[u]^2_{H^{s}(\Omega)}:=\frac{\gamma_{n,s}}{2}\iint_{\Omega\times \Omega} \frac{|u(x)-u(y)|^2}{|x-y|^{n+2s}}\,\de x\de y<\infty\,,\quad \gamma_{n,s}:=s\,2^{2s}\pi^{-\frac{n}{2}}\frac{\Gamma\big(\frac{n+2s}{2}\big)}{\Gamma(1-s)} \,.$$ 
It is a separable Hilbert space normed by $\|\cdot\|^2_{H^{s}(\Omega)}:= \|\cdot\|^2_{L^2(\Omega)}+[\cdot]^2_{H^{s}(\Omega)}$. 
The space $ H^{s}_{\rm loc}(\Omega;\R^d)$ denotes the class of functions whose restriction to any relatively compact  open subset $\Omega'$ of $\Omega$ belongs to $H^{s}(\Omega';\R^d)$. 
The linear subspace $H^{s}_{00}(\Omega;\R^d) \subset H^{s}(\mathbb{R}^n;\R^d)$ is in turn defined by 
$$H^{s}_{00}(\Omega;\R^d):=\big\{u\in H^{s}(\mathbb{R}^n;\R^d) :  u=0 \text{ a.e. in } \R^n\setminus\Omega\big\}\,. $$
Endowed with the  induced norm,   
 $H^{s}_{00}(\Omega;\R^d)$ is also an Hilbert space, and 
$$[u]^2_{H^{s}(\mathbb{R}^n)}=2\mathcal{E}_{s}(u,\Omega) 
=[u]^2_{H^{s}(\Omega)} + \int_{\Omega} \boldsymbol{\rho}_\Omega(x) |u(x)|^2\,\de x\quad\forall u\in H^{s}_{00}(\Omega;\R^d)\,,$$
where $\mathcal{E}_{s}(\cdot,\Omega)$ is  defined in \eqref{defenergE}, and 
$ \boldsymbol{\rho}_\Omega(x) := \gamma_{n,s}\int_{\R^n\setminus \Omega}|x-y|^{-n-2s}\,\de y$. 
Since $s\in(0,1/2)$, if $\Omega$ is  bounded and has a Lipschitz boundary, then 
$ \int_{\Omega} \boldsymbol{\rho}_\Omega(x) |u(x)|^2\,\de x\leq C_\Omega \|u\|^2_{H^{s}(\Omega)}$  for all functions $u\in H^s(\Omega;\R^d)$, and for  
 a constant $C_\Omega=C_\Omega(s)>0$. As a consequence, if $u\in H^s(\Omega;\R^d)$ and $\widetilde u$ denotes the extension of $u$ by zero  outside $\Omega$, then  
$ \|u\|_{H^{s}(\Omega)}\leq \|\widetilde u\|_{H^{s}(\R^n)}\leq (C_\Omega+1)^{\frac{1}{2}}\|u\|_{H^{s}(\Omega)}$. 
In particular, if $\partial \Omega$ is smooth enough, then $H^{s}_{00}(\Omega)=\big\{\widetilde u : u\in  H^s(\Omega;\R^d)\big\}$ 
(see \cite[Corollary~1.4.4.5]{G}), and  (see \cite[Theorem~1.4.2.2]{G}) 
\begin{equation}\label{densitysmoothH1/200}
H^{s}_{00}(\Omega;\R^d)= \overline{\mathscr{D}(\Omega;\R^d)}^{\,\|\cdot\|_{H^{s}(\mathbb{R}^n)}}\,.
 \end{equation}
The topological dual  of the vector space space $H^{s}_{00}(\Omega;\R^d)$ is denoted by $H^{-s}(\Omega;\R^d)$.
\vskip5pt

For our purposes and following \cite{MilSirW}, we have to consider the further class of functions 
$$\widehat{H}^{s}(\Omega;\R^d):=\Big\{u\in L^{2}_{\rm loc}(\mathbb{R}^n;\R^d) : \mathcal{E}_s(u,\Omega)<\infty\Big\} \,.$$
It is a Hilbert space for the scalar product induced by the norm $u\mapsto \big(\|u\|^2_{L^2(\Omega)}+ \mathcal{E}_s(u,\Omega)\big)^{1/2}$. 
Finally, we set for $g\in \widehat{H}^{s}(\Omega;\R^d)$, 
$$H^{s}_g(\Omega;\R^d):= \left\{u\in \widehat{H}^{s}(\Omega;\mathbb{R}^d) : u=g\, \text{ a.e. in $\R^n\setminus\Omega$}\right\}=g+H^{s}_{00}(\Omega;\R^d)\,,$$ 
and we refer to \cite{MilPegSch,MilSirW} for further details about the space $\widehat{H}^{s}(\Omega;\R^d)$. 

\vskip5pt

In a bounded open set $\Omega\subset \mathbb{R}^n$,  the fractional  Laplacian is defined as the continuous linear operator $(-\Delta)^s: \widehat{H}^{s}(\Omega;\R^d)\to (\widehat{H}^{s}(\Omega;\R^d))^\prime$ induced by the quadratic form $\mathcal{E}_s(\cdot,\Omega)$. In other words,  for a function $u\in \widehat{H}^{s}(\Omega;\R^d)$, we define its {\it distributional fractional Laplacian} $ (-\Delta)^{s} u$ through 
its action on $ \widehat{H}^{s}(\Omega;\R^d)$ by setting 
$$\big\langle  (-\Delta)^{s} u, \varphi\big\rangle_\Omega:=\frac{\gamma_{n,s}}{2}\iint_{(\R^n\times\R^n)\setminus(\Omega^c\times\Omega^c)}  \frac{\big(u(x)-u(y)\big)\cdot\big(\varphi(x)-\varphi(y)\big)}{|x-y|^{n+2s}}\,\de x\de y\,.$$
In this way, $(-\Delta)^{s} u$ appears to be the first outer variation of  the quadratic form $\mathcal{E}_s(\cdot,\Omega)$ at $u$ with respect to perturbations supported in $\Omega$, i.e.,  
$$\big\langle  (-\Delta)^{s} u, \varphi\big\rangle_\Omega=\left[\frac{\de}{\de t} \mathcal E_s (u+t \varphi , \Omega)\right]_{t=0} $$
for all $\varphi \in H^{s}_{00}(\Omega)$. If $u$ is a smooth bounded function, then the distribution $ (-\Delta)^{s} u$ can be rewritten as a pointwise defined function which coincides with the one given by  formula~\eqref{formpvfraclap}. We refer again to \cite{MilPegSch,MilSirW} for further details.

\vskip5pt

We complete this subsection recalling a useful  Poincar\'e type inequality. The proof is fairly standard, but we provide it for completeness. 

\begin{lemma}\label{poincareHs}
There exists a constant $\boldsymbol{\lambda}_{n,s}>0$ depending only on $n$ and $s$ such that for every $r>0$ and every $u\in H^s(D_r)$, 
\begin{equation}\label{poincfracL1}
\big\|u-[u]_{r}\big\|_{L^1(D_r)}\leq \boldsymbol{\lambda}_{n,s}r^{\frac{n+2s}{2}}[u]_{H^s(D_r)}\,, 
\end{equation}
where $[u]_{r}$ denotes the average of $u$ over $D_r$. 
\end{lemma}

\begin{proof}
By H\"older's inequality, \eqref{poincfracL1} holds if we prove that there exists a constant $\widetilde{\boldsymbol{\lambda}}_{n,s}>0$ depending only on $n$ and $s$ such that for every $r>0$ and every $u\in H^s(D_r)$,
 \begin{equation}\label{poincfracL2}
 \big\|u-[u]_{r}\big\|_{L^2(D_r)}\leq \widetilde{\boldsymbol{\lambda}}_{n,s}r^{s}[u]_{H^s(D_r)}\,,
 \end{equation}
 By rescaling, it is enough to prove \eqref{poincfracL2} for $r=1$, and replacing $u$ by $u$ minus its average, it is enough to consider the case  $[u]_{1}=0$. 
We argue by contradiction assuming that there exists a sequence $\{u_n\}_{n\in\mathbb{N}}\subset H^s(D_1)$  such that $[u_n]_1=0$, $\|u_n\|_{L^2(D_1)}=1$, and $[u_n]_{H^s(D_1)}\to 0$ as $n\to\infty$. Then  $\{u_n\}_{n\in\mathbb{N}}$ is bounded in $H^s(D_1)$. Hence there exists a (non relabeled) subsequence and $u_*\in H^s(D_1)$ such that $u_n\rightharpoonup u_*$ weakly in $H^s(D_1)$. By the compact embedding $H^s(D_1)\hookrightarrow L^2(D_1)$, $u_n\to u_*$ strongly in $L^2(D_1)$ (and thus in $L^1(D_1)$), so that $\|u_*\|_{L^2(D_1)}=1$ and $[u_*]_1=0$. On the other hand, by lower semi-continuity of the $H^s$-seminorm, $[u_*]_{H^s(D_1)}\leq \lim_n [u_n]_{H^s(D_1)}=0$. Therefore, $u_*$ is constant, and since $[u_*]_1=0$, we have $u_*=0$ which contradicts $\|u_*\|_{L^2(D_1)}=1$. 
\end{proof}


 \subsection{Weighted spaces, fractional Poisson kernel, and the Dirichlet-to-Neumann operator}\label{secext}


For an open set $G\subset \R^{n+1}$, we define the weighted $L^2$-space  
$$L^2(G;\R^d,|z|^a\de \mathbf{x}):= \Big\{v\in L^1_{\rm loc}(G;\R^d) :  |z|^{\frac{a}{2}} v \in L^2(G;\R^d)\Big\} \quad \text{with }a:=1-2s\,,$$
normed by $\|v\|^2_{L^2(G,|z|^a\de \mathbf{x})}:=\int_G |z|^{a}|v|^2\,\de \mathbf{x}$. 
Accordingly, we introduce the weighted Sobolev space 
$$H^1(G;\R^d,|z|^a\de \mathbf{x}):= \Big\{v\in L^2(G;\R^d,|z|^a\de \mathbf{x}) : \nabla v \in L^2(G;\R^d,|z|^a\de \mathbf{x})\Big\} \,,$$ 
normed by $\|u\|_{H^1(G,|z|^a\de \mathbf{x})}:=\|u\|_{L^2(G,|z|^a\de \mathbf{x})}+\|\nabla u\|_{L^2(G,|z|^a\de \mathbf{x})}$. 
Both spaces $L^2(G;\R^d,|z|^a\de \mathbf{x})$ and $H^1(G;\R^d,|z|^a\de \mathbf{x})$ are separable Hilbert spaces when equipped with the scalar product induced by their respective Hilbertian norm. Throughout our analysis, we shall repeatedly use that, in case $G$ is a bounded admissible open set, 
\begin{equation}\label{compactembL1trace}
v\in  H^1(G;\R^d,|z|^a\de \mathbf{x})\mapsto v_{|\partial^0G}\in L^1(\partial^0 G;\R^d) \text{ is a compact linear operator}\,,
\end{equation}
see e.g. \cite[Remark 2.4]{MilSirW}.
\vskip3pt 

On $H^1(G;\R^d,|z|^a\de \mathbf{x})$, we define {\sl the weighted Dirichlet energy} $\mathbf{E}_s(\cdot,G)$ by setting
\begin{equation}\label{defEbold}
\mathbf{E}_s(v,G):=\frac{\boldsymbol{\delta}_s}{2}\int_G|z|^a|\nabla v|^2\,\de {\bf x}\quad\text{with } \boldsymbol{\delta}_s:=2^{2s-1}\frac{\Gamma(s)}{\Gamma(1-s)}\,,
\end{equation}
where the normalisation constant $\boldsymbol{\delta}_s>0$ is carefully chosen to relate $\mathbf{E}_s$ to the fractional energy $\mathcal{E}_s$ (through identity \eqref{identenergnonlocloc} below).  For a map  $v\in H^1(G;\R^d,|z|^a\de \mathbf{x})$, the trace of $v$ on $\partial^0G$ is well defined and belongs to $H^s_{\rm loc}(\partial^0G;\R^d)$ (see e.g. \cite[Lemma~2.8]{MilPegSch}). A converse property also holds for maps defined on $\R^n\simeq \partial \R^{n+1}_+$ when suitably extended. The appropriate extension is given by convolution with the so-called {\sl fractional Poisson kernel}. 

Introduced in  \cite{CaffSil} (see also \cite{MolOs}), the fractional Poisson kernel is the  function $\mathbf{P}_{n,s}:\R^{n+1}_+\to [0,\infty)$ defined by 
$$\mathbf{P}_{n,s}(\mathbf{x}):=\sigma_{n,s}\,\frac{z^{2s}}{|\mathbf{x}|^{n+2s}}\,, \qquad \sigma_{n,s}:=\pi^{-\frac{n}{2}}\frac{\Gamma(\frac{n+2s}{2})}{\Gamma(s)}\,,$$
where $\mathbf{x}:=(x,z)\in\mathbb{R}^{n+1}_+:=\R^n\times(0,\infty)$. The choice of the constant $\sigma_{n,s}$ is made in such a way that $\mathbf{P}_{n,s}(\cdot,z)$ is a probability measure for every $z>0$.  The function $\mathbf{P}_{n,s}$ solves 
$$\begin{cases}
{\rm div}(z^{a}\nabla \mathbf{P}_{n,s})= 0 & \text{in $\R^{n+1}_+$}\,,\\
\mathbf{P}_{n,s}=\delta_0 & \text{on $\partial\R^{n+1}_+$}\,,
\end{cases}$$
where $\delta_0$ is the Dirac distribution at the origin. 
From now on, for a measurable function $u$ defined over $\mathbb{R}^n$, we shall denote by 
$u^\e$ its extension to the half-space $\mathbb{R}^{n+1}_+$ given by the convolution (in the $x$-variables) of $u$ with the 
 fractional Poisson kernel $\mathbf{P}_{n,s}$,  i.e., 
\begin{equation}\label{poisson}
u^\e(x,z):=\mathbf{P}_{n,s}*u(x,z)= \sigma_{n,s}\int_{\mathbb{R}^n}\frac{z^{2s} u(y)}{(|x-y|^2+z^2)^{\frac{n+2s}{2}}}\,\de y\,.
\end{equation}

The extension $u^\e$ is well defined whenever $u$ belongs to the Lebesgue space $L^1$ over $\R^n$ with respect to the probability measure 
$\sigma_{n,s}(1+|y|^2)^{-(n+2s)/2}\,\de y$.   
In particular,  $u^\e$ can be defined  whenever $u\in \widehat{H}^{s}(\Omega;\R^d)$ for some bounded  open set $\Omega\subset\R^n$ (see e.g. \cite[Lemma 2.1]{MilPegSch}). 
For  such an admissible function $u$, the extension  $u^\e$ has a pointwise trace on $\partial\R^{n+1}_+=\R^n$ which is equal to $u$ at every Lebesgue point. 
 In addition, $u^\e$ solves the equation 
$$\begin{cases} 
{\rm div}(z^{a}\nabla u^\e) = 0 & \text{in $\mathbb{R}_+^{n+1}$}\,,\\
u^\e = u  & \text{on $\partial\R^{n+1}_+$}\,. 
\end{cases}$$
It has been proved in \cite{CaffSil} that $u^\e$  belongs to $H^1(\mathbb{R}_+^{n+1};\R^d,|z|^a\de\mathbf{x})$ whenever 
$u\in H^{s}(\mathbb{R}^n;\R^d)$, and  
\begin{equation}\label{identenergnonlocloc}
[u]^2_{H^{s}(\mathbb{R}^n)}=2\boldsymbol{\delta}_s\mathbf{E}_s(u^\e,\R^{n+1}_+)\,.
\end{equation}  
In addition, $u^\e$ provides the minimal $\mathbf{E}_s$-energy over all possible extensions of $u$ to  $\mathbb{R}_+^{n+1}$. 

For an open set $\Omega\subset\R^n$ and an admissible bounded  open set $G\subset\R^{n+1}_+$ satisfying $\overline{\partial^0G}\subset \Omega$, the extension $u\mapsto u^\e$ defines a continuous linear operator from $\widehat{H}^{s}(\Omega;\R^d)$ into $H^1(G;\R^d,|z|^a\de \mathbf{x})$ (see \cite[Corollary 2.10]{MilPegSch}). 

If $\Omega\subset\R^n$ is a bounded domain with Lipschitz boundary and $u\in \widehat H^s(\Omega;\R^d)$, 
the divergence free vector field $z^{a}\nabla u^\e$ admits a distributional normal trace on $\Omega$, that we denote by $\mathbf{\Lambda}^{(2s)}u$.  
More precisely, $\mathbf{\Lambda}^{(2s)} u$ is defined through its action on a test function $\varphi\in \mathscr{D}(\Omega;\R^d)$ by 
$$\left\langle \mathbf{\Lambda}^{(2s)} u, \varphi\right\rangle_\Omega := \int_{\mathbb{R}^{n+1}_+}z^{a}\nabla u^\e\cdot\nabla\Phi\,\de \mathbf{x}\,,$$
where $\Phi$ is any smooth extension of $\varphi$ compactly supported in  $\mathbb{R}_+^{n+1}\cup\Omega$. 
By the divergence theorem, 
this integral
 does not depend on the choice of the extension $\Phi$, and the linear operator
  $\mathbf{\Lambda}^{(2s)}:\widehat H^{s}(\Omega;\R^d)\to H^{-s}(\Omega;\R^d)$ is  continuous. 
  Moreover, whenever $u$ is smooth, the distribution $\mathbf{\Lambda}^{(2s)}u$ coincides the pointwise defined function 
$$\mathbf{\Lambda}^{(2s)} u(x)=-\lim_{z\downarrow0}z^{a}\partial_z u^\e(x,z)=2s\, \lim_{z\downarrow0} \frac{u^\e(x,0)-u^\e(x,z)}{z^{2s}}\,.$$  
It has been proved in \cite{CaffSil} that $\mathbf{\Lambda}^{(2s)}$ coincides with $ (-\Delta)^{s} $, up to a constant multiplicative factor. In our  setting, this identity still holds, see  \cite[Lemma 2.12]{MilSirW} and \cite[Lemma 2.9]{MilSir}, that is 
$$ (-\Delta)^{s} \equiv \boldsymbol{\delta}_s \mathbf{\Lambda}^{(2s)} \text{ on $\widehat H^{s}(\Omega;\R^d)$}\,.$$
Once again, we refer to  \cite{MilPegSch,MilSirW} for further details. 
\vskip3pt

Lastly, we recall from  \cite{MilPegSch,MilSirW} that the extension $u^\e$ of a map $u\in \widehat H^{s}(\Omega;\R^d)$ provides a way to compute the first inner variation of $\mathcal{E}_s$ at $u$ in terms of the first inner variation of $\mathbf{E}_s$ at $u^\e$. Given a vector field $X\in C^1(\R^n;\R^n)$ with compact support in $\Omega$, the first variation $\delta\mathcal{E}_s(u,\Omega)$ evaluated at $X$ is defined by means of the integral flow $\{\phi_t\}_{t\in\R}$ generated by $X$. More precisely, setting $t\mapsto \phi_t(x)$ to be the integral flow of $X$, i.e., the unique solution of the system of ODEs
\begin{equation}\label{defintegflow}
\begin{cases}
\displaystyle\frac{\de}{\de t} \phi_t(x)=  X\big(\phi_t(x)\big) \,,\\[5pt]
\phi_0(x)=x\,,
\end{cases}
\end{equation}
the first variation $\delta\mathcal{E}_s(u,\Omega)$ is defined by 
$$\delta\mathcal{E}_s(u,\Omega)[X] :=\left[ \frac{\de}{\de t} \,\mathcal{E}_s(u\circ\phi_{-t},\Omega)\right]_{t=0} \,.$$ 
In turn, for an admissible bounded open set $G\subset\R^{n+1}_+$ such that $\overline{\partial^0G}\subset\Omega$, the first inner variation  $\delta\mathbf{E}_s(\cdot,G)$ {\sl up to} $\partial^0G$ at a map $v\in H^1(G;\R^d,|z|^a\de \mathbf{x})$ is computed as follows. For a vector field ${\bf X}=(X,{\bf X}_{n+1})\in C^1(\overline G;\R^{n+1})$ compactly supported in $G\cup\partial^0G$ and satisfying  ${\bf X}_{n+1}=0$ on $\partial^0G$, we consider the integral flow $\{\boldsymbol{\Phi}_t\}_{t\in\R}$ generated by ${\bf X}$, i.e., the unique solution of 
$$
\begin{cases}
\displaystyle\frac{\de}{\de t} \boldsymbol{\Phi}_t(x)=  {\bf X}\big(\boldsymbol{\Phi}(x)\big) \,,\\[5pt]
\boldsymbol{\Phi}_0(x)=x\,. 
\end{cases}
$$
It satisfies $\boldsymbol{\Phi}_t=(\phi_t,0)$  on $\partial^0G$, and 
\begin{equation}\label{firstvarEsdef}
\delta\mathbf{E}_s(v,G)[{\bf X}] :=\left[ \frac{\de}{\de t} \,\mathbf{E}_s(v\circ\boldsymbol{\Phi}_{-t},G)\right]_{t=0} \,.
\end{equation}
Standard computations (see e.g. \cite[Chapter 2.2]{Sim}) show that 
\begin{multline}\label{firstvarEsformul}
\delta\mathbf{E}_s(v,G)[{\bf X}] =\frac{\boldsymbol{\delta}_s}{2}\int_Gz^a\bigg(|\nabla v|^2{\rm div}{\bf X} -2\sum_{i,j=1}^{n+1}(\partial_iv\cdot\partial_jv)\partial_j{\bf X}_i\bigg)\,\de {\bf x} \\
+\frac{\boldsymbol{\delta}_s a}{2}\int_G z^{a-1}|\nabla v|^2{\bf X}_{n+1}\,\de {\bf x}\,.
\end{multline}
It has been proved in \cite[Corollary 2.14]{MilSirW} that 
\begin{equation}\label{identfirstvarfracext}
\delta\mathcal{E}_s(u,\Omega)[X] = \delta\mathbf{E}_s(u^\e,G)[{\bf X}]
\end{equation}
for every ${\bf X}=(X,{\bf X}_{n+1})\in C^1(\overline G;\R^{n+1})$ compactly supported in $G\cup\partial^0G$ and satisfying  ${\bf X}_{n+1}=0$ on $\partial^0G$.

															 
\section{The vectorial fractional Allen-Cahn equation}\label{FractAC}    
															 

We consider throughout in this section a bounded open set $\Omega\subset \mathbb{R}^n$ with smooth boundary (at least Lipschitz regular). We are interested  
in weak solutions  $u_\varepsilon\in \widehat H^{s}(\Omega;\R^d)\cap L^p(\Omega)$ of the vectorial fractional Allen-Cahn  equation 
\begin{equation}\label{eqfractGL}
 (-\Delta)^{s} u_\varepsilon+\frac{1}{\varepsilon^{2s}}\nabla W(u_\varepsilon) =0\quad\text{in $\Omega$}\,.
\end{equation}
The notion of weak solution is understood in the duality sense according to the weak formulation of the fractional Laplacian, i.e., 
\begin{equation}\label{weakformeq}
\big\langle  (-\Delta)^{s} u_\varepsilon, \varphi\big\rangle_{\Omega} + \frac{1}{\varepsilon^{2s}}\int_{\Omega}\nabla W(u_\varepsilon)\cdot \varphi\,\de x=0 \qquad 
\forall \varphi\in H^{s}_{00}(\Omega;\R^d)\cap L^p(\Omega)\,.
\end{equation}
Such solutions correspond  to  critical points of 
the  fractional Allen-Cahn energy $\mathcal{E}_{s,\varepsilon}(\cdot,\Omega)$ defined  in~\eqref{defFGLenerg}. 
In other words, we are interested in maps $u_\varepsilon\in \widehat H^{s}(\Omega;\R^d)\cap L^p(\Omega)$ satisfying  
\begin{equation}\label{ptcritAC}
\left[\frac{\de}{\de t} \mathcal{E}_{s,\varepsilon}(u_\varepsilon+t\varphi,\Omega) \right]_{t=0} =0\qquad 
\forall \varphi\in H^{s}_{00}(\Omega;\R^d)\cap L^p(\Omega)\,.
\end{equation}
Note that for $u\in L^p(\Omega;\R^d)$, we have $\nabla W(u)\in L^{p^\prime}(\Omega;\R^d)$ with $1/p+1/p^\prime=1$ by our assumption $(H3)$ on the potential $W$, and thus  \eqref{weakformeq} and \eqref{ptcritAC} are well defined. 
\vskip3pt

Among the whole class of critical points of $\mathcal{E}_{s,\varepsilon}(\cdot,\Omega)$, there are the minimizing maps. 

\begin{definition}\label{defminimizerAC}
A map $u_\varepsilon\in \widehat H^{s}(\Omega;\R^d)\cap L^p(\Omega)$ is  said to be $\mathcal{E}_{s,\varepsilon}$-minimizing  in $\Omega$ if 
$$\mathcal{E}_{s,\varepsilon}(u_\eps,\Omega)\leq \mathcal{E}_{s,\varepsilon}(u,\Omega) $$
for all $u\in \widehat H^{s}(\Omega;\R^d)\cap L^p(\Omega)$ such that $u-u_\eps$ is compactly supported in $\Omega$. 
\end{definition}

Note that a map which is $\mathcal{E}_{s,\varepsilon}$-minimizing  in $\Omega$ is indeed a solution of \eqref{weakformeq} since the set of test functions $\varphi$ which are compactly supported in $\Omega$ are dense in $H^{s}_{00}(\Omega;\R^d)\cap L^p(\Omega)$ by \eqref{densitysmoothH1/200}. Minimizing maps arise for instance as solutions of the minimization problem 
\begin{equation}\label{GLminProb}
\min\Big\{ \mathcal{E}_{s,\varepsilon}(u,\Omega) : u\in H^{s}_{g} (\Omega;\R^d)\cap L^p(\Omega) \Big\}\,, 
\end{equation}
for a given exterior Dirichlet condition $g\in \widehat H^{s}(\Omega;\R^d)\cap L^p(\Omega)$. The resolution of \eqref{GLminProb} follows with no difficulties from the Direct Method of Calculus of Variations and provides  particular examples of weak solutions to \eqref{eqfractGL}.


\subsection{Degenerate Allen-Cahn boundary reactions}\label{subsectDegAllCahneq}


To obtain a priori estimates and energy identities on weak solutions of \eqref{eqfractGL}, we rely on the fractional harmonic extension to $\R^{n+1}_+$ 
introduced in Section~\ref{prelim}.  
 If $u_\varepsilon\in \widehat H^{s}(\Omega;\R^d)\cap L^p(\Omega)$ is a weak solution of \eqref{eqfractGL}, then its 
 extension $u^\e_\varepsilon$ given by \eqref{poisson}  satisfies 
\begin{equation}\label{varformbdgleq}
\boldsymbol{\delta}_s \int_{\R^{n+1}_+} z^{a}\, \nabla u^\e_\varepsilon \cdot\nabla\phi\,\de \mathbf{x}+ \frac{1}{\varepsilon^{2s}}\int_\Omega \nabla W(u^\e_\varepsilon)\cdot\phi\,\de x=0 
\end{equation}
for every smooth function $\phi :\overline{\mathbb{R}^{n+1}_+}\to\mathbb{R}^d$  compactly supported in $\R^{n+1}_+\cup\Omega$, or equivalently (by density, see \cite[Remark 2.6]{MilSirW}),   
for every $\phi\in H^1(\mathbb{R}_+^{n+1};\R^d,|z|^a\de\mathbf{x})\cap L^p(\Omega)$  compactly supported in $\R^{n+1}_+\cup\Omega$. 
In particular, given an admissible bounded open set $G\subset \R^{n+1}_+$ such that $\overline{\partial^0G}\subset\Omega$, the extension $u_\eps^\e$  satisfies  \eqref{varformbdgleq} 
 for every test function $\phi \in H^1(G;\R^d,|z|^a\de \mathbf{x})\cap L^p(\partial^0G)$ compactly supported in $G\cup\partial^0G$. In other words,   $u^\e_\varepsilon$ is a critical point of the functional $\mathbf{E}_{s,\eps}(\cdot,G)$ defined on the weighted space 
 $ H^1(G;\R^d,|z|^a\de \mathbf{x})\cap L^p(\partial^0G)$ by 
\begin{equation}\label{defEbfeps}
\mathbf{E}_{s,\varepsilon}(v,G):=\mathbf{E}_s(v,G)+ \frac{1}{\varepsilon^{2s}}\int_{\partial^0G} W(v)\,\de x\,, 
\end{equation}
where ${\bf E}_s(\cdot,G)$ is {\it the weighted Dirichlet energy} defined in \eqref{defEbold}.
\vskip3pt

If a function $v_\eps$ is a critical point of $\mathbf{E}_{s,\eps}(\cdot,G)$ such that both $v_\eps$ and $z^{a}\partial_z v_\eps$ are continuous in  $G$ up to $\partial^0G$, then $v_\eps$ satisfies in the classical sense the Euler-Lagrange equation 
\begin{equation}\label{eqext}
\begin{cases}
{\rm div}(z^{a}\nabla v_\eps)= 0 & \text{in $G$}\,,\\[10pt]
\displaystyle \boldsymbol{\delta}_s \boldsymbol{\partial}^{(2s)}_{z} v_\eps=\frac{1}{\varepsilon^{2s}} \nabla W (v_\eps) & \text{on $\partial^0G$}\,,
\end{cases}
\end{equation} 
where we have set for $\mathbf{x}=(x,0)\in\partial^0G$, 
$$\boldsymbol{\partial}^{(2s)}_{z} v_\eps(\mathbf{x}):=\lim_{z\downarrow0} z^{a}\partial_{z} v_\eps(x,z)\,.$$
We shall refer  to as {\it weak solution} of  equation \eqref{eqext} a critical point of $\mathbf{E}_{s,\eps}(\cdot,G)$, i.e., a map $v_\varepsilon\in H^1(G;\R^d,|z|^a\de\mathbf{x})\cap L^p(\partial^0G)$ satisfying 
\begin{equation}\label{defcritpointEexten}
\boldsymbol{\delta}_s \int_{G} z^{a}\, \nabla v_\varepsilon \cdot\nabla\phi\,\de \mathbf{x}+ \frac{1}{\varepsilon^{2s}}\int_{\partial^0G} \nabla W(v_\varepsilon)\cdot\phi\,\de x=0 
\end{equation}
 for every  $\phi \in H^1(G;\R^d,|z|^a\de \mathbf{x})\cap L^p(\partial^0G)$ compactly supported in $G\cup\partial^0G$.
\vskip3pt

If a weak solution $v_\eps$ is known a priori to be bounded, then the regularity theory for  \eqref{eqext} implies that $v_\eps$ is smooth in $G$ and partially  up to $\partial^0G$. In particular, bounded weak solutions are classical solutions. To be more precise, the following regularity result holds. Since its proof follows from line to line the proof of \cite[Theorem~3.3]{MilSirW}, we shall omit it.

\begin{theorem}\label{regint}
Let   $v_\varepsilon\in H^1(B_{R}^+;\R^d,|z|^a\de \mathbf{x})\cap L^\infty(B_{R}^+)$ be a bounded weak solution of 
\begin{equation}\label{pfff}
\begin{cases} 
{\rm div}(z^{a}\nabla v_\varepsilon)=0 & \text{in $B_{R}^+$}\,,\\[8pt]
\displaystyle \boldsymbol{\delta}_s\boldsymbol{\partial}^{(2s)}_z v_\varepsilon= \frac{1}{\varepsilon^{2s}}\nabla W(v_\varepsilon) & \text{on $D_{R}$}\,.
\end{cases}
\end{equation}
Then $v_\eps\in C^\infty(B_R^+)$, $v_\varepsilon\in C^{0,\alpha}_{\rm loc}\big(B_R^+\cup D_R\big)$, $\nabla_xv_\eps\in C^{0,\alpha}_{\rm loc}(B_R^+\cup D_R)$, and  $z^{a}\partial_{z}v_\eps\in C^{0,\alpha}_{\rm loc}( B_R^+\cup D_R)$ for some exponent $\alpha=\alpha(n,s)\in(0,1)$ depending only on $n$ and $s$.
\end{theorem}
 
Again, exactly as in \cite[Corollary 3.5]{MilSirW}, the previous regularity result implies that bounded weak solutions of \eqref{eqext} are stationary points of $\mathbf{E}_{s,\eps}(\cdot,G)$, i.e., critical points with respect to inner variations up to $\partial^0G$. Recalling the definition and value of the first inner variation $\delta{\bf E}_s$ in \eqref{firstvarEsdef}-\eqref{firstvarEsformul}, we have

\begin{corollary}\label{statAC}
Let $G\subset \mathbb{R}^{n+1}_+$ be an  admissible bounded  open set. If  $v_\eps\in H^1(G;\R^d,|z|^a\de\mathbf{x})\cap L^\infty(G)$ is a bounded weak solution of  \eqref{eqext}, then 
\begin{equation}\label{statcondEseps}
\delta{\bf E}_s\big(v_\eps,G\cup\partial^0G\big)[{\bf X}]
+\frac{1}{\eps^{2s}}\int_{\partial^0 G}W(v_\eps)\,{\rm div}X\,\de x=0
\end{equation}
for every vector field $\mathbf{X}=(X,\mathbf{X}_{n+1})\in C^1(\overline G;\R^{n+1})$ compactly supported in $G\cup \partial^0G$ such that $\mathbf{X}_{n+1}=0$ on $\partial^0G$. 
\end{corollary}

In turn the stationarity condition \eqref{statcondEseps} yields a crucial monotonicity formula, see \cite[Lemma~4.2]{MilSirW}. 

\begin{corollary}[\bf Monotonicity Formula]\label{monotformACeq}
Let $v_\eps \in  H^1(B_R^+;\R^d,|z|^a\de \mathbf{x})\cap L^\infty(B^+_R)$ 
 be a bounded weak solution  of~\eqref{pfff}. 
For every $\mathbf{x}_0=(x_0,0)\in D_R\times\{0\}$, the function $r\in(0,R-|\mathbf{x}_0|\,]\mapsto \boldsymbol{\Theta}_{s,\eps}(v_\eps,x_0,r)$ defined by 
$$ \boldsymbol{\Theta}_{s,\eps}(v_\eps,x_0,r):=\frac{1}{r^{n-2s}} \mathbf{E}_{s,\eps}\big(v_\eps,B_r^+(\mathbf{x}_0)\big)$$
 is non-decreasing.
\end{corollary}

\subsection{Regularity and Maximum Principle}

In view of Section \ref{secext}, a bounded weak solution $u_\eps$ of the fractional equation \eqref{eqfractGL} provides after the extension procedure $u_\eps\mapsto u^\e_\eps$ a bounded weak solution of \eqref{eqext} in every admissible bounded open $G\subset\R^{n+1}_+$ satisfying $\overline{\partial^0G}\subset\Omega$. As a consequence, Theorem~\ref{regint}  yields  the following  interior regularity for the fractional equation. 

 \begin{corollary}\label{corintregfrac}
If $u_\varepsilon \in  \widehat H^{s}(\Omega;\R^d)\cap L^\infty(\R^n)$ is a bounded weak solution of \eqref{eqfractGL}, then  $u_\varepsilon \in C^{1,\alpha}_{\rm loc}(\Omega)$ for some $\alpha=\alpha(n,s)\in(0,1)$ depending only on $n$ and $s$. 
\end{corollary}
 
In the  case where \eqref{eqfractGL} is complemented with a smooth and bounded exterior Dirichlet condition $g$, one can prove an a priori $L^\infty$-bound exactly as in \cite[Lemma 3.7]{MilSirW}    (which rests on assumption ({\rm H3})), so that Corollary~\ref{corintregfrac} applies.

 \begin{lemma}\label{Linftybd}
 Let $g \in C^{0,1}_{\rm loc}(\R^n;\R^d)\cap L^{\infty}(\R^n)$. If $u_\varepsilon \in H_{g}^{s}(\Omega;\R^d)\cap L^p(\Omega)$ is
 a weak solution of \eqref{eqfractGL}, then $u_\varepsilon \in L^\infty(\R^n)$. 
 \end{lemma}

 Continuity up to the boundary for the Dirichlet exterior problem can be obtained from~\cite[Theorem~2]{SerVal}. Combining this result with Corollay \ref{corintregfrac} and Lemma \ref{Linftybd}, we reach the following theorem. 

\begin{theorem}\label{regDirich}
Assume that $\partial\Omega$ is smooth and that $g\in   C^{0,1}_{\rm loc}(\R^n;\R^d)\cap L^\infty(\R^n)$. If $u_\varepsilon \in H_{g}^{s}(\Omega;\R^d)\cap L^p(\Omega)$ 
is a weak solution of \eqref{eqfractGL}, then 
$u_\eps \in C^{1,\alpha}_{\rm loc}(\Omega)\cap C^0(\R^n)$ for some $\alpha\in(0,1)$. 
\end{theorem}  

With regularity in hands, we can now establish the following maximum principle for equation \eqref{eqfractGL}.  
 
\begin{corollary}\label{modless1}
Let $\Omega$ and $g$ be as in Theorem \ref{regDirich}. If $u_\varepsilon \in H_{g}^{s}(\Omega;\R^d)\cap L^p(\Omega)$ 
is a weak solution of~\eqref{eqfractGL}, then  
\begin{equation}\label{ptbd}
\|u_\varepsilon\|_{L^\infty(\R^n)}\leq \max\left(\big(\boldsymbol{c}^2_W\big)^{\frac{1}{p-1}},\|g\|_{L^\infty(\R^n\setminus\Omega)}\right)\,,
\end{equation}
where $\boldsymbol{c}^2_W>0$ is the second constant given in  assumption {\rm (H3)}.
\end{corollary}
 
\begin{proof}
Let $\lambda$ be the constant in the right hand side of \eqref{ptbd}.
By Theorem  \ref{regDirich},  $m_\varepsilon:=\lambda^2-|u^\e_\varepsilon|^2$ is continuous in $\overline\R^{n+1}_+$, and $z^a\partial_z m_\eps$ is continuous up to~$\Omega$. Moreover, 
$m_\varepsilon$ satisfies (in the pointwise sense)
$$\begin{cases} 
-{\rm div}\big(z^{a} \nabla m_\varepsilon\big)\geq 0 & \text{in $\R^{n+1}_+$}\,,\\[5pt]
\displaystyle \boldsymbol{\delta}_s\boldsymbol{\partial}^{(2s)}_z m_\varepsilon= -\frac{2}{\varepsilon^{2s}}u_\eps^\e\cdot\nabla W(u^\e_\eps) & \text{on $\Omega$}\,,\\
m_\varepsilon \geq 0 & \text{on $\R^n\setminus\Omega$}\,.
\end{cases}
$$
Assume that $m_\varepsilon$ achieves its minimum over $\R^n$ at a point $x_0\in\Omega$. Then $x_0$ is a point of maximum of $|u_\varepsilon|$, and thus $\mathbf{x}_0=(x_0,0)$  is a point of maximum of $|u_\eps^\e|$ since $\|u_\eps^\e\|_{L^\infty(\R^{n+1}_+)}\leq \|u_\eps\|_{L^\infty(\R^n)}$. Hence the point $\mathbf{x}_0$  is an absolute minima of 
$m_\varepsilon$ over $\overline\R^{n+1}_+$. If $m_\eps(\mathbf{x}_0)< 0$,  then $|u_\eps^\e(\mathbf{x}_0)|>\lambda$, and we infer from assumption {\rm (H3)} 
that 
$$\boldsymbol{\delta}_s\boldsymbol{\partial}^{(2s)}_z m_\varepsilon(\mathbf{x}_0)\leq -\frac{2}{\eps^{2s}}\boldsymbol{c}^1_W|u^\e_\eps({\bf x}_0)|\big(|u^\e_\eps({\bf x}_0)|^{p-1}-\boldsymbol{c}^2_W\big)< 0\,.$$ 
By the strong maximum maximum principle of \cite[Corollary 2.3.10]{FKS}, we have $m_\eps>m_\eps(\mathbf{x}_0)$ in $\mathbb{R}^{n+1}_+$. Then, the 
Hopf boundary lemma of \cite[Proposition 4.11]{CS1} yields $\boldsymbol{\partial}^{(2s)}_z m_\varepsilon(\mathbf{x}_0)>0$ which gives a contradiction.  
\end{proof} 
 
\begin{remark}\label{remarkLinftymin}
If $u_\eps$ is a solution of the minimization problem \eqref{GLminProb} for an arbitrary exterior condition $g\in \widehat H^s(\Omega;\R^d)\cap L^\infty(\R^n)$, then $u_\eps\in L^\infty(\R^n)$ and the a priori bound \eqref{ptbd} holds. Indeed, assume by contradiction that the set $A:=\{|u_\eps|>\lambda\}$ has positive measure with $\lambda$  
 the constant in the right hand side of \eqref{ptbd}. Since $u_\eps=g$ outside $\Omega$, we have $A\subset\Omega$. Setting ${\rm p}_\lambda:\R^d\to\R^d$ to be the Lipschitz mapping  $y\mapsto \min\{\lambda,|y|\}\frac{y}{|y|}$, we consider the competitor $w_\eps:={\rm p}_\lambda(u_\eps)\in  H_g^s(\Omega;\R^d)$. Since ${\rm p}_\lambda$ is $1$-Lipschitz (being the nearest point projection on the closed ball of radius $\lambda$ centered at the origin), we have $\mathcal{E}_s(w_\eps,\Omega)\leq \mathcal{E}_s(u_\eps,\Omega)$. By assumption {\rm (H3)}, we have  for $x\in A$, 
\begin{multline*}
W\big(u_\eps(x)\big)-W\big(w_\eps(x)\big) =\int_\lambda^{|u_\eps(x)|}\frac{u_\eps(x)}{|u_\eps(x)|}\cdot\nabla W\bigg(t\frac{u_\eps(x)}{|u_\eps(x)|}\bigg)\,\de t\\
\geq \boldsymbol{c}^1_W\int_\lambda^{|u_\eps(x)|}\big(t^{p-1}-\boldsymbol{c}^2_W\big)\,\de t>0\,.
\end{multline*}
Hence $\int_\Omega W(w_\eps)\,\de x<\int_\Omega W(u_\eps)\,\de x$, and it follows that $\mathcal{E}_{s,\eps}(w_\eps,\Omega)< \mathcal{E}_{s,\eps}(u_\eps,\Omega)$ contradicting the minimality of $u_\eps$. 
\end{remark}


  \subsection{The clearing-out property and consequences}
  The purpose of this subsection is to prove a fundamental {\sl clearing out property} for  a bounded weak solution $v_\varepsilon$ of \eqref{pfff}. 
  This property asserts that if the energy of $v_\varepsilon$ in $B_R^+$ is small enough, then $v_\varepsilon$ is very close to one of the wells of $W$ on $D_{R/2}$. Since the potential $W$ is convex near each of its well by assumption {\rm (H2)}, the clearing out property shall provide minimality of $v_\varepsilon$ in smaller balls, whence compactness in the energy space, as well as the convergence rates claimed in Theorem \ref{main1new}. 
  
This clearing   out property is the content  the following lemma which rests on the monotonicity formula in Corollary~\ref{monotformACeq} and the linear regularity estimates from \cite{CS1}. 

\begin{lemma}\label{clear2new}
Given $b\geq 1$, there exists a non-decreasing function ${\boldsymbol{\eta}}_{b}:(0,1]\to (0,\infty)$ 
depending only $n$, $s$, $b$, and $W$, such that the following holds. For $R>0$ and $\eps\in(0,R]$,  
if $v_\eps\in H^1(B_R^+;\R^d,|z|^a\de \mathbf{x})\cap L^\infty(B^+_R)$ is a bounded weak solution of  \eqref{pfff} satisfying $\|v_\eps\|_{L^\infty(B^+_R)}\leq b$, and for some $\delta\in(0,1]$, 
\begin{equation}\label{smallenerg}
\boldsymbol{\Theta}_{s,\eps}(v_\eps,0,R)\leq \boldsymbol{\eta}_{b}(\delta)\,,
\end{equation}
then there exists ${\bf a}\in\mathcal{Z}$ such that $|v_\eps-{\bf a}|\leq \delta$  on  ${D}_{R/2}$. 
\end{lemma}

\begin{proof}
{\it Step 1.} We first assume that $\varepsilon\geq R/2$, and we claim that there exists $\widetilde{\boldsymbol{\eta}}_{b}(\delta)>0$ depending only on $\delta$, $n$, $s$, $b$, and $W$, such that the condition 
$\boldsymbol{\Theta}_{s,\eps}(v_\eps,0,R) \leq \widetilde{\boldsymbol{\eta}}_{b}(\delta)$
implies $ {\rm dist}(v_\eps,\mathcal{Z})\leq \delta$ in $\overline B^+_{R/2}$. Indeed, the rescaled function $\widetilde v_\varepsilon(\mathbf{x}):=v_\varepsilon(R \mathbf{x}) $, 
 satisfies
$$\begin{cases}
{\rm div}(z^{a} \nabla \widetilde v_\varepsilon) = 0 & \text{in $B_1^+$}\,,\\[8pt]
\displaystyle \boldsymbol{\delta}_s\boldsymbol{\partial}_z^{(2s)} \widetilde v_\varepsilon =\frac{R^{2s}}{\eps^{2s}}\nabla W(\widetilde v_\eps) & \text{on $D_1$}\,,
\end{cases}$$
with $\eps/R\in[ 1/2,1]$. Since $\|\widetilde v_\eps\|_{L^\infty(B_1^+)}\leq b$, we infer from \cite[Lemma 4.5]{CS1} that 
\begin{equation}\label{holdbdseq}
\|\widetilde v_\eps\|_{ C^{0,\alpha}(\overline B^+_{1/2})}\leq C_{b}\,,
\end{equation} 
for an exponent $\alpha=\alpha(n,s)\in(0,1)$ and a constant $C_{b}$ depending only on $n$, $s$, $b$, and $W$.   

We now argue by contradiction assuming that for some sequences $\{R_k\}_{k\in\mathbb{N}}\subset(0,+\infty)$, $\{\varepsilon_k\}_{k\in\mathbb{N}}\subset[R_k/2,R_k]$,  $\{\mathbf{x}_k\}_{k\in\mathbb{N}}\subset \overline B^+_{1/2}$, 
the function $\widetilde v_k:=\widetilde v_{\eps_k}$ satisfies ${\rm dist}\big(\widetilde v_k(\mathbf{x}_k),\mathcal{Z}\big)>\delta$ for every $k$, 
and
$$ \mathbf{E}_{s,\varepsilon_k/R_k}(\widetilde v_k,B_1^+)=\frac{1}{R_k^{n-2s}}\mathbf{E}_{s,\eps_k}(v_{\eps_k},B_{R_k}^+)\leq \boldsymbol{\Theta}_{s,\eps_k}(v_{\eps_k},0,R_k)\to 0\quad \text{as $k\to\infty$}\,.$$
By \eqref{holdbdseq} and the Arzel\`a-Ascoli Theorem, we can find a (not relabeled) subsequence  such that $\widetilde u_k$ converges uniformly on $\overline B^+_{1/2}$. 
Since  $\mathbf{E}_{s,\varepsilon_k/R_k}(\widetilde v_k,B_1^+)\to 0$, the limit has to be a constant map belonging to the set $\mathcal{Z}$. In particular, ${\rm dist}(\widetilde v_k,\mathcal{Z})\to 0$  uniformly in $\overline B^+_{1/2}$, which contradicts our assumption  ${\rm dist}\big(\widetilde v_k(\mathbf{x}_k),\mathcal{Z}\big)>\delta$. 
\vskip3pt

\noindent{\it Step 2.} Define
$$\boldsymbol{\eta}_{b}(\delta):=2^{2s-n}\inf_{t\in[\delta,1]} \widetilde{\boldsymbol{\eta}}_{b}(t)\,.$$
Let $\delta\in(0,1]$ and assume that \eqref{smallenerg} holds for $R>0$ and $\eps\in(0,R]$. We fix an arbitrary point $\mathbf{x}_0=(x_0,0)\in \overline D_{R/2}\times\{0\}$. If $\eps\geq R/2$, then ${\rm dist}\big(v_\eps(\mathbf{x}_0),\mathcal{Z}\big)\leq \delta$ by Step~1. If $\eps<R/2$, then $\varepsilon<R-|\mathbf{x}_0|$ and  by Corollary~\ref{monotformACeq} we have 
$$\boldsymbol{\Theta}_{s,\eps}(v_\eps,x_0,\eps)\leq \boldsymbol{\Theta}_{s,\eps}\big(v_\eps,x_0,R-|x_0|\big)\leq 2^{n-2s}\boldsymbol{\Theta}_{s,\eps}(v_\eps,0,R)\,.$$
Our choice of $\boldsymbol{\eta}_{b}(\delta)$ then implies $\boldsymbol{\Theta}_{s,\eps}(v_\eps,x_0,\eps)\leq \widetilde{\boldsymbol{\eta}}_{b}(\delta)$,  
and  we infer from Step 1 that ${\rm dist}(v_\eps,\mathcal{Z})\leq \delta$ in $\overline B^+_{\eps/2}(\mathbf{x}_0)$.
\end{proof}

By our assumptions (H1)-(H2) on $W$, there exists a radius $\boldsymbol{\varrho}_W\in(0,1]$ and a constant $\boldsymbol{\kappa}_W>0$ such that 
\begin{equation}\label{lwbdhess}
D^2W(y)\geq \boldsymbol{\kappa}_W I_{d} \quad\text{whenever $y\in\R^d$ satisfies }{\rm dist}(y,\mathcal{Z})\leq \boldsymbol{\varrho}_W\,, 
\end{equation}
where the inequality above holds in the sense of quadratic forms.

  \begin{lemma}\label{estifond}
Let $R>0$ and $\eps>0$.   If $v_\eps\in H^1(B_R^+;\R^d,|z|^a\de \mathbf{x})\cap L^\infty(B^+_R)$ is a bounded weak solution of  \eqref{pfff} satisfying  $|v_\eps-{\bf a}|\leq \boldsymbol{\varrho}_W$  on ${D}_{R}$  for some  ${\bf a}\in\mathcal{Z}$, then 
\begin{equation}\label{estispeedconv}
|v_\eps-{\bf a}|\leq \boldsymbol{\kappa}^{-1}_W\|v_\eps-{\bf a}\|_{L^\infty(B_R^+)}\frac{\eps^{2s}}{R^{2s}}\quad \text{on $D_{R/2}$}\,. 
\end{equation}
In particular,
\begin{equation}\label{conseqestifond}
 \sqrt{W(v_\eps) }+ |\nabla W(v_\eps)| \leq C_W \|v_\eps-{\bf a}\|_{L^\infty(B_R^+)}\frac{\eps^{2s}}{R^{2s}}\quad \text{on $D_{R/2}$}\,, 
 \end{equation}
for a constant $C_W$ depending only on $W$. 
\end{lemma}

\begin{proof}
By rescaling equation  \eqref{pfff}, it is enough to consider the case $R=1$. First, we notice that \eqref{conseqestifond} follows from \eqref{estispeedconv} expanding $W$ near the point ${\bf a}$. 
To prove \eqref{estispeedconv}, we proceed as follows. Fix an arbitrary parameter $\eta\in(0,1)$, and consider the nonnegative smooth convex function $\psi_\eta:\R^d\to\R$ given by 
$$\psi_\eta(y):=\sqrt{|y|^2+\eta^2}-\eta\,. $$  
Set $w_\eta:=\psi_\eta(v_\eps-{\bf a})\in H^1(B^+_{R},|z|^a\de{\bf x})\cap L^\infty(B^+_{R})$, and we observe that $w_\eta$ satisfies in the weak sense
$$ \begin{cases}
\displaystyle{\rm div}(z^{a}\nabla w_\eta)= z^a\sum_{j=1}^{n+1}(D^2\psi_\eta(v_\eps-{\bf a})\partial_jv_\eps)\cdot \partial_j v_\eps 
& \text{in $B_{R}^+$}\,,\\[12pt]
\displaystyle \boldsymbol{\delta}_s\boldsymbol{\partial}_z^{(2s)} w_\eta=\frac{1}{\varepsilon^{2s}}\nabla\psi_\eta(v_\eps-{\bf a})\cdot\nabla W(v_\eps) & \text{on $D_{R}$}\,.
\end{cases}$$
On the other hand, \eqref{lwbdhess} implies that 
$$(y-z)\cdot\big(\nabla W(y)-\nabla W(z)\big)\geq  \boldsymbol{\kappa}_W|y-z|^2\quad\forall y,z\in \overline B({\bf a},\boldsymbol{\varrho}_W)\,.$$
 Consequently, 
$$\nabla\psi_\eta(v_\eps-{\bf a})\cdot\nabla W(v_\eps) =  \frac{(v_\eps-{\bf a})\cdot\big(\nabla W(v_\eps)-\nabla W({\bf a})\big)}{\sqrt{|v_\eps-{\bf a}|^2+\eta^2}}\geq  \boldsymbol{\kappa}_W \frac{|v_\eps-{\bf a}|^2}{\sqrt{|v_\eps-{\bf a}|^2+\eta^2}} \geq  \boldsymbol{\kappa}_W w_\eta\quad\text{on $D_R$}\,.$$
Therefore $w_\eta$ satisfies 
$$ \begin{cases}
{\rm div}(z^{a}\nabla w_\eta)\geq 0 & \text{in $B_{R}^+$}\,,\\[5pt]
\displaystyle \boldsymbol{\delta}_s\boldsymbol{\partial}_z^{(2s)} w_\eta\geq \frac{\boldsymbol{\kappa}_W}{\varepsilon_k^{2s}}\,w_\eta  & \text{on $D_{R}$}\,.
\end{cases}$$
By \cite[Lemma 3.5]{TVZ} it implies that 
$$\|w_\eta\|_{L^\infty(D_{R/2})}\leq \frac{\eps_k^{2s}}{\boldsymbol{\kappa}_W} \|w_\eta\|_{L^\infty(B^+_{R})} \leq \frac{\eps_k^{2s}}{\boldsymbol{\kappa}_W} \sqrt{\|v_k-{\bf a}\|_{L^\infty(B_R^+)}^2+\eta^2}\,. $$
Letting $\eta\to 0$, we deduce that \eqref{estispeedconv} holds.
\end{proof}

For a fixed ${\bf a}\in\mathcal{Z}$, we consider the modified potential $\widetilde W_{\bf a}:\R^n\to[0,+\infty[$ defined by 
\begin{equation}\label{defmodifpotconv}
\widetilde W_{{\bf a}}(x) := \sup_{y\in \overline B({\bf a},\boldsymbol{\varrho}_W)}\Big\{ W(y)+(x-y)\cdot \nabla W(y)\Big\}\,.
\end{equation}
By construction $\widetilde W_{\bf a}$ is Lipschitz continuous, $ \widetilde W_{\bf a}$  is convex as a supremum over a family of affine functions, and $\widetilde W_{{\bf a}}(x)=W(x)$ for all $x\in \overline B({\bf a},\boldsymbol{\varrho}_W)$ thanks to \eqref{lwbdhess}.

\begin{lemma}\label{minimalitynew}
Let $R>0$ and $\eps>0$.   If $v_\eps\in H^1(B_R^+;\R^d,|z|^a\de \mathbf{x})\cap L^\infty(B^+_R)$ is a weak solution of  \eqref{pfff} satisfying  $|v_\eps-{\bf a}|\leq \boldsymbol{\varrho}_W$  on ${D}_{R}$  for some  ${\bf a}\in\mathcal{Z}$, then 
$${\bf E}_s(v_\eps,B_R^+)+\frac{1}{\varepsilon^{2s}}\int_{D_{R}}\widetilde W_{\bf a}(v_\varepsilon)\,\de x
\leq {\bf E}_s(w,B_R^+)+\frac{1}{\varepsilon^{2s}}\int_{D_{R}}\widetilde W_{\bf a}(w)\,\de x \,,$$
for every $w\in H^1(B^+_R;\R^d,|z|^a\de\mathbf{x})\cap L^p(D_R)$ such that $w-v_\varepsilon$ is compactly supported in $B^+_{R}\cup D_{R}$. 
\end{lemma}

\begin{proof}
Let $w\in H^1(B^+_R,|z|^a\de\mathbf{x})\cap L^p(D_R)$ such that $w-v_\varepsilon$ is compactly supported in $B^+_{R}\cup D_{R}$. Setting $\phi:=w-v_\varepsilon$, then $\phi$ is compactly supported in $B^+_{R}\cup D_{R}$, and since $w=v_\varepsilon+\phi$, 
$${\bf E}_s(w,B_R^+) \geq {\bf E}_s(v_\varepsilon,B_R^+)+ \boldsymbol{\delta}_s \int_{B_R^+} z^{a}\, \nabla v_\varepsilon \cdot\nabla\phi\,\de \mathbf{x}\,.$$
On the other hand, $|v_\eps-{\bf a}|\leq \boldsymbol{\varrho}_W$  on ${D}_{R}$, so that the very definition of $\widetilde W_{{\bf a}}$ implies that 
$$\int_{D_{R}}\widetilde W_{\bf a}(w)\,\de x \geq \int_{D_{R}} W(v_\varepsilon)\,\de x+\int_{D_R}\nabla W(v_\varepsilon)\cdot \phi\,\de x=  \int_{D_{R}} \widetilde W_{\bf a}(v_\varepsilon)\,\de x+\int_{D_R}\nabla W(v_\varepsilon)\cdot \phi\,\de x\,.$$
Combining the two inequalities above and the fact that $v_\varepsilon$ satisfies \eqref{defcritpointEexten} (with $G=B_R^+$), we obtain the announced minimality. 
\end{proof}

 														 
\section{Asymptotics for the fractional Allen-Cahn equation}\label{FGLasymp}   		 
														 
 
 In this section, our main goal is to prove a first part of Theorems \ref{main1new} and \ref{main2new}, the complete proofs of Theorems \ref{main1new} and \ref{main2new} being a combination of 
 Theorem \ref{main1part1}, Remark \ref{smoothbdrycondrem}, Theorem \ref{main1part1mincase}, and Theorem~\ref{main1part1mincase2} below, together with additional estimates that we shall obtain in Theorem \ref{thmlastone} from Section \ref{imprsect}. As a matter of facts, Theorems  \ref{main1part1} and \ref{main1part1mincase} deal with a more general case since no specific exterior boundary condition is required, and only an a priori $L^\infty$-bound on the solutions is needed (in replacement). 
 \vskip3pt
 
 We start with the case of general critical points.

 \begin{theorem}\label{main1part1}
 Let $\Omega\subset\R^n$ be a smooth bounded open set, and $\eps_k\downarrow 0$ a given sequence. 
 Let $\{u_k\}_{k\in\N}\subset \widehat H^s(\Omega;\R^d)\cap L^p(\Omega)$ be a sequence such that $u_k$ weakly solves
 \begin{equation}\label{fracalllcahnf}
 \displaystyle  (-\Delta)^{s} u_k+\frac{1}{\varepsilon_k^{2s}}\nabla W(u_k) =0 \quad \text{in $\Omega$}\,,
\end{equation}
Assume that $\sup_k\|u_k\|_{L^\infty(\R^n)}<\infty$ and that $u_k(x)$ converges to an element of $\mathcal{Z}$ for a.e. $x\in\R^n\setminus\Omega$ as $k\to\infty$.   If $\sup_k \mathcal{E}_{s,\varepsilon_k}(u_k,\Omega)<\infty$, then there exist a (not relabeled) subsequence,  $u_*\in \widehat H^s(\Omega;\mathcal{Z})$ and a partition $\mathfrak{E}=(E^*_1,\ldots,E^*_m)\in\mathscr{A}_m(\Omega)$ 
such that 
\begin{equation}\label{ustarpartcond}
u_*=\sum_{j=1}^m\chi_{E_j^*}{\bf a}_j\quad \text{a.e. in $\R^n$,}
\end{equation}
$u_k\to u_*$ strongly in $H^s_{\rm loc}(\Omega)$ and $L^2_{\rm loc}(\R^n)$. Moreover, each $E^*_j\cap \Omega$ is an open set, $P_{2s}(E_j^*,\Omega)<\infty$, and setting $\boldsymbol{\sigma}:=(|{\bf a}_i-{\bf a}_j|^2)_{i,j}\in\mathscr{S}^2_m$, we have 
 \begin{equation}\label{meancurveq}
\delta \mathscr{P}^{\boldsymbol{\sigma}}_{2s}(\mathfrak{E}^*,\Omega)[X]=0 \quad\text{for every $X\in C_c^1(\Omega;\R^n)$}\,.
 \end{equation}
In addition, for every smooth open subset $\Omega'\subset\Omega$ such that $\overline{\Omega'}\subset\Omega$, 
 \begin{itemize}[leftmargin=22pt]
\item[ \rm  (i)] $ \mathcal{E}_s(u_k,\Omega') \to \mathcal{E}_{s}(u_*,\Omega')=\frac{\gamma_{n,s}}{2}\mathscr{P}^{\boldsymbol{\sigma}}_{2s}(\mathfrak{E}^*,\Omega^\prime)$; 
\vskip5pt

\item[ \rm  (ii)] $\displaystyle\frac{1}{\varepsilon_k^{2s}}W(u_k)\to 0$ in $L^1(\Omega^\prime)$; 
\vskip5pt

\item[\rm  (iii)] $\displaystyle \frac{1}{\varepsilon_k^{2s}}\nabla W(u_k)\to - \sum_{j=1}^m V_{E^*_j}{\bf a}_j$ strongly in  $H^{-s}(\Omega^\prime)$ where  each potential $V_{E^*_j}$ belongs to  $L^1(\Omega)$ and  given by
$$V_{E^*_j}(x):=\Big(\gamma_{n,s}\int_{\R^n}\frac{|\chi_{E^*_j}(x)-\chi_{E^*_j}(y)|^2}{|x-y|^{n+2s}}\,\de y\Big) \big(2\chi_{E^*_j}(x)-1\big) \,.$$
\vskip5pt 

\noindent Setting $\partial\mathfrak{E}^*:=\bigcup_j\partial E_j^*$, 

\item[\rm  (iv)] $u_k\to u_*$ in $C^{1,\alpha}_{\rm loc}(\Omega\setminus\partial\mathfrak{E}^*)$;
\vskip5pt

\item[\rm  (v)] $\|u_k-u_*\|_{L^\infty(K)}\leq C_K\eps_k^{2s}$ for every compact set $K\subset \Omega\setminus\partial\mathfrak{E}^*$; 
\vskip5pt 

\item[\rm  (vi)] there exists $t_W>0$ (depending only on $W$) such that the superlevel set $L_k^t:=\{{\rm dist}(u_k,\mathcal{Z})\geq t\}$ converges locally uniformly  in $\Omega$ to $\partial\mathfrak{E}^*$ for each $t\in(0,t_W)$. 
\end{itemize}
 \end{theorem}

\begin{remark}[\bf Smooth Dirichlet condition]\label{smoothbdrycondrem}
If \eqref{fracalllcahnf} supplemented with a smooth and uniformly bounded Dirichlet condition, then the condition $\sup_k\|u_k\|_{L^\infty(\R^n)}<\infty$ is satisfied. Indeed, consider  $\{g_k\}_{k\in\N}\subset C^{0,1}_{\rm loc}(\R^n;\R^d)$    such that  $\sup_k\|g_k\|_{L^\infty(\R^n\setminus\Omega)}<\infty$. If, in addition to \eqref{fracalllcahnf}, $u_k$ satisfies the Dirichlet condition $u_k=g_k$ in $\R^n\setminus\Omega$, then Corollary \ref{modless1} applies. 
\end{remark}

Then we consider the case of minimizers in sense of Definition \ref{defminimizerAC}. 
 
 \begin{theorem}\label{main1part1mincase}
 In addition to Theorem \ref{main1part1}, if $u_k$ is assumed to be $\mathcal{E}_{s,\eps_k}$-minimizing in $\Omega$, then $u_*$ is $\mathcal{E}_{s}$-minimizing in $\Omega$ among $\mathcal{Z}$-valued maps, that is $\mathcal{E}_{s}(u_*,\Omega)\leq  \mathcal{E}_{s}(u,\Omega)$ 
 for every $u\in \widehat H^s(\Omega;\mathcal{Z})$ such that $u-u_*$ is compactly supported in $\Omega$. In other words, the partition $\mathfrak{E}^*$  is 
 $\mathscr{P}^{\boldsymbol{\sigma}}_{2s}$-minimizing in $\Omega$ in the sense of Definition \ref{defminpart}. 
 \end{theorem}

And lastly, we address the minimization problem with prescribed exterior condition. 
 
 \begin{theorem}\label{main1part1mincase2}
 Let $\Omega\subset\R^n$ be a smooth bounded open set, and $\eps_k\downarrow 0$ a given sequence.  Let $\{g_k\}_{k\in\N}\subset \widehat H^s(\Omega;\R^d)\cap L^\infty(\R^n)$ satisfying $\sup_k\|g_k\|_{L^\infty(\R^n\setminus\Omega)}<\infty$ and such  that $g_k(x)\to g(x)$ with $g(x)\in\mathcal{Z}$ for a.e. $x\in\R^n\setminus\Omega$ as $k\to\infty$.  If $u_k$ is a solution of the minimization problem 
 $$\min\Big\{ \mathcal{E}_{s,\varepsilon_k}(u,\Omega) : u\in H^{s}_{g_k} (\Omega;\R^d)\cap L^p(\Omega) \Big\}\,, $$
 then $\sup_k\|u_k\|_{L^\infty(\R^n)}<\infty$ and $\sup_k \mathcal{E}_{s,\varepsilon_k}(u_k,\Omega)<\infty$. In particular, Theorem \ref{main1part1} applies and in addition to it, 
 the limiting partition $\mathfrak{E}^*$  solves the minimization problem
 \begin{equation}\label{limitminpb}
 \min\bigg\{ \mathscr{P}^{\boldsymbol{\sigma}}_{2s}(\mathfrak{E},\Omega) : \mathfrak{E}\in\mathscr{A}_m(\Omega)\,,\;\sum_j\chi_{E_j}{\bf a}_{j=1}^m=g \text{ a.e. in }\R^n\setminus\Omega
  \bigg\}\,.
 \end{equation}
 Moreover, $u_k\to u_*$ strongly in $H^s(\Omega)$, $\mathcal{E}_{s}(u_k,\Omega) \to \mathcal{E}_{s}(u_*,\Omega)$, and $\eps_k^{-2s}W(u_k)\to 0$ in $L^1(\Omega)$. 
 \end{theorem}
 
  \subsection{Compactness for boundary reactions}\label{subsecbdryreac}

As in \cite{MilSirW}, the proof of Theorem~\ref{main1part1} relies on a slightly more general theorem for the local equation with boundary reactions. We shall essentially derive Theorem~\ref{main1part1} 
from Theorem \ref{thmasympbdryreact} below, applying the extension procedure \eqref{poisson} and  the results from Section~\ref{FractAC}.   

\begin{theorem}\label{thmasympbdryreact}
Let $G\subset\R^{n+1}_+$ be an admissible bounded  open set, and $\varepsilon_k\downarrow0$  a given sequence.  
Let  $\{v_k\}_{k\in\mathbb{N}}\subset H^1(G;\R^d,|z|^a\de\mathbf{x})\cap L^\infty(G)$ satisfying $\sup_k\|v_k\|_{L^\infty(G)}< \infty$, and  such that each $v_k$ weakly solves  
\begin{equation}\label{equinGnew}
 \begin{cases}
{\rm div}\big(z^{a}\nabla v_k)= 0 & \text{in $G$}\,,\\[8pt]
\displaystyle \boldsymbol{\delta}_s\boldsymbol{\partial}^{(2s)}_z v_k=\frac{1}{\varepsilon^{2s}_k}\nabla W(v_k)  & \text{on $\partial^0G$}\,. 
\end{cases}
\end{equation}
If  $\sup_k \mathbf{E}_{s,\varepsilon_k}(v_k,G)<\infty$, then there exist a (not relabeled) subsequence, $v_*\in H^1(G;\R^d,|z|^a\de\mathbf{x})$ and $m$ disjoint open subsets  $E_1^*,\ldots,E^*_m\subset\partial^0G$ such that 
\begin{equation}\label{vstaraubord}
\sum_{j=1}^m\chi_{E^*_j}=1 \text{ and  }v_*=\sum_{j=1}^m\chi_{E_j^*}{\bf a}_j\; \text{ a.e. on $\partial^0G$,}
\end{equation}
 $v_k\rightharpoonup v_*$ weakly in $H^1(G,|z|^a\de\mathbf{x})$, and $v_k\to v_*$ strongly in $H^1_{{\rm loc}}(G\cup\partial^0G,|z|^a\de\mathbf{x})$ as $k\to\infty$. In addition, 
 setting $\Sigma:=\bigcup_j\partial E_j^*\cap\partial^0 G$, 
\vskip5pt

\begin{itemize}[leftmargin=22pt]
\item[ \rm  (i)] $\varepsilon^{-2s}_kW(v_k)\to 0$ in $L^1_{\rm loc}(\partial^0G)$; 
\vskip5pt

\item[\rm  (ii)]  $v_k\to v_*$ in $C^{1,\alpha}_{\rm loc}(\partial^0G\setminus\Sigma)$ for some $\alpha=\alpha(n,s)\in(0,1)$;
\vskip5pt

\item[\rm  (iii)] $\|v_k-v_*\|_{L^\infty(K)}\leq C_K \eps_k^{2s}$ for every compact set $K\subset \partial^0G\setminus\Sigma$,
\vskip5pt

\item[\rm  (iv)] 
there exists $t_W>0$ (depending only on $W$) such that the superlevel set $L_k^t:=\{{\rm dist}(v_k,\mathcal{Z})\geq t\}$ converges locally uniformly  in $\partial^0G$ to $\Sigma$ for each $t\in(0,t_W)$.
\vskip5pt

\item[\rm  (v)] the function $v_*$ satisfies 
$$\delta{\bf E}_s\big(v_*,G\cup\partial^0G\big)[{\bf X}]=0$$ 
for every vector field $\mathbf{X}=(X,\mathbf{X}_{n+1})\in C^1(\overline G;\R^{n+1})$ compactly supported in $G\cup \partial^0G$ such that $\mathbf{X}_{n+1}=0$ on $\partial^0G$. 
\end{itemize}
\end{theorem}

The proof of this theorem will be divided into different steps. Classically, one of the main ingredients is a ``small energy regularity'' property stating essentially that, if the ${\bf E}_{s,\eps_k}$-energy of $v_k$ is small enough in a (half) ball, then the sequence $\{v_k\}$ is compact in a smaller (half) ball, in the energy space and beyond. This is the purpose of the following proposition. 

\begin{proposition}\label{epsregnew}
Given $b\geq 1$, there exist a constant 
$\boldsymbol{\theta}_{b}>0$ 
(depending only on  $n$, $s$, $b$, and $W$) such that the following holds. Let $\varepsilon_k\downarrow0$ be a given sequence and 
 $\{v_k\}_{k\in\mathbb{N}} \subset H^1(B_R^+;\R^d,|z|^a\de\mathbf{x})\cap L^\infty(B_R^+)$  such that $\|v_k\|_{L^\infty(B_R^+)}\leq b$, and $v_k$ solves in the weak sense
 \begin{equation}\label{tim1439}
\begin{cases}
{\rm div}(z^{a}\nabla v_k)= 0 & \text{in $B_R^+$}\,,\\[5pt]
\displaystyle \boldsymbol{\delta}_s\boldsymbol{\partial}_z^{(2s)} v_k=\frac{1}{\varepsilon_k^{2s}}\nabla W(v_k)  & \text{on $D_R$}\,.
\end{cases}
\end{equation}
Assume that $v_k\to v_*$ in $L^1(B_R^+)$. If 
\begin{equation}\label{condliminfnew} 
\liminf_{k\to\infty} {\bf E}_{s,\eps_k}(v_k,B_R^+)< \boldsymbol{\theta}_{b}R^{n-2s} \,,
\end{equation}
then $v_*\in H^1(B_R^+;\R^d,|z|^a\de\mathbf{x})$ and it satisfies $v_*={\bf a}$ on $D_{R/2}$ for some ${\bf a}\in\mathcal{Z}$. Moreover, up to a subsequence, 
\vskip3pt
\begin{enumerate}
\item[\rm (i)] $v_k\to v_*$ strongly in $H^1(B_{R/4}^+,|z|^a\de\mathbf{x})\,$; 
\vskip3pt
\item[\rm (ii)] $\eps_k^{-2s}\int_{D_{R/4}} W(v_k)\,\de x \to 0\,$;
\vskip3pt
\item[\rm (iii)] $v_k\to {\bf a}$  in $C^{1,\alpha}(D_{R/16})\,$ for some $\alpha=\alpha(n,s)\in(0,1)$.
 \end{enumerate} 
\end{proposition}

\begin{proof}
Let $\boldsymbol{\theta}_{b}:={\boldsymbol{\eta}}_{b}(\boldsymbol{\rho}_W)$ where the constant $\boldsymbol{\rho}_W$ is given by \eqref{lwbdhess}, and ${\boldsymbol{\eta}}_{b}$ is given by Lemma~\ref{clear2new}. We select a (not relabeled) subsequence which achieves the $\liminf$ in \eqref{condliminfnew}.  
By the uniform energy bound, we have $v_k\rightharpoonup v_*$ weakly in $H^1(B_R^+,|z|^a\de\mathbf{x})$. By assumption \eqref{condliminfnew}, we then have  $\boldsymbol{\Theta}_{s,\eps_k}(v_k,0,R)\leq {\boldsymbol{\eta}}_{b}(\boldsymbol{\rho}_W)$ for $k$ large enough. Up to a further subsequence, Lemma   \ref{clear2new} implies that there exists ${\bf a}\in\mathcal{Z}$ such that $|v_k-{\bf a}|\leq \boldsymbol{\rho}_W$ on $D_{R/2}$ for $k$ large enough. From the uniform energy bound and the continuity of the trace operator \eqref{compactembL1trace}, we infer that $v_*={\bf a}$ on $D_{R/2}$. 

To prove that that {\rm (i)} and {\rm (ii)} hold, we shall use the argument in  \cite[Corollary 4.7]{MilSirW} together with Lemma~\ref{minimalitynew}. We postpone the proof {\rm (i)} and {\rm (ii)} below, and complete the proof of the proposition. Since $|v_k-{\bf a}|\leq \boldsymbol{\rho}_W$ on $D_{R/2}$,  we have $|\nabla W(v_k)|\leq CR^{-2s}\eps_k^{2s}$ on $D_{R/4}$  by Lemma~\ref{estifond}. Applying  the a priori estimates in \cite[Lemma 4.5]{CS1}, we deduce that $v_k$ is bounded in $C^{0,\alpha}(D_{R/8})$ for some $\alpha=\alpha(n,s)\in(0,1)$. Then, arguing exactly as the proof of \cite[Proposition 4.10]{MilSirW} (based again on  \cite[Lemma 4.5]{CS1} and following the computations in the proof of Lemma \ref{estifond}), we conclude that $v_k$ is bounded in $C^{1, \alpha}(D_{R/16})$, and {\rm (iii)} follows. 
\vskip5pt

\noindent{\it Proof of {\rm (i)} and {\rm (ii)}.}
Let us set 
$$M:=\sup_k \big\{\mathbf{E}_{s,\varepsilon_k}(v_k,B_R^+) +\|v_k\|_{L^\infty(B_R^+)}\big\}<\infty\,.$$
We fix some $r\in(0,R/2)$ arbitrary, and we select a subsequence $\{v_{k_j}\}_{j\in\mathbb{N}}$ such that 
$$\limsup_{k\to+\infty}  \mathbf{E}_{s,\eps_{k}}(v_{k},B_r^+)=\lim_{j\to+\infty}  \mathbf{E}_{s,\eps_{k_j}}(v_{k_j},B_r^+)\,.$$
For $\theta\in(0,1)$, we set $r_\theta:=(1-\theta)(R/2)+\theta r \in(r, R/2)$ and $L_\theta:= r_\theta -r $.  Given an arbitrary integer $\ell\geq 1$, we define $r_i:=r+i\delta_m$ where $i\in\{0,\ldots, \ell\}$ and $\delta_\ell:=L_\theta/\ell$. 
Since 
$$\sum_{i=0}^{\ell-1} \mathbf{E}_{s,\eps_{k_j}}(v_{k_j},B^+_{r_{i+1}}\setminus B^+_{r_i})\leq M\,,  $$
we can find a good index  $i_\ell\in \{0,\ldots,\ell-1\}$ and a (not relabeled)  further subsequence of $\{v_{k_j}\}_{j\in\mathbb{N}}$ such that 
\begin{equation}\label{tranche1}
\mathbf{E}_{s,\eps_{k_j}}(v_{k_j},B^+_{r_{i_\ell+1}}\setminus B^+_{r_{i_\ell}}) \leq \frac{M+1}{\ell}\quad \forall j\in\N\,.
\end{equation}
From the weak convergence of $v_{k_j}$ to $v_*$ and the lower semicontinuity of ${\bf E}_s$, we deduce that 
\begin{equation}\label{tranche2}
{\bf E}_s\big(v_*,B^+_{r_{i_\ell+1}}\setminus B^+_{r_{i_\ell}}\big) \leq \frac{M+1}{\ell}\,.
\end{equation}
Now consider a smooth cut-off function $\chi\in C_c^\infty(B_{R/2};[0,1])$ such that $\chi=1$ in $B_{r_{i_\ell}}$, $\chi=0$ in $B_{R/2}\setminus B_{r_{i_\ell+1}}$, and satisfying 
$|\nabla\chi|\leq C\delta_\ell^{-1}$ for a constant $C$ only depending on~$n$. Then, we define 
$$w_j:=\chi v_* +(1-\chi)v_{k_j}\,, $$
so that $w_j\in H^1(B_{R/2}^+;\R^d,|z|^a\de\mathbf{x})$ and $w_j-v_{k_j}$ is compactly supported in $B^+_{R/2}\cup D_{R/2}$. 

Since $|v_{k_j}-{\bf a}|\leq \boldsymbol{\rho}_W$  and $|w_j-1|\leq  \boldsymbol{\rho}_W$ on $D_{R/2}$, we have $\widetilde W_{\bf a}(v_{k_j})=W(v_{k_j})$ and  
$\widetilde W_{\bf a}(w_{j})=W(w_{j})$ on $D_{R/2}$, where $\widetilde W_{\bf a}$ is the potential defined in \eqref{defmodifpotconv}. Therefore,  Lemma \ref{minimalitynew} yields 
$$ \mathbf{E}_{s,\varepsilon_{k_j}}(v_{k_j},B_{R/2}^+)\leq \mathbf{E}_{s,\varepsilon_{k_j}}(w_j,B_{R/2}^+)\,,$$
which implies that 
 \begin{equation}\label{conseqtranch}
 \mathbf{E}_{s,\varepsilon_{k_j}}(v_{k_j},B_{r}^+)\leq {\bf E}_s(v_*,B^+_{r_\theta}) + \mathbf{E}_{\varepsilon_{k_j}}(w_j,B_{r_{i_{\ell}+1}}^+\setminus B_{r_{i_{\ell}}}^+)\,.
 \end{equation}
Recalling again that $W=\widetilde W_{\bf a}$ in $\overline B({\bf a},\boldsymbol{\rho}_W)$, the potential $W$ is convex on the set $\overline B({\bf a},\boldsymbol{\rho}_W)$, so that 
$$W(w_j)=W\big(\chi {\bf a}+(1-\chi)v_{k_j}\big)\leq (1-\chi)W(v_{k_j})\quad\text{on $D_{R/2}$}\,. $$
Therefore, 
\begin{multline}\label{tranche3}
\mathbf{E}_{s,\varepsilon_{k_j}}(w_j,B_{r_{i_{\ell}+1}}^+\setminus B_{r_{i_{\ell}}}^+)\leq  {\bf E}_s\big(v_*,B_{r_{i_{\ell}+1}}^+\setminus B_{r_{i_{\ell}}}^+\big)\\
+\mathbf{E}_{s,\varepsilon_{k_j}}\big(v_{k_j},B_{r_{i_{\ell}+1}}^+\setminus B_{r_{i_{\ell}}}^+\big)
+C\delta_\ell^{-2}\int_{B_{r_{i_{\ell}+1}}^+\setminus B_{r_{i_{\ell}}}^+} z^{a}|v_{k_j}-v_*|^2\,\de \mathbf{x}\,.
\end{multline}
Since $v_{k_j}\to v_*$ in $L^1(B_R^+)$ and  $|v_{k_j}|\leq M$ in $B_R^+$, we have 
 $v_{k_j}\to v_*$ strongly in $L^2(B_{R}^+,|z|^a\de\mathbf{x})$.  Consequently,  gathering \eqref{tranche1}-\eqref{tranche2}-\eqref{tranche3} leads to 
$$\limsup_{j\to\infty} \mathbf{E}_{s,\varepsilon_{k_j}}(w_j,B_{r_{i_{\ell+1}}}^+\setminus B_{r_{i_{\ell}}}^+)\leq \frac{2(M+1)}{\ell}\,.$$
In view of \eqref{conseqtranch}, we have thus obtained
$$ \lim_{j\to\infty}\mathbf{E}_{s,\varepsilon_{k_j}}(v_{k_j},B_{r}^+)\leq {\bf E}_s(v_*,B_{r_\theta}^+) + \frac{2(M+1)}{\ell}\,.$$
Finally, letting first $\ell\to\infty$ and then $\theta\to 1$, we conclude that 
$$\lim_{j\to+\infty} \mathbf{E}_{s,\varepsilon_{k_j}}(v_{k_j},B_{r}^+)\leq    {\bf E}_s(v_*,B_{r}^+)\,.$$
On the other hand, $\liminf_j  \mathbf{E}_s(v_{k_j},B_{r}^+)\geq {\bf E}_s(v_*,B_{r}^+)$ by lower semicontinuity, and consequently, 
$$\lim_{j\to\infty} \mathbf{E}_s(v_{k_j},B_{r}^+)= {\bf E}_s(v_*,B_{r}^+)\quad\text{and}\quad \lim_{j\to\infty} \frac{1}{\eps_{k_j}^{2s}}\int_{D_r}W(v_{k_j})\,\de x=0\,. $$
From the weak convergence of $v_{k_j}$, it classically follows that the sequence $\{v_{k_j}\}_{j\in\mathbb{N}}$ converges strongly in $H^1(B_r^+,|z|^a\de\mathbf{x})$ towards $v_*$.
\end{proof}

We are now ready to prove Theorem \ref{thmasympbdryreact}. With Proposition \ref{epsregnew} in hands, the strategy follows  \cite[Proof of Theorem 4.1]{MilSirW}. We shall provide all details for completeness.  

\begin{proof}[Proof of Theorem \ref{thmasympbdryreact}] 
{\it Step 1.} Set  $b:= 1+\sup_k\|v_k\|_{L^\infty(G)}$. By the energy bound assumption,  $\{v_k\}$ is bounded in $H^1(G,|z|^a\de \mathbf{x})$. 
Hence there is a (not relabeled) subsequence  such that 
$v_k\rightharpoonup v_*$ weakly in $H^1(G,|z|^a\de \mathbf{x})$. By the compact embedding  $H^1(G,|z|^a\de\mathbf{x})\hookrightarrow L^1(G)$  (see e.g. \cite[Section 2.3]{MilPegSch}), we also have $v_k\to v_*$ strongly in $L^1(G)$. 
Since $|v_k|\leq b$, it implies that $|v_*|\leq b$ in $G$, and  $v_k\to v_*$ strongly in $L^2(G,|z|^a\de\mathbf{x})$. By equation \eqref{equinGnew} and standard elliptic regularity, $v_k\to v_*$ in $ C^\ell_{\rm loc}(G)$ for all $\ell\in\mathbb{N}$. As a consequence, ${\rm div}\big(z^{a}\nabla v_*)= 0$  in $G$.  On the other hand, the uniform energy bound also implies that $W(v_k)\to 0$ in $L^1(\partial^0G)$, and we infer from the continuity of the trace operator that $v_*$ takes values in $\mathcal{Z}$  on $\partial^0G$. 

Next we introduce the sequence of measures  
$$\mu_k:= \frac{\boldsymbol{\delta}_s}{2}z^{a}|\nabla v_k|^2\mathscr{L}^{n+1}\LL G +\frac{1}{\varepsilon^{2s}_k}W(v_k)\mathscr{H}^n \LL\partial^0G\,.$$ 
By the energy bound assumption, the total masses of the $\mu_k$'s are bounded. Hence we can find a further subsequence such that 
\begin{equation}\label{tim2253}
\mu_k\rightharpoonup \mu:=\frac{\boldsymbol{\delta}_s}{2}z^{a}|\nabla v_*|^2\mathscr{L}^{n+1}\LL G+\mu_{\rm sing}\,,
\end{equation}
weakly* as Radon measures on $G\cup\partial^0G$ for some finite nonnegative measure $\mu_{\rm sing}$\footnote{Indeed, writing 
$|\nabla v_k|^2=|\nabla v_*|^2+2\nabla v_*\cdot\nabla(v_k-v_*)+|\nabla(v_k-v_*)|^2$ and using the weak convergence of $v_k$ to $v_*$, we obtain the defect  measure $\mu_{\rm sing}$ as the weak* limit of $\frac{\boldsymbol{\delta}_s}{2}z^{a}|\nabla(v_k-v_*)|^2\mathscr{L}^{n+1}\LL G +\frac{1}{\varepsilon^{2s}_k}W(v_k)\mathscr{H}^n \LL\partial^0G\,.$}. 
From the monotonicity formula in Corollary \ref{monotformACeq}, we infer that 
\begin{equation}\label{tim1213new}
\rho^{2s-n}\mu(B_\rho(\mathbf{x}))\leq r^{2s-n}\mu(B_r(\mathbf{x}))
\end{equation}
for every $\mathbf{x}\in\partial^0G$ and every $0<\rho<r<\min\big(1,{\rm dist}(\mathbf{x},\partial^+G)\big)$. As a consequence, the $(n-2s)$-dimensional density 
\begin{equation}\label{existdens}
\Theta^{n-2s}(\mu,\mathbf{x}):=\lim_{r\downarrow 0}\, \frac{\mu(B_r(\mathbf{x}))}{{\omega}_{n-2s}r^{n-2s}}
\end{equation}
exists\footnote{Here we have set $\displaystyle\omega_{n-2s}:=\frac{\pi^{\frac{n-2s}{2}}}{\Gamma(1+\frac{n-2s}{2})}$.} 
and is finite at every point $\mathbf{x}\in\partial^0G$. 
Note that \eqref{tim1213new} also yields
\begin{equation}\label{upbddensity}
\Theta^{n-2s}(\mu,\mathbf{x})\leq \frac{C}{\big({\rm dist}(\mathbf{x},\partial^+G)\big)^{n-2s}} \,\sup_{k} \mathbf{E}_{s,\eps_k}(v_k,G)<\infty 
\end{equation}
for all $\mathbf{x}\in\partial^0G$. On the other hand, by the smooth convergence of $v_k$ toward $v_*$ in $G$,  $\Theta^{n-2s}(\mu,\mathbf{x})=0$ for all $x\in G$. 
In addition, we observe that $\mathbf{x}\in\partial^0G\mapsto \Theta^{n-2s}(\mu,\mathbf{x})$ is upper semicontinuous \footnote{Indeed, assume that ${\bf x}_j\to {\bf x}\in \partial^0G$, and choose a sequence $r_m\downarrow0$ such that $\mu(\partial B_{r_m}({\bf x}))=0$. By  \eqref{tim1213new}, we have $\limsup_j\Theta^{n-2s}(\mu,\mathbf{x}_j)\leq \omega^{-1}_{n-2s}r_m^{n-2s}\mu(B_{r_m}({\bf x}))$, and the conclusion follows letting $r_m\to0$.}.
\vskip3pt

We define the concentration set  as
\begin{multline}\label{defconcentrset}
\Sigma:=\bigg\{\mathbf{x}\in \partial^0G :  \inf_r\Big\{ \liminf_{k\to\infty}\, r^{2s-n}\mu_k\big(B_r(\mathbf{x})\big) : \\
0<r<\min\big(1,{\rm dist}(\mathbf{x},\partial^+G)\big)\Big\}\geq\boldsymbol{\theta}_{b}\bigg\}\,, 
\end{multline}
where $\boldsymbol{\theta}_{b}>0$ is the constant given by Proposition~\ref{epsregnew}. From Corollary \ref{monotformACeq} and \eqref{tim1213new}, we infer that 
\begin{multline*}
\Sigma =\bigg\{\mathbf{x}\in \partial^0G :  \lim_{r\downarrow 0} \,\liminf_{k\to\infty}\, r^{2s-n}\mu_k(B_r(\mathbf{x})) \geq\boldsymbol{\theta}_{b}\bigg\} \\
 = \bigg\{\mathbf{x}\in \partial^0G :  \lim_{r\downarrow 0}\, r^{2s-n}\mu(B_r(\mathbf{x})) \geq \boldsymbol{\theta}_{b}\bigg\} \,,$$
\end{multline*} 
and consequently, 
\begin{equation}\label{tim1721}
\Sigma=\bigg\{\mathbf{x}\in \partial^0G: \Theta^{n-2s}(\mu,\mathbf{x})\geq \frac{\boldsymbol{\theta}_{b}}{{\omega}_{n-2s}} \bigg\} \,. 
\end{equation}
In particular, $\Sigma$ is a relatively closed subset of $\partial^0G$ since $\Theta^{n-2s}(\mu,\cdot)$ is upper semicontinuous. 
Moreover, by a well known property of  densities (see e.g. \cite[Theorem~2.56]{AFP}), we have
\begin{equation}\label{train}
 \frac{\boldsymbol{\theta}_{b}}{{\omega}_{n-2s}}\mathscr{H}^{n-2s}(\Sigma)  \leq \mu(\Sigma)<\infty\,.
 \end{equation}
 On the other hand, it follows from \eqref{upbddensity} and \cite[Theorem~2.56]{AFP} that $\mu\LL \Sigma=\mu_{\rm sing}\LL \Sigma$ is absolutely continuous with respect to $\mathscr{H}^{n-2s}\LL\Sigma$. 

We now claim that ${\rm spt}(\mu_{\rm sing})\subset\Sigma$. Indeed, for an arbitrary $\mathbf{x}_0\in \partial^0G\setminus\Sigma$, we can find a radius $0<r <{\rm dist}(\mathbf{x}_0,\partial^+G\cup\Sigma)$ 
 such  that $r^{2s-n}\mu(B_r(\mathbf{x}_0))< \boldsymbol{\theta}_{b}$ and $\mu(\partial B_r(\mathbf{x}_0))=0$. Then $\lim_{k} \mathbf{E}_{s,\eps_k}(v_k,B^+_r(\mathbf{x}_0))=\mu(B_r(\mathbf{x}_0))< \boldsymbol{\theta}_{b}r^{n-2s}$, 
and $\mu_{\rm sing}(B_{r/4}(\mathbf{x}_0))=0$ by Proposition \ref{epsregnew}. Hence 
$\mu_{\rm sing}(\partial^0G\setminus \Sigma)=0$,  
and thus $\mu_{\rm sing}$ is indeed supported by $\Sigma$. As a consequence, $\mu_{\rm sing}$ is absolutely continuous with respect to $\mathscr{H}^{n-2s}\LL\Sigma$. 
\vskip3pt

We  now claim that $\mu_{\rm sing}\equiv 0$. We argue by contradiction assuming that $\mu_{\rm sing}(\Sigma)>0$. By \cite[Corollary 3.2.3]{Zi}, we can find a Borel subset $\widetilde \Sigma\subset\Sigma$ such that $\mathscr{H}^{n-2s}(\Sigma\setminus\widetilde\Sigma)=0$ and 
$$\lim_{r\downarrow 0} \frac{1}{r^{n-2s}}{\bf E}_s\big(v_*, B^+_r(\mathbf{x}_0)\big)=0\quad\text{for every $\mathbf{x}_0\in\widetilde\Sigma$}\,.$$
Then $\mu_{\rm sing}(\widetilde\Sigma)=\mu_{\rm sing}(\Sigma)>0$. 
Moreover, by our choice of $\widetilde\Sigma$, the density 
$$\Theta^{n-2s}(\mu_{\rm sing},\mathbf{x}_0):=\lim_{r\downarrow 0}\, \frac{\mu_{\rm sing}(B_r(\mathbf{x}_0))}{{\omega}_{n-2s}r^{n-2s}}$$ 
exists at every $\mathbf{x}_0\in \widetilde\Sigma$, and $\Theta^{n-2s}(\mu_{\rm sing},\mathbf{x}_0)=\Theta^{n-2s}(\mu,\mathbf{x}_0)\in(0,\infty)$. 
By Marstrand's Theorem (see e.g. \cite[Theorem 14.10]{Matti}), it implies that $(n-2s)$ is an integer, which is an obvious contradiction. Hence $\mu_{\rm sing}\equiv 0$. 

Note that \eqref{train}  now yields $\mathscr{H}^{n-2s}(\Sigma)=0$. 
Moreover, we infer from \eqref{tim2253} that for every admissible open set  $G'$ such that  $\overline{G'}\subset G\cup\partial^0G$, 
$${\bf E}_s(v_*,G^\prime) 
\leq \liminf_{k\to\infty} {\bf E}_s(v_k,G^\prime) \\
\leq \lim_{k\to\infty}\mathbf{E}_{s,\eps_k}(v_k,G')= {\bf E}_s(v_*,G^\prime)\,.$$
Therefore $v_k\to v_*$ strongly in $H^1_{{\rm loc}}(G\cup\partial^0G,|z|^a\de\mathbf{x})$, and $\eps_k^{-2s}W(v_k)\to 0$ in $L^1_{\rm loc}(\partial^0G)$. 
\vskip5pt

\noindent{\it Step 2.}
We define for each $j\in\{1,\ldots,m\}$ the set 
$$E^*_j:=\Big\{\mathbf{x}=(x,0)\in\partial^0G : \text{$v_*={\bf a}_j$  a.e. on $D_r(x)$ for some $r\in(0,{\rm dist}(\mathbf{x},\partial^+G))$}\Big \} \,.$$
By construction, the $E^*_j$'s are disjoint relatively open subsets of $\partial^0G$. 

We claim that $E^*_j\cap \Sigma=\emptyset$ for every $j=1,\ldots,m$. Indeed, assume for instance that  $\mathbf{x}_0=(x_0,0)\in E^*_j$ some $j$. Then we can find $r>0$ such that $v_*={\bf a}_j$   on $D_r(x_0)$.  According to \cite[Lemma 4.8]{MilSirW} and Step 1, we have 
$$\Theta^{n-2s}(\mu,\mathbf{x}_0)=\lim_{\rho\to 0} \frac{1}{\rho^{n-2s}}{\bf E}_s\big(v_*,  B^+_\rho(\mathbf{x}_0)\big)=0\,,$$
whence $\mathbf{x}_0\not \in \Sigma$. 

Next we claim that $\partial^0G=\Sigma\cup\bigcup_j E^*_j$. Indeed, if $\mathbf{x}_0=(x_0,0)\in\partial^0G\setminus \Sigma$, then we can find a radius $r>0$ such that  
$\lim_k \mathbf{E}_{s,\eps_k}(v_k,B^+_{r}(\mathbf{x}_0))< \boldsymbol{\theta}_{b}r^{n-2s}$. By Proposition~\ref{epsregnew},  $v_k\to {\bf a}_j$ for some $j$ uniformly in $D_{r/4}(x_0)$. Therefore,  $v_*={\bf a}_j$  on  $D_{r/4}(x_0)$. Hence $\mathbf{x}_0\in E^*_j$. 

Since $\mathscr{L}^n(\Sigma)=0$, it implies in particular that \eqref{vstaraubord} holds. Next we show that $\Sigma=\bigcup_{j}\partial E^*_j\cap\partial^0G$.  
Indeed, if $\mathbf{x}_0=(x_0,0)\in \partial E^*_j\cap\partial^0G$ for some $j$, then $D_r(x_0)\cap E^*_j\not=\emptyset$ for every $r>0$. Since $E^*_j$ is open, $D_r(x_0)\cap E^*_j$ contains a small disc for every $r>0$. Thus $D_r(x_0)\not\subset E^*_l$ for every $l\not=j$ and every $r>0$, and thus  $x_0\in\Sigma$.  This shows that $\bigcup_j\partial E^*_j\cap\partial^0G\subset\Sigma$.  The other way around, if $x_0\in\Sigma$, then $x_0\not\in \bigcup_j E^*_j$.  
Since $\mathscr{L}^n(\Sigma)=0$, we easily deduce that for every $r>0$,  $ \bigcup_j E^*_j\cap D_r(x_0)\not=\emptyset$. Hence $\Sigma\subset \partial (\bigcup_jE^*_j)\cap\partial^0G\subset \bigcup_j\partial E^*_j\cap\partial^0G$. 
\vskip3pt

We now prove claims (ii) and (iii).  We fix an arbitrary compact set $K\subset \partial^0G\setminus \Sigma$, and notice that for every $\bar{\mathbf{x}}=(\bar x,0)\in K\times\{0\}$ and $0<r<\min\{{\rm dist}(K,\Sigma), {\rm dist}(K,\partial^+G)\}$, there exists $j_{\bar{\mathbf{x}}}\in\{1,\ldots,m\}$ such that $D_r(\bar x)\subset E^*_{j_{\bar{\mathbf{x}}}}$, that is $v_*={\bf a}_{j_{\bar{\mathbf{x}}}}$ on $D_r(\bar x)$.  By Lemma \cite[Lemma 4.8]{MilSirW}, we can find a radius $r_K<\min\{{\rm dist}(K,\Sigma), {\rm dist}(K,\partial^+G)\}$ such that 
$${\bf E}_s\big(v_*,B^+_{r_K}(\bar{\mathbf{x}})\big) <\boldsymbol{\theta}_{b}r_K^{n-2s}$$
for every $\bar{\mathbf{x}}\in K\times\{0\}$. Then we deduce from Step 1 that 
$$\lim_{k\to\infty} \mathbf{E}_{s,\eps_k}\big(v_k,B^+_{r_K}(\bar{\mathbf{x}})\big) < \boldsymbol{\theta}_{b}r_K^{n-2s}$$ 
for every $\bar{\mathbf{x}}\in K\times\{0\}$. Then (ii) follows  Proposition \ref{epsregnew} and a standard covering argument. In turn, (iii) follows from  Lemma \ref{estifond} and again a covering argument. 
\vskip5pt

\noindent{\it Step 3: Convergence of superlevel sets.} We now prove (iv)  choosing $t_W:=\boldsymbol{\varrho}_W>0$  (given by \eqref{lwbdhess}). 
We fix $t\in(0,t_W)$, a compact set $K\subset\partial^0G$, and a radius $r>0$. First, from (ii) or (iii) we deduce that ${\rm dist}(v_k,\mathcal{Z})\to 0$ uniformly on $K\setminus \mathscr{T}_r(\Sigma)$. Therefore, $L_k^t\cap K\subset \mathscr{T}_r(\Sigma)$ for $k$ large enough. Then we consider a covering of $\Sigma\cap K$ made by finitely many discs $D_{r/2}(x_1),\ldots, D_{r/2}(x_J)$ centered at points of $\Sigma$ (and included in $\partial^0G$, choosing a smaller radius if necessary). We claim that for $k$ large enough, each disc 
$ D_{r/2}(x_j)$ contains at least one point $z^k_j\in L_k^t$.  We argue by contradiction assuming that for some (not relabeled) subsequence, there is a disc $ D_{r/2}(x_{j_0})$ such that 
${\rm dist}(v_k,\mathcal{Z})<t<\boldsymbol{\varrho}_W$ on  $ D_{r/2}(x_{j_0})$  for every $k$. Noticing that \eqref{lwbdhess} implies that $|{\bf a}-{\bf a}^\prime|> \boldsymbol{\varrho}_W$ for ${\bf a}, {\bf a}^\prime\in\mathcal{Z}$ with ${\bf a}\not={\bf a}^\prime$, we infer that $|v_k-{\bf a}|\leq \boldsymbol{\varrho}_W$ on  $ D_{r/2}(x_{j_0})$ for some ${\bf a}\in\mathcal{Z}$ for every $k$ (choosing a further subsequence if necessary). As a consequence, $v_*={\bf a}$ on  $ D_{r/2}(x_{j_0})$, which implies that $x_{j_0}\not\in\Sigma$ by Step 2, a contradiction. 

Now if $x$ is an arbitrary point in $\Sigma\cap K$, then $x\in D_{r/2}(x_{j_x})$ for some $j_x\in\{1,\ldots,J\}$, and thus $|x-z^k_{j_x}|<r$ for $k$ large enough. 
Hence $\Sigma\cap K\subset \mathscr{T}_r(L_k^t)$ whenever $k$ is sufficiently large.
\vskip5pt

\noindent{\it Step 4: Proof of {\rm (v)}.} Let $\mathbf{X}=(X,\mathbf{X}_{n+1})\in C^1(\overline G;\R^{n+1})$ be a compactly supported  vector field in $G\cup \partial^0G$ such that $\mathbf{X}_{n+1}=0$ on $\partial^0G$.
By Corollary \ref{statAC}, we have 
$$\delta{\bf E}_s\big(v_k,G\cup\partial^0G\big)[{\bf X}]
+\frac{1}{\eps_k^{2s}}\int_{\partial^0 G}W(v_k)\,{\rm div}X\,\de x=0\,.$$
From formula \eqref{firstvarEsformul} and the convergences established in Step 1, we can pass to the limit $k\to\infty$ in this identity to infer that 
$\delta{\bf E}_s\big(v_*,G\cup\partial^0G\big)[{\bf X}]=0$, and the proof is complete. 
\end{proof}

\subsection{Proof of Theorem \ref{main1part1}, Theorem \ref{main1part1mincase}, and Theorem \ref{main1part1mincase2}}

\begin{proof}[Proof  of Theorem \ref{main1part1}]
{\it Step 1.} By the energy bound assumption, the sequence $\{u_k\}$ is bounded in $\widehat H^s(\Omega)$. According to \cite[Remark 2.2 \& Lemma 2.1]{MilPegSch}, we can  find a (not relabeled) subsequence and $u_*\in \widehat H^s(\Omega;\R^d)$ such that $u_k\rightharpoonup u_*$ weakly in $\widehat H^s(\Omega)$ and weakly in $L^2_{\rm loc}(\R^n)$.  By the compact embedding $H^s(\Omega)\hookrightarrow L^2(\Omega)$, we first deduce that $u_k\to u_*$ strongly in $L^2(\Omega)$. On the other hand, $\{u_k\}$ being uniformly bounded and a.e. convergent in $\R^n\setminus\Omega$, we infer that $u_k\to u_*$ strongly in $L^2_{\rm loc}(\R^n\setminus\Omega)$. Hence,  $u_k\to u_*$ strongly in $L^2_{\rm loc}(\R^n)$, and (up to a further subsequence if necessary) $u_k\to u_*$ a.e. in $\R^n$. By assumption, $u_*(x)\in\mathcal{Z}$ for a.e. $x\in \R^n\setminus\Omega$, and from the energy bound we deduce that $u_*(x)\in\mathcal{Z}$ for a.e. $x\in\Omega$. Therefore $u_*\in \widehat H^s(\Omega;\mathcal{Z})$. Setting for $j=1,\ldots, m$, 
$$E^*_j:=\big\{x\in\R^n: \text{ $x$ is Lebesgue point of $u_*$ and $u_*(x)={\bf a}_j$}\big\}\,, $$
the family $E^*_1,\ldots, E_m^*$ is a family of Borel  subsets of $\R^n$ such that $\sum_j\chi_{E^*_j}=1$ a.e. in $\R^n$, and \eqref{ustarpartcond} holds. 
Moreover, setting $\mathfrak{E}^*:=(E_1^*,\ldots,E^*_m)$, we  have 
$$\frac{\gamma_{n,s}}{2}\mathscr{P}_{2s}^{\boldsymbol{\sigma}}(\mathfrak{E}^*,\Omega)=\mathcal{E}_s(u_*,\Omega)<\infty\,,$$ 
so that 
$\mathfrak{E}^*$ belongs to $\mathscr{A}_m(\Omega)$. 
\vskip3pt

For each $j\in\{1,\ldots,m\}$, we now consider the Lipschitz function ${\rm p}_j:\R^d\to \R$ defined by 
\begin{equation}\label{projtrickEj}
{\rm p}_j(y):=\frac{{\rm dist}(y,\mathcal{Z}\setminus\{{\bf a}_j\})}{{\rm dist}({\bf a}_j,\mathcal{Z}\setminus\{{\bf a}_j\})} \,.
\end{equation}
Since $\chi_{E^*_j}={\rm p}_j(u_*)$ and $u_*\in \widehat H^s(\Omega;\R^d)$, we have $\chi_{E^*_j}\in \widehat H^s(\Omega)$, which rephrases as $\mathcal{E}_s(\chi_{E^*_j},\Omega)=\gamma_{n,s}P_{2s}(E^*_j,\Omega)<\infty$, i.e., $E^*_j$ has finite $2s$-perimeter in $\Omega$. 
\vskip3pt

Finally, recalling that the linear operator $u\mapsto u^\e$ is a continuous from $\widehat{H}^{s}(\Omega;\R^d)$ into $H^1(G,|z|^a\de \mathbf{x})$ for every admissible bounded  open set $G\subset\R^{n+1}_+$ satisfying $\overline{\partial^0G}\subset \Omega$ (see \cite[Corollary 2.10]{MilPegSch}), we infer that  $u_k^\e\rightharpoonup u_*^\e$ weakly in $H^1_{{\rm loc}}(\R^{n+1}_+\cup\Omega,|z|^a\de \mathbf{x})$.

\vskip3pt

\noindent{\it Step 2.} Let us now consider an increasing  sequence $\{G_l\}_{l\in\N}$ of  bounded admissible open sets such that  $\overline{\partial^0G_l}\subset\Omega$ for every $l\in\N$, $\cup_l G_l=\R^{n+1}_+$, and $\cup_l\partial^0G_l=\Omega$. Since $\|u_k^\e\|_{L^\infty(\R^{n+1}_+)}\leq \|u_k\|_{L^\infty(\R^n)}$, the sequence $\{u_k^\e\}$ is bounded in $L^\infty(G_l)$ for each $l\in\N$. According to Section \ref{subsectDegAllCahneq}, $u_k^\e\in H^1(G_l,|z|^a\de \mathbf{x})\cap L^\infty(G_l)$ solves 
$$ \begin{cases}
{\rm div}(z^a\nabla u^\e_k)= 0 & \text{in $G_l$}\,,\\[8pt]
\displaystyle \boldsymbol{\delta}_s\boldsymbol{\partial}^{(2s)}_z u^\e_k=\frac{1}{\varepsilon^{2s}_k}\nabla W(u^\e_k) & \text{on $\partial^0 G_l$}\,,  
\end{cases}
$$
for every $l\in\N$. In addition, by the energy bound assumption on $u_k$ and \cite[Corollary 2.10]{MilPegSch}), we have $\sup_k {\bf E}_{s,\eps_k}(u_k^\e,G_l)<\infty$ for every $l\in\N$. Therefore, we can apply Theorem~\ref{thmasympbdryreact}  in each $G_l$, and since  $u_k^\e\rightharpoonup u_*^\e$ weakly in $H^1_{{\rm loc}}(\R^{n+1}_+\cup\Omega,|z|^a\de \mathbf{x})$ by Step 1,  we conclude that $u_k^\e\to u_*^\e$ strongly in $H^1_{{\rm loc}}(\R^{n+1}_+\cup\Omega,|z|^a\de \mathbf{x})$.  By  \cite[Lemma 2.8]{MilPegSch}) and a standard covering argument, it implies that $u_k\to u_*$ strongly in $H^s_{\rm loc}(\Omega)$.

For each $l\in\mathbb{N}$ and $j\in\{1,\ldots,m\}$, we denote by $\widetilde E^*_{j,l}$ the limiting open subset of $\partial^0G_l$ provided by Theorem~\ref{thmasympbdryreact}, and we observe that $\widetilde E^*_{j,l}=\widetilde E^*_{j,l+1}\cap \partial^0G_l$ for every $l\in\N$ (see Step 2 in the proof of Theorem~\ref{thmasympbdryreact}). Then we define $\widetilde E^*_{j,\Omega}:=\cup_l \widetilde E^*_{j,l}$, so that $\widetilde E^*_{j,\Omega}$ is an open subset of $\Omega$,   
$\widetilde E^*_{j,l}=\widetilde E^*_{j,\Omega}\cap  \partial^0G_l$ for every $l\in\N$, and $u_*=\sum_j\chi_{\widetilde E^*_{j,\Omega}}{\bf a}_j$ a.e. in $\Omega$.  Therefore 
$\widetilde E^*_{j,\Omega}=E^*_j\cap \Omega$ up to a set of vanishing measure for every $j\in\{1,\ldots,m\}$. With a slight abuse of notation, we may then write $E^*_j\cap \Omega=\widetilde E^*_{j,\Omega}$, so that $E^*_j\cap \Omega$ is indeed an open set for each $j$. 
Then the conclusions of  Theorem~\ref{thmasympbdryreact} in each $G_l$ clearly imply the announced results stated in (ii), (iv), (v), and (vi). 
\vskip3pt

To show that $\mathfrak{E}*=(E^*_1,\ldots,E_m^*)$ satisfies \eqref{meancurveq}, we consider a vector field $X\in C^1(\R^n;\R^n)$ compactly supported in $\Omega$, and $\mathbf{X}=(\mathbf{X}_1,\ldots,\mathbf{X}_{n+1})\in C^1(\overline{\R^{n+1}_+};\mathbb{R}^{n+1})$ compactly supported in $\R^{n+1}_+\cup\Omega$ satisfying $\mathbf{X}=(X,0)$ on $\Omega$. 
Since the support of $\mathbf{X}$ is contained in $G_l\cup\partial^0G_l$ for $l$ large enough, we can apply (v) in Theorem \ref{thmasympbdryreact}. In view of  \eqref{identfirstvarfracext}, we obtain 
$$\delta\mathcal{E}(u_*,\Omega)[X]
=\delta{\bf E}\big(u^\e_*,G_l\cup\partial^0G_l\big)[{\bf X}]= 0\,.$$
On the other hand, if $\{\phi_t\}_{t\in\R}$ denotes the flow generated by $X$, we observe that for every $t\in\R$, 
\begin{equation}\label{deformupart}
u_*\circ\phi_{-t}=\sum_{j=1}^m\chi_{\phi_t(E_j^*)} {\bf a}_j\,,\text{ and }\;\mathcal{E}(u_*\circ\phi_{-t},\Omega)=\frac{\gamma_{n,s}}{2}\mathscr{P}_{2s}^{\boldsymbol{\sigma}}\big(\phi_t(\mathfrak{E}^*),\Omega\big)\,,
\end{equation}
so that 
\begin{equation}\label{deformupart2}
\delta\mathscr{P}_{2s}^{\boldsymbol{\sigma}}(\mathfrak{E}^*,\Omega)[X]= \frac{2}{\gamma_{n,s}}\delta\mathcal{E}(u_*,\Omega)[X]=0\,.
\end{equation}
\vskip3pt

\noindent{\it Step 3.} Now we show items (i) and (iii). Let $\Omega^\prime\subset\Omega$ be a smooth open subset compactly included in $\Omega$. By Step 2, $u_k\to u_*$ strongly in $H^s(\Omega^\prime)$. On the other hand, since $\partial \Omega^\prime$ is smooth, 
$$\int_{\R^n\setminus\Omega^\prime}\frac{|u_k(x)-u_k(y)|^2}{|x-y|^{n+2s}}\,\de y\leq \big(4\sup_k\|u_k\|^2_{L^\infty(\R^n)}\big) \int_{\R^n\setminus\Omega^\prime}\frac{1}{|x-y|^{n+2s}}\,\de y\in L^1(\Omega^\prime)\,,$$
so that 
$$\iint_{\Omega^\prime\times(\R^n\setminus\Omega^\prime)}\frac{|u_k(x)-u_k(y)|^2}{|x-y|^{n+2s}}\,\de y\de x\to \iint_{\Omega^\prime\times(\R^n\setminus\Omega^\prime)}\frac{|u_*(x)-u_*(y)|^2}{|x-y|^{n+2s}}\,\de y\de x$$
by dominated convergence, and (i) follows. Observe that, since $u_k\rightharpoonup u_*$ in $\widehat H^s(\Omega^\prime)$, it implies that $\mathcal{E}_s(u_k-u_*,\Omega^\prime)\to 0$ as $k\to \infty$. 

To prove (iii), we first recall from \cite[Lemma 6.35]{MilSirW} (and the fact that $E^*_j$ has finite $2s$-perimeter in~$\Omega$) that the function $w_j:=\chi_{E^*_j}-\chi_{( E^*_j)^c}\in \widehat H^s(\Omega)$ satisfies $(-\Delta)^sw_j\in L^1(\Omega)$ with 
$$(-\Delta)^sw_j(x)=\Big(\frac{\gamma_{n,s}}{2}\int_{\R^n}\frac{|w_j(x)-w_j(y)|^2}{|x-y|^{n+2s}}\,\de y\Big)w_j(x) \quad\text{a.e. in $\Omega$}\,,$$
for each $j=1,\ldots,m$. Writing $w_j=2\chi_{E^*_j}-1$, we deduce that  $(-\Delta)^s\chi_{E^*_j}\in L^1(\Omega)$ with 
$$(-\Delta)^s\chi_{E^*_j}(x)=\Big(\gamma_{n,s}\int_{\R^n}\frac{|\chi_{E^*_j}(x)-\chi_{E^*_j}(y)|^2}{|x-y|^{n+2s}}\,\de y\Big) \big(2\chi_{E^*_j}(x)-1\big)= V_{E^*_j}(x)$$
for a.e. $x\in \Omega$. Therefore, $(-\Delta)^su_*\in L^1(\Omega)$ with 
\begin{equation}\label{eqpartquelq1}
 (-\Delta)^su_*=\sum_{j=1}^m V_{E^*_j}{\bf a}_j\quad\text{a.e. in $\Omega$}\,.
 \end{equation}
Since $\|(-\Delta)^s(u_k-u_*)\|_{H^{-s}(\Omega^\prime)}\leq 2\mathcal{E}_s(u_k-u_*,\Omega^\prime)\to 0$, we have $(-\Delta)^su_k\to (-\Delta)^su_*$ strongly in $H^{-s}(\Omega^\prime)$, and (iii) follows from equations \eqref{fracalllcahnf} and \eqref{eqpartquelq1}. 
\end{proof}

 \begin{proof}[Proof of Theorem \ref{main1part1mincase}] 
 Let $u\in \widehat H^s(\Omega;\mathcal{Z})$ be such that $u-u_*$ is compactly supported in $\Omega$. Since $\partial \Omega$ is assumed to be smooth, we can find an open subset $\Omega^\prime$ with smooth boundary such that $\overline{\Omega^\prime}\subset \Omega$ and $u-u_*$ is supported in $\Omega^\prime$. Then $\chi_{\Omega^\prime}$ has finite $2s$-perimeter in $\R^n$, and consequently, setting $\bar u_k:=\chi_{\Omega^\prime}u+(1-\chi_{\Omega^\prime})u_k$, 
 $$\mathcal{E}_s(\bar u_k,\Omega)\leq \frac{1}{2} [u]^2_{H^s(\Omega^\prime)}+ \mathcal{E}_s(u_k,\Omega)+CP_{2s}(\Omega^\prime,\R^n)<\infty\,,$$
 for a constant $C$ depending on $\|u\|_{L^\infty(\R^n)}$ and $\sup_k\|u_k\|_{L^\infty(\R^n)}$. In other words, $\bar u_k\in \widehat H^s(\Omega;\R^d)$, and by construction $\bar u_k-u_k$ is compactly supported in $\Omega^\prime$ (hence in $\Omega$). By minimality of $u_k$, we have 
\begin{equation}\label{eqmistak}
0\leq \mathcal{E}_{s,\eps_k}(\bar u_k,\Omega)-\mathcal{E}_{s,\eps_k}(u_k,\Omega)
 = \mathcal{E}_{s,\eps_k}(\bar u_k,\Omega^\prime)-\mathcal{E}_{s,\eps_k}(u_k,\Omega^\prime) \,.
 \end{equation}
Since $\bar u_k$ is $\mathcal{Z}$-valued in $\Omega^\prime$, we have
 $$ \mathcal{E}_{s,\eps_k}(\bar u_k,\Omega^\prime)= \mathcal{E}_{s}(\bar u_k,\Omega^\prime)=\frac{1}{2}[u]^2_{H^s(\Omega^\prime)} + \frac{\gamma_{n,s}}{2}\iint_{\Omega^\prime\times(\Omega^\prime)^c}\frac{|u(x)-u_k(y)|^2}{|x-y|^{n+2s}}\,\de x\de y\,.$$
 On the other hand,  
 $$ \frac{|u(x)-u_k(y)|^2}{|x-y|^{n+2s}}\leq \frac{C}{|x-y|^{n+2s}} \in L^1(\Omega^\prime\times(\Omega^\prime)^c)\,,$$
and $u_k\to u_*=u$ a.e. in $\R^n\setminus\Omega^\prime$. Hence, 
$$ \mathcal{E}_{s,\eps_k}(\bar u_k,\Omega^\prime)\to  \mathcal{E}_{s}(u,\Omega^\prime) \quad\text{as $k\to\infty$}\,.$$
 In view of (i) and (ii) in Theorem \ref{main1part1}, we also have  $\mathcal{E}_{s,\eps_k}(u_k,\Omega^\prime)\to  \mathcal{E}_{s}(u_*,\Omega^\prime)$. 
Therefore, letting $k\to\infty$ in \eqref{eqmistak}  yields 
 $$0\leq \mathcal{E}_{s}(u,\Omega^\prime)-\mathcal{E}_{s}(u_*,\Omega^\prime)= \mathcal{E}_{s}(u,\Omega)-\mathcal{E}_{s}(u_*,\Omega)\,,$$
 since $u-u_*$ is supported in $\Omega^\prime$. 
 
 Finally, if $\mathfrak{E}=(E_1,\ldots,E_m)\in\mathscr{A}_m(\Omega)$ is such that each $E_j\triangle E^*_j$ is compactly included in $\Omega$ , 
 then the map $u:=\sum_j\chi_{F_j}{\bf a}_j$ belongs to $\widehat H^s(\Omega;\mathcal{Z})$ and $u-u_*$ is compactly supported in $\Omega$. Hence, by minimality  of $u_*$, we have 
 \begin{equation}\label{minimPproof}
 \mathscr{P}_{2s}^{\boldsymbol{\sigma}}(\mathfrak{E}^*,\Omega)
 = \frac{2}{\gamma_{n,s}}\mathcal{E}(u_*,\Omega)\leq \frac{2}{\gamma_{n,s}}\mathcal{E}(u,\Omega)= \mathscr{P}_{2s}^{\boldsymbol{\sigma}}(\mathfrak{E},\Omega)\,,
 \end{equation}
 and the minimality of $\mathfrak{E}^*$ is proved. 
 \end{proof}

 \begin{proof}[Proof of Theorem \ref{main1part1mincase2}] 
According to Remark \ref{remarkLinftymin}, $\sup_k\|u_k\|_{L^\infty(\R^n)}\leq C$ for a constant $C$ depending on $W$ and $\sup_k\|g_k\|_{L^\infty(\R^n\setminus\Omega)}$. As in the previous proof, since $\Omega$ has a smooth boundary, $\Omega$ has finite $2s$-perimeter in $\R^n$. Setting $\bar u_k:=\chi_\Omega{\bf a}_1+(1-\chi_\Omega)g_k$, we have $\bar u_k\in H_{g_k}^s(\Omega;\R^d)$ with  
$$ \mathcal{E}_{s,\eps_k}(\bar u_k,\Omega)=\frac{\gamma_{n,s}}{2}\iint_{\Omega\times\Omega^c}\frac{|{\bf a}_1-g_k(y)|^2}{|x-y|^{n+2s}}\,\de x\de y\leq C P_{2s}(\Omega,\R^n)\,.$$
for a constant $C$ depending on $\sup_k\|g_k\|_{L^\infty(\R^n\setminus\Omega)}$. By minimality of $u_k$, we have $\mathcal{E}_{s,\eps_k}(u_k,\Omega)\leq \mathcal{E}_{s,\eps_k}(\bar u_k,\Omega)\leq \widetilde C$ for a constant $\widetilde C$ independent of $k$. Therefore Theorem \ref{main1part1} applies, and recalling the Step 1 of its proof, we have $u_k\rightharpoonup u_*$ weakly in $\widehat H^s(\Omega)$ with $ u_*\in \widehat H_g^s(\Omega;\mathcal{Z})$. Then we consider the competitor $\widetilde u_k:=\chi_\Omega u_*+(1-\chi_\Omega)g_k\in H_{g_k}^s(\Omega;\R^d)$ which satisfies
$$ \mathcal{E}_{s,\eps_k}(\widetilde u_k,\Omega)=\frac{1}{2}[u_*]^2_{H^s(\Omega)}+\frac{\gamma_{n,s}}{2}\iint_{\Omega\times\Omega^c}\frac{|u_*(x)-g_k(y)|^2}{|x-y|^{n+2s}}\,\de x\de y \,.$$
We have 
$$\frac{|u_*(x)-g_k(y)|^2}{|x-y|^{n+2s}}\leq \frac{\widehat C}{|x-y|^{n+2s}}\in L^1(\Omega\times\Omega^c)\text{ for a.e. }(x,y)\in  \Omega\times\Omega^c\,,$$
For a constant $\widehat C$ depending on $\mathcal{Z}$ and  $\sup_k\|g_k\|_{L^\infty(\R^n\setminus\Omega)}$. Since $g_k\to g$ a.e. in $\Omega^c$, we have 
$  \mathcal{E}_{s,\eps_k}(\widetilde u_k,\Omega)\to  \mathcal{E}_{s}(u_*,\Omega)$ by dominated convergence. On the other hand, by lower semi-continuity of $  \mathcal{E}_{s}$ with respect to weak convergence in $\widehat H^s$ (see \cite[Remark 2.2]{MilPegSch}), we have 
$$ \mathcal{E}_{s}(u_*,\Omega)\leq \liminf_{k}  \mathcal{E}_{s}(u_k,\Omega)\,.$$  
However, by minimality of $u_k$, 
$$ \limsup_{k}  \mathcal{E}_{s,\eps_k}( u_k,\Omega)\leq \lim_{k}  \mathcal{E}_{s,\eps_k}(\widetilde u_k,\Omega)=  \mathcal{E}_{s}(u_*,\Omega)\,.$$
Consequently, $\eps_k^{2s}\int_\Omega W(u_k)\,\de x\to 0$ and  $\mathcal{E}_{s}( u_k,\Omega)\to  \mathcal{E}_{s}( u_*,\Omega)$. Since $u_k\rightharpoonup u_*$ weakly in $\widehat H^s(\Omega)$, it classically implies that  $u_k\to u_*$ strongly in $\widehat H^s(\Omega)$, and in particular strongly in $H^s(\Omega)$. 

Finally, if $\widehat u\in \widehat H_g^s(\Omega;\mathcal{Z})$ is arbitrary, we set $\widehat u_k:=\chi_\Omega \widehat u+(1-\chi_\Omega)g_k$. Arguing as above, we have  $\widehat u_k\in H_{g_k}^s(\Omega;\R^d)$, and 
$$ \mathcal{E}_{s}( u_*,\Omega)\mathop{\longleftarrow}\limits_{k\to\infty}  \mathcal{E}_{s,\eps_k}( u_k,\Omega) \leq \mathcal{E}_{s,\eps_k}(\widehat u_k,\Omega)\mathop{\longrightarrow}\limits_{k\to\infty} \mathcal{E}_{s}( \widehat u,\Omega)\,.$$
Arguing once again as above, for $\mathfrak{E}=(E_1,\ldots,E_m)\in\mathscr{A}_m(\Omega)$ satisfying $u:=\sum_j\chi_{F_j}{\bf a}_j = g$ a.e. in $\R^n\setminus\Omega$, we have 
$u\in \widehat H_g^s(\Omega;\mathcal{Z})$, and thus  \eqref{minimPproof} holds,  i.e.,  $\mathfrak{E}*$ solves the minimization problem~\eqref{limitminpb}. 
 \end{proof}


\section{Nonlocal minimal partitions}\label{GHSNMC}


In this section, we investigate regularity properties in a smooth bounded open set $\Omega\subset\R^n$ of a partition $\mathfrak{E}^*=(E^*_1,\ldots,E^*_m)\in\mathscr{A}_m(\Omega)$ 
which is stationary in $\Omega$ for the functional $\mathscr{P}^{\boldsymbol{\sigma}}_{2s}(\cdot,\Omega)$, that is $\mathscr{P}^{\boldsymbol{\sigma}}_{2s}\big(\cdot,\Omega\big)$ has a vanishing first inner variation at  $\mathfrak{E}$, i.e., 
\begin{equation}\label{SectParteq1}
\delta \mathscr{P}^{\boldsymbol{\sigma}}_{2s}(\mathfrak{E}^*,\Omega)[X]=0 \qquad\forall X\in C_c^1(\Omega;\R^n)\,,
\end{equation}
where $\delta \mathscr{P}^{\boldsymbol{\sigma}}_{2s}$  is the first variation defined by \eqref{deffirstvarparti}. 
Recall that the set of admissible partitions $\mathscr{A}_m(\Omega)$  is made of finite energy, 
 or equivalently made of partitions $\mathfrak{E}=(E_1,\ldots,E_m)$ such that each $E_j$ has finite $2s$-perimeter in $\Omega$ since 
$$\frac{\boldsymbol{\sigma}_{\rm min}}{2}\sum_{j=1}^mP_{2s}(E_j,\Omega) \leq \mathscr{P}^{\boldsymbol{\sigma}}_{2s}(\mathfrak{E},\Omega)\leq \frac{\boldsymbol{\sigma}_{\rm max}}{2}\sum_{j=1}^mP_{2s}(E_j,\Omega) \,,$$
where $\boldsymbol{\sigma}_{\rm min}:=\min_{i\not=j}|\sigma_{ij}|>0$ and $\boldsymbol{\sigma}_{\rm max}:=\max_{i\not=j}|\sigma_{ij}|>0$. 
\vskip3pt
As we have seen in the previous section,  stationary partitions arise as limits of critical points of the fractional Allen-Cahn energy with multiple wells. 
Among the whole class of  stationary partitions  in $\Omega$,
 one has the family of minimizing partitions  in $\Omega$.  

\begin{definition}\label{defminpart}
A partition $\mathfrak{E}\in\mathscr{A}_m(\Omega)$  is said to be $\mathscr{P}^{\boldsymbol{\sigma}}_{2s}$-minimizing in $\Omega$ if 
$$ \mathscr{P}^{\boldsymbol{\sigma}}_{2s}(\mathfrak{E},\Omega)\leq  \mathscr{P}^{\boldsymbol{\sigma}}_{2s}\big(\mathfrak{F},\Omega\big)$$
for every  $\mathfrak{F}=(F_1,\ldots,F_m)\in\mathscr{A}_m(\Omega)$ such that each $\overline{F_j\triangle E_j}\subset\Omega$,  $j\in\{1,\ldots,m\}$.  
\end{definition}

Minimizing partitions arise for instance as solutions of the minimization problem 
\begin{equation}\label{PartminProb}
\min\Big\{\mathscr{P}^{\boldsymbol{\sigma}}_{2s}(\mathfrak{E},\Omega) : \text{$\mathfrak{E}=(E_1,\ldots,E_m) \in\mathscr{A}_m(\Omega)$, $E_j\setminus\Omega=\bar E_j\setminus\Omega$ $\forall j$} \Big\}\,, 
\end{equation}
for a given exterior data   $\bar{\mathfrak{E}}=(\bar E_1,\ldots,\bar E_m) \in\mathscr{A}_m(\Omega)$.  
Once again the resolution of \eqref{PartminProb} follows from the Direct Method of Calculus of Variations and provides examples of  stationary partitions in $\Omega$, i.e., solutions to \eqref{SectParteq1}. Solutions of \eqref{PartminProb} can also be obtained as singular limits of solutions of minimization problems for the Allen-Cahn energy, as seen in Theorem~\ref{main1part1mincase2}. However, both in the minimizing and in the stationary case, this ``Allen-Cahn approximation'' holds if each coefficient of the matrix $\boldsymbol{\sigma}$ corresponds to the Euclidean distance squared between points in some finite dimensional linear space. We shall also use this assumption here.  
\vskip5pt

Throughout this section,  we shall assume that the matrix of coefficients $\boldsymbol{\sigma}$ belongs to $\mathscr{S}_m^2$ which means that the matrix $(\sqrt{\sigma_{ij}})_{1\leq i,j\leq m}$ is {\sl $\ell^2$-embeddable}: there exists a dimension  $d\in\N\setminus\{0\}$ and a finite set $\mathcal{Z}_{\boldsymbol{\sigma}}=\{{\bf a}_1,\ldots,{\bf a}_m\}\subset\R^d$ such that 
$$\sqrt{{\sigma}_{ij}}=|{\bf a}_i-{\bf a}_j|\quad\forall i,j\in\{1,\ldots,m\}\,. $$
In this case, introducing for a partition $\mathfrak{E}=(E_1,\ldots,E_m)\in\mathscr{A}_m(\Omega)$ the {\sl ``phase indicator function"} 
\begin{equation}\label{basdec1}
u:=\sum_{j=1}^m\chi_{E_j}{\bf a}_j\in\widehat H^s(\Omega;\mathcal{Z}_{\boldsymbol{\sigma}})\,,
\end{equation}
the energy rewrites 
\begin{equation}\label{basdec2}
\frac{\gamma_{n,s}}{2}\mathscr{P}^{\boldsymbol{\sigma}}_{2s}(\mathfrak{E},\Omega)=\mathcal{E}_s(u,\Omega)\,,
\end{equation}
and condition  \eqref{SectParteq1} is equivalent to the stationarity of the indicator function $u_*$ of $\mathfrak{E}_*$, i.e., 
\begin{equation}\label{campdev1847}
\delta\mathcal{E}_s(u_*,\Omega)[X]=0\qquad\forall X\in C_c^1(\Omega;\R^n)\,,
\end{equation}
as we have seen in \eqref{deformupart}-\eqref{deformupart2}.    
\vskip3pt

Similarly, the partition $\mathfrak{E}^*$ is $\mathscr{P}^{\boldsymbol{\sigma}}_{2s}$-minimizing in $\Omega$ if and only if $u_*$  
is a $\mathcal{E}_s$-minimizing map into $\mathcal{Z}_{\boldsymbol{\sigma}}$ in $\Omega$ (by \eqref{minimPproof}) according to the following definition. 

\begin{definition}
A map $u\in \widehat H^s(\Omega;\mathcal{Z}_{\boldsymbol{\sigma}})$ is said to be a $\mathcal{E}_s$-minimizing map into $\mathcal{Z}_{\boldsymbol{\sigma}}$ in $\Omega$ if 
$$\mathcal{E}_s(u,\Omega)\leq \mathcal{E}_s(\widetilde u,\Omega) $$
for all $\widetilde u\in \widehat H^s(\Omega;\mathcal{Z}_{\boldsymbol{\sigma}})$ such that $\widetilde u-u$ is compactly supported in $\Omega$. 
\end{definition}
\vskip5pt

As we already did for the fractional Allen-Cahn equation, our analysis rely on the fractional harmonic extension $u_*^\e$ defined in  \eqref{poisson} of  $u_*$  which satisfies 
\begin{equation}\label{campdev2}
\begin{cases}
{\rm div}(z^a\nabla u_*^\e)=0 & \text{in $\R^{n+1}_+$}\,,\\
u_*^\e\in\mathcal{Z}_{\boldsymbol{\sigma}} & \text{on $\R^n$}\,,\\
|u_*^\e|\leq |{\bf a}|_{\rm max} & \text{in $\R^{n+1}_+$}\,,
\end{cases}
\end{equation}
where we have set $|{\bf a}|_{\rm max}:=\max_j|{\bf a}_j|$. According to \eqref{identfirstvarfracext}, equation \eqref{campdev1847} is in turn equivalent to 
\begin{equation}\label{stateeq1458}
\delta{\bf E}_s\big(u_*^\e,G\cup\partial^0G\big)[{\bf X}]= 0
\end{equation}
for every vector field ${\bf X}=(X,{\bf X}_{n+1})\in C^1(\overline G;\R^{n+1})$ compactly supported in $G\cup\partial^0G$ satisfying ${\bf X}_{n+1}=0$ on $\partial^0 G$, whenever $G\subset\R^{n+1}_+$ is an admissible bounded open set such that $\overline{\partial^0G}\subset \Omega$. 
\vskip5pt

By our assumption $\boldsymbol{\sigma}\in \mathscr{S}_m^2$, our aim to investigate the regularity properties in $\Omega$ of stationary or minimizing partitions  rephrases as a regularity problem for the discontinuity set of the indicator function $u_*$. In turn this can be further rephrased  as a boundary regularity problem for solutions of \eqref{campdev2}-\eqref{stateeq1458}. Following the approach in \cite{MilSirW}, and similarly to Sections \ref{FractAC}  \& \ref{FGLasymp}, we shall deal with the following more general situation. 
 We consider an admissible bounded  open set $G\subset \R^{n+1}_+$ and a  map $v\in H^{1}(G;\R^d,|z|^a\de{\bf x})\cap L^\infty(G)$ satisfying 
\begin{equation}\label{startequsurfnew}
\begin{cases}
{\rm div}(z^a\nabla v)=0 & \text{in $G$}\,,\\
v\in\mathcal{Z}_{\boldsymbol{\sigma}} & \text{on  $\partial^0 G$}\,,\\
|v|\leq b & \text{in $G$}\,,
\end{cases}
\end{equation}
for a given parameter $b\geq 1$, and 
\begin{equation}\label{critnonlocexteqnew}
\delta{\bf E}_s\big(v,G\big)[{\bf X}]= 0\,,
\end{equation}
for every vector field ${\bf X}=(X,{\bf X}_{n+1})\in C^1(\overline G;\R^{n+1})$ compactly supported in $G\cup\partial^0G$ satisfying ${\bf X}_{n+1}=0$ on $\partial^0 G$. 
\vskip3pt

Boundary regularity estimates at $\partial^0 G$ on the function $v$ satisfying \eqref{startequsurfnew}-\eqref{stateeq1458}  will be our main concern in this section (note that $v$ is smooth inside $G$ by classical elliptic theory). The application to solutions of \eqref{campdev1847} is the object of the very last subsection with some specific results. Improvements in the minimizing case will be the object of Section \ref{regminpartsect} with partial regularity results for the boundaries. In the sequel, many arguments are borrowed from~\cite{MilSirW}, but we often prefer to include full proofs and details for completeness.


\subsection{Energy monotonicity and clearing-out}\label{subres1} 


In this subsection, we consider a given admissible bounded  open set $G\subset \R^{n+1}_+$ and 
an arbitrary solution $v\in H^{1}(G;\R^d,|z|^a\de{\bf x})\cap L^\infty(G)$ to \eqref{startequsurfnew}-\eqref{critnonlocexteqnew}. Since the trace of $v$ on $\partial^0 G$ is $\mathcal{Z}_{\boldsymbol{\sigma}}$-valued, we can find $m$ disjoint measurable subsets $E^v_j\subset \partial^0 G$ such that 
\begin{equation}\label{partitonthebdry}
\sum_{j=1}^m \chi_{E^v_j}=1\quad\text{and}\quad v=\sum_{j=1}^m \chi_{E^v_j}{\bf a}_j\quad\text{a.e. on $\partial^0G$}\,.
\end{equation}
We begin with the fundamental monotonicity formula involving the following density function:  
for a point $\mathbf{x}_0=(x_0,0)\in\partial^0 G$ and $r>0$ such that $B^+_r(x_0)\subset G$, we set  
$$\boldsymbol{\Theta}_s(v,x_0,r):=\frac{1}{r^{n-2s}}\mathbf{E}_s\big(v,B^+_r({\bf x}_0)\big)\,.$$ 
Arguing exactly as in the proof of \cite[Lemma 6.2]{MilSirW}, equation 
 \eqref{critnonlocexteqnew} combined with formula \eqref{firstvarEsformul} yields the following monotonicity formula. 

\begin{lemma}\label{monotformsurf}
For every $\mathbf{x}_0=(x_0,0)\in  \partial^0 G$ and $r>\rho>0$ such that $B^+_r(\mathbf{x}_0)\subset G$,  
$$\boldsymbol{\Theta}_s(v,x_0,r)- \boldsymbol{\Theta}_s(v,x_0,\rho)= 
\boldsymbol{\delta}_s\int_{B^+_r({\bf x}_0)\setminus B^+_\rho({\bf x}_0)} z^a\frac{|(\mathbf{x}-{\bf x}_0)\cdot \nabla v|^2}{|\mathbf{x}-{\bf x}_0|^{n+2-2s}}\,\de{\bf x}\,.$$
\end{lemma}

In turn, the monotonicity above allows to define to so-called density function, a main tool in the regularity theory as illustrated in the following Lemma \ref{epsregminsurf}. 

\begin{corollary}\label{directcorlmonot}
For every $\mathbf{x}=(x,0)\in  \partial^0 G$, the limits 
$$\boldsymbol{\Theta}_s(v,x):=\lim_{r\downarrow 0} \boldsymbol{\Theta}_s(v,x,r)=\lim_{r\downarrow 0}  \frac{1}{r^{n-2s}}\mathbf{E}_s\big(v,B^+_r({\bf x}_0)\big)$$
exist, and the function $\boldsymbol{\Theta}_s(v,\cdot):\partial^0 G \to[0,\infty)$ is upper semicontinuous. In addition, 
\begin{equation}\label{cdv1117}
\boldsymbol{\Theta}_s(v,x_0,r)- \boldsymbol{\Theta}_s(v,x_0)= \boldsymbol{\delta}_s\int_{B^+_r({\bf x}_0)} z^a\frac{|(\mathbf{x}-{\bf x}_0)\cdot \nabla v|^2}{|\mathbf{x}-{\bf x}_0|^{n+2-2s}}\,\de{\bf x}\,.
\end{equation}
\end{corollary}

\begin{proof}
The existence defining $\boldsymbol{\Theta}_s(v,x)$ is of course a direct consequence of the monotonicity of the density function stated in Lemma \ref{monotformsurf}.  
Then $\boldsymbol{\Theta}_s(v,\cdot)$ is upper semicontinuous as a pointwise limit of a decreasing family of continuous functions. Finally, letting $\rho\to 0$ in Lemma \ref{monotformsurf} yields  \eqref{cdv1117}.  
\end{proof}

As briefly mentioned above, the following clearing-out property can be seen as a {\sl small-energy} regularity result. 

\begin{lemma}[\bf Clearing-out]\label{epsregminsurf}
There exists  a constant ${\boldsymbol{\eta}}_0>0$ (depending only on ${\boldsymbol{\sigma}}_{\rm min}$, $n$, and $s$)  such that the following holds. For 
$\mathbf{x}_0=(x_0,0)\in \partial^0 G$ and $r >0$ such that $B^+_r(\mathbf{x}_0)\subset G$,  the condition 
$$\boldsymbol{\Theta}_s(v,x_0,r)\leq {\boldsymbol{\eta}}_0$$ 
implies that   $v={\bf a}$ a.e. on $D_{r/2}(x_0)$ for some ${\bf a}\in\mathcal{Z}_{\boldsymbol{\sigma}}$ (i.e., $D_{r/2}(x_0)\subset E^v_j$ up to a negligeable set for some $j\in\{1,\ldots,m\}$). 
\end{lemma}

\begin{proof}
By Lemma \ref{monotformsurf}, for $0<\rho<r/4$ and $y\in D_{r/2}(x_0)$, we have
\begin{equation}\label{clearouteq1}
\boldsymbol{\Theta}_s(v,y,2\rho) \leq \boldsymbol{\Theta}_s(v,y,r/2)\leq 2^{n-2s} \boldsymbol{\Theta}_s(v,x_0,r)\leq 2^{n-2s}{\boldsymbol{\eta}}_0\,.
\end{equation}
On the other hand, according to \cite[Lemma 2.8]{MilPegSch}, 
\begin{equation}\label{clearouteq2}
[v]^2_{H^s(D_{\rho}(y))}\leq {\rm c}_{n,s}  \mathbf{E}_s\big(v,B^+_{2\rho}({\bf y})\big) \,,
\end{equation}
where ${\bf y}=(y,0)$ and ${\rm c}_{n,s}$ is a constant depending only on $n$ and $s$. 
Then, considering for each $j\in\{1,\ldots,m\}$, the Lipschitz function ${\rm p}_j:\R^d\to \R$ defined by \eqref{projtrickEj}, we have 
$$\chi_{E^v_j}={\rm p}_j\big(v_{|\partial^0G}\big)\,.$$ 
Noticing that ${\rm p}_j$ is ${\boldsymbol{\sigma}}^{-1/2}_{\rm min}$-Lipschitz, we infer from \eqref{clearouteq1} and   \eqref{clearouteq2}  that
\begin{equation}\label{clearouteq3}
  \frac{1}{\rho^{n-2s}}[\chi_{E^v_j}]^2_{H^s(D_{\rho}(y))}\leq \frac{1}{\boldsymbol{\sigma}_{\rm min}\rho^{n-2s}}[v]^2_{H^s(D_{\rho}(y))}\leq \frac{2^{2n-4s}{\rm c}_{n,s}}{\boldsymbol{\sigma}_{\rm min}}\, {\boldsymbol{\eta}}_0\quad \text{for $j=1,\ldots,m$}\,,
  \end{equation}
for every $y\in  D_{r/2}(x_0)$ and $0<\rho<r/4$. 

Setting for $j=1,\ldots,m$, $[v]_{y,\rho}$ to be the average of $v$ over $D_\rho(y)$, and 
$$A^j_\rho(y):=\frac{1}{\rho^n}\int_{D_\rho(y)}\big|\chi_{E^v_j}-[\chi_{E^v_j}]_{y,\rho}\big|\,\de x\,,$$ 
we deduce from \eqref{clearouteq3} and the  Poincar\'e inequality in Lemma \ref{poincareHs} that 
$$\sup_{y\in D_{r/2}(x_0)\,,\;\rho\in(0,r/4)}A^j_\rho(y)\leq  \frac{2^{n-2s}\boldsymbol{\lambda}_{n,s}\sqrt{{\rm c}_{n,s}}}{\boldsymbol{\sigma}^{1/2}_{\rm min}}\sqrt{{\boldsymbol{\eta}}_0}\qquad\forall j\in\{0,\ldots,m\}\,.$$
Then we notice that for each $j\in\{1,\ldots,m\}$, 
$$A^j_\rho(y)= 2\omega_n\left(1-\frac{|E^v_j\cap D_\rho(y)|}{|D_\rho|}\right)\frac{|E^v_j\cap D_\rho(y)|}{|D_\rho|}\,.$$
Choosing 
$${\boldsymbol{\eta}}_0:=\frac{9\omega_n^2\boldsymbol{\sigma}_{\rm min}}{2^{2n+6-4s}\boldsymbol{\lambda}^2_{n,s}} $$
leads to  $A^j_\rho(y)\leq 3\omega_n/8$ for each $j\in\{1,\ldots,m\}$.  In turn, this inequality implies that for each $j\in\{1,\ldots,m\}$, we have 
$$\frac{|E^v_j\cap D_\rho(y)|}{|D_\rho|} \in [0,1/4]\cup[3/4,1]\,.$$ 
Since each function $(y,\rho)\in D_{r/2}(x_0)\times(0,r/4)\mapsto |E^v_j\cap D_\rho(y)|/|D_\rho|$ is continuous, we infer that for  $j=1,\ldots,m$, either 
\begin{equation}\label{density0}
\frac{|E^v_j\cap D_\rho(y)|}{|D_\rho|} \in [0,1/4] \quad\text{for every $y\in D_{r/2}(x_0)$ and every $0<\rho<r/2$}\,,
\end{equation}
or 
\begin{equation}\label{density1}
\frac{|E^v_j\cap D_\rho(y)|}{|D_\rho|} \in [3/4,1] \quad\text{for every $y\in D_{r/2}(x_0)$ and every $0<\rho<r/2$}\,.
\end{equation}
Given $j\in\{1,\ldots,m\}$, assume that \eqref{density1} holds (the other case being analogous). Then, by the Lebesgue differentiation theorem, we deduce that a.e. $y\in D_{r/2}(x_0)$ is a point of density $1$ for the set~$E^v_j$. Consequently, $\chi_{E^v_j}=1$ a.e. in $D_{r/2}(x_0)$. On the contrary, if  \eqref{density0} holds, then $\chi_{E^v_j}=0$ a.e. in $D_{r/2}(x_0)$.`

We have thus proved that for each $j\in\{1,\ldots,m\}$, either $\chi_{E^v_j}=1$ a.e. in $D_{r/2}(x_0)$ or $\chi_{E^v_j}=0$ a.e. in $D_{r/2}(x_0)$. Since $\sum_j\chi_{E^v_j}=1$, it implies that there exists  $j_0\in\{1,\ldots,m\}$ such that $\chi_{E^v_{j_0}}=1$ a.e. in $D_{r/2}(x_0)$ and  $\chi_{E^v_{k}}=0$ a.e. in $D_{r/2}(x_0)$ for $k\not=j_0$. In other words, $v={\bf a}_{j_0}$ on $D_{r/2}(x_0)$. 
\end{proof}

\begin{corollary}\label{openset}
For every ${\bf x}=(x,0)\in\partial^0G$, either $\boldsymbol{\Theta}_s(v,x)=0$ or $\boldsymbol{\Theta}_s(v,x)\geq {\boldsymbol{\eta}}_0$. As a consequence,  each $E^v_j \subset \partial^0G$ is essentially (relatively) open, i.e.,  there are $m$ disjoint open subsets $\widetilde E^v_1,\ldots, \widetilde E^v_m$ of $\partial^0G$ such that $|\widetilde E^v_j\triangle E^v_j |=0$ for $j=1,\ldots,m$. 
In addition, setting $\Sigma(v):=\big\{\boldsymbol{\Theta}_s(v,\cdot)\geq {\boldsymbol{\eta}}_0\big\}$, we have $\Sigma(v)=\bigcup_j(\partial\widetilde E^v_j\cap\partial^0G)$. 
\end{corollary}

\begin{proof}
Let ${\bf x}=(x,0)\in\partial^0G$ be such that $\boldsymbol{\Theta}_s(v,x)< {\boldsymbol{\eta}}_0$. By Corollary \ref{directcorlmonot}, we can find $r>0$ such that $B^+_r({\bf x})\subset G$ and $\boldsymbol{\Theta}_s(v,x,r)\leq {\boldsymbol{\eta}}_0$. Then Lemma \ref{epsregminsurf} implies that $v={\bf a}$ a.e. on $D_{r/2}(x)$ for some ${\bf a}\in\mathcal{Z}_{\boldsymbol{\sigma}}$. Consequently, $v\in H^1(B^+_{r/2}({\bf x}))\cap L^\infty(B^+_{r/2}({\bf x}))$ satisfies
$$\begin{cases} 
{\rm div}(z^{a}\nabla v) = 0 & \text{in $B^+_{r/2}({\bf x})$}\,,\\
v = {\bf a}  & \text{on $D_{r/2}(x)$}\,.
\end{cases}$$
By \cite[Lemma 4.8]{MilSirW}, we have 
$$\boldsymbol{\Theta}_s(v,x)=\lim_{\rho\to 0}  \frac{1}{\rho^{n-2s}}\mathbf{E}_s\big(v,B^+_\rho({\bf x}_0)\big)=0\,.$$
Consequently, for every ${\bf x}=(x,0)\in\partial^0G$, either $\boldsymbol{\Theta}_s(v,x)=0$ or $\boldsymbol{\Theta}_s(v,x)\geq {\boldsymbol{\eta}}_0$. 

By upper semicontinuity of $\boldsymbol{\Theta}_s(v,\cdot)$, the set $\Sigma(v):=\{\boldsymbol{\Theta}_s(v,\cdot)\geq {\boldsymbol{\eta}}_0\}$ is relatively closed in $\partial^0G$. By \cite[Corollary 3.2.3]{Zi}, we  have $\mathscr{H}^{n-2s}(\Sigma(v))=0$, and in particular $\mathscr{L}^n(\Sigma(v))=0$.

Next we set for $j=1,\ldots,m$, 
$$\widetilde E^v_j :=\Big\{{\bf x}=(x,0)\in\partial^0G : \text{$v={\bf a}_j$  a.e. in $D_r(x)$ for some $r\in(0,{\rm dist}(x,\partial^+B_r({\bf x})))$}\Big \}\,.$$
By construction, the  $\widetilde E^v_j$'s are disjoint open subsets of $\partial^0G$, and each $\widetilde E^v_j$ is disjoint from $\Sigma(v)$ by the argument above. 
Arguing as in the proof of Theorem \ref{thmasympbdryreact}, Step 2, we obtain $\partial^0G=\bigcup_j\widetilde E^v_j\cup\Sigma(v)$. Since each point of  $\widetilde E^v_j$ is obviously a point of density $1$ of $E^v_j$, and $\mathscr{L}^n(\Sigma(v))=0$, we infer that $|\widetilde E^v_j\triangle E_j|=0$ for each $j=1,\ldots,m$. 
Arguing again as in the proof of Theorem \ref{thmasympbdryreact}, we finally derive that $\Sigma(v):=\bigcup_j(\partial\widetilde E^v_j\cap\partial^0G)$. 
\end{proof}

\noindent{\bf Convention.} In what follows, we may and we shall identify the open sets $\widetilde E^v_j$ with $E^v_j$. In other words, in the statement of Corollary \ref{openset}, we can read ``$E^v_j$ is an open subset of $\partial^0G$ for every $j\in\{1,\ldots,m\}$'', and we have 
\begin{equation}\label{unionofbdeqsigm}
\Sigma(v)=\bigcup_{j=1}^m(\partial E^v_j\cap\partial^0G)\,.
\end{equation}


\subsection{Compactness}\label{subsectcompsurf}


In this subsection, we are interested with the compactness issues of sequences  $\{v_k\}_{k\in\mathbb{N}}\subset H^{1}(G;\R^d,|z|^a\de{\bf x})\cap L^\infty(G)$ satisfying for each $k\in\mathbb{N}$, 
$$\begin{cases}
{\rm div}(z^a\nabla v_k) =0 & \text{in $G$}\,,\\
v_k\in\mathcal{Z}_{\boldsymbol{\sigma}} & \text{on $\partial^0 G$}\,,\\
|v_k|\leq b & \text{in $G$}\,,
\end{cases}$$
and 
\begin{equation}\label{critnonlocexteqcoucou}
\delta\mathbf{E}_s(v_k,G)[X]= 0
\end{equation}
for every vector field ${\bf X}=(X,{\bf X}_{n+1})\in C^1(\overline G;\R^{n+1})$ compactly supported in $G\cup\partial^0G$ satisfying ${\bf X}_{n+1}=0$ on $\partial^0 G$, 
and for some parameter $b\geq 1$ independent of $k$. 

\begin{theorem}\label{compactminsurf}
If $\sup_k\mathbf{E}_s(v_k,G)<\infty$,  then there exist a (not relabeled) subsequence and a map $v\in H^{1}(G;\R^d,|z|^a\de{\bf x})\cap L^\infty(G)$ satisfying \eqref{startequsurfnew}-\eqref{critnonlocexteqnew} such that $v_k\rightharpoonup v $  weakly in $ H^{1}(G,|z|^a\de{\bf x})$ and strongly in $ H_{\rm loc}^{1}(G\cup\partial^0G ,|z|^a\de{\bf x})$.
\end{theorem}

\begin{proof}
The argument  follows the proof of Theorem \ref{thmasympbdryreact} and this reason, we only sketch the main points. 
First, by assumption on the energy and standard elliptic theory,  we can find a (not relabelled) subsequence and $v\in H^{1}(G;\R^d,|z|^a\de{\bf x})\cap L^\infty(G)$ 
satisfying \eqref{startequsurfnew} such that $v_k\rightharpoonup v$ weakly in  $H^{1}(G,|z|^a\de{\bf x})$ and strongly in $ H_{\rm loc}^{1}(G ,|z|^a\de{\bf x})$. 
Then we consider the sequence of measures $\mu_k:=\frac{\boldsymbol{\delta}_s}{2}z^a|\nabla v_k|^2 \mathscr{L}^{n+1}\LL G$. Up to a further subsequence, the sequence $\{\mu_k\}$  
 weakly* converges towards a limiting measure $\mu:= \frac{d_s}{2}z^a|\nabla v|^2 \mathscr{L}^{n+1}\LL G +\mu_{\rm sing}$ with ${\rm spt}(\mu_{\rm sing})\subset\partial^0G$. From Lemma \ref{monotformsurf}, we infer that $\mu$ satisfies the monotonicity inequality \eqref{tim1213new}. As a consequence, the density $\Theta^{n-2s}(\mu,{\bf x})$  (as defined in \eqref{existdens}) exists and is finite for every ${\bf x}\in\partial^0G$. It also defines an upper semicontinuous function on $\partial^0G$. We define the concentration set $\Sigma_{\rm sing}$ as in \eqref{defconcentrset} with $\boldsymbol{\theta}_{b}$ replaced by $\boldsymbol{\eta}_0$.  
Then $\Sigma_{\rm sing}=\big\{\Theta^{n-2s}(\mu,\cdot)\geq \boldsymbol{\eta}_0/\omega_{n-2s} \big\}\subset\partial^0G$, and $\mathscr{H}^{n-2s}(\Sigma_{\rm sing})$ is finite. We continue exactly as in the proof of Theorem \ref{thmasympbdryreact} to show that $\mu_{\rm sing}$ is absolutely continuous with respect $\mathscr{H}^{n-2s}\res\Sigma_{\rm sing}$, and that $\Theta^{n-2s}(\mu_{\rm sing},{\bf x})\in[0,\infty)$ exists at $\mathscr{H}^{n-2s}$-a.e. ${\bf x}\in \Sigma_{\rm sing}$. By Marstrand's Theorem, we must have $\mu_{\rm sing}\equiv 0$. In other words, $v_k\to v$ strongly in $H^1_{\rm loc}(G\cup\partial^0G,|z|^a\de{\bf x})$. In view of \eqref{firstvarEsformul},  this strong convergence allows us to pass to the limit $k\to\infty$ in \eqref{critnonlocexteqcoucou} and obtain \eqref{critnonlocexteqnew}. 
\end{proof}

\begin{remark}
If $v_k\to v$ strongly in $ H_{\rm loc}^{1}(G\cup\partial^0G ,|z|^a\de{\bf x})$, ${\bf x}_k=(x_k,0)\to {\bf x}=(x,0)$  
 and $r_k\to r>0$ with $ \overline B^+_r({\bf x})\subset G\cup\partial^0G$, then $\boldsymbol{\Theta}_s(v_k,x_k,r_k)\to\boldsymbol{\Theta}_{s}(v,x,r)$. 
\end{remark}

\begin{lemma}\label{updenstheta}
In addition to the conclusion of Theorem \ref{compactminsurf}, if $\{(x_k,0)\}_{k\in\mathbb{N}}\subset\partial^0G$ is a sequence converging to $(x,0)\in\partial^0G$, then 
$$ \limsup_{k\to\infty} \,\boldsymbol{\Theta}_s(v_k,x_k)\leq \boldsymbol{\Theta}_s(v,x)\,.$$
\end{lemma}

\begin{proof}
Assume for simplicity that $x=0$. Applying Corollary \ref{directcorlmonot}, we obtain for $r>0$ sufficiently small and $r_k:=|x_k|<r$, 
$$ \boldsymbol{\Theta}_s(v_k,x_k) \leq  \boldsymbol{\Theta}_s(v_k,x_k,r) \leq \frac{1}{r^{n-2s}} {\bf E}_s(v_k,B^+_{r+r_k}) \,. $$
Since $r_k\to0$ and $v_k$ converges strongly to $v$ in $H^1(B^+_{2r},|z|^a\de{\bf x})$, we have $ {\bf E}_s(v_k,B^+_{r+r_k}) \to  {\bf E}_s(v,B^+_{r}) $. Hence
$$\limsup_{k\to\infty} \, \boldsymbol{\Theta}_s(v_k,x_k) \leq   \frac{1}{r^{n-2s}} {\bf E}_s(v,B^+_{r}) \,.$$
Letting $r\downarrow0$ now leads to the conclusion.
\end{proof}

\begin{corollary}\label{Hausdcvbd}
In addition to the conclusion of Theorem \ref{compactminsurf} and according to \eqref{unionofbdeqsigm}, the sequence of sets $\Sigma(v_k)$ converges locally uniformly in $\partial^0G$ to $ \Sigma(v)$, i.e., for every compact subset $K\subset\partial^0G$ and every $r>0$, 
$$\Sigma(v_k)\cap K \subset\mathscr{T}_r(\Sigma(v)) \quad\text{and}\quad \Sigma(v)\cap K\subset \mathscr{T}_r(\Sigma(v_k)) $$
for $k$ large enough. 
\end{corollary}

\begin{proof}
We start proving the first inclusion. By Corollary \ref{openset}, $\boldsymbol{\Theta}_{s}(v,x)=0$ for every point $x\in K\setminus \mathscr{T}_r(\Sigma(v))$. Since $\boldsymbol{\Theta}_s(v_k,\cdot)$ is upper semicontinuous, we can find  a point $x_k\in  K\setminus \mathscr{T}_r(\Sigma(v))$ such that 
$$\boldsymbol{\Theta}_s(v_k,x_k)= \sup_{x\in K\setminus \mathscr{T}_r(\Sigma_v)}\boldsymbol{\Theta}_s(v_k,x)\,.$$
Then select a subsequence $\{k_j\}_{j\in\mathbb{N}}$ such that $\lim_j\boldsymbol{\Theta}_s(v_{k_j},x_{k_j})=\limsup_k\boldsymbol{\Theta}_s(v_k,x_k)$. Extracting a further subsequence if necessary, we can assume that $x_{k_j}\to x_*\in K\setminus \mathscr{T}_r(\Sigma(v))$. Since $\boldsymbol{\Theta}_{s}(v,x_*)=0$, we infer from Lemma \ref{updenstheta} that 
$\limsup_k\boldsymbol{\Theta}_s(v_k,x_k)=0$. Consequently, $\boldsymbol{\Theta}_s(v_k,x_k)< {\boldsymbol{\eta}}_0$ for $k$ large enough, and for such integers $k$, $(\Sigma(v_k)\cap K)\setminus \mathscr{T}_r(\Sigma(v))=\emptyset$.  

To prove the second inclusion, we consider a covering of $\Sigma(v)\cap K$ made by finitely many discs $D_{r/2}(x_1),\ldots, D_{r/2}(x_L)$ (included in $\partial^0G$, choosing a smaller radius if necessary). In view of \eqref{unionofbdeqsigm} and \eqref{partitonthebdry}, for each $\ell$ there exist $i_\ell,j_\ell\in\{1,\ldots,m\}$ with $i_\ell\neq j_\ell$ such that the open sets 
$D_{r/2}(x_\ell)\cap E^v_{i_\ell}$ and $D_{r/2}(x_\ell)\cap E^v_{j_\ell}$ are not empty. Hence, for each $\ell$, we can find a point $x^+_\ell\in D_{r/2}(x_\ell)\cap E^v_{i_\ell}$ and a point $x^-_\ell\in D_{r/2}(x_\ell)\cap E^v_{j_\ell}$. The sets $D_{r/2}(x_\ell)\cap E^v_{i_\ell}$ and $D_{r/2}(x_\ell)\cap E^v_{j_\ell}$ being open, we can find a radius $\varrho>0$ such that $D_{2\varrho}(x_\ell^+)\subset D_{r/2}(x_\ell)\cap E^v_{i_\ell}$ and $D_{2\varrho}(x_\ell^-)\subset  D_{r/2}(x_\ell)\cap  E^v_{j_\ell}$ for each $\ell\in\{1,\ldots,L\}$. Consequently, $v={\bf a}_{i_\ell}$ on $D_{2\varrho}(x_\ell^+)$ and $v={\bf a}_{j_\ell}$ on $D_{2\varrho}(x_\ell^-)$  for each $\ell\in\{1,\ldots,L\}$.
In particular, $\boldsymbol{\Theta}_s(v,x)=0$ for every $x\in \overline D_{\varrho}(x_\ell^\pm)$ and each $\ell\in\{1,\ldots,L\}$. Arguing as before (for the first inclusion), 
we infer from Lemma \ref{updenstheta} that 
$$\lim_{k\to\infty}\Big(\sup_{x\in \overline D_{\varrho}(x_\ell^\pm)}\boldsymbol{\Theta}_s(v_k,x)\Big)=0 \qquad\forall \ell\in\{1,\ldots,L\}\,. $$
Then Corollary  \ref{openset} implies that $\boldsymbol{\Theta}_s(v_k,x)=0$ for every $x\in  \overline D_{\varrho}(x_\ell^\pm)$ and $\ell\in\{1,\ldots,L\}$, whenever $k$ is large enough. Since each $D_{\varrho}(x_\ell^\pm)$ is connected, the map $v_k$ is constant ($\mathcal{Z}_{\boldsymbol{\sigma}}$-valued) on each $D_{\varrho}(x_\ell^\pm)$.
 On the other hand, $v_k\to v$ in $L^1(D_{\varrho}(x_\ell^\pm))$ by \eqref{compactembL1trace}, and we conclude that 
 $v_k=v={\bf a}_{i_\ell}$ on $D_{\varrho}(x_\ell^+)$ and  $v_k=v={\bf a}_{j_\ell}$ on $D_{\varrho}(x_\ell^-)$ for each  $\ell\in\{1,\ldots,L\}$, whenever $k$ is large enough. Hence, $D_{r/2}(x_{\ell})\cap E^{v_k}_{i_\ell}\not=\emptyset$ and $D_{r/2}(x_\ell)\cap  E^{v_k}_{j_\ell}\not=\emptyset$ for $k$ large, which implies  that $\Sigma(v_k)\cap D_{r/2}(x_\ell)\not=\emptyset$  for each  $\ell\in\{1,\ldots,L\}$, whenever $k$ is large enough. Therefore, $\Sigma(v)\cap K\subset \bigcup_\ell D_{r/2}(x_\ell)\subset \mathscr{T}_r(\Sigma(v_k))$ for $k$  sufficiently large. 
\end{proof}


\subsection{Tangent maps}\label{secttangmap}


We now return to the analysis of a given solution $v\in H^1(G;\R^d,|z|^a\de {\bf x})\cap L^\infty(G)$ to  \eqref{startequsurfnew}-\eqref{critnonlocexteqnew}, and we apply the results of Subsection \ref{subsectcompsurf} to define the so-called ``tangent maps" of $v$ at a given point. To this purpose, we 
fix the point of study ${\bf x_0}=(x_0,0)\in\partial^0 G$ and a reference radius $\rho_0>0$ such that $B^+_{\rho_0}({\bf x}_0)\subset G$. We introduce the rescaled functions 
\begin{equation}\label{defrescmap}
v_{x_0,\rho}({\bf x}):=v({\bf x}_0+\rho{\bf x})\,,
\end{equation}
which are defined for $0<\rho<\rho_0/r$ and ${\bf x}\in \partial^0B^+_r$. Changing variables, we observe that 
\begin{equation}\label{1912tang}
\boldsymbol{\Theta}_s(v_{x_0,\rho}, 0, r)
=\boldsymbol{\Theta}_{s}(v,x_0,\rho r)\,.
\end{equation}
This identity, together with Lemma \ref{monotformsurf}, leads to
\begin{equation}\label{1912tangbd}
\frac{1}{r^{n-2s}}{\bf E}_s(v_{x_0,\rho},B_r^+) 
\leq \boldsymbol{\Theta}_{s}(v,x_0,\rho r)\leq \boldsymbol{\Theta}_s(v,x_0,\rho_0)
\leq \frac{1}{\rho_0^{n-2s}}{\bf E}_s(v,G)\,.
\end{equation}
Given a sequence $\rho_k\to0$, we deduce that  
\begin{equation}\label{enbdresc}
\limsup_{k\to\infty}{\bf E}_s(v_{x_0,\rho_k},B_r^+) <\infty \quad\text{for every $r>0$}\,.
\end{equation} 
As a consequence of Theorem \ref{compactminsurf}, we have the following 

\begin{lemma}\label{tangmapstat}
Every sequence $\rho_k\to0$ admits a subsequence $\{\rho^\prime_k\}_{k\in\mathbb{N}}$ such that  $v_{x_0,\rho^\prime_k}\to \varphi$ strongly in $H^1(B_r^+,|z|^a\de{\bf x})$ for every $r>0$, where $\varphi$ satisfies 
\begin{equation}\label{eqtangmap}
\begin{cases}
{\rm div}(z^a\nabla\varphi)=0 & \text{in $\R^{n+1}_+$}\,,\\
\varphi\in\mathcal{Z}_{\boldsymbol{\sigma}}  & \text{on $\R^n$}\,,\\
|\varphi|\leq b & \text{in $\R^{n+1}_+$}\,,
\end{cases}
\end{equation}
and for each $r>0$, 
\begin{equation}\label{stattangmap}
\delta{\bf E}_s(\varphi,B_r^+)[{\bf X}]=0
\end{equation}
for every vector field ${\bf X}=(X,{\bf X}_{n+1})\in C^1(\overline B_r^+,\R^{n+1})$ compactly supported in $B_R^+\cup D_r$ such that ${\bf X}_{n+1}=0$ on $D_r$. 
\end{lemma}

\begin{proof}
In view of \eqref{enbdresc}, Theorem \ref{compactminsurf} yields the announced convergence and \eqref{eqtangmap}. Then we observe that $v_{x_0,\rho}$ also satisfies 
$\delta{\bf E}_s(v_{x_0,\rho},B_r^+)[{\bf X}]=0$. This identity persists under strong convergence, and thus \eqref{stattangmap} holds. 
\end{proof}

\begin{definition}
Every function $\varphi$ obtained by this process will be referred to as {\it tangent map of $v$ at the point $x_0$}.  The family of all tangent maps of $v$ at $x_0$ will be denoted by 
$T_{x_0}(v)$. 
\end{definition}

\begin{lemma}\label{homo2203}
If $\varphi\in T_{x_0}(v)$, then  
$$\boldsymbol{\Theta}_s(\varphi,0,r)=\boldsymbol{\Theta}_s(\varphi,0)= \boldsymbol{\Theta}_s(v,x_0)\quad\forall r>0\,,$$
and $\varphi$ is $0$-homogeneous, i.e., $\varphi(\lambda {\bf x})=\varphi({\bf x})$ for every $\lambda>0$ and every ${\bf x}\in\R^{n+1}_+$. 
\end{lemma}

\begin{proof}
From the strong convergence of $v_{x_0,\rho^\prime_k}$ toward $\varphi$ and the identity in \eqref{1912tang}, we first infer that
$$ \boldsymbol{\Theta}_\varphi(0,0,r)=\lim_{k\to\infty} \boldsymbol{\Theta}_s(v_{x_0,\rho^\prime_k},0,r)=\boldsymbol{\Theta}_s(v,x_0)\quad\forall r>0\,.$$
Then the monotonicity formula in Lemma \ref{monotformsurf} applied to $\varphi$ implies that ${\bf x}\cdot \nabla\varphi({\bf x})=0$ for every ${\bf x}\in\R^{n+1}_+$, and the conclusion  follows.
\end{proof}


\subsection{Homogeneous solutions}  


The previous lemma indicates that the study of tangent maps leads to  the study of $0$-homogeneous solutions, which is the purpose of this subsection. 

\begin{lemma}\label{lemhomogsol}
Let $\varphi\in H^1(B_1^+;\R^d,|z|^a\de {\bf x})\cap L^\infty(B_1^+)$  be a solution of 
\begin{equation}\label{eqcone}
\begin{cases}
{\rm div}(z^a\nabla\varphi)=0 & \text{in $B_1^+$}\,,\\
\varphi \in  \mathcal{Z}_{\boldsymbol{\sigma}}   & \text{on $D_1$}\,,\\
|\varphi|\leq b & \text{in $B_1^+$}\,,
\end{cases}
\end{equation}
for some constant $b\geq 1$. Assume that 
\begin{equation}\label{statinB1}
\delta{\bf E}_s\big(\varphi,B_1^+\big)[{\bf X}]=0\,,
\end{equation}
for every vector field ${\bf X}=(X,{\bf X}_{n+1})\in C^1(\overline B_1^+,\R^{n+1})$ compactly supported in $B_1^+\cup D_1$ such that ${\bf X}_{n+1}=0$ on $D_1$. 
If $\boldsymbol{\Theta}_s(\varphi,0,1)=\boldsymbol{\Theta}_s(\varphi,0) $, then $\varphi$ is $0$-homogeneous. 
\end{lemma}

\begin{proof}
Applying Corollary \ref{directcorlmonot} at $x_0=0$ leads to the homogeneity of $\varphi$ in $B_1^+$. In turn, the homogeneity of $\varphi$ implies that 
$T_0(\varphi)=\{\varphi\}$, and the conclusion  follows from Lemma~\ref{tangmapstat}. 
\end{proof}

\begin{definition}
We say that a function $\varphi\in L^1_{\rm loc}(\R^{n+1}_+;\R^d)$ is a {\it nonlocal stationary $\mathcal{Z}_{\boldsymbol{\sigma}}$-cone} if $\varphi$ is $0$-homogeneous, $\varphi\in H^1(B_1^+;\R^d,|z|^a\de {\bf x})\cap L^\infty(B_1^+)$, and $\varphi$ satisfies \eqref{eqtangmap}-\eqref{stattangmap} (for some constant $b\geq 1$). 
\end{definition}

In view of the previous subsection, tangent maps to a solution of \eqref{startequsurfnew}-\eqref{critnonlocexteqnew} are thus nonlocal stationary $\mathcal{Z}_{\boldsymbol{\sigma}}$-cones. We shall present in details the main properties of those ``cones''. We start with the following lemma explaining somehow the terminology. 

\begin{lemma}\label{relattangmap}
If $\varphi$ is a nonlocal stationary $\mathcal{Z}_{\boldsymbol{\sigma}}$-cone, then there are $m$ open cones $\mathcal{C}_1^\varphi,\ldots,  \mathcal{C}_m^\varphi\subset \R^n$ such that 
$$\sum_{j=1}^m\chi_{\mathcal{C}_j^\varphi}=1\text{ a.e. on $\R^n$ and }  \varphi=\sum_{j=1}^m\big(\chi_{\mathcal{C}_j^\varphi}\big)^\e{\bf a}_j\,, $$
as defined in \eqref{poisson}. In particular, $|\varphi|\leq |{\bf a}|_{\rm max}$ in $\R^{n+1}_+$. 
\end{lemma}

\begin{proof}
By Corollary \ref{openset}, there are $m$ open sets $\mathcal{C}_1^\varphi,\ldots, \mathcal{C}_m^\varphi\subset \R^n$ such that $\sum_{j}\chi_{\mathcal{C}_j^\varphi}=1$ a.e. in $\R^n$ and 
$\varphi=\sum_{j}\chi_{\mathcal{C}_j^\varphi}{\bf a}_j$ a.e. on $\R^n$. Since  $\varphi$ is $0$-homogeneous, we easily infer that  $\mathcal{C}_\varphi $ is an open cone. We  set 
$$w:=\varphi- \sum_{j=1}^m\big(\chi_{\mathcal{C}_j^\varphi}\big)^\e{\bf a}_j\,.$$ 
The map $w$ being $0$-homogeneous, it belongs to $\in H^1_{\rm loc}\big(\overline{\mathbb{R}^{n+1}_+};\R^d,|z|^a\de{\bf x}\big)\cap L^\infty(\mathbb{R}^{n+1}_+)$ with the upper bound  $\|w\|_{L^\infty(\R^{n+1}_+)}\leq |{\bf a}|_{\rm max}+\|\varphi\|_{L^\infty(\R^{n+1}_+)}$, and it satisfies
$$\begin{cases} 
{\rm div}(z^a\nabla w)=0& \text{in $\R^{n+1}_+$}\,,\\
w=0 & \text{on $\partial \R^{n+1}_+$}\,. 
\end{cases}$$
Then $w$ and $z^a\partial_z w$ are H\"older continuous up to $\partial\R^{n+1}_+$ (see e.g. \cite[Proof of Lemma 4.8]{MilSirW}), and smooth in $\R^{n+1}_+$ by  elliptic regularity. Since $w$ is bounded, the Liouville type theorem in \cite[Corollary 3.5]{CS1} tells us that $w\equiv 0$. 
\end{proof}

\begin{remark}
If  $\varphi$ is a nonlocal stationary $\mathcal{Z}_{\boldsymbol{\sigma}}$-cone, then $ \boldsymbol{\Theta}_s(\varphi,\lambda y)=\boldsymbol{\Theta}_s(\varphi,y)$ for every $y\in\R^n\setminus\{0\}$ and $\lambda>0$. Indeed, by homogeneity of $\varphi $  we have for  each $\rho>0$, 
$$\boldsymbol{\Theta}_s(\varphi,\lambda y,\rho) = \boldsymbol{\Theta}_s(\varphi,y,\rho/\lambda)\,,$$
and the assertion follows letting $\rho\to 0$. 
\end{remark}

\begin{lemma}\label{lemSphi}
Let $\varphi$ be a nonlocal stationary $\mathcal{Z}_{\boldsymbol{\sigma}}$-cone. Then, 
$$ \boldsymbol{\Theta}_s(\varphi,y)\leq \boldsymbol{\Theta}_s(\varphi,0)\quad \forall y\in\R^n\,.$$
In addition, the set 
$$S(\varphi):=\Big\{y\in\R^n:  \boldsymbol{\Theta}_s(\varphi,y)=\boldsymbol{\Theta}_s(\varphi,0)\Big\} $$
is a linear subspace of $\R^n$, and $\varphi({\bf x}+{\bf y})=\varphi({\bf x})$ for every ${\bf y}\in S(\varphi)\times\{0\}$ and ${\bf x}\in\R^{n+1}_+$. 
\end{lemma}

\begin{proof}
By Corollary \ref{directcorlmonot}, we have for every ${\bf y}=(y,0)\in\partial\R^{n+1}_+$ and every $\rho>0$, 
\begin{equation}\label{2124}
\boldsymbol{\Theta}_s(\varphi,y)+d_s\int_{B_\rho^+({\bf y})} z^a\frac{|(\mathbf{x}-{\bf y})\cdot \nabla \varphi({\bf x})|^2}{|\mathbf{x}-{\bf y}|^{n+2-2s}}\,\de{\bf x} = \boldsymbol{\Theta}_s(\varphi,y,\rho)\,.
\end{equation}
On the other hand, by homogeneity of $\varphi$, 
$$\boldsymbol{\Theta}_s(\varphi,y,\rho) \leq \frac{(\rho+|z|)^{n-2s}}{\rho^{n-2s}}\,\boldsymbol{\Theta}_s(\varphi,0,\rho+|y|) =  \frac{(\rho+|y|)^{n-2s}}{\rho^{n-2s}}\,\boldsymbol{\Theta}_s(\varphi,0)\,.  $$
Inserting this inequality in \eqref{2124} and letting $\rho\to\infty$, we deduce that 
$$ \boldsymbol{\Theta}_s(\varphi,y)+d_s\int_{\R^{n+1}_+} z^a\frac{|(\mathbf{x}-{\bf y})\cdot \nabla \varphi({\bf x})|^2}{|\mathbf{x}-{\bf y}|^{n+2-2s}}\,\de{\bf x} \leq \boldsymbol{\Theta}_s(\varphi,0)\,.$$
Next, if $ \boldsymbol{\Theta}_s(\varphi,y)= \boldsymbol{\Theta}_s(\varphi,0)$, then $(\mathbf{x}-{\bf y})\cdot \nabla \varphi({\bf x})=0$ for every ${\bf x}\in\R^{n+1}_+$. By homogeneity of $\varphi$, we deduce that ${\bf y}\cdot \nabla \varphi({\bf x})=0$ for every ${\bf x}\in\R^{n+1}_+$, i.e, 
\begin{equation}\label{transinv}
\varphi({\bf x}+{\bf y})=\varphi({\bf x}) \quad \forall {\bf x}\in\R^{n+1}_+\,.
\end{equation}
The other way around, if ${\bf y}=(y,0)$ satisfies \eqref{transinv}, then $(\mathbf{x}-{\bf y})\cdot \nabla \varphi({\bf x})=0$ for every ${\bf x}\in\R^{n+1}_+$ (using again the homogeneity of $\varphi$). By \eqref{2124} and \eqref{transinv},  it implies that for each radius $\rho>0$, 
$$\boldsymbol{\Theta}_s(\varphi,y)= \boldsymbol{\Theta}_s(\varphi,y,\rho)= \boldsymbol{\Theta}_s(\varphi,0,\rho)=\boldsymbol{\Theta}_s(\varphi,0)\,,$$
and thus $y\in S(\varphi)$. From \eqref{transinv} it now follows that $S(\varphi)$ is a linear subspace of $\R^n$. 
\end{proof}

\begin{remark}\label{remSphisubbound}
If  $\varphi$ is a non constant nonlocal stationary $\mathcal{Z}_{\boldsymbol{\sigma}}$-cone, then $ \boldsymbol{\Theta}_s(\varphi,0)\geq \boldsymbol{\eta}_0>0$ by Lemma~\ref{epsregminsurf}. In turn, we infer from Corollary \ref{openset} that  $S(\varphi)\subset \Sigma(\varphi)$.
\end{remark}

\begin{remark}\label{spinedimn}
If  $\varphi$ is a nonlocal stationary $\mathcal{Z}_{\boldsymbol{\sigma}}$-cone such that ${\rm dim}\,S(\varphi)=n$, then  $\mathcal{C}_{j_0}^\varphi=\R^n$ for some $j_0\in\{1,\ldots,m\}$ and 
 $\varphi={\bf a}_{j_0}$. As a consequence, if $\varphi\in T_{x_0}(v)$ for some solution $u$  of \eqref{startequsurfnew}-\eqref{critnonlocexteqnew}, then $ \boldsymbol{\Theta}_s(v,x_0)=\boldsymbol{\Theta}_s(\varphi,0)=0$, and Corollary \ref{openset} yields $x_0\not\in \Sigma(v)$. In other words,
$$x_0\in\Sigma(v) \Longleftrightarrow\,{\rm dim}\,S(\varphi)\leq n-1 \text{ for {\sl all} $\varphi\in T_{x_0}(v)$}\,.$$
\end{remark}

\begin{remark}\label{remhalfspa}
If  $\varphi$ is a nonlocal stationary $\mathcal{Z}_{\boldsymbol{\sigma}}$-cone such that ${\rm dim}\,S(\varphi)=n-1$, then there are two distinct indices $i_0,j_0\in\{1,\ldots,m\}$ and a half-space $H\subset\R^n$ such that $\varphi=\chi_{{\rm int}(H)}{\bf a}_{i_0}+\chi_{{\rm int}(H^c)}{\bf a}_{j_0}$ on $\R^n$.  
 Indeed, up to a rotation, we may assume that $S(\varphi)=\{0\}\times \R^{n-1}$, and Lemma~\ref{lemSphi} yields $\varphi({\bf x})=\varphi(x_1,z)$ for all ${\bf x}=(x_1,\ldots,x_n,z)\in\R^{n+1}_+$. Since 
 $\varphi$ is $0$-homogeneous, the conclusion follows.
\end{remark}
\vskip5pt

For a constant $\Lambda\geq 0$ and $\ell\in\{0,\ldots,n\}$, we now introduce the following class of nonlocal stationary $\mathcal{Z}_{\boldsymbol{\sigma}}$-cones,  
$$\mathscr{C}_\ell(\Lambda):=\Big\{\text{nonlocal stationary $\mathcal{Z}_{\boldsymbol{\sigma}}$-cone $\varphi$ such that ${\rm dim}\,S(\varphi)\geq \ell$ and $\boldsymbol{\Theta}_s(\varphi,0)\leq\Lambda$}\Big\}\,. $$
Note that $\mathscr{C}_{\ell+1}(\Lambda)\subset \mathscr{C}_\ell(\Lambda)$, and $\mathscr{C}_n(\Lambda)=\mathcal{Z}_{\boldsymbol{\sigma}}$ by Remark \ref{spinedimn}.

\begin{lemma}\label{compactcone}
For each $\ell\in\{0,\ldots,n\}$ and $r>0$, the set $\big\{\varphi_{| B_r^+}: \varphi\in\mathscr{C}_\ell(\Lambda)\big\} $ is strongly compact in $H^1(B_r^+,|z|^a\de{\bf x})$. 
\end{lemma}

\begin{proof}
By homogeneity, it is enough to consider the case $r=1$. Let $\{\varphi_k\}_{k\in\mathbb{N}}\subset \mathscr{C}_\ell(\Lambda)$ be an arbitrary sequence. Still by homogeneity, we have $\boldsymbol{\Theta}_s(\varphi_k,0,2)=\boldsymbol{\Theta}_s(\varphi_k,0)\leq \Lambda$, so that 
$${\bf E}_s(\varphi_k,B^+_{2})\leq 2^{n-2s}\Lambda\,. $$
Since $|\varphi_k|\leq |{\bf a}|_{\rm max}$ by Lemma \ref{relattangmap}, we can apply Theorem \ref{compactminsurf} to find a (not relabeled) subsequence   such that $\varphi_k\to\psi$ strongly in $H^1(B_1^+,|z|^a\de{\bf x})$ for a function $\psi$ satisfying \eqref{eqcone}-\eqref{statinB1} with $b={\rm max}(1,|{\bf a}|_{\rm max})$. Then we deduce from Lemma \ref{updenstheta} that 
$$\boldsymbol{\Theta}_s(\psi,0)\geq \lim_{k\to\infty}\boldsymbol{\Theta}_s(\varphi_k,0)=\lim_{k\to\infty}\boldsymbol{\Theta}_s(\varphi_k,0,1) =\boldsymbol{\Theta}_s(\psi,0,1)\,.$$
In turn, Corollary \ref{directcorlmonot} shows that $\boldsymbol{\Theta}_s(\psi,0)=\boldsymbol{\Theta}_s(\psi,0,1)$, and thus $\psi$ is $0$-homogeneous by Lemma~\ref{lemhomogsol}, and $\boldsymbol{\Theta}_s(\psi,0)=\lim_k\boldsymbol{\Theta}_s(\varphi_k,0)\leq \Lambda$. Consequently, $\psi$ is a nonlocal stationary $\mathcal{Z}_{\boldsymbol{\sigma}}$-cone, and it remains to show that  ${\rm dim}\,S(\psi)\geq \ell$. 

Extracting a further subsequence if necessary, we may assume that ${\rm dim}\,S(\varphi_k)$ is a constant integer $\bar d\geq \ell$, and $S(\varphi_k)\to V$ in the Grassmannian  $G(\bar d,n)$ of all $\bar d$-dimensional linear subspaces of $\mathbb{R}^n$. For an arbitrary $y\in V\cap D_1$, there exists a sequence $\{y_k\}_{k\in\mathbb{N}}\subset D_1$ such that $y_k\in S(\varphi_k)$ and $y_k\to y$.  By Lemma \ref{updenstheta}, we have 
$$\boldsymbol{\Theta}_s(\psi,y)\geq \lim_{k\to\infty} \boldsymbol{\Theta}_s(\varphi_k,y_k)=\lim_{k\to\infty} \boldsymbol{\Theta}_s(\varphi_k,0)=\boldsymbol{\Theta}_s(\psi,0)\,,$$
and we deduce from Lemma \ref{lemSphi} that $y\in S(\psi)$. Therefore $V\subset S(\psi)$ and  ${\rm dim}\,S(\psi)\geq \bar d\geq\ell$. 
\end{proof}


\subsection{Quantitative stratification}\label{subsectstrat}  


In this subsection, we are back again to the analysis  of the map  $v\in H^1(G;\R^d,|z|^a\de{\bf x})\cap L^\infty(G)$ solving \eqref{startequsurfnew}-\eqref{critnonlocexteqnew}.  
We are interested in some regularity properties of the set  $\Sigma(v)\subset \partial^0G$ (provided by Corollary~\ref{openset}). To this purpose, we introduced the following (standard)  stratification of the singular set of $v$, 
$${\rm Sing}^\ell(v):=\Big\{{\bf x}=(x,0)\in\partial^0G: {\rm dim}\,S(\varphi)\leq \ell \text{ for all $\varphi\in T_x(v)$}\Big\}\,,\quad \ell=0,\ldots,n-1\,.$$ 
Obviously, 
$${\rm Sing}^\ell(v)\subset {\rm Sing}^{\ell+1}(v)\,, $$
and by Remark \ref{spinedimn}, 
\begin{equation}\label{identboundtopstrat}
\Sigma(v) =  {\rm Sing}^{n-1}(v)\,.
\end{equation}
We also introduce the ``regular part'' $\Sigma_{\rm reg}(v)$ of $\Sigma(v)$, 
\begin{equation}\label{defsigmareg}
\Sigma_{\rm reg}(v):= \Sigma(v)\setminus  {\rm Sing}^{n-2}(v)\,.
\end{equation}
The terminology {\sl regular part} is motivated by the fact that at least one blow-up limit of $\Sigma(v)$ is an hyperplane at every point of $\Sigma_{\rm reg}(v)$.

\begin{proposition}\label{regblowupset}
For every $x\in \Sigma_{\rm reg}(v)$, there exists $\varphi\in T_x(v)$ such that ${\rm dim}\,S(\varphi)=n-1$. In particular, if $x\in \Sigma_{\rm reg}(v)$, then there exists a sequence $\rho_k\downarrow 0$, two distinct indices $i_0,j_0\in\{1,\ldots,m\}$,  and a half space $H\subset\mathbb{R}^n$, with $0\in \partial H$, such that  
\begin{equation}\label{conL1halfsp}
v_{x,\rho_k} \to \chi_{{\rm int}(H)}{\bf a}_{i_0}+\chi_{{\rm int}(H^c)}{\bf a}_{j_0}\quad \text{in $L^1_{\rm loc}(\R^n)$}\,,
\end{equation}
and 
\begin{equation}\label{cnvhaushyperpl}
\Sigma_k:=\Sigma(v_{x,\rho_k}) =(\Sigma(v) -x)/ \rho_k
\end{equation}
converge locally uniformly to the hyperplane $\partial H$, i.e., for every compact set $K\subset \R^n$ and every $r>0$,
$$\Sigma_k\cap K\subset \mathscr{T}_r(\partial H) \quad\text{and}\quad   \partial H\cap K\subset  \mathscr{T}_r(\Sigma_k)$$
whenever $k$ is large enough.
\end{proposition}

\begin{proof}
By the very  definition of $\Sigma_{\rm reg}(v)$ and \eqref{identboundtopstrat}, if $x\in \Sigma_{\rm reg}(v)$, then there exists $\varphi_0\in  T_x(v)$ such that ${\rm dim}\,S(\varphi_0)=n-1$.  By  Remark \ref{remhalfspa}, we have  $\varphi_0=\chi_{{\rm int}(H)}{\bf a}_{i_0}+\chi_{{\rm int}(H^c)}{\bf a}_{j_0}$ on $\R^n$ for some distinct indices $i_0,j_0\in\{1,\ldots,m\}$  and a half space $H\subset\mathbb{R}^n$, with $0\in \partial H$. Then $\Sigma(\varphi_0)=S(\varphi_0)=\partial H$. By definition of $\varphi_0$, there exists a sequence 
 $\rho_k\downarrow 0$  
  such that $v_{x,\rho_k}\to \varphi_0$ strongly in $H^1(B_r^+,|z|^a\de{\bf x})$ for every $r>0$.  Since $H^1(B_r^+,|z|^a\de{\bf x})$ embeds into $L^1(D_r)$ by \eqref{compactembL1trace}, \eqref{conL1halfsp} follows, 
  and \eqref{cnvhaushyperpl} is implied by Corollary \ref{Hausdcvbd}. 
\end{proof}

We are now ready to prove one of the main result of this section: the optimal estimate for the dimension of $\Sigma(v)$ (here ${\rm dim}_{\mathscr{M}} $ denotes the Minkowski dimension, or box-counting dimension, see \cite[Section~3.1]{Falc}). 

\begin{theorem}\label{dimMink}
We have ${\rm dim}_{\mathscr{M}} (\Sigma(v)\cap\Omega^\prime) = n-1$ for every open subset $\Omega^\prime\subset \partial^0G$ such that $\overline{\Omega^\prime}\subset\partial^0G$ and $\Sigma(v)\cap\Omega^\prime\neq\emptyset$. In addition, ${\rm dim}_{\mathscr{H}}{\rm Sing}^\ell(v)\leq \ell$ for each $\ell\in\{1,\ldots,n-2\}$, and ${\rm Sing}^0(v)$ is countable. 
\end{theorem}

We will prove Theorem \ref{dimMink} using the abstract stratification principle of \cite{FMS}, originally introduced in~\cite{CheegNab}. To fit the setting of \cite{FMS}, we first need to introduce some notations. 

For a radius $r>0$, we set 
\begin{equation}\label{defOmegar}
\Omega^{r}:=\big\{x\in\R^n: B^+_{2r}((x,0))\subset G\big\}\,.
\end{equation}
In what follows, we {\it fix  two constants} $r_0> 0$ and $\Lambda_0\geq 0$ such that 
\begin{equation}\label{controldensit}
\sup\Big\{\boldsymbol{\Theta}_s(v,x,\rho) : x\in\Omega^{r_0}\,,\;0<\rho\leq r_0\Big\}\leq \Lambda_0\,. 
\end{equation}
Note that the supremum above is indeed finite  by \eqref{1912tangbd}, and for $0<\rho\leq r_0$, 
$$\boldsymbol{\Theta}_s(v,x,\rho)\leq \frac{1}{r_0^{n-2s}}{\bf E}_s(v,G)
\quad\forall x\in\Omega^{r_0} \,.$$

For each $\ell\in \{0,\ldots,n\}$, $\rho\in(0,r_0)$, $x_0\in\Omega^{r_0}$ and ${\bf x}_0=(x_0,0)$, we now introduce the function ${\bf d}_\ell(\cdot,x_0,\rho): H^1(B^+_\rho({\bf x}_0);\R^d,|z|^a\de{\bf x})\to [0,\infty)$
defined by
$${\bf d}_\ell(v,x_0,\rho):=\inf\Big\{\|v_{x_0,\rho}-\varphi\|_{L^1(B_1^+)}: \varphi\in\mathscr{C}_\ell(\Lambda_0)\Big\} \,,$$
where $v_{x_0,\rho}({\bf x}):=v({\bf x}_0+\rho {\bf x})$. Note that the infimum above is well defined since $H^1(B^+_1({\bf x}_0),|z|^a\de{\bf x})$ compactly embeds into $L^1(B_1^+)$ (see e.g. \cite[Remark 2.4]{MilSirW}), and it is always achieved by Lemma~\ref{compactcone}. Moreover, 
$${\bf d}_0(\cdot,x_0,\rho)\leq {\bf d}_1(\cdot,x_0,\rho)\leq \ldots\leq {\bf d}_{n}(\cdot,x_0,\rho)\,,$$
and
$${\bf d}_n(v,x_0,\rho):=\min\Big\{ \|v_{x_0,\rho}-{\bf a}\|_{L^1(B_1^+)} : {\bf a}\in \mathcal{Z}_{\boldsymbol{\sigma}} \Big\}\,. $$
Observe that each functional ${\bf d}_\ell(\cdot,x_0,\rho)$ is a (rescaled) $L^1$-distance function, and consequently they are $\rho^{-n}$-Lipschitz functions with respect to the $L^1(B_\rho^+({\bf x}_0))$-norm. In particular, each functional ${\bf d}_\ell(\cdot,x_0,\rho)$ is continuous with respect to strong convergence in $L^1(B_\rho^+({\bf x}_0))$. 
\vskip3pt

In the terminology of \cite[Section 2.1]{FMS}, the functions $\boldsymbol{\Theta}_s(v,\cdot,\rho)$ and ${\bf d}_\ell(v,\cdot,\cdot)$ will play the roles of {\sl density function} and {\sl control functions},  respectively (thanks to Lemma \ref{monotformsurf}). We need to show that they satisfy the structural assumptions of  \cite[Section 2.2]{FMS}. This is the purpose of the following lemmas.

\begin{lemma}\label{gapdistlemma}
There exists a constant 
$$\delta_0(r_0)=\delta_0(r_0,\Lambda_0,\mathcal{Z}_{\boldsymbol{\sigma}},b,n,s)\in(0,1)$$ 
(independent of $v$) such that for every for every $x\in\Omega^{r_0}$ and $\rho\in(0,r_0)$, 
$$\boldsymbol{\Theta}_s(v,x)>0\quad\Longrightarrow\quad  {\bf d}_n(v,x,\rho) \geq\delta_0 \,.$$
\end{lemma}

\begin{proof}
Assume by contradiction that there exist a sequence $\{v_k\}_{k\in\mathbb{N}}$ solving \eqref{startequsurfnew}-\eqref{critnonlocexteqnew} and satisfying \eqref{controldensit}, points $\{x_k\}_{k\in\mathbb{N}}\subset\Omega^{r_0}$, and radii $\{\rho_k\}_{k\in\mathbb{N}}\subset (0,r_0)$ such that $\boldsymbol{\Theta}_s(v_k,x_k)>0$ and ${\bf d}_n(v_k,x_k,\rho_k) \leq 2^{-k}$. 
We continue with a general first step that we shall use again in the sequel. 
\vskip3pt
 
\noindent{\it Step 1, general compactness.} We consider the rescaled maps $\widetilde v_k:=(v_k)_{x_k,\rho_k}$ 
as defined in \eqref{defrescmap}. By our choice of $\Lambda_0$, a simple change of variables yields $\boldsymbol{\Theta}_s(\widetilde v_k,0,1)\leq \Lambda_0$. 
By Theorem \ref{compactminsurf}, we can find a (not relabeled) subsequence such that $\widetilde v_k\to v_*$ weakly in $H^1(B_1^+,|z|^a\de {\bf x})$ and strongly in $H^1(B_r^+,|z|^a\de {\bf x})$ for every $0<r<1$, 
where $u_*$ satisfies \eqref{eqcone}-\eqref{statinB1}. Then $\widetilde v_k\to v_*$ strongly in $L^1(B_1^+)$, so that 
$${\bf d}_\ell(\widetilde v_k,0,1)\to {\bf d}_\ell(v_*,0,1)\quad\text{for each $\ell\in\{0,\ldots,n\}$}\,.$$ 
In addition, by lower semicontinuity of ${\bf E}_s(\cdot, B_1^+)$, we have 
 \begin{equation}\label{lscthetw1q}
 \boldsymbol{\Theta}_{s}(v_*,0,1)\leq \liminf_{k\to\infty}\boldsymbol{\Theta}_s(\widetilde v_k,0,1)\leq\Lambda_0\,.
 \end{equation}
\vskip3pt

\noindent{\it Step 2, conclusion.} Since ${\bf d}_n(\widetilde v_k,0,1)\leq 2^{-k}$, we have ${\bf d}_n(v_*,0,1)=0$. In other words,  $v_*={\bf a}$ for some ${\bf a}\in\mathcal{Z}_{\boldsymbol{\sigma}}$, and consequently 
$\boldsymbol{\Theta}_s(v_*,0)=0$. On the other hand, by Corollary~\ref{openset}, $\boldsymbol{\Theta}_s(\widetilde v_k,0)=\boldsymbol{\Theta}_s(v_k,0)\geq {\boldsymbol{\eta}}_0>0$. 
Then Lemma \ref{updenstheta} yields $\boldsymbol{\Theta}_s(v_*,0)\geq \limsup_k \boldsymbol{\Theta}_s(v_k,0)>0$, which contradicts $\boldsymbol{\Theta}_s(v_*,0)=0$. 
\end{proof}

\begin{lemma}\label{lemme2stratsurf}
For every $\delta>0$, there exist constants 
$$\eta_1(\delta,r_0)=\eta_1(\delta,r_0,\Lambda_0,\mathcal{Z}_{\boldsymbol{\sigma}},b,n,s)\in (0,1/4)$$ 
and   
$$\lambda_1(\delta,r_0)=\lambda_1(\delta,r_0,\Lambda_0,\mathcal{Z}_{\boldsymbol{\sigma}},b,n,s)\in (0,1/4)$$ 
(independent of $v$) such that for every $x\in\Omega^{r_0}$ and $\rho\in(0,r_0)$,   
$$\boldsymbol{\Theta}_s(v,x,\rho)-\boldsymbol{\Theta}_s(v,x,\lambda_1\rho) \leq \eta_1 \quad \Longrightarrow\quad {\bf d}_0(v,x,\rho) \leq\delta \,.$$ 
\end{lemma}

\begin{proof}
Assume by contradiction that for some $\delta>0$, there exist a sequence  $\{v_k\}_{k\in\mathbb{N}}$ solving \eqref{startequsurfnew}-\eqref{critnonlocexteqnew} and satisfying \eqref{controldensit}, points $\{x_k\}_{k\in\mathbb{N}}\subset\Omega^{r_0}$, and radii $\{\rho_k\}_{k\in\mathbb{N}}\subset (0,r_0)$ such that  
$$\boldsymbol{\Theta}_s(v_k,x_k,\rho_k)-\boldsymbol{\Theta}_s(v_k,x_k,\lambda_k\rho_k) \leq 2^{-k} \quad \text{and}\quad {\bf d}_0(v_k,x_k,\rho_k) \geq\delta \,,$$
where $\lambda_k\to0$ as $k\to\infty$. 
We consider the rescaled map $\widetilde v_k:=(v_k)_{x_k,\rho_k}$ as defined in \eqref{defrescmap}, so that 
$$\boldsymbol{\Theta}_s(\widetilde v_k,0,1)-\boldsymbol{\Theta}_s(\widetilde v_k,0,\lambda_k) \leq 2^{-k} \quad \text{and}\quad {\bf d}_0(\widetilde v_k,0,1) \geq\delta \,. $$ 
Then we apply Step 1 in the proof of Lemma \ref{gapdistlemma} to find a (not relabeled) sequence along which $\widetilde v_k$  converges to $v_*$. As a consequence of the established convergences, we  deduce that   ${\bf d}_0(v_*,0,1)\geq \delta$. 

On the other hand, by Lemma \ref{monotformsurf} we can estimate for $0<r<1$ and $k$ large enough  (in such a way that $\lambda_k<r$), 
$$\boldsymbol{\Theta}_s(\widetilde v_k,0,1)-\boldsymbol{\Theta}_s(\widetilde v_k,0,r) \leq 2^{-k}\,. $$
Using \eqref{lscthetw1q} and the strong convergence of $\widetilde v_k$ in $H^1(B_r^+,|z|^a\de{\bf x})$, we can let $k\to\infty$ to deduce that 
$$  \boldsymbol{\Theta}_s(v_*,0,1)-  \boldsymbol{\Theta}_s(v_*,0,r)\leq 0\,. $$
Letting $r\to 0$, we infer from  Corollary \ref{directcorlmonot} that  $\boldsymbol{\Theta}_s(v_*,0,1)=\boldsymbol{\Theta}_s(v_*,0)$. From Lemma~\ref{lemhomogsol}, we infer that $v_*$ is a nonlocal  stationary $\mathcal{Z}_{\boldsymbol{\sigma}}$-cone. Moreover, \eqref{lscthetw1q} yields the estimate $\boldsymbol{\Theta}_s(v_*,0)\leq \Lambda_0$, so that $v_*\in\mathscr{C}_0(\Lambda_0)$. Hence ${\bf d}_0(v_*,0,1)=0$, which contradicts the previous estimate ${\bf d}_0(v_*,0,1)\geq \delta$. 
\end{proof}

\begin{lemma}\label{condiistrat}
For every $\delta,\tau\in(0,1)$, there exists a constant 
$$\eta_2(\delta,\tau,r_0)=\eta_2(\delta,\tau,r_0,\Lambda_0,\mathcal{Z}_{\boldsymbol{\sigma}},b,n,s)\in (0,\delta\,]$$
(independent of $v$) such that the following holds for every $\rho\in(0,r_0/5)$ and $x\in\Omega^{r_0}$. If 
$${\bf d}_\ell(v,x,4\rho) \leq\eta_2\quad \text{and}\quad {\bf d}_{\ell+1}(v,x,4\rho)\geq \delta\,,$$
hold for some $\ell\in\{0,\ldots,n-1\}$, then there exists a $\ell$-dimensional linear subspace $V\subset\R^n$ for which
$${\bf d}_0(v,y,4\rho)> \eta_2 \quad \forall y\in D_{\rho}(x)\setminus \mathscr{T}_{\tau \rho}(x+V)\,.$$ 
\end{lemma}

\begin{proof}
The proof is again by contradiction. Assume that for some $\delta,\tau\in(0,1)$ and some index $\ell\in\{0,\ldots,n-1\}$, 
there exist a sequence $\{v_k\}_{k\in\mathbb{N}}$ solving \eqref{startequsurfnew}-\eqref{critnonlocexteqnew} and satisfying \eqref{controldensit}, points $\{x_k\}_{k\in\mathbb{N}}\subset\Omega^{r_0}$, and radii $\{\rho_k\}_{k\in\mathbb{N}}\subset (0,r_0/5)$ such that 
$${\bf d}_\ell(v_k,x_k,4\rho_k) \leq 2^{-k}\quad \text{and}\quad {\bf d}_{\ell+1}(v_k,x_k,4\rho_k)\geq \delta\,,$$
and such that the conclusion of the lemma does not hold. 

We consider  the rescaled map $\widetilde v_k:=(v_k)_{x_k,4\rho_k}$.
\vskip3pt

\noindent{\it Step 1.} For each $k$, we select $\varphi_k\in\mathscr{C}_\ell(\Lambda_0)$ such that $\|\widetilde v_k-\varphi_k\|_{L^1(B_1^+)}\leq 2^{-k}$ (which is possible by Lemma~\ref{compactcone}). Since 
\begin{equation}\label{20147fin}
{\bf d}_{\ell+1}(\varphi_k,0,1)\geq {\bf d}_{\ell+1}(\widetilde v_k,0,1)-\|\widetilde v_k-\varphi_k\|_{L^1(B_1^+)}\geq \delta-2^{-k}\,, 
\end{equation}
we infer that ${\rm dim}\,S(\varphi_k)=\ell$ for $k$ large enough. Extracting a (not relabeled) subsequence and rotating coordinates if necessary, we may assume that $S(\varphi_k)=V$ for some fixed linear subspace $V\subset\R^n$ of dimension $\ell$.  Then, by Lemma \ref{compactcone} we can find a further (not relabeled) subsequence such that $\varphi_k\to \varphi$ strongly in $H^1(B_r^+,|z|^a\de {\bf x})$ for every $r>0$ and some $\varphi\in \mathscr{C}_\ell(\Lambda_0)$. In particular,  
$$\boldsymbol{\Theta}_s(\varphi,0)=\boldsymbol{\Theta}_s(\varphi,0,1)=\lim_{k\to\infty} \boldsymbol{\Theta}_s(\varphi_k,0,1)=\lim_{k\to\infty} \boldsymbol{\Theta}_s(\varphi_k,0)\,. $$
On the other hand, by Lemma \ref{updenstheta},
$$\boldsymbol{\Theta}_s(\varphi,y)\geq \lim_{k\to\infty} \boldsymbol{\Theta}_s(\varphi_k,y)=\lim_{k\to\infty} \boldsymbol{\Theta}_s(\varphi_k,0)= \boldsymbol{\Theta}_s(\varphi,0)\quad\forall y\in V\,. $$
Therefore, $V\subset S(\varphi)$ by Lemma \ref{lemSphi}. But letting $k\to \infty$ in \eqref{20147fin}, we deduce that  ${\bf d}_{\ell+1}(\varphi,0,1)\geq \delta$, and thus $S(\varphi)=V$. 
Since the conclusion of the lemma does not hold, for each $k$ we can find a point $y_k\in D_{1/4}\setminus \mathscr{T}_{\tau/4}(V)$ such that 
${\bf d}_0(\widetilde v_k,y_k,1)\to 0$ as $k\to\infty$.  
\vskip3pt

\noindent{\it Step 2.} Consider the translated function $\widehat v_k:=(\widetilde v_k)_{y_k,1}$, and select 
 $\psi_k\in\mathscr{C}_0(\Lambda_0)$ such that 
 $$\|\widehat v_k-\psi_k\|_{L^1(B_1^+)}= {\bf d}_0(\widetilde v_k,y_k,1)\to 0\,.$$ 
 By Lemma \ref{compactcone} we can find a further (not relabeled) subsequence such that $\psi_k\to \psi$ strongly in $L^1(B_1^+)$ for  some $\psi\in \mathscr{C}_0(\Lambda_0)$. Then $\widehat u_k\to \psi$ strongly in $L^1(B_1^+)$. Now we extract a further (not relabeled) subsequence such that $y_k\to y_*$ for some $y_*\in \overline D_{1/4}\setminus \mathscr{T}_{\tau/4}(V)$. Observe that 
\begin{multline*}
\|\psi-\varphi_{y_k,1}\|_{L^1(B^+_{3/4})} \leq \|\psi-\widehat v_k\|_{L^1(B^+_{3/4})}+ \|(\widetilde v_k)_{y_k,1}-\varphi_{y_k,1}\|_{L^1(B^+_{3/4})}\\
 \leq  \|\psi-\widehat v_k\|_{L^1(B^+_{1})}+ \|\widetilde v_k-\varphi\|_{L^1(B^+_{1})}\,.
\end{multline*}
By continuity of translations in $L^1$, and since $\widetilde v_k\to \varphi$, we infer that 
$$\|\psi-\varphi_{y_*,1}\|_{L^1(B^+_{3/4})}=\lim_{k\to\infty} \|\psi-\varphi_{y_k,1}\|_{L^1(B^+_{3/4})} =0\,.$$
In other words, $\psi=\varphi_{y_*,1}$ on $B^+_{3/4}$.  As a consequence, setting ${\bf y}_*:=(y_*,0)$, for every ${\bf x}\in B^+_{1/2}$ and $t\in(0,1)$, 
$$\varphi\big({\bf x}+t({\bf y}_*-{\bf x})\big) =\psi\big((1-t){\bf x}+(t-1){\bf y}_*\big)=\psi({\bf x}-{\bf y}_*)\,.$$ 
Differentiating first this identity with respect to $t$, and then letting $t\to 0$, we discover that 
$$0=({\bf y}_*-{\bf x})\cdot\nabla\varphi({\bf x})={\bf y}_*\cdot\nabla\varphi({\bf x}) \text{ for every ${\bf x}\in B^+_{1/2}$}\,. $$
By homogeneity of $\varphi$, it implies that ${\bf y}_*\cdot\nabla\varphi({\bf x})=0$ for every ${\bf x}\in\R^{n+1}_+$. Arguing as in the proof of Lemma \ref{lemSphi}, we deduce that 
$y_*\in S(\varphi)=V$, which contradicts the fact that  $y_*\in \overline D_{1/4}\setminus \mathscr{T}_{\tau/4}(V)$. 
\end{proof}

We finally prove the following corollary whose importance will be revealed in Section \ref{imprsect}. 

\begin{corollary}\label{corstrat3surf}
For every $\delta,\tau\in(0,1)$, there exists a constant 
$$\eta_3(\delta,\tau,r_0)=\eta_3(\delta,\tau,r_0,\Lambda_0,\mathcal{Z}_{\boldsymbol{\sigma}},b,n,s)\in (0,\delta]$$
(independent of $v$) such that for every $\rho\in(0,r_0/5]$ and $x\in\Omega^{r_0}$, the conditions
$${\bf d}_0(v,x,4\rho) \leq\eta_3\quad \text{and}\quad {\bf d}_{n}(v,x,4\rho)\geq \delta\,,$$
imply the existence of a linear subspace $V\subset\R^n$, with ${\rm dim}\,V\leq n-1$, for which
$${\bf d}_0(v,y,4\rho)> \eta_3 \quad \forall y\in D_{\rho}(x)\setminus \mathscr{T}_{\tau \rho}(x+V)\,.$$ 
\end{corollary}

\begin{proof}
We argue by induction on the dimension $\ell\in\{1,\ldots,n\}$ assuming that there exists a constant $\eta_{*,\ell}(\delta,\tau,r_0)\in (0,\delta]$ 
such that for every $\rho\in(0,r_0/5]$ and $x\in\Omega^{r_0}$, the conditions
$${\bf d}_0(v,x,4\rho) \leq\eta_{*,\ell}\quad \text{and}\quad {\bf d}_{\ell}(v,x,4\rho)\geq \delta\,,$$
imply the existence of a linear subspace $V$, with ${\rm dim}\,V\leq \ell-1$, for which
$${\bf d}_0(v,y,4\rho)> \eta_{*,\ell} \quad \forall y\in D_{\rho}(x)\setminus \mathscr{T}_{\tau \rho}(x+V)\,.$$ 
By Lemma \ref{condiistrat} this property holds for $\ell=1$ with $\eta_{*,1}(\delta,\tau)=\eta_2(\delta,\tau)$. 

Now we assume that the property holds at step $\ell$, and we prove that it also holds at step $\ell+1$. To this purpose, we choose 
$$ \eta_{*,\ell+1}(\delta,\tau,r_0):=\eta_{*,\ell}\big( \eta_{*,\ell}(\delta,\tau,r_0) ,\tau,r_0\big)\in \big(0, \eta_{*,\ell}(\delta,\tau,r_0)\big]\subset (0,\delta] \,.$$
Then we distinguish two cases. 

\noindent {\it Case 1).} If ${\bf d}_{\ell}(v,x,4\rho)\leq \eta_{*,\ell}$, then ${\bf d}_{\ell}(v,x,4\rho)\leq \eta_{2}$ and 
we can apply Lemma \ref{condiistrat}  to find the required linear subspace $V$ of dimension $\ell=(\ell+1)-1$. 

\noindent {\it Case 2).} If ${\bf d}_{\ell}(v,x,4\rho)>\eta_{*,\ell}$, then we apply the induction hypothesis to find the required linear subspace $V$ of dimension less than $\ell-1$. 

Now the conclusion follows for $\eta_3(\delta,\tau,r_0):= \eta_{*,n}(\delta,\tau,r_0)$.
\end{proof}

We now introduce the so-called {\sl singular strata} of $v$. For parameters $\delta\in(0,1)$, $0<r\leq r_0$, and an index $\ell\in\{0,\ldots,n-1\}$, we set 
$$\mathcal{S}^\ell_{r_0,r,\delta}(v):=\Big\{x\in\Omega^{r_0}: \boldsymbol{\Theta}_s(v,x)>0\text{ and }{\bf d}_{\ell+1}(v,x,\rho)\geq\delta\text{ for all } r\leq\rho\leq r_0\Big\}\,,$$
$$\mathcal{S}^\ell_{r_0,\delta}(v):=\bigcap_{0<r\leq r_0} \mathcal{S}^\ell_{r_0,r,\delta}(v)\quad\text{and}\quad \mathcal{S}^\ell_{r_0}(v):=\bigcup_{0<\delta<1}\mathcal{S}^\ell_{r_0,\delta}(v)\,.$$
According to \cite{FMS}, we have the following result. 

\begin{theorem}\label{volsingsrat}
For every $\kappa_0\in(0,1)$, there exists a constant 
$$C=C(\kappa_0,r_0,\Lambda_0,\mathcal{Z}_{\boldsymbol{\sigma}},b,n,s)>0$$ 
such that 
\begin{equation}\label{pff1710}
\mathscr{L}^n\Big(\mathscr{T}_r\big(\mathcal{S}^{n-1}_{r_0}(v)\big)\Big)\leq Cr^{1-\kappa_0}  \quad\forall r\in(0,r_0)\,. 
\end{equation}
In addition, ${\rm dim}_{\mathscr{H}}\big( \mathcal{S}^\ell_{r_0}(v)\big)\leq \ell$ for each $\ell\in\{1,\ldots,n-2\}$, and $ \mathcal{S}^0_{r_0}(v)$ is countable. 
\end{theorem}

\begin{proof}
By Lemma \ref{lemme2stratsurf} and Lemma \ref{condiistrat}, the functions $\boldsymbol{\Theta}_s(v,\cdot,\cdot)$ and ${\bf d}_\ell(v,\cdot,\cdot)$ satisfy the assumptions in \cite[Section 2.2]{FMS}. Then the dimension estimate on  $\mathcal{S}^\ell_{r_0}(v)$ for each $\ell\in\{1,\ldots,n-2\}$, and the fact that $ \mathcal{S}^0_{r_0}(v)$ is countable, follow from \cite[Theorem 2.3]{FMS}. 

According to Lemma \ref{gapdistlemma}, $\mathcal{S}^{n-1}_{r_0,\delta}(v)=\mathcal{S}^{n-1}_{r_0,\delta_0(r_0)}(v)$ for every $\delta\in(0,\delta_0(r_0)]$. Since the sets $\mathcal{S}^{n-1}_{r_0,\delta}(v)$ are decreasing with respect to $\delta$, we deduce that $\mathcal{S}^{n-1}_{r_0}(v)=\mathcal{S}^{n-1}_{r_0,\delta_0(r_0)}(v)$. Then, estimate \eqref{pff1710} follows from \cite[Theorem 2.2]{FMS}. 
\end{proof}

\begin{proof}[Proof of Theorem \ref{dimMink}]
We  choose $r_0>0$ in such a way that $\Omega^\prime\subset\Omega^{r_0}$. By Corollary~\ref{openset} and Lemma~\ref{gapdistlemma}, we have $\Sigma(v)\cap\Omega^\prime\subset \mathcal{S}^{n-1}_{r_0}(v)$. According to estimate \eqref{pff1710}, for every $\alpha\in(0,1)$ there exists a constant $C=C(\alpha,r_0)$ such that 
\begin{equation}
\mathscr{L}^n\big(\mathscr{T}_r(\Sigma(v)\cap\Omega^\prime)\big)\leq Cr^\alpha \quad\forall r\in(0,r_0)\,.
\end{equation}
Hence, 
$$\limsup_{r\downarrow 0}\left( n-\frac{\log\Big(\mathscr{L}^n\big(\mathscr{T}_r(\Sigma(v)\cap\Omega^\prime)\big)\Big)}{\log r}\right)\leq n-\alpha\quad\forall\alpha\in(0,1)\,, $$
and we obtain that the upper Minkowski dimension $\overline{\rm dim}_{\mathscr{M}}(\Sigma(v)\cap\Omega^\prime)$ is less than $n-1$. On the other hand, 
since $\Sigma(v)\cap \Omega^\prime\not=\emptyset$, we can find $j_0\in\{1,\ldots,m\}$ such that $\partial E^v_{j_0}\cap \Omega^\prime\not=\emptyset$. Hence  $E^v_{j_0}\cap \Omega^\prime$ 
is a not empty open subset of $\Omega^\prime$, distinct from $\Omega^\prime$.  Since $\Sigma(v)\cap \Omega^\prime\subset \partial E^v_{j_0}\cap \Omega^\prime$, we derive that 
$${\rm dim}_{\mathscr{H}}(\Sigma(v)\cap \Omega^\prime)\geq {\rm dim}_{\mathscr{H}}(\partial E^v_{j_0}\cap \Omega^\prime)\geq n-1\,.$$ 
On the other hand, the lower Minkowski dimension $\underline{\rm dim}_{\mathscr{M}}(\Sigma(v))$ is greater than than the Hausdorff dimension, and we conclude that ${\rm dim}_{\mathscr{M}}(\Sigma(v)\cap\Omega^\prime)=n-1$. 

To complete the proof, we show that 
\begin{equation}\label{identstrata}
{\rm Sing}^\ell(v)\cap\Omega^{r_0}\subset \mathcal{S}^\ell_{r_0}(v)\quad \text{for each } \ell\in\{0,\ldots,n-2\}\,,
\end{equation}
so that the conclusion follows from Theorem \ref{volsingsrat} (letting $r_0\to0$ along a decreasing sequence). To prove \eqref{identstrata}, 
we argue by contradiction assuming that there exists a point $x\in {\rm Sing}^\ell(v)\cap\Omega^{r_0}\setminus \mathcal{S}^\ell_{r_0}(v)$. Then, $x\not\in \mathcal{S}^{\ell}_{r_0,2^{-k}}(v)$ for every $k\in\mathbb{N}$. Hence, for each $k\in\mathbb{N}$, there exists a radius $r_k\in(0,r_0]$ such that $x\not\in  \mathcal{S}^{\ell}_{r_0,r_k,2^{-k}}(v)$, and therefore a radius $\rho_k\in[r_k,r_0]$ such that ${\bf d}_{\ell+1}(v,x,\rho_k)<2^{-k}$. Now we extract a (not relabeled) subsequence such that $\rho_k\to \rho_*$ for some $\rho_*\in[0,r_0]$. We distinguish the two following cases: 
\vskip3pt

\noindent{\it Case 1).} If $\rho_*=0$, then we can extract a further subsequence such that $v_{x,\rho_k}\to\varphi$ strongly in $H^1(B_1^+,|z|^a\de{\bf x})$ for some $\varphi\in T_x(v)$ (by Lemma \ref{tangmapstat}). In addition, 
$$ {\bf d}_{\ell+1}(\varphi,0,1)=\lim_{k\to\infty} {\bf d}_{\ell+1}(v_{x,\rho_k},0,1)=\lim_{k\to\infty} {\bf d}_{\ell+1}(v,x,\rho_k)=0\,,$$
so that $\varphi\in\mathscr{C}_{\ell+1}(\Lambda_0)$. Then ${\rm dim}\,S(\varphi)\geq \ell+1$ which contradicts $x\in {\rm Sing}^\ell(v)$. 
\vskip3pt

\noindent{\it Case 2).} If $\rho_*>0$,  then 
$${\bf d}_{\ell+1}(v_{x,\rho_*},0,1)=\lim_{k\to\infty} {\bf d}_{\ell+1}(v_{x,\rho_k},0,1)=\lim_{k\to\infty} {\bf d}_{\ell+1}(v,x,\rho_k)=0\,.$$
Hence there exists $\varphi\in\mathscr{C}_{\ell+1}(\Lambda_0)$ such that $v_{x,\rho_*}=\varphi$ on  $B^+_{1}$. Clearly, it implies that $T_x(v)=\{\varphi\}$, which contradicts  $x\in {\rm Sing}^\ell(v)$ as in Case 1). 
\end{proof}


\subsection{Application to stationary partitions}\label{complrestprescrNMMC}


In this subsection, we apply the previous results to our stationary point $\mathfrak{E}^*=(E^*_1,\ldots,E^*_m)\in\mathscr{A}_m(\Omega)$ of $\mathscr{P}^{\boldsymbol{\sigma}}_{2s}$ in $\Omega$, i.e., to a solution to \eqref{SectParteq1}. Recalling our assumption $\boldsymbol{\sigma}\in\mathscr{S}^2_m$, we introduce its indicator function 
$u_*\in H^s(\Omega;\mathcal{Z}_{\boldsymbol{\sigma}})$ given by \eqref{basdec1} and satisfying \eqref{campdev1847}, and then its extended function $u_*^\e$ which satisfies \eqref{campdev2}-\eqref{stateeq1458}. 

We  consider an increasing sequence of admissible bounded open sets $\{G_l\}_{l\in\mathbb{N}}$ such that $\overline{\partial^0G_l}\subset \Omega$, $\bigcup_l G_l=\R^{n+1}_+$, and $\bigcup_l\partial^0G_l=\Omega$. In view of \eqref{campdev2}-\eqref{stateeq1458}, we can apply to $u_*^\e$ the different results from Subsection \ref{subres1} to Subsection \ref{subsectstrat} in each $G_l$ to reach the following main conclusions. 
\vskip3pt

According to our previous notations, we have 
$$E^*_j=\bigcup_{l\in\mathbb{N}} E^{{u_*^\e}_{|G_l}}_j\quad \text{and}\quad \partial\mathfrak{E}^*\cap\Omega=\Sigma(u^\e_*)\cap\Omega:=\bigcup_{l\in\mathbb{N}}\Sigma({u^\e_*}_{|G_l})\,,$$ 
so that 
\begin{enumerate}
\item each set $E^*_j\cap\Omega$ is open (when identified with its set of density $1$ points); 
\vskip5pt

\item  
${\rm dim}_{\mathscr{M}}(\partial\mathfrak{E}^*\cap\Omega^\prime)\leq n-1$ for every open subset $\Omega^\prime$ such that $\overline{\Omega^\prime}\subset\Omega$ with equality  if and only if $\partial\mathfrak{E}^*\cap\Omega^\prime\not=\emptyset$;
\vskip5pt
\item There is a subset $\Sigma_{\rm sing}\subset\partial\mathfrak{E}^*\cap \Omega$ with ${\rm dim}_{\mathscr{H}}\Sigma_{\rm sing}\leq n-2$ (countable if $n=2$) such that the following holds: 

If $x_0\in  (\partial\mathfrak{E}^*\cap\Omega) \setminus \Sigma_{\rm sing}$, then there exist a sequence $\rho_k\downarrow 0$, distinct indices $i,j\in\{1,\ldots,m\}$, and a half space $H\subset \R^n$ with $0\in\partial H$ such that 
\begin{itemize}
\item $(u_*)_{x_0,\rho_k}\to \chi_{{\rm int}(H)}{\bf a}_i+\chi_{{\rm int}(H^c)}{\bf a}_j$ 
in $L^1_{\rm loc}(\R^n)$;
\vskip3pt

\item $\Sigma_k:=(\Sigma(u^\e_*)-x_0)/\rho_k$ converges locally uniformly to the hyperplane $\partial H$, i.e.,  for every compact set $K\subset \R^n$ and every $r>0$, $\Sigma_k\cap K\subset \mathscr{T}_r(\partial H)$ and $ \partial H\cap K\subset  \mathscr{T}_r(\Sigma_k)$ whenever $k$ is large enough.
\end{itemize}
\end{enumerate}

Our objective for the rest of this subsection is to show that the Minkowski dimension estimate on $\partial\mathfrak{E}^*\cap\Omega=\Sigma(u^\e_*)\cap\Omega$ leads to the following higher regularity result. 

\begin{theorem}\label{improvperim}
For every $s^\prime\in(0,1/2)$ and every open subset $\Omega^\prime\subset\Omega$ such that $\overline{\Omega^\prime}\subset \Omega$, we have  $u_*\in H^{s^\prime}(\Omega^\prime;\R^d)$ and 
$$P_{2s^\prime}(E^*_j,\Omega^\prime)<\infty\quad\forall j\in\{1,\ldots,m\}\,.$$
\end{theorem}
\vskip3pt

The proof of Theorem \ref{improvperim} (postponed to this end of the subsection) rests on the following regularity estimate taken from \cite[Proposition 6.34]{MilSirW}. 

\begin{proposition}[\cite{MilSirW}]\label{keypropimpr}
Let $w\in \widehat H^s(\Omega)$ be such that $w\in L^{\infty}_{\rm loc}(\Omega)$ and $(-\Delta)^s w\in L^{\bar p}_{\rm loc}(\Omega)$ for  some exponent $\bar p\in(1,\infty)$. Then, for every $s^\prime\in(0,s)$ and every open subsets $\Omega^\prime,\Omega^{\prime\prime}$ of $\Omega$ such that $\overline{\Omega^{\prime\prime}}\subset \Omega^\prime$ and $\overline{\Omega^\prime}\subset \Omega$, 
\begin{equation}\label{sameditroptard}
\left(\iint_{\Omega^{\prime\prime}\times \Omega^{\prime\prime}} \frac{|w(x)-w(y)|^{\bar p}}{|x-y|^{n+2s^\prime\bar p}}\,\de x\de y\right)^{1/\bar p}\leq C\Big(\|(-\Delta)^s w\|_{L^{\bar p}(\Omega^\prime)}+\|w\|_{L^{\infty}(\Omega^\prime)}\Big)\,,
\end{equation}
for some constant $C=C(n,s,s^\prime,\bar p,\Omega^{\prime},\Omega^{\prime\prime})$ independent of $w$.  
\end{proposition}

We also recall the following observation from  \cite[Lemma 6.35]{MilSirW}. 

\begin{lemma}[\cite{MilSirW}]\label{lemeqharmmap}
Let $F\subset\R^n$ be a Borel set such that $P_{2s}(F,\Omega)<\infty$. Then the function $w_F:=\chi_F-\chi_{F^c}$ belongs to $\widehat H^s(\Omega)$, and $(-\Delta)^sw_F\in L^1(\Omega)$ with 
$$ (-\Delta)^sw_F(x)=\left(\frac{\gamma_{n,s}}{2}\int_{\R^n}\frac{|w_F(x)-w_F(y)|^2}{|x-y|^{n+2s}}\,\de y\right)w_F(x)\quad \text{for a.e. $x\in \Omega$}\,.$$
\end{lemma}

Back to our map $u_*$, we combine Lemma \ref{lemeqharmmap} with the estimate on the Minkowski dimension to obtain 

\begin{proposition}\label{imprintsharp}
We have $(-\Delta)^su_*\in L^{\bar p}_{\rm loc}(\Omega;\R^d)$ for every $\bar p<1/2s$.
\end{proposition}

\begin{proof}
Let us fix two open subsets $\Omega^\prime,\Omega^{\prime\prime}$ of $\Omega$ such that  $\overline{\Omega^{\prime\prime}}\subset\Omega^{\prime}$ and $\overline{\Omega^\prime}\subset\Omega$. Setting  $w_j:=\chi_{E_j^*}-\chi_{(E_j^*)^c}$, we have $u_*=\frac{1}{2}\sum_{j}(w_j+1){\bf a}_j$. Hence 
\begin{equation}\label{decomplaplsu}
(-\Delta)^su_*=\frac{1}{2}\sum_{j=1}^m \big((-\Delta)^sw_j\big){\bf a}_j\,.
\end{equation}
On the other hand, 
by Lemma~\ref{lemeqharmmap} we have $(-\Delta)^sw_j\in L^{1}(\Omega^\prime)$ for each $j\in\{1,\ldots,m\}$.  
\vskip3pt

We claim that for each $j\in\{1,\ldots,m\}$, 
\begin{equation}\label{dim1819tard}
\big|(-\Delta)^sw_j(x)|\leq \frac{C(\Omega^\prime,\Omega^{\prime\prime})}{{\rm dist}(x,\Sigma(u^\e_*)\cap\Omega^\prime)^{2s}}\quad\text{for a.e. $x\in\Omega^{\prime\prime}\setminus\Sigma(u_*^\e)$}\,, 
\end{equation}
for some constant $C(\Omega^\prime,\Omega^{\prime\prime})$ independent of $u$. For $x\in \Omega^{\prime\prime}\setminus \Sigma(u_*^\e)$, we set 
$$r_x:=\frac{1}{2}\min\left({\rm dist}(x,\Sigma(u_*)\cap\Omega^\prime), \min\Big\{|z-y|: z\in\overline{\Omega^{\prime\prime}}\,,\;y\in\R^n\setminus\Omega^\prime\Big\}\right)\,.$$
Since $D_{r_x}(x)\subset \Omega^\prime\setminus \Sigma(u^\e_*)$, there exists $j_x\in\{1,\ldots,m\}$ such that $D_{r_x}(x)\subset E^*_{j_x}$ and $D_{r_x}(x)\subset (E^*_{j})^c$ for $j\not=j_x$. 
Then we can deduce from Lemma~\ref{lemeqharmmap} that
$$\big|(-\Delta)^sw_j(x)|\leq 2\gamma_{n,s}\int_{\R^n\setminus D_{r_x}(x)}\frac{1}{|x-y|^{n+2s}}\,\de y =\frac{C_{n,s}}{(r_x)^{2s}}\quad\forall j\in\{1,\ldots,m\}\,,$$
and \eqref{dim1819tard} follows.

Let us now fix an exponent $\alpha\in(2s\bar p, 1)$. Since ${\rm dim}_{\mathscr{M}}(\Sigma(u^\e_*)\cap\Omega^\prime)\leq n-1$, we can find a radius $R_\alpha\in(0,1)$ such that $\mathscr{L}^n(\mathscr{T}_r(\Sigma(u^\e_*)\cap \Omega^\prime))\leq r^\alpha$ for every $r\in(0,2R_\alpha)$. Then, we estimate for $j\in\{1,\ldots,m\}$ and an arbitrary integer $k\geq 1$, 
\begin{multline*}
\int_{\Omega^{\prime\prime}\setminus\mathscr{T}_{2^{-k}R_\alpha}(\Sigma(u_*^\e)\cap\Omega^\prime)}\big|(-\Delta)^sw_j|^{\bar p}\,\de x\leq 
\int_{\Omega^{\prime\prime}\setminus\mathscr{T}_{R_\alpha}(\Sigma(u_*^\e)\cap\Omega^\prime)}\big|(-\Delta)^sw_j|^{\bar p}\,\de x\\
+\sum_{j=0}^{k-1}\int_{\Omega^{\prime\prime}\cap\mathscr{A}_j}\big|(-\Delta)^sw_j|^{\bar p}\,\de x\,.
\end{multline*}
where we have set $\mathscr{A}_j:=\mathscr{T}_{2^{-j}R_\alpha}(\Sigma(u_*^\e)\cap\Omega^\prime)\setminus\mathscr{T}_{2^{-(j+1)}R_\alpha}(\Sigma(u_*^\e)\cap\Omega^\prime)$. 
Inserting \eqref{dim1819tard}, we derive 
$$\int_{\Omega^{\prime\prime}\setminus\mathscr{T}_{2^{-k}R_\alpha}(\Sigma(u_*^\e)\cap\Omega^\prime)}\big|(-\Delta)^sw_j|^{\bar p}\,\de x\leq CR_\alpha^{-2s\bar p}\left(1+\sum_{j=0}^{\infty}\frac{1}{2^{(\alpha-2s\bar p)j}} \right) <\infty\,.$$
Letting $k\to\infty$, we can now conclude by dominated convergence that $(-\Delta)^sw_j \in L^{\bar p}(\Omega^{\prime\prime})$ for each $j\in\{1,\ldots,m\}$. In view of  \eqref{decomplaplsu}, we then have $(-\Delta)^su_* \in L^{\bar p}(\Omega^{\prime\prime})$. 
\end{proof}

\begin{proof}[Proof of Theorem \ref{improvperim}]
Obviously, it is enough to consider the case $s^\prime\in(s,1/2)$. Fix two open subsets $\Omega^\prime,\Omega^{\prime\prime}$ of $\Omega$ such that  $\overline{\Omega^{\prime\prime}}\subset\Omega^{\prime}$ and $\overline{\Omega^\prime}\subset\Omega$. We choose a number $\theta>2$ such that $s^\prime<1/\theta$. We set 
$\bar p:=1/(\theta s)\in(s^\prime/s,1/2s)$, and $\bar s:=s^\prime/\bar p<s$. By Proposition \ref{imprintsharp}, we have $(-\Delta)^su_*\in L^{\bar p}_{\rm loc}(\Omega)$. Recalling that 
$$\chi_{E^*_j}={\rm p}_j(u_*)\quad \forall j\in\{1,\ldots,m\}\,,$$
where ${\rm p}_j:\R^d\to \R$ is the $\boldsymbol{\sigma}^{-1/2}_{\rm min}$-Lipschitz function given by \eqref{projtrickEj}, 
we infer from Proposition \ref{keypropimpr} that  for each $ j\in\{1,\ldots,m\}$, 
\begin{multline*}
\iint_{\Omega^\prime\times\Omega^\prime}\frac{|\chi_{E^*_j}(x)-\chi_{E^*_j}(y)|^2}{|x-y|^{n+2s^\prime}}\,\de x\de y=\iint_{\Omega^\prime\times\Omega^\prime}\frac{|\chi_{E^*_j}(x)-\chi_{E^*_j}(y)|^{\bar p}}{|x-y|^{n+2\bar s\bar p}}\,\de x\de y\\
\leq \frac{1}{(\boldsymbol{\sigma}_{\rm min})^{\bar p/2}} \iint_{\Omega^\prime\times\Omega^\prime}\frac{|u_*(x)-u_*(y)|^{\bar p}}{|x-y|^{n+2\bar s\bar p}}\,\de x\de y<\infty
\end{multline*}
 In other words, $\chi_{E^*_j}\in H^{s^\prime}(\Omega^\prime)$,  so that $u_*\in H^{s^\prime}(\Omega^\prime;\R^d)$ since $u_*=\sum_j\chi_{E^*_j}{\bf a}_j$.

Finally, we observe that for each $ j\in\{1,\ldots,m\}$, 
$$P_{2s^\prime}(E^*_j,\Omega^{\prime\prime})\leq \frac{1}{2}\iint_{\Omega^\prime\times\Omega^\prime}\frac{|\chi_{E^*_j}(x)-\chi_{E^*_j}(y)|^2}{|x-y|^{n+2s^\prime}}\,\de x\de y+C\,,$$
for a constant $C$ depending only $\Omega^\prime$ and $\Omega^{\prime\prime}$, $n$, and $s^\prime$. 
\end{proof}

We conclude this section with a corollary of Theorem~\ref{improvperim}. 

\begin{corollary}\label{corintegrpotentE}
In addition to Theorem~\ref{improvperim}, for each $j\in\{1,\ldots,m\}$, the potential 
$$V_{E_j^*}(x):= \Big(\gamma_{n,s}\int_{\R^n}\frac{|\chi_{E^*_j}(x)-\chi_{E^*_j}(y)|^2}{|x-y|^{n+2s}}\,\de y\Big) \big(2\chi_{E^*_j}(x)-1\big) $$
belongs to $L^{\bar p}_{\rm loc}(\Omega)$ for every $\bar p<1/2s$. 
\end{corollary}

\begin{proof}
We fix $j\in\{1,\ldots,m\}$ and two open subsets $\Omega^\prime,\Omega^{\prime\prime}$ of $\Omega$ such that  $\overline{\Omega^{\prime\prime}}\subset\Omega^{\prime}$ and $\overline{\Omega^\prime}\subset\Omega$.  We also fix an exponent $\bar p\in(1,1/2s)$, and  a constant 
$0<\delta<\frac{1-2s\bar p}{\bar p -1}$. We define $s^\prime$ by the relation
$$\frac{2s^\prime+\delta}{2s+\delta}=\bar p\,. $$
Note that our choice of $\delta$ implies $s^\prime\in(s,1/2)$. According to Theorem~\ref{improvperim}, we have 
\begin{equation}\label{integpotprf1}
\iint_{\Omega^\prime\times\Omega^\prime}\frac{|\chi_{E^*_j}(x)-\chi_{E^*_j}(y)|^2}{|x-y|^{n+2s^\prime}}\,\de x\de y<\infty\,.
\end{equation}
On the other hand, we have 
\begin{equation}\label{integpotprf2}
|V_{E_j^*}(x)|= \gamma_{n,s}\int_{\R^n}\frac{|\chi_{E^*_j}(x)-\chi_{E^*_j}(y)|^2}{|x-y|^{n+2s}}\,\de y
\leq \gamma_{n,s}\int_{\Omega^\prime}\frac{|\chi_{E^*_j}(x)-\chi_{E^*_j}(y)|^2}{|x-y|^{n+2s}}\,\de y+C
\end{equation}
for a.e. $x\in \Omega^{\prime\prime}$ and a constant $C$ depending only on $n$, $s$, $\Omega^{\prime}$, and $\Omega^{\prime\prime}$. Then we apply H\"older inequality to infer that 
\begin{align}
\nonumber\int_{\Omega^\prime}\frac{|\chi_{E^*_j}(x)-\chi_{E^*_j}(y)|^2}{|x-y|^{n+2s}}\,\de y & \leq \Big(\int_{\Omega^\prime}\frac{|\chi_{E^*_j}(x)-\chi_{E^*_j}(y)|^{2\bar p}}{|x-y|^{n+2s^\prime}}\,\de y\Big)^{\frac{1}{\bar p}}\Big(\int_{\Omega^\prime}\frac{dy}{|x-y|^{n-\delta}}\Big)^{1-\frac{1}{\bar p}}\\
&\leq  C  \Big(\int_{\Omega^\prime}\frac{|\chi_{E^*_j}(x)-\chi_{E^*_j}(y)|^{2}}{|x-y|^{n+2s^\prime}}\,\de y\Big)^{\frac{1}{\bar p}} \label{integpotprf3}
\end{align}
for a.e. $x\in \Omega^{\prime\prime}$ and a constant $C$ depending on $\delta$ and $\bar p$.  Combining \eqref{integpotprf1}, \eqref{integpotprf2} and \eqref{integpotprf3}, we conclude that 
$$\int_{\Omega^{\prime\prime}}|V_{E_j^*}(x)|^{\bar p}\,\de x\leq C\Big(\iint_{\Omega^\prime\times\Omega^\prime}\frac{|\chi_{E^*_j}(x)-\chi_{E^*_j}(y)|^2}{|x-y|^{n+2s^\prime}}\,\de x\de y+1\Big)<\infty\,, $$
which completes the proof. 
\end{proof}


\section{Volume of transition sets and  improved estimates}\label{imprsect}


In this section, our goal is to apply the quantitative stratification principle to improve the convergence results of  Theorem \ref{main1part1}. Following \cite{MilSirW},  it allows to obtain a quantitative volume estimate on the transition set (i.e., where a solution remains far from the set of wells $\mathcal{Z}$ of the potential $W$). This estimate, combined with Lemma~\ref{estifond}, provides improved estimates on the potential part of the energy and on the reaction term in the equation.   
 As performed in Section~\ref{FGLasymp}, we start with estimates on the Allen-Cahn  boundary  reactions equation. 


\subsection{Quantitative estimates for boundary reactions}\label{subsectstratAC}


In this subsection, we are back to the analysis initiated in Section~\ref{subsecbdryreac}. We  consider a bounded admissible open set $G\subset \R^{n+1}_+$,  $\eps\in(0,1)$, and a weak solution $v_\eps\in H^{1}(G;\R^d,|z|^a\de{\bf x})\cap L^\infty(G)$ of  
\begin{equation}\label{eqACagain}
\begin{cases}
{\rm div}(z^a\nabla v_\eps) =0 & \text{in $G$}\,,\\[8pt]
\displaystyle \boldsymbol{\delta}_s\boldsymbol{\partial}^{(2s)}_z v_\eps=\frac{1}{\varepsilon^{2s}}\nabla W(v_\eps)  & \text{on $\partial^0 G$}\,.
\end{cases}
\end{equation}
We shall {\it fix constants} $r_0> 0$, $b\geq 1$, and $\Lambda_0\geq 0$ such that 
\begin{equation}\label{controlLinftyueps}
 \|v_\eps\|_{L^\infty(G)}\leq b\,, 
 \end{equation}
 and 
\begin{equation}\label{controldensiteps}
\sup\Big\{\boldsymbol{\Theta}_{s,\eps}(v_\eps,x,\rho) : x\in\Omega^{r_0}\,,\;0<\rho\leq r_0\Big\}\leq \Lambda_0\,, 
\end{equation}
where the domain $\Omega^{r_0}$ is defined in \eqref{defOmegar}. 
\vskip3pt

The following volume estimate is the key result of this section.

\begin{theorem}\label{volesti}
For each $\alpha\in(0,1)$, there exist  two constants ${\bf k_*}={\bf k_*}(\alpha,r_0,\Lambda_0,W,b,n,s)>0$  
and 
$C=C(\alpha,r_0,\Lambda_0,W,{\rm diam}(\partial^0G),b,n,s)$ such that  
\begin{equation}\label{volestieq}
\mathscr{L}^n\Big(\mathscr{T}_r\big(\{{\rm dist}(v_\eps,\mathcal{Z})>\boldsymbol{\varrho}_W\}\cap \Omega^{r_0}\big)\Big)\leq Cr^\alpha\qquad\forall r\in({\bf k_*}\eps,r_0)\,, 
\end{equation}
where $\boldsymbol{\varrho}_W\in(0,1]$ is given by \eqref{lwbdhess}.
\end{theorem}

The proof of Theorem \ref{volesti} is once again  close to \cite[Proof of Theorem 7.1]{MilSirW} and inspired from \cite[Theorem 2.2]{FMS}.  
The solution $v_\varepsilon$ being a smooth map, there is of  course no singular set $\Sigma$ in this setting. Instead, 
the idea is to consider $v_\varepsilon$ at scales of order larger than $\eps$ and consider the transition set  $\{{\rm dist}(v_\eps,\mathcal{Z})>\boldsymbol{\varrho}_W\}$ in place of a singular set, knowing that, formally at least, this set should have a thickness of order $\eps$. This consideration explains the  lower bound on the admissible radii appearing in the theorem above and in the following ``preparatory lemmae".

\begin{lemma}\label{lemme1strateps}
There exist constants 
$$\widetilde\delta_0(r_0)=\delta_0(r_0,\Lambda_0,W,b,n,s)\in(0,1)$$ 
and 
$${\bf k}_0(r_0)={\bf k}_0(r_0,\Lambda_0,W,b,n,s) \geq 1$$ 
(independent of $\eps$ and $v_\eps$)  such that for every $x\in\Omega^{r_0}$ and 
$\rho\in(0,r_0)$,
$$ {\rm dist}(v_\eps(x,0),\mathcal{Z})>\boldsymbol{\varrho}_W \quad\text{and}\quad {\bf k}_0\eps\leq \rho\quad\Longrightarrow\quad {\bf d}_n(v_\eps,x,\rho) \geq\widetilde\delta_0 \,.$$
\end{lemma}

\begin{proof}
Assume by contradiction that there exist sequences $\{\eps_k\}_{k\in\mathbb{N}}\subset(0,1)$, $\{v_k\}_{k\in\mathbb{N}}$ satisfying \eqref{eqACagain}-\eqref{controlLinftyueps}-\eqref{controldensiteps}, points $\{x_k\}_{k\in\mathbb{N}}\subset\Omega^{r_0}$, and radii $\{\rho_k\}_{k\in\mathbb{N}}\subset(0,r_0)$ such that  $ {\rm dist}(v_k(x_k,0),\mathcal{Z})>\boldsymbol{\varrho}_W$,  
${\bf d}_n(v_k,x_k,\rho_k)\to 0$, and $\eps_k/\rho_k\leq 2^{-k}$. 

Setting $\widetilde \eps_k:=\eps_k/\rho_k$,  consider the rescaled map $\widetilde v_k:=(v_k)_{x_k,\rho_k}$ as defined in \eqref{defrescmap}. Rescaling variables, we derive that 
\begin{equation}\label{rescACstrat} 
\begin{cases}
{\rm div}(z^a\nabla \widetilde v_k) =0 & \text{in $B_1^+$}\,,\\[8pt]
\displaystyle \boldsymbol{\delta}_s\boldsymbol{\partial}^{(2s)}_z \widetilde v_k=\frac{1}{(\widetilde\varepsilon_k)^{2s}}W^\prime(\widetilde v_k)  & \text{on $D_1$}\,,
\end{cases}
\end{equation}
and
\begin{equation}\label{contrrescACstrat}
\| \widetilde v_k\|_{L^\infty(B_1^+)}\leq b\,,
\end{equation}
as well as
\begin{equation}\label{contrdensrescACstrat}
\boldsymbol{\Theta}_{s,\widetilde \eps_k}(\widetilde v_k,0,1) =\boldsymbol{\Theta}_{s,\eps_k}(v_k,x_k,\rho_k)\leq \Lambda_0\,. 
\end{equation}
By Theorem \ref{thmasympbdryreact}, we can find a (not relabeled) subsequence such that $\widetilde v_k\to v_*$ weakly in $H^1(B_1^+,|z|^a\de{\bf x})$ (and strongly in $H^1(B_r^+,|z|^a\de{\bf x})$ for every $r\in(0,1)$). Hence $\widetilde v_k\to v_*$  in $L^1(B^+_1)$ (see e.g. \cite[Remark 2.4]{MilSirW}).  
On the other hand, ${\bf d}_n(\widetilde v_k,0,1)={\bf d}_n(v_k,x_k,\rho_k)\to 0$, so that ${\bf d}_n(v_*,0,1)=0$, i.e., $v_*={\bf a}$ for some ${\bf a}\in \mathcal{Z}$.  
Then Theorem \ref{thmasympbdryreact} tells us that $\widetilde v_k\to {\bf a}$ uniformly on $D_r$ for every $r\in(0,1)$. In particular ${\rm dist}(\widetilde v_k(0),\mathcal{Z})\to 0$ which contradicts ${\rm dist}(\widetilde v_k(0),\mathcal{Z})={\rm dist}(v_k(x,0),\mathcal{Z})>\boldsymbol{\varrho}_W$. 
\end{proof}

\begin{lemma}\label{lemme2strateps}
For every $\delta>0$, there exist constants 
$$\widetilde\eta_1(\delta)=\widetilde\eta_1(\delta,\Lambda_0,\mathcal{Z},b,n,s)\in (0,1/4)\,,$$ 
$$\widetilde \lambda_1(\delta,r_0)=\widetilde\lambda_1(\delta,r_0,\Lambda_0,W,b,n,s)\in (0,1/4)\,,$$ 
and 
$${\bf k}_1(\delta,r_0)={\bf k}_1(\delta,r_0,\Lambda_0,W,b,n,s) \geq 1$$ 
(independent of $\eps$ and $v_\eps$) such that for every 
$\rho\in(0,r_0/5)$ and $x\in\Omega^{r_0}$,   
$$\boldsymbol{\Theta}_{s,\eps}(v_\eps,x,\rho)-\boldsymbol{\Theta}_{s,\eps}(v_\eps,x,\widetilde\lambda_1\rho) \leq \widetilde\eta_1\;\text{ and }\; {\bf k}_1\eps\leq \rho \quad \Longrightarrow\quad {\bf d}_0(v_\eps,x,\rho) \leq\delta \,.$$ 
\end{lemma}

\begin{proof}
We choose 
$$\widetilde\eta_1(\delta,r_0):=\eta_1(\delta/2,2/5,\Lambda_0,\mathcal{Z},b,n,s)\,, $$
where $\eta_1$ is given by Lemma \ref{lemme2stratsurf}. Once again we argue by contradiction assuming that for some constant $\delta>0$, there exist sequences $\{\eps_k\}_{k\in\mathbb{N}}\subset(0,1)$, $\{v_k)\}_{k\in\mathbb{N}}$ satisfying \eqref{eqACagain}-\eqref{controlLinftyueps}-\eqref{controldensiteps}, points $\{x_k\}_{k\in\mathbb{N}}\subset\Omega^{r_0}$, radii $\{\rho_k\}_{k\in\mathbb{N}}\subset(0,r_0/5)$, and $\lambda_k\to 0$  such that $\eps_k/\rho_k\leq 2^{-k}$, 
$$\boldsymbol{\Theta}_{s,\eps_k}(v_k,x_k,\rho_k)-\boldsymbol{\Theta}_{s,\eps_k}(v_k,x_k,\lambda_k\rho_k) \leq \widetilde\eta_1\,,\quad\text{and}\quad {\bf d}_0(v_k,x_k,\rho_k) >\delta\,.$$ 
We proceed as in the proof of Lemma \ref{lemme1strateps} rescaling variables as $\widetilde \eps_k:=\eps_k/(5\rho_k)$, $\widetilde u_k:=(u_k)_{x_k,5\rho_k}$.
Then, \eqref{rescACstrat}, \eqref{contrrescACstrat}, and \eqref{contrdensrescACstrat}  hold, as well as 
\begin{equation}\label{1451jhu}
 \sup\Big\{\boldsymbol{\Theta}_{s,\widetilde \eps_k}(\widetilde v_k,x,\rho) : x\in D_{1/5}\,,\;0<\rho\leq 2/5\Big\}\leq \Lambda_0\,.
 \end{equation}
Now our assumptions lead to 
$$\boldsymbol{\Theta}_{s,\widetilde \eps_k}(\widetilde v_k,0,1/5)-\boldsymbol{\Theta}_{s,\widetilde \eps_k}(\widetilde v_k,0,\lambda_k/5) \leq \widetilde\eta_1\,,\quad\text{and}\quad {\bf d}_0(\widetilde v_k,0,1/5) >\delta\,.$$ 
By Theorem \ref{thmasympbdryreact}, there is a (not relabeled) subsequence such that $\widetilde v_k\to v_*$ strongly in $H^1(B_r^+,|z|^a\de{\bf x})$  for every $r\in(0,1)$ (and weaky for $r=1$),  
where $v_*$ solves  \eqref{eqcone}-\eqref{statinB1}. 
In addition, by Theorem   \ref{thmasympbdryreact}, we can deduce from \eqref{1451jhu} that
\begin{equation}\label{peuplusdutotu}
\sup\Big\{\boldsymbol{\Theta}_s(v_*,x,\rho) : x\in D_{1/5}\,,\;0<\rho\leq 2/5\Big\}\leq \Lambda_0\,.
\end{equation}
Thanks to Corollary \ref{monotformACeq}, for $0<r<1/5$ and $k$ large enough  (in such a way that $\lambda_k<r$), we have 
$$\boldsymbol{\Theta}_{s,\widetilde \eps_k}(\widetilde v_k,0,1/5)-\boldsymbol{\Theta}_{s,\widetilde \eps_k}(\widetilde v_k,0,r) \leq \widetilde\eta_1\,. $$
Using Theorem  \ref{thmasympbdryreact} again, we can let $k\to\infty$ in this inequality to derive
$$\boldsymbol{\Theta}_s(v_*,0,1/5)- \boldsymbol{\Theta}_s(v_*,0,r)\leq  \widetilde\eta_1\,.$$
Choosing $r$ small enough in such a way that 
$$r\leq \frac{1}{5} \lambda_1(\delta/2,2/5,\Lambda_0,b,n,s)\,, $$
where $\lambda_1$ is given Lemma \ref{lemme2stratsurf}, we infer from Lemma \ref{monotformsurf} that
$$\boldsymbol{\Theta}_s(v_*,0,1/5)- \boldsymbol{\Theta}_s(v_*,0,\lambda_1/5)\leq \widetilde\eta_1\,.$$
Then Lemma \ref{lemme2stratsurf} yields $ {\bf d}_0(v_*,0,1/5) \leq\delta/2$.  On the other hand, $\widetilde v_k\to v_*$ in $L^1(B_1^+)$, and thus 
$ {\bf d}_0(v_*,0,1/5) =\lim_k{\bf d}_0(\widetilde v_k,0,1/5) \geq\delta$, a contradiction. 
\end{proof}

\begin{lemma}\label{lemme3strateps}
For every $\delta,\tau\in(0,1)$, there exist two constants 
$$\widetilde \eta_2(\delta,\tau)=\widetilde \eta_2(\delta,\tau,\Lambda_0,\mathcal{Z},b,n,s)\in (0,\delta]$$
and 
$${\bf k}_2(\delta,\tau,r_0)= {\bf k}_2(\delta,\tau,r_0,\Lambda_0,W,b,n,s)\geq 1$$
(independent of $\eps$ and $v_\eps$) such that for every $\rho\in(0,r_0/25)$ and $x\in\Omega^{r_0}$, the conditions
$${\bf k}_2\eps\leq \rho\,,\quad{\bf d}_0(v_\eps,x,4\rho) \leq\widetilde \eta_2\quad \text{and}\quad {\bf d}_{n}(v_\eps,x,4\rho)\geq \delta\,,$$
imply the existence of a linear subspace $V\subset\R^n$, with ${\rm dim}\,V\leq n-1$, for which
$${\bf d}_0(v_\eps,y,4\rho)> \widetilde\eta_2 \quad \forall y\in D_{\rho}(x)\setminus \mathscr{T}_{\tau \rho}(x+V)\,.$$ 
\end{lemma}

\begin{proof}
We choose 
$$\widetilde\eta_2(\delta,\tau):=\frac{1}{2}\eta_3(\delta,\tau,2/5,\Lambda_0,\mathcal{Z}, b,n,s)\,, $$
where $\eta_3$ is given by Corollary \ref{corstrat3surf}. We still argue by contradiction assuming that for some constants $\delta,\tau\in(0,1)$, there exist sequences $\{\eps_k\}_{k\in\mathbb{N}}\subset(0,1)$, $\{v_k\}_{k\in\mathbb{N}}$ satisfying \eqref{eqACagain}-\eqref{controlLinftyueps}-\eqref{controldensiteps}, points $\{x_k\}_{k\in\mathbb{N}}\subset\Omega^{r_0}$, and radii $\{\rho_k\}_{k\in\mathbb{N}}\subset(0,r_0/25)$  such that 
$$\eps_k/\rho_k\leq 2^{-k}\,,\quad{\bf d}_0(v_k,x_k,4\rho_k) \leq\widetilde \eta_2\quad \text{and}\quad {\bf d}_{n}(v_k,x_k,4\rho_k)\geq \delta\,,$$
and such that the conclusion of the lemma fails. 

As before, we rescale variables setting $\widetilde \eps_k:=\eps_k/(25\rho_k)$, $\widetilde v_k:=(v_k)_{x_k,25\rho_k}$,  
so that \eqref{rescACstrat}, \eqref{contrrescACstrat}, \eqref{contrdensrescACstrat}, and \eqref{1451jhu} hold. Then we reproduce the proof of Lemma \ref{lemme2strateps} to find a (not relabeled) subsequence along which $\widetilde v_k$ converges to some $v_*$ solving \eqref{eqcone}-\eqref{statinB1}, and satisfying \eqref{contrrescACstrat}-\eqref{contrdensrescACstrat}-\eqref{peuplusdutotu}. In particular, $\widetilde v_k\to v_*$  in $L^1(B^+_{1/5})$, and as a consequence,  
$${\bf d}_0(v_*,0,4/25) \leq\widetilde \eta_2\quad\text{and}\quad {\bf d}_{n}(v_*,0,4/25)\geq \delta\,.$$ 
By Corollary \ref{corstrat3surf}, there exists a linear subspace $V\subset \R^n$, with ${\rm dim}\,V\leq n-1$, such that 
\begin{equation}\label{1825Yanhouse}
{\bf d}_0(v_*,y,4/25)> \eta_3 \quad \forall y\in D_{1/25}\setminus \mathscr{T}_{\tau /25}(V)\,.
\end{equation}
Since the conclusion of the lemma does not hold, we can find for each $k$ a point $y_k\in  D_{1/25}\setminus \mathscr{T}_{\tau /25}(V)$ such that 
${\bf d}_0(\widetilde v_k,y_k,4/25)\leq\widetilde\eta_2$.  Then extract a further subsequence such that $y_k\to y_*$ for some $y_*\in \overline D_{1/25}\setminus \mathscr{T}_{\tau /25}(V)$. Noticing that 
$$\|(v_*)_{y_*,1}-(\widetilde v_k)_{y_k,1} \|_{L^1(B^+_{4/25})} \leq \|(v_*)_{y_*,1}-(v_*)_{y_k,1} \|_{L^1(B^+_{4/25})}+ \|v_*-\widetilde v_k \|_{L^1(B^+_{1/5})}\,,$$
by continuity of translations in $L^1$, we have  $\|(v_*)_{y_*,1}-(\widetilde v_k)_{y_k,1} \|_{L^1(B^+_{4/25})}\to 0$.  Consequently, 
$${\bf d}_0(v_*,y_*,4/25)={\bf d}_0\big((v_*)_{y_*,1},0,4/25\big)
=\lim_{k\to\infty}{\bf d}_0\big((\widetilde v_k)_{y_k,1},0,4/25)=\lim_{k\to\infty}{\bf d}_0(\widetilde v_k,y_k,4/25) \,, $$
and thus ${\bf d}_0(u_*,y_*,4/25)\leq \widetilde\eta_2$.  However \eqref{1825Yanhouse} (together with the continuity of translations in $L^1$) yields  ${\bf d}_0(v_*,y_*,4/25)\geq \eta_3=2 \widetilde\eta_2$, a contradiction.
\end{proof}

We are now ready to prove Theorem \ref{volesti}. Here again, we follow a proof in \cite{MilSirW}, which itself deeply rests on \cite{FMS}, in particular concerning certain covering lemmae. 

\begin{proof}[Proof of Theorem \ref{volesti}]
For $0<r\leq r_0$, we consider the set 
$$\mathcal{S}^\eps_{r_0,r}:=\bigg\{x\in\Omega^{r_0} : {\bf d}_n(v_\eps,x,\rho)\geq \widetilde\delta_0(r_0)\;\;\forall \,r\leq \rho\leq r_0\bigg\} \,,$$
where  $ \widetilde\delta_0(r_0)>0$ is given by Lemma \ref{lemme1strateps}. We fix  $\alpha\in(0,1)$, and we set $\kappa_0:=1-\alpha\in(0,1)$. 
\vskip3pt

We shall prove that there exist two constants ${\bf k_*}={\bf k_*}(\kappa_0,r_0,\Lambda_0,W,b,n,s)\geq {\bf k}_0(r_0)\geq 1$  
and $C=C(\kappa_0,r_0,\Lambda_0,W,b,n,s)$ such that  
\begin{equation}\label{estvolstarteps}
\mathscr{L}^n\big(\mathscr{T}_r(\mathcal{S}^\eps_{r_0,r})\big)\leq C r^{1-\kappa_0} \quad\forall r\in({\bf k}_*\eps,r_0)\,,
\end{equation}
where $ {\bf k}_0(r_0)$ is  given by Lemma  \ref{lemme1strateps}. Note that, since ${\bf k_*}\geq {\bf k}_0(r_0)$, we have 
$$\big\{{\rm dist}(v_\eps,\mathcal{Z})>\boldsymbol{\varrho}_W\big\}\cap\Omega^{r_0}\subset \mathcal{S}^\eps_{r_0,r} \quad\forall r\in ({\bf k}_*\eps,r_0)\,,$$
by Lemma  \ref{lemme1strateps}. In other words, estimates \eqref{estvolstarteps} implies Theorem \ref{volesti}. 
\vskip3pt

Now the proof follows closely the arguments in \cite[proof of Theorem~2.2]{FMS} once adjusted to our setting, but  for the sake of clarity we partially reproduce it.  

We start {\sl fixing} a number $\tau=\tau(\kappa_0,n)\in(0,1)$ such that $\tau^{\kappa_0/2}\leq 20^{-n}$. 
We consider the following constants according to Lemma \ref{lemme1strateps}, Lemma \ref{lemme2strateps}, and Lemma~\ref{lemme3strateps}:
\begin{enumerate}
\item[(i)] $\widetilde\eta_2:=\widetilde\eta_2\big(\widetilde\delta_0(r_0),\tau\big)$ and ${\bf k}_2:={\bf k}_2\big(\widetilde\delta_0(r_0),\tau,r_0\big)$; 
\vskip3pt

\item[(ii)] $\widetilde\eta_1:=\widetilde\eta_1\big(\widetilde\eta_2\big)$, $\widetilde\lambda_1:=\widetilde\lambda_1\big(\widetilde\eta_2,r_0\big)$, and ${\bf k}_1:={\bf k}_1\big(\widetilde\eta_2,r_0\big)$; 
\vskip3pt

\item[(iii)] ${\bf k}_3:=\max\{{\bf k}_0(r_0), {\bf k}_1, {\bf k}_2\}$.   
\end{enumerate}
Next we fix an integer $q_0\in\mathbb{N}$ such that $\tau^{q_0}\leq \widetilde\lambda_1$, and we set $M:=\lfloor q_0\Lambda_0/ \widetilde\eta_1\rfloor$ (the integer part of).  Set 
$p_0:=q_0+M+1$ and define 
$$\boldsymbol{\eps}_0:=\min\left\{1,\frac{r_0\tau^{p_0+1}}{25{\bf k}_3}\right\}\,,\quad {\bf k}_*:=\frac{{\bf k}_3}{\tau}\,.$$
Without loss of generality, we may assume that $\eps\in(0,\boldsymbol{\eps}_0)$ (since \eqref{estvolstarteps} is straightforward for $\eps\geq \boldsymbol{\eps}_0$). Let ${\bf k}_4={\bf k}_4(\eps)$ be defined by the relation 
$$r_0\tau^{{\bf k}_4|\log\eps|}=25{\bf k}_3\eps\,,$$ 
and set 
$$p_1=p_1(\eps):=\lfloor {\bf k}_4|\log\eps|\rfloor$$ 
(the integer part of).  Note that our choice of $\boldsymbol{\eps}_0$ and ${\bf k}_*$ insures that 
$$p_1\geq p_0+1 \;\text{ and }\;{\bf k}_3\eps\leq\frac{r_0\tau^{p_1}}{25}\leq {\bf k}_*\eps\,.$$
\vskip3pt

\noindent{\it Step 1. Reduction to $\tau$-adic radii.} We argue exactly as in \cite[Proof of Theorem 2.2, Step~1]{FMS} to show that  it suffices to prove \eqref{estvolstarteps} for each radius $r$ of the form  $r=\frac{r_0\tau^k}{25}$ for an integer $k$ satisfying $p_0\leq k\leq p_1$. 
\vskip3pt

\noindent{\it Step 2. Selection of good scales.} We fix an integer $k$ with $p_0\leq k\leq p_1$ and set $r:=\frac{r_0\tau^k}{25}$.  
For an arbitrary $x\in\Omega^{r_0}$, we have 
\begin{multline*}
\sum_{l=q_0}^k \boldsymbol{\Theta}_{s,\eps}(v_\eps,x,4r_0\tau^{l})-\boldsymbol{\Theta}_{s,\eps}(v_\eps,x,4r_0\tau^{l+q_0})\\  
=\sum_{l=q_0}^k \sum_{i=l}^{l+q_0-1}\boldsymbol{\Theta}_{s,\eps}(v_\eps,x,4r_0\tau^{i})-\boldsymbol{\Theta}_{s,\eps}(v_\eps,x,4r_0\tau^{i+1})  \\
\leq q_0 \sum_{l=q_0}^{k+q_0-1}\boldsymbol{\Theta}_{s,\eps}(v_\eps,x,4r_0\tau^{l})-\boldsymbol{\Theta}_{s,\eps}(v_\eps,x,4r_0\tau^{l+1})\,,
\end{multline*}
and thus
$$\sum_{l=q_0}^k \boldsymbol{\Theta}_{s,\eps}(v_\eps,x,4r_0\tau^{l})-\boldsymbol{\Theta}_{s,\eps}(v_\eps,x,4r_0\tau^{l+q_0}) \leq q_0 \boldsymbol{\Theta}_{s,\eps}(v_\eps,x,4r_0\tau^{q_0}) \leq  q_0\Lambda_0\,. $$
Hence there exists a (possibly empty) subset $A(x)\subset\{q_0,\ldots,k\}$ with ${\rm Card}(A(x))\leq M$ such that  for every $l\in\{q_0,\ldots,k\}\setminus A(x)$, 
\begin{equation}\label{jhu1333}
\boldsymbol{\Theta}_{s,\eps}(v_\eps,x,4r_0\tau^{l})-\boldsymbol{\Theta}_{s,\eps}(v_\eps,x,4r_0\tau^{l+q_0})\leq \widetilde\eta_1\,.
\end{equation}
Next define $\mathfrak{A}:=\{A\subset\{q_0,\ldots,k\} : {\rm Card}(A)= M\}$, and set for $A\in\mathfrak{A}$,  
$$\mathcal{S}_A:= \bigg\{x\in\mathcal{S}^\eps_{r_0,r}:  \text{ \eqref{jhu1333} holds for each } l\in\{q_0,\ldots,k\}\setminus A\bigg\}\,.$$
As previously observed, we have $\mathcal{S}^\eps_{r_0,r}\subset \bigcup_{A\in\mathfrak{A}} \mathcal{S}_A$. 

In the next step, we shall prove that for any $A\in\mathfrak{A}$, 
\begin{equation}\label{Yanikoumouk}
\mathscr{L}^n\big(\mathscr{T}_r(\mathcal{S}_A)\big) \leq C r^{1-\kappa_0/2}\,.
\end{equation}
Since ${\rm Card}(\mathfrak{A})  \leq k^M\leq C|\log r|^M$, the conclusion follows from this estimate, i.e., 
$$\mathscr{L}^n\big( \mathscr{T}_r(\mathcal{S}^\eps_{r_0,r})\big) \leq\sum_{A\in\mathfrak{A}} \mathscr{L}^n\big(\mathscr{T}_r(\mathcal{S}_A)\big)\leq  C |\log r|^M  r^{1-\kappa_0/2}\leq C  r^{1-\kappa_0}\,,$$
for some constants $C=C(\kappa_0,r_0,\Lambda_0,W, {\rm diam}(\partial^0G),b,n,s)$. 
\vskip3pt

\noindent{\it Step 3. Proof of \eqref{Yanikoumouk}.} Again we follow \cite[Proof of Theorem 2.2, Step~3]{FMS}. We first consider a finite cover of $\mathscr{T}_{r_0\tau^{q_0}/25}(\mathcal{S}_A)$ made of discs $\{D_{r_0\tau^{q_0}}(x_{i,q_0})\}_{i\in I_{q_0}}$ with $x_{i,q_0}\in\mathcal{S}_A$, and 
$${\rm Card}(I_{q_0})\leq 5^n\tau^{-nq_0}r_0^{-n}({\rm diam}(\partial^0G)+1)^n\,. $$
We  argue now by iteration on the integer $j\in\{q_0+1,\ldots,k\}$, assuming that  we already have a cover $\{D_{r_0\tau^{j-1}}(x_{i,j-1})\}_{i\in I_{j-1}}$ of $\mathscr{T}_{r_0\tau^{j-1}/25}(\mathcal{S}_A)$ such that $x_{i,j-1}\in\mathcal{S}_A$. 
We select the next cover $\{D_{r_0\tau^{j}}(x_{i,j})\}_{i\in I_{j}}$ (still centered at points of $\mathcal{S}_A$) of $\mathscr{T}_{r_0\tau^{j}/25}(\mathcal{S}_A)$ according to the following two cases: $j-1\in A$ or $j-1\not\in A$. 

\noindent {\it Case 1)} If $j-1\in A$, then we proceed exactly as in  \cite[Proof of Theorem 2.2, Step~3, Case~(a)]{FMS} to produce the new cover $\{D_{r_0\tau^{j}}(x_{i,j})\}_{i\in I_{j}}$  in such a way that 
$${\rm Card}(I_{j})\leq 20^n{\rm Card}(I_{j-1})\tau^{-n}\,.  $$

\noindent {\it Case 2)} If $j-1\not\in A$, then \eqref{jhu1333} holds with $l=j-1$. By our choice of $q_0$ and Corollary~\ref{monotformACeq}, we infer that 
\begin{multline*}
\boldsymbol{\Theta}_{s,\eps}(v_\eps,x_i,4r_0\tau^{j-1})-\boldsymbol{\Theta}_{s,\eps}(v_\eps,x_i,4r_0\widetilde\lambda_1\tau^{j-1})\\
\leq \boldsymbol{\Theta}_{s,\eps}(v_\eps,x_i,4r_0\tau^{j-1})-\boldsymbol{\Theta}_{s,\eps}(v_\eps,x_i,4r_0\tau^{j-1+q_0})\leq \widetilde\eta_1\quad\forall x\in\mathcal{S}_A\,.
\end{multline*}
Then Lemma \ref{lemme2strateps} yields ${\bf d}_0(v_\eps,x,4r_0\tau^{j-1}) \leq \widetilde\eta_2$ for every $x\in\mathcal{S}_A$. On the other hand, by the  definition of $\mathcal{S}_A$ we have ${\bf d}_n(v_\eps,x,4r_0\tau^{j-1}) \geq \widetilde\delta_0$ for every $x\in\mathcal{S}_A$. Applying Lemma~\ref{lemme3strateps} at each point $x_{i,j-1}$, we infer that for each $i\in I_{j-1}$, there is a linear subspace $V_i$, with ${\rm dim}\,V_i\leq n-1$, such that  $\mathcal{S}_A\cap D_{r_0\tau^{j-1}}(x_{i,j-1})\subset \mathscr{T}_{r_0\tau^j}(x_{i,j-1}+V_i)$. 
From this inclusion, we estimate for each $i\in I_{j-1}$, 
$$\mathscr{L}^n\bigg(\mathscr{T}_{r_0\tau^j}\big(\mathcal{S}_A\cap D_{r_0\tau^{j-1}}(x_{i,j-1})\big)\bigg)\leq 2^{n+1}\omega_{n-1}r^n_0\tau^{nj-n+1} \,.$$
By the covering lemma in \cite[Lemma 3.2]{FMS}), we can find a cover of $\mathscr{T}_{r_0\tau^j/25}(\mathcal{S}_A)$ by discs 
$\{D_{r_0\tau^{j}}(x_{i,j})\}_{i\in I_{j}}$ centered on $\mathcal{S}_A$  such that 
$$ {\rm Card}(I_{j})\leq 10^n\frac{2\omega_{n-1}}{\omega_n}{\rm Card}(I_{j-1})  \tau^{-(n-1)}\leq  20^n{\rm Card}(I_{j-1})  \tau^{-(n-1)} \,.$$
The iteration procedure stops at $j=k$, and it yields a cover $\{D_{r_0\tau^{k}}(x_{i,k})\}_{i\in I_{k}}$ of $\mathscr{T}_{r}(\mathcal{S}_A)$. Collecting the estimates from Case 1 and Case 2 (and using ${\rm Card}\,A=M$), we derive
\begin{multline*}
 {\rm Card}(I_{k})\leq   5^n\tau^{-nq_0}r_0^{-n}({\rm diam}(\partial^0G)+1)^n(20^n\tau^{-n})^M\left(20^n\tau^{-(n-1)}\right)^{k-q_0-M} \\
 \leq C\tau^{-k(n-1+\kappa_0/2)}\,,
 \end{multline*}
where $C$ depends on the announced parameters (recall that $\tau^{\kappa_0/2}\leq 20^{-n}$). Consequently, 
$$\mathscr{L}^n\big(\mathscr{T}_r(\mathcal{S}_A)\big) \leq\omega_n {\rm Card}(I_{k})r^n\leq C \tau^{k(1-\kappa_0/2)}\leq C r^{1-\kappa_0/2}\,, $$
and the proof is complete.  
\end{proof}

\begin{corollary}\label{corpotquant}
For  every $\alpha\in(0,1)$,  
$$\int_{\Omega^{2r_0}}W(v_\eps)\,\de x\leq C \eps^{\min(4s,\alpha)} \,,$$
for some constant $C=C(\alpha,r_0,\Lambda_0,W,b,n,s)$.
\end{corollary}

\begin{proof} Without loss of generality, we may assume that $\alpha\not=4s$. 
We use the notation of the proof of Theorem  \ref{volesti}, and we assume (without loss of generality) that $\eps\in(0,\boldsymbol{\eps}_0)$.  

Let us write $\mathcal{V}_\eps:=\{{\rm dist}(v_\eps,\mathcal{Z})>\boldsymbol{\varrho}_W\}$, and $\rho_k:=\frac{r_0\tau^{k}}{25}$ for $k\in\mathbb{N}$. Notice that 
$$\rho_{p_1(\eps)-1}\in({\bf k}_*\varepsilon,{\bf k}_*\tau^{-1}\eps)\,. $$
Hence, by Theorem  \ref{volesti}, we have 
\begin{equation}\label{olala2251}
\mathscr{L}^n\left(\mathscr{T}_{\rho_k}(\mathcal{V}_\eps\cap\Omega^{r_0})\right)\leq C \rho_k^{\alpha}\leq C \tau^{\alpha k}\quad\text{for $k=0,\ldots,p_1(\eps)-1$}\,,
\end{equation}
where the constant $C$ may depend on the announced parameters. In particular, 
\begin{equation}\label{toto2253}
\int_{\mathscr{T}_{\rho_{p_1(\eps)-1}}(\mathcal{V}_\eps\cap\Omega^{r_0})}W(v_\eps)\,\de x\leq C\|W\|_{L^\infty(-b,b)}\rho_{p_1(\eps)-1}^\alpha\leq C\eps^\alpha\,.
\end{equation}
On the other hand, by Lemma \ref{estifond}, we have 
\begin{equation}\label{titi2254}
W\big(v_\eps(x,0)\big)\leq \frac{C\eps^{4s}}{\big({\rm dist}(x,\mathcal{V}_\eps)\big)^{4s}}\leq \frac{C\eps^{4s}}{\big({\rm dist}(x,\mathcal{V}_\eps\cap\Omega^{r_0})\big)^{4s}} \quad\text{for $x\in\Omega^{2r_0}\setminus \mathcal{V}_\eps$}\,.
\end{equation}
Writing $\mathscr{A}_k:=\big(\mathscr{T}_{\rho_{k-1}}(\mathcal{V}_\eps\cap\Omega^{r_0})\setminus\big(\mathscr{T}_{\rho_k}(\mathcal{V}_\eps\cap\Omega^{r_0})\big)$, 
we have
\begin{multline*}
\int_{\Omega^{2r_0}}W(v_\eps)\,\de x  = 
\int_{\Omega^{2r_0}\setminus\mathscr{T}_{\rho_0}(\mathcal{V}_\eps)}W(v_\eps)\,\de x+\int_{\mathscr{T}_{\rho_{p_1(\eps)-1}}(\mathcal{V}_\eps\cap\Omega^{r_0})\cap\Omega^{2r_0}}W(v_\eps)\,\de x \\
+\sum_{k=1}^{p_1(\eps)-1}\int_{\mathscr{A}_k\cap\Omega^{2r_0}}W(v_\eps)\,\de x\,,
 \end{multline*}
 We may now estimate by \eqref{olala2251}, \eqref{toto2253}, and \eqref{titi2254}, 
 \begin{equation}\label{tata2256}
\int_{\Omega^{2r_0}}W(v_\eps)\,\de x \leq C\left(\eps^{4s}+ \eps^\alpha+\eps^{4s}\sum_{k=1}^{p_1(\eps)-1}\tau^{k(\alpha -4s)}\right)\,.
\end{equation}
If $\alpha>4s$, then $\sum_{k\geq 1}\tau^{k(\alpha-4s)}<\infty$, and the result is proved. If $\alpha<4s$, then 
$$ \sum_{k=1}^{p_1(\eps)-1}\tau^{k(\alpha -4s)}\leq C\tau^{p_1(\eps)(\alpha-4s)}\leq C\eps^{\alpha-4s}\,.$$
Inserting this estimate in \eqref{tata2256} still yields the announced result. 
\end{proof}

\begin{corollary}\label{corpotquantbis}
For every $\bar p <1/2s$, 
$$\left\|\nabla W(v_\eps)\right\|_{L^{\bar p}(\Omega^{2r_0})}\leq C\eps^{2s} \,,$$
for some constant $C=C(\bar p,r_0,\Lambda_0,W,b,n,s)$. 
\end{corollary}

\begin{proof}
We proceed as in the proof Corollary \ref{corpotquant}, using $\alpha\in(2s\bar p,1)$. Keeping the same notation, we first derive as in \eqref{toto2253}, 
\begin{equation}\label{toto2253bis}
\int_{\mathscr{T}_{\rho_{p_1(\eps)-1}}(\mathcal{V}_\eps\cap\Omega^{r_0})}\big|\nabla W(v_\eps)\big|^{\bar p}\,\de x \leq C\eps^\alpha\,.
\end{equation}
Then Lemma \ref{estifond} yields, 
\begin{equation}\label{titi2254bis}
\big|\nabla W\big(v_\eps(x,0)\big)\big|\leq \frac{C\eps^{2s}}{\big({\rm dist}(x,\mathcal{V}_\eps\cap\Omega^{r_0})\big)^{2s}} \quad\text{in $\Omega^{2r_0}\setminus \mathcal{V}_\eps$}\,.
\end{equation}
Writing 
\begin{multline*}
\int_{\Omega^{2r_0}}\big|\nabla W (v_\eps)\big|^{\bar p}\,\de x  = 
\int_{\Omega^{2r_0}\setminus\mathscr{T}_{\rho_0}(\mathcal{V}_\eps)}\big|\nabla W(v_\eps)\big|^{\bar p}\,\de x\\
+\int_{\mathscr{T}_{\rho_{p_1(\eps)-1}}(\mathcal{V}_\eps\cap\Omega^{r_0})\cap\Omega^{2r_0}}\big|\nabla W(v_\eps)\big|^{\bar p}\,\de x 
+\sum_{k=1}^{p_1(\eps)-1}\int_{\mathscr{A}_k\cap\Omega^{2r_0}}\big|\nabla W(v_\eps)\big|^{\bar p}\,\de x\,,
 \end{multline*}
 we estimate by means of \eqref{olala2251}, \eqref{toto2253bis}, and \eqref{titi2254bis}, 
 $$\int_{\Omega^{2r_0}}\big|\nabla W (v_\eps)\big|^{\bar p}\,\de x \leq  C\left(\eps^{2s\bar p}+ \eps^\alpha+\eps^{2s\bar p}\sum_{k=1}^{p_1(\eps)-1}\tau^{k(\alpha -2s\bar p)}\right)\leq C\eps^{2s\bar p}\,,$$
 and the proof is complete. 
\end{proof}


\subsection{Application to the fractional Allen-Cahn equation}


Applying the estimates obtained in the previous section to the fractional Allen-Cahn equation, we obtain the following improvement of Theorem~\ref{main1part1}. Combined with Theorem \ref{main1part1}, Remark \ref{smoothbdrycondrem}, and Theorem \ref{main1part1mincase2}, it completes the proof of Theorems \ref{main1new} \& \ref{main2new}. 

\begin{theorem}\label{thmlastone}
In addition to Theorem \ref{main1part1},  
for every open subset $\Omega^\prime\subset \Omega$ such that $\overline{\Omega^\prime}\subset \Omega$, 
\begin{enumerate}
\item[\rm (i)]  $u_k\to u_*$ strongly in $H^{s^\prime}(\Omega^\prime)$ for every $s^\prime\in\big(0,\min(2s,1/2)\big)$;
\vskip5pt

\item[\rm (ii)] $\int_{\Omega^\prime}W(v_k)\,\de x=O\big(\eps_k^{\min(4s,\alpha)}\big)$ for every $\alpha\in(0,1)$;  
\vskip5pt

\item[\rm (iii)] $\frac{1}{\eps_k^{2s}}\nabla W(v_k(x))\rightharpoonup  - \sum_{j=1}^m V_{E^*_j}{\bf a}_j$ 
weakly in $L^{\bar p}(\Omega^\prime)$ where each $V_{E^*_j}$ belongs to $L^{\bar p}(\Omega^\prime)$ 
\vskip3pt

\noindent for every $\bar p<1/2s$.  
\end{enumerate}
\end{theorem}

\begin{proof}
The proof departs from the end of the proof of Theorem \ref{main1part1}. We apply the results of Subsection~\ref{subsectstratAC} to the extended function $u_k^\e$. Then items (ii) and the weak convergence in $L^{\bar p}(\Omega^\prime)$ in (iii) are straightforward consequences of Corollaries \ref{corpotquant} and \ref{corpotquantbis} (together with item (iii) in Theorem \ref{main1part1}).  The fact that each potential $V_{E^*_j}$ belongs to $L^{\bar p}(\Omega^\prime)$ is given by Corollary \ref{corintegrpotentE}. 

Let us now fix an open subset $\Omega^{\prime\prime}\subset \Omega^\prime$ with  Lipschitz boundary such that $\overline{\Omega^{\prime\prime}}\subset\Omega^\prime$. 
Given $s^\prime$ as in~${\rm (i)}$, we can find  $\theta>\max(2,1/2s)$ such that $ \max(s,s^\prime)<1/\theta$. We set 
$\bar p:=1/(\theta s)<\min(1/2s,2)$, and $\bar s:=s^\prime/\bar p<s$. 

Then it also follows from Corollary \ref{corpotquantbis} and equation \eqref{fracalllcahnf}
that $\{(-\Delta)^su_k\}_{k\in\mathbb{N}}$ is bounded in $L^{\bar p}(\Omega^\prime)$, which,  
combined with  Proposition~\ref{keypropimpr}, yields 
\begin{multline}\label{verylastruc}
\iint_{\Omega^{\prime\prime}\times\Omega^{\prime\prime}}\frac{|u_k(x)-u_k(y)|^{2}}{|x-y|^{n+2s^\prime}}\,\de x\de y\\
\leq 2^{2-\bar p}\|u_k\|^{2-\bar p}_{L^\infty(\R^n)} \iint_{\Omega^{\prime\prime}\times\Omega^{\prime\prime}}\frac{|u_k(x)-u_k(y)|^{\bar p}}{|x-y|^{n+2\bar s\bar p}}\,\de x\de y\leq C\,, 
\end{multline}
for some constant $C$ independent of $k$. The  sequence $\{u_k\}_{k\in\mathbb{N}}$ is thus bounded in $H^{s^\prime}(\Omega^{\prime\prime})$. Finally, for an arbitrary $s^{\prime\prime}\in(0,s^\prime)$, the embedding $H^{s^{\prime\prime}}(\Omega^{\prime\prime})\subset H^{s^\prime}(\Omega^{\prime})$ is compact, and consequently $\{u_k\}_{k\in\mathbb{N}}$ is strongly relatively compact in 
$H^{s^{\prime\prime}}(\Omega^{\prime\prime})$ which proves (i). 
\end{proof}


\section{Regularity for minimizing nonlocal partitions} \label{regminpartsect}   
								 

In this subsection, we aim to provide a partial regularity result  for the boundaries of a partition which is minimizing the functional $\mathscr{P}^{\boldsymbol{\sigma}}_{2s}$ 
in $\Omega$, still in case where $\boldsymbol{\sigma}\in\mathscr{S}^2_m$. We shall rely on the results from Section \ref{GHSNMC} and our discussion departs from Subsection \ref{complrestprescrNMMC}. In other words, throughout this section, we assume that  $\mathfrak{E}^*=(E^*_1,\ldots,E^*_m)\in\mathscr{A}_m(\Omega)$ is given and it is $\mathscr{P}^{\boldsymbol{\sigma}}_{2s}$- minimizing on $\Omega$ in the sense of Definition \ref{defminpart}. In particular,  $\mathfrak{E}^*$ satisfies the stationarity condition \eqref{SectParteq1}, 
and the results from Subsection~\ref{complrestprescrNMMC}~hold.

\subsection{The Non Infiltration Property}\label{NIPsect}

The strategy to prove a partial regularity for $\partial \mathfrak{E}^*\cap \Omega$ is to show that in a small neighborhood of every point in the set $\Sigma_{\rm reg}(u_*^\e)\cap\Omega$, 
there are only two phases of the partition  $\mathfrak{E}^*$. In turn, it will imply that these two phases are almost minimizers of the $2s$-perimeter $P_{2s}$ in a smaller neighborhood, and existing regularity results from \cite{CRS,CapGui,DVV} will apply.  By Proposition~\ref{regblowupset}, the feature of points in $\Sigma_{\rm reg}(u_*^\e)\cap\Omega$ is that only two phases are dominant {\sl in proportion} in small neighborhood of the point. Our aim is thus to prove that, in a smaller neighborhood, those two phases are the only one. This property is often referred to as {\it Non Infiltration Property} of the remaining phases. In our context, it has been first proved by \cite{ColMag} in the homogeneous case where $\sigma_{ij}=1$ for every $i\not= j$, and stated without proof in \cite{CesNov}
in the so-called additive case where $\sigma_{ij}=\alpha_i+\alpha_j$ for every $i\not= j$ with $\alpha_i,\alpha_j>0$ (as a matter of of facts, it is claimed without proof that the argument in \cite{ColMag} can be applied verbatim in the additive case, but there is some evidence that it is not the case, at least for $m\geq 4$). 
\vskip3pt

Here, we shall prove a non infiltration property  in three different cases for the matrix  $\boldsymbol{\sigma}$. The first case is the nearly homogeneous case where the coefficients of $\boldsymbol{\sigma}$ are nearly identical, with an explicit condition on their admissible deviation. The second case is restricted to three phases ($m=3$) and requires a {\sl strict triangle inequality} between the coefficient of $\boldsymbol{\sigma}$ as in the classical regularity theory for the standard perimeter, see \cite{Leo,W}. The third case, the most surprising one, is still for three phases, but requires that the triangle inequality fails {\sl strictly}. In this latter situation, we shall prove that non infiltration holds for two of the phases, and this would be enough to conclude regularity thanks to a ``simple'' classification argument of tangent maps. 

We may now underline that, in the three cases, the non infiltration property does not require the matrix $\boldsymbol{\sigma}$ to be in the set $\mathscr{S}^2_m$. Let us start with the nearly homogeneous case. 

\begin{proposition}[\bf Nearly homogeneous case]\label{NIPnearlyhom}
Assume that $\boldsymbol{\sigma}\in\mathscr{S}_m$ satisfies
\begin{equation}\label{CondNearHomCoeff}
q:=\frac{(m-2)\boldsymbol{\sigma}_{\rm max}}{(m-1)\boldsymbol{\sigma}_{\rm \min}}<1\,.
\end{equation}
There exists a constant $\mathfrak{p}_1=\mathfrak{p}_1(q,s,n)>0$ (depending only on $q$, $s$, and $n$) such  that for every ball $\overline D_{r}(x_0)\subset\Omega$ and 
every $h\in\{1,\ldots,m\}$,  the following implication holds: 
$$|\{E^*_h\cap D_r(x_0)\}|\leq  \mathfrak{p}_1r^n\quad\Longrightarrow\quad |\{E^*_h\cap D_{r/2}(x_0)\}|=0\,.$$
\end{proposition}
\vskip5pt

The proof of Proposition \ref{NIPnearlyhom} is inspired from \cite{ColMag}, and follows its general strategy.  However, we tried to made the argument more transparent to handle different cases.  
To this purpose, we need some preliminary lemmata, starting with the following observation. 

\begin{lemma}\label{minimlaityinset}
If $\mathfrak{F}=(F_1,\ldots,F_m)\in\mathscr{A}_m(\Omega)$ satisfies $F_j\triangle E^*_j\subset A$ for each $j\in\{1,\ldots,m\}$ and  some measurable subset $A\subset\Omega$, then 
$$\mathscr{P}^{\boldsymbol{\sigma}}_{2s}(\mathfrak{E}^*,A)\leq \mathscr{P}^{\boldsymbol{\sigma}}_{2s}(\mathfrak{F},A)\,.$$
\end{lemma}

\begin{proof}
For an admissible competitor $\mathfrak{F}=(F_1,\ldots,F_m)$, we observe that 
\begin{multline*}
\mathcal{I}_{2s}(F_i\cap\Omega,F_j\cap\Omega)= \mathcal{I}_{2s}(F_i\cap A,F_j\cap A)+ \mathcal{I}_{2s}\big(F_i\cap A,F_j\cap (\Omega\setminus A)\big)\\
+ \mathcal{I}_{2s}\big(F_i\cap (\Omega\setminus A),F_j\cap A\big) +  \mathcal{I}_{2s}\big(E^*_i\cap (\Omega\setminus A),E^*_j\cap  (\Omega\setminus A)\big)\,,
\end{multline*}
and 
$$\mathcal{I}_{2s}(F_i\cap\Omega,F_j\cap\Omega^c)= \mathcal{I}_{2s}(F_i\cap A,F_j\cap \Omega^c)+ \mathcal{I}_{2s}\big(E^*_i\cap (\Omega\setminus A),E^*_j\cap \Omega^c\big)\,,
$$
and 
$$\mathcal{I}_{2s}(F_i\cap\Omega^c,F_j\cap\Omega)= \mathcal{I}_{2s}(F_i\cap \Omega^c,F_j\cap A)+ \mathcal{I}_{2s}\big(E^*_i\cap \Omega^c,E^*_j\cap (\Omega\setminus A)\big)\,.
$$
Summing these identities, we obtain
\begin{multline*}
\mathscr{P}^{\boldsymbol{\sigma}}_{2s}(\mathfrak{F},\Omega)=\mathscr{P}^{\boldsymbol{\sigma}}_{2s}(\mathfrak{F},A)\\
 + \frac{1}{2}\sum^m_{i,j=1}\sigma_{ij}\Big[\mathcal{I}_{2s}\big(E^*_i\cap(\Omega\setminus A),E^*_j\cap(\Omega\setminus A)\big)+\mathcal{I}_{2s}\big(E^*_i\cap(\Omega\setminus A),E^*_j\cap\Omega^c\big)+\mathcal{I}_{2s}(E^*_i\cap\Omega^c,E^*_j\cap\big(\Omega\setminus A)\big)\Big]\,.
 \end{multline*}
Obviously, the same identity holds for  $\mathfrak{E}^*$, and the conclusion follows from the minimality of $\mathfrak{E}^*$ in~$\Omega$. 
\end{proof}

We shall make use of the following classical property, very well known as {\sl De Giorgi iteration}. Its proof is elementary and simply goes by induction. 

\begin{lemma}\label{DGiter}
Let $\alpha\in(0,1)$ and $M>0$. If $\{c_k\}_{k\in\mathbb{N}}$ is a sequence of non negative numbers satisfying
$$c_{k+1}^{1-\alpha}\leq 2^kM c_k \qquad\forall k\in\mathbb{N}\,,$$
with 
\begin{equation}\label{condIterDG}
c_0\leq  \frac{M^{-1/\alpha}}{2^{(1-\alpha)/\alpha^2}}\,,
\end{equation}
then $c_k\leq 2^{-k\alpha}c_0$ for every $k\in\mathbb{N}$. In particular, $c_k\to 0$ as $k\to\infty$.  
\end{lemma}

The previous iteration lemma allows us to show one of the key tools which  is  the following extinction property for an integro-differential inequation, inspired from \cite{FFMMM}. 

\begin{lemma}\label{integineqlemma}
Let $\alpha,\gamma\in(0,1)$, $R>0$, $K>0$, and $f:[0,R]\to [0,\infty]$  an absolutely continuous non decreasing function satisfying 
\begin{equation}\label{ineqinteg}
f(r)^{1-\alpha}\leq K \int_0^r\frac{f^\prime(t)}{(r-t)^{\gamma}}\,\de t\quad\text{for every $r\in(0,R]$}\,.
\end{equation}
Setting
$$\mathfrak{p}:=  
\bigg(\frac{1-\gamma}{2^{\frac{1+\alpha}{\alpha}}K}\bigg)^{\frac{1}{\alpha}}\,,$$
if $f(R)\leq \mathfrak{p}  R^{\gamma/\alpha}$, then $f(R/2)=0$. 
\end{lemma}

\begin{proof}
First we integrate \eqref{ineqinteg} on $(0,\ell)\subset (0,R]$ and apply Fubini's theorem to obtain 
$$\int_{0}^\ell f(r)^{1-\alpha}\,\de r\leq \frac{K}{1-\gamma}\ell^{1-\gamma}f(\ell)\qquad\forall\ell\in(0,R]\,.$$
For $k\in\mathbb{N}$, we set $\ell_k:=\frac{R}{2}+2^{-(k+1)}R$ and $c_k:=f(\ell_k)$, so that $\ell_0=R$ and $c_0=f(R)$.  
Then, 
$$ 2^{-(k+2)}R\,c^{1-\alpha}_{k+1} \leq \int_{\ell_{k+1}}^{\ell_k} f(r)^{1-\alpha}\,\de r\leq \int_{0}^{\ell_k} f(r)^{1-\alpha}\,\de r\leq \frac{K}{1-\gamma}\ell^{1-\gamma}_kc_k
\leq \frac{KR^{1-\gamma}}{1-\gamma} c_k\,.$$
Hence, 
$$ c^{1-\alpha}_{k+1} \leq 2^kM c_k\quad\text{with}\quad M:=\frac{4K}{(1-\gamma)}\, R^{-\gamma}\,.$$
By our choice of $\mathfrak{p}$, condition \eqref{condIterDG} holds and Lemma \ref{DGiter} applies. It shows that  $c_k\to 0$, and since $c_k\to f(R/2)$, the conclusion follows. 
\end{proof}

\begin{proof}[Proof of Proposition~\ref{NIPnearlyhom}]
We  fix an arbitrary point $x_0\in\Omega$ and we set $r_0:={\rm dist}(x_0,\partial\Omega)$ so that $D_{r_0}(x_0)\subset\Omega$. Without loss of generality, we may assume that $x_0=0$ . Then we consider the non decreasing Lipschitz function $f:[0,r_0]\to[0,\infty)$ given by 
$$f(r):=|E^*_h\cap  D_r| \,, $$
which satisfies  $f^\prime(r)=\mathcal{H}^{n-1}(E^*_h\cap \partial D_r)$ for a.e. $r\in(0,r_0)$. 
\vskip5pt

\noindent{\it Step 1.} We fix for the moment an index $k\in\{1,\ldots,m\}\setminus\{h\}$ and $r\in(0,r_0)$. We consider the competitor $\mathfrak{F}^k=(F^k_1,\ldots,F^k_m)$ given by 
$$F^k_j:=\begin{cases}
E^*_j & \text{if $j\not\in\{k,h\}$}\,,\\
E^*_h\setminus D_r & \text{if $j=h$}\,,\\
E^*_k\cup(E^*_h\cap D_r) & \text{if $j=k$}\,.
\end{cases}$$
By construction,  $F^k_j\triangle E^*_j $ is either empty or $F^k_j\triangle E^*_j\subset E^*_h\cap D_r$ which is compactly included in $\Omega$. By minimality of $\mathfrak{E}^*$ and Lemma~\ref{minimlaityinset} (applied with $A=E^*_h\cap D_r$), we have 
\begin{equation}\label{ineqminmalinsetinfiltr}
\mathscr{P}^{\boldsymbol{\sigma}}_{2s}(\mathfrak{E}^*,E^*_h\cap D_r)\leq \mathscr{P}^{\boldsymbol{\sigma}}_{2s}(\mathfrak{F}^k,E^*_h\cap D_r)\,.
\end{equation}
On the other hand, elementary manipulations yield 
\begin{align}
\nonumber \mathscr{P}^{\boldsymbol{\sigma}}_{2s}(\mathfrak{E}^*,E^*_h\cap D_r)&=\sum_{j\not=h}\sigma_{hj}\,\mathcal{I}_{2s}(E^*_h\cap D_r,E^*_j\cap D_r) 
+ \sum_{j=1}^m\sigma_{hj}\,\mathcal{I}_{2s}(E^*_h\cap D_r,E^*_j\cap D^c_r)\\
&\geq \sum_{j\not=h}\sigma_{hj}\,\mathcal{I}_{2s}(E^*_h\cap D_r,E^*_j\cap D_r) \,, \label{ineqperpartnoinfil}
\end{align}
where we use $E^*_h\cap D_r=((E^*_h)^c\cap D_r)\cup D_r^c$, the symmetry of the matrix $\boldsymbol{\sigma}$, and the disjointness between the~$E^*_j$'s. Similarly, we obtain 
\begin{align}
\nonumber \mathscr{P}^{\boldsymbol{\sigma}}_{2s}(\mathfrak{F}^k,E^*_h\cap D_r)& = \sum_{j\not=h}\sigma_{kj}\,\mathcal{I}_{2s}(E^*_h\cap D_r,E^*_j\cap D_r) 
+ \sum_{j=1}^m\sigma_{kj}\,\mathcal{I}_{2s}(E^*_h\cap D_r,E^*_j\cap D^c_r)\\
&\leq \sum_{j\not=h}\sigma_{kj}\,\mathcal{I}_{2s}(E^*_h\cap D_r,E^*_j\cap D_r)\! +\boldsymbol{\sigma}_{\rm max} \,\mathcal{I}_{2s}(E^*_h\cap D_r,D^c_r)\,, \label{ineqperpartnoinfil2}
\end{align}
using the fact that $F^k_i\cap (E^*_h\cap D_r)=E^*_h\cap D_r$ for $i=k$, and $F^k_i\cap (E^*_h\cap D_r)=\emptyset$ for $i\not=k$. 

Moreover, using the coarea formula, we estimate  
\begin{multline}\label{coareainfiltr}
\mathcal{I}_{2s}(E^*_h\cap D_r,D^c_r)\leq \int_{E^*_h\cap D_r}\Big(\int_{D^c_{r-|x|}(x)}\frac{1}{|x-y|^{n+2s}}\,\de y \Big)\de x\\
 =\frac{n\omega_n}{2s}\int_{E^*_h\cap D_r}\frac{1}{(r-|x|)^{2s}}\,\de x=\frac{n\omega_n}{2s}\int_0^r \frac{\mathcal{H}^{n-1}(E^*_h\cap \partial D_t)}{(r-t)^{2s}}\,\de t \,.
\end{multline}
Combining \eqref{ineqminmalinsetinfiltr} with \eqref{ineqperpartnoinfil}-\eqref{ineqperpartnoinfil2}-\eqref{coareainfiltr}, we conclude that for every $k\in\{1,\ldots,m\}\setminus\{h\}$, 
\begin{multline}\label{generineqinfiltrprop}
 \sum_{j\not=h}\sigma_{hj}\,\mathcal{I}_{2s}(E^*_h\cap D_r,E^*_j\cap D_r) \leq  \sum_{j\not=h}\sigma_{kj}\,\mathcal{I}_{2s}(E^*_h\cap D_r,E^*_j\cap D_r) \\
 +\frac{n\omega_n\boldsymbol{\sigma}_{\rm max}}{2s}\int_0^r \frac{f^\prime(t)}{(r-t)^{2s}}\,\de t \,.
 \end{multline}

\noindent{\it Step 2.}  
Summing up over $k\not=h$ the inequalities in \eqref{generineqinfiltrprop}, we infer that 
\begin{multline}\label{ineqfondnearhomoginfiltr}
(m-1)\sum_{j\not=h}\sigma_{hj}\,\mathcal{I}_{2s}(E^*_h\cap D_r,E^*_j\cap D_r) \leq  \sum_{j\not=h}\Big(\sum_{k\not=h}\sigma_{kj}\Big)\mathcal{I}_{2s}(E^*_h\cap D_r,E^*_j\cap D_r) \\
 +(m-1)\frac{n\omega_n\boldsymbol{\sigma}_{\rm max}}{2s}\int_0^r \frac{f^\prime(t)}{(r-t)^{2s}}\,\de t \,.
 \end{multline}
 On one hand, we have 
\begin{multline}\label{ineqfondnearhomoginfiltr2}
(m-1)\sum_{j\not=h}\sigma_{hj}\,\mathcal{I}_{2s}(E^*_h\cap D_r,E^*_j\cap D_r)\geq (m-1)\boldsymbol{\sigma}_{\rm min} \sum_{j\not=h}\mathcal{I}_{2s}(E^*_h\cap D_r,E^*_j\cap D_r) \\
=(m-1)\boldsymbol{\sigma}_{\rm min}\, \mathcal{I}_{2s}(E^*_h\cap D_r,(E^*_h)^c\cap D_r)\,,
 \end{multline}
while 
\begin{multline}\label{ineqfondnearhomoginfiltr3}
\sum_{j\not=h}\Big(\sum_{k\not=h}\sigma_{kj}\Big)\mathcal{I}_{2s}(E^*_h\cap D_r,E^*_j\cap D_r) \leq (m-2) \boldsymbol{\sigma}_{\rm max}\, \mathcal{I}_{2s}(E^*_h\cap D_r,(E^*_h)^c\cap D_r)\\
=(m-1) \boldsymbol{\sigma}_{\rm min}q\, \mathcal{I}_{2s}(E^*_h\cap D_r,(E^*_h)^c\cap D_r)\,,
 \end{multline}
where we used $\sigma_{jj}=0$ to obtain the factor $(m-2)$ (instead of $(m-1)$). 

Gathering \eqref{ineqfondnearhomoginfiltr}, \eqref{ineqfondnearhomoginfiltr2}, and \eqref{ineqfondnearhomoginfiltr3} leads to  
$$ \mathcal{I}_{2s}(E^*_h\cap D_r,(E^*_h)^c\cap D_r)\leq \frac{n\omega_n}{2s}\frac{(m-1)q}{(m-2)(1-q)}\int_0^r \frac{f^\prime(t)}{(r-t)^{2s}}\,\de t\leq \frac{n\omega_n}{s(1-q)}\int_0^r \frac{f^\prime(t)}{(r-t)^{2s}}\,\de t\,.$$
Since $(E^*_h\cap D_r)^c=((E_h^*)^c\cap D_r)\cup D_r^c$, we derive from \eqref{coareainfiltr} again that 
$$P_{2s}(E^*_h\cap D_r)= \mathcal{I}_{2s}\big(E^*_h\cap D_r,(E^*_h\cap D_r)^c\big)\leq \frac{3n\omega_n}{2s(1-q)} \int_0^r \frac{f^\prime(t)}{(r-t)^{2s}}\,\de t\,.$$
By the fractional isoperimetric inequality (see e.g. \cite{FLS,FFMMM}), we have 
\begin{equation}\label{isoperinieq}
P_{2s}(E^*_h\cap D_r)\geq\frac{P_{2s}(D_1)}{|D_1|^{\frac{n-2s}{n}}}|E^*_h\cap D_r|^{\frac{n-2s}{n}}=\frac{P_{2s}(D_1)}{\omega_n^{1-\frac{2s}{n}}} f(r)^{1-\frac{2s}{n}}\,.
\end{equation}
Therefore, 
$$  f(r)^{1-\frac{2s}{n}} \leq K \int_0^r \frac{f^\prime(t)}{(r-t)^{2s}}\,\de t\quad\text{with}\quad K:= \frac{3n\omega_n^{2-\frac{2s}{n}}}{2s(1-q)P_{2s}(D_1)}\,.$$
Choosing 
$$\mathfrak{p}_1:=\bigg(\frac{s(1-q)(1-2s)P_{2s}(D_1)}{2^{n/2s}3n\,\omega_n^{2-2s/n}}\bigg)^{n/2s} \,,$$
if $f(r)\leq \mathfrak{p}_1 r^n$, then Lemma  \ref{integineqlemma} applies and it implies that $f(r/2)=0$. 
\end{proof}

As already mentioned, the two remaining cases concerns three phases problems, that is $m=3$. We shall need the following elementary algebraic property. 

\begin{lemma}\label{lemmacoeffm=3}
For every $\boldsymbol{\sigma}\in \mathscr{S}_3$, the disjoint following properties hold. 
\begin{itemize}
\item[\bf (STI):]{\sl{Strict Triangle Inequality}}. If $\sigma_{ij}<\sigma_{ik}+\sigma_{kj}$  for every triplet $\{i,j,k\}=\{1,2,3\}$, then there exists $\alpha_1>0$, $\alpha_2>0$, and $\alpha_3>0$ such that 
$\sigma_{ij}=\alpha_i+\alpha_j$ for $i\not=j$. 
\vskip3pt

\item[\bf (SITI):]{\sl{Strict Inverse Triangle Inequality}}.  If $\sigma_{ij}>\sigma_{ik}+\sigma_{kj}$ for some pair $(i,j)$ of distinct indices, then  $\sigma_{ik}<\sigma_{ij}+\sigma_{jk}$ and $\sigma_{kj}<\sigma_{ki}+\sigma_{ij}$ for $k\not\in\{i,j\}$. 
\end{itemize}
\end{lemma}

\begin{proof}
{\it Case (i)}. Resolving explicitly the associated linear system yields
\begin{equation}\label{coeffalpham=3}
\alpha_1=\frac{1}{2}\big(\sigma_{12}+\sigma_{13}-\sigma_{23}\big)\,,\; \alpha_2=\frac{1}{2}\big(\sigma_{12}+\sigma_{23}-\sigma_{13}\big)\,,\;\alpha_3=\frac{1}{2}\big(\sigma_{13}+\sigma_{23}-\sigma_{12}\big)\,.
\end{equation}

\noindent {\it Case (ii)}. Without loss of generality, we may assume that $\sigma_{23}>\sigma_{12}+\sigma_{13}$. Assume by contradiction (and without loss of generality) that $\sigma_{12}\geq \sigma_{23}+\sigma_{13}$. Then,  $\sigma_{12}> \sigma_{12}+2\sigma_{13}$ so that $\sigma_{13}<0$, a contradiction. 
\end{proof}

\begin{remark}
Notice that, in the case $m=3$, condition \eqref{CondNearHomCoeff} implies the {\bf (STI)}-condition above, but not the converse.  Hence, for $m=3$,  the {\bf (STI)}-condition is strictly weaker than \eqref{CondNearHomCoeff}.  
\end{remark}

\begin{proposition}[\bf $m=3$ with (STI)]\label{NIP3STI}
Assume that $m=3$ and that $\sigma_{ij}<\sigma_{ik}+\sigma_{kj}$  for every triplet $\{i,j,k\}=\{1,2,3\}$. There exists a constant $\mathfrak{p}_2=\mathfrak{p}_2(\boldsymbol{\sigma},s,n)>0$ (depending only on $\boldsymbol{\sigma}$, $s$, and $n$) such  that for every ball $\overline D_{r}(x_0)\subset\Omega$ and 
every $h\in\{1,2,3\}$,  
$$|\{E^*_h\cap D_r(x_0)\}|\leq  \mathfrak{p}_2r^n\quad\Longrightarrow\quad |\{E^*_h\cap D_{r/2}(x_0)\}|=0\,.$$
\end{proposition}

\begin{proof}
The proof departs from \eqref{generineqinfiltrprop} at the end of Step 1, in the proof of Proposition \ref{NIPnearlyhom}. Using Lemma~\ref{lemmacoeffm=3} in \eqref{generineqinfiltrprop}, we obtain (since $m=3$)
\begin{multline*}
(\alpha_h+\alpha_k)\mathcal{I}_{2s}(E^*_h\cap D_r,E^*_k\cap D_r) + (\alpha_h+\alpha_{j(h,k)})\mathcal{I}_{2s}(E^*_h\cap D_r,E^*_{j(h,k)}\cap D_r)\\
\leq (\alpha_k+\alpha_{j(h,k)})\mathcal{I}_{2s}(E^*_h\cap D_r,E^*_{j(h,k)}\cap D_r) +\frac{n\omega_n\boldsymbol{\sigma}_{\rm max}}{2s}\int_0^r \frac{f^\prime(t)}{(r-t)^{2s}}\,\de t \,,
\end{multline*}
where $j(h,k)$ denotes the unique element of $\{1,2,3\}\setminus\{h,k\}$. 
Hence, 
\begin{multline}\label{ineqm3infiltr}
 \alpha_h\,\mathcal{I}_{2s}(E^*_h\cap D_r,(E^*_{h})^c\cap D_r)+ \alpha_k\,\mathcal{I}_{2s}(E^*_h\cap D_r,E^*_k\cap D_r) \\
\leq \alpha_k\,\mathcal{I}_{2s}(E^*_h\cap D_r,E^*_{j(h,k)}\cap D_r) +\frac{n\omega_n\boldsymbol{\sigma}_{\rm max}}{2s}\int_0^r \frac{f^\prime(t)}{(r-t)^{2s}}\,\de t \,.
\end{multline}
Let $k(h)\in\{1,2,3\}\setminus\{h\}$ be the index for which the value of $\mathcal{I}_{2s}(E^*_h\cap D_r,E^*_{k}\cap D_r)$ is maximal (among all $k\in\{1,2,3\}\setminus\{h\}$), so that 
$\mathcal{I}_{2s}(E^*_h\cap D_r,E^*_{j(h,k(h))}\cap D_r)\leq \mathcal{I}_{2s}(E^*_h\cap D_r,E^*_{k(h)}\cap D_r)$. Choosing $k=k(h)$ in \eqref{ineqm3infiltr} leads to 
$$\mathcal{I}_{2s}(E^*_h\cap D_r,(E^*_{h})^c\cap D_r) \leq \frac{n\omega_n\boldsymbol{\sigma}_{\rm max}}{2s \alpha_h}\int_0^r \frac{f^\prime(t)}{(r-t)^{2s}}\,\de t\,. $$
Using \eqref{coareainfiltr} as in Step 2 of the proof of Proposition \ref{NIPnearlyhom}, we deduce that  
$$P_{2s} (E^*_h\cap D_r)\leq \frac{n\omega_n(\alpha_h+\boldsymbol{\sigma}_{\rm max})}{2s \alpha_h}\int_0^r \frac{f^\prime(t)}{(r-t)^{2s}}\,\de t\leq \frac{n\omega_n \boldsymbol{\sigma}_{\rm max}}{s \alpha_{\rm min}}\int_0^r \frac{f^\prime(t)}{(r-t)^{2s}}\,\de t \,,$$
where $\alpha_{\rm min}:=\min_j\alpha_j$. Then, we infer from \eqref{isoperinieq} that
$$  f(r)^{1-\frac{2s}{n}} \leq K \int_0^r \frac{f^\prime(t)}{(r-t)^{2s}}\,\de t\quad\text{with}\quad K:= \frac{n\omega_n^{2-\frac{2s}{n}}\boldsymbol{\sigma}_{\rm max}}{s\alpha_{\rm min}P_{2s}(D_1)}\,.$$
Choosing
$$\mathfrak{p}_2:=\bigg(\frac{s\alpha_{\rm min}(1-2s)P_{2s}(D_1)}{2^{1+n/2s}n\,\omega_n^{2-2s/n}\boldsymbol{\sigma}_{\rm max}}\bigg)^{n/2s} \,,$$
then $f(r)\leq \mathfrak{p}_2 r^n$ implies $f(r/2)=0$, again by Lemma  \ref{integineqlemma}. 
\end{proof}

\begin{proposition}[\bf $m=3$ with (SITI)]\label{NIP3ISTI}
Assume that $m=3$ and that $\sigma_{i_0j_0}>\sigma_{i_0k}+\sigma_{kj_0}$  for some pair $(i_0,j_0)$  of distinct indices. There exists a constant $\mathfrak{p}_3=\mathfrak{p}_3(\boldsymbol{\sigma},s,n)>0$ (depending only on $\boldsymbol{\sigma}$, $s$, and~$n$) such  that for every ball $\overline D_{r}(x_0)\subset\Omega$ and 
 $h\in\{i_0,j_0\}$,  
$$|\{E^*_h\cap D_r(x_0)\}|\leq  \mathfrak{p}_3r^n\quad\Longrightarrow\quad |\{E^*_h\cap D_{r/2}(x_0)\}|=0\,.$$

\begin{proof}
Let us fix $h\in\{i_0,j_0\}$, and write $j(h)$ the element of $\{i_0,j_0\}\setminus\{h\}$, and $k_0$ the element of $\{1,2,3\}\setminus\{i_0,j_0\}$. 
Once again, we start from \eqref{generineqinfiltrprop} at the end of Step 1, in the proof of Proposition \ref{NIPnearlyhom}, which gives for $k=k_0$, 
\begin{multline*}
\sigma_{hk_0}\mathcal{I}_{2s}(E^*_h\cap D_r,E^*_{k_0}\cap D_r) + \sigma_{hj(h)}\mathcal{I}_{2s}(E^*_h\cap D_r,E^*_{j(h)}\cap D_r)\\
\leq \sigma_{k_0j(h)}\mathcal{I}_{2s}(E^*_h\cap D_r,E^*_{j(h)}\cap D_r) +\frac{n\omega_n\boldsymbol{\sigma}_{\rm max}}{2s}\int_0^r \frac{f^\prime(t)}{(r-t)^{2s}}\,\de t \,.
\end{multline*}
Inserting our assumption $\sigma_{hj(h)}>\sigma_{hk_0}+\sigma_{k_0j(h)}$, we derive that 
$$\boldsymbol{\sigma}_{\rm min} \mathcal{I}_{2s}(E^*_h\cap D_r,(E^*_{h})^c\cap D_r) \leq \sigma_{hk_0}\mathcal{I}_{2s}(E^*_h\cap D_r,(E^*_{h})^c\cap D_r)\leq  \frac{n\omega_n\boldsymbol{\sigma}_{\rm max}}{2s}\int_0^r \frac{f^\prime(t)}{(r-t)^{2s}}\,\de t \,. $$
Then we can argue exactly as in the proof of Proposition \ref{NIP3STI} to show that 
$$  f(r)^{1-\frac{2s}{n}} \leq K \int_0^r \frac{f^\prime(t)}{(r-t)^{2s}}\,\de t\quad\text{with}\quad K:= \frac{n\omega_n^{2-\frac{2s}{n}}\boldsymbol{\sigma}_{\rm max}}{s\boldsymbol{\sigma}_{\rm min}P_{2s}(D_1)}\,.$$
Choosing
$$\mathfrak{p}_3:=\bigg(\frac{s\boldsymbol{\sigma}_{\rm min}(1-2s)P_{2s}(D_1)}{2^{1+n/2s}n\,\omega_n^{2-2s/n}\boldsymbol{\sigma}_{\rm max}}\bigg)^{n/2s} \,,$$
leads to the announced result, still by Lemma  \ref{integineqlemma}. 
\end{proof}
\end{proposition}

\subsection{Partial regularity for minimizing partitions}

We now ready to prove the partial regularity results from Theorems \ref{mainthm3} \&  \ref{mainthm4} \&  \ref{mainthm5}. As explained in Subsection~\ref{NIPsect}, the non infiltration property allows us to reduce the regularity problem for the partition to a regularity problem for almost minimizers of the fractional perimeter $P_{2s}$. This is the purpose of the following lemma. 

\begin{lemma}
Let $x_0\in\Omega$ and $r>0$ such that $D_{2r}(x_0)\subset \Omega$. Assume that there is a pair $(h,k)$ of distinct indices such that $E^*_j\cap D_{2r}(x_0)=\emptyset$ 
for each $j\in\{1,\ldots,m\}\setminus\{h,k\}$. Then, 
\begin{equation}\label{minprescribcurvreduc}
P_{2s}\big(E^*_h,D_r(x_0)\big) -\int_{E^*_h\cap D_r(x_0)} f_{hk}(x)\,\de x\leq P_{2s}\big(\widetilde E,D_r(x_0)\big) -\int_{\widetilde E\cap D_r(x_0)} f_{hk}(x)\,\de x
\end{equation}
for every $\widetilde E\subset \R^n$ such that $\widetilde E\triangle E^*_h\subset D_r(x_0)$, where $f_{hk}\in C^\infty(D_{2r}(x_0))$ is given by
\begin{equation}\label{deffcthk}
f_{hk}(x):=\sum_{j\not\in\{h,k\}} \frac{\sigma_{hk}+\sigma_{kj}-\sigma_{hj}}{\sigma_{hk}} \int_{E^*_j}\frac{1}{|x-y|^{n+2s}}\,\de y\,.
\end{equation}
\end{lemma}

\begin{proof}
Assume without loss of generality that $x_0=0$. We observe that 
\begin{multline*}
\mathscr{P}^{\boldsymbol{\sigma}}_{2s}(\mathfrak{E}^*,D_r)=\sigma_{hk}\,\mathcal{I}_{2s}(E^*_h\cap D_r,E^*_k\cap D_r)
+\sum_{j\not=h} \sigma_{hj}\,\mathcal{I}_{2s}(E^*_h\cap D_r,E^*_j\cap D^c_r)\\
+ \sum_{j\not=k} \sigma_{kj}\,\mathcal{I}_{2s}(E^*_j\cap D^c_r,E^*_k\cap D_r)\,.
\end{multline*}
Since $E^*_k\cap D_r=(E^*_h)^c\cap D_r$ and $E^*_j\cap D_r^c=E^*_j$ for $j\not\in\{h,k\}$, we have 
\begin{multline*}
\mathscr{P}^{\boldsymbol{\sigma}}_{2s}(\mathfrak{E}^*,D_r)=\sigma_{hk}\,\mathcal{I}_{2s}(E^*_h\cap D_r,(E^*_h)^c\cap D_r) +  \sigma_{hk}\,\mathcal{I}_{2s}(E^*_h\cap D_r,E^*_k\cap D^c_r)
\\
+ \sigma_{hk}\,\mathcal{I}_{2s}(E^*_h\cap D^c_r,(E^*_h)^c\cap D_r)+\sum_{j\not\in\{h,k\}} \sigma_{hj}\,\mathcal{I}_{2s}(E^*_h\cap D_r,E^*_j)+ \sum_{j\not\in\{h,k\}} \sigma_{kj}\,\mathcal{I}_{2s}(E^*_j,(E^*_h)^c\cap D_r)\,,
\end{multline*}
that we can further rewrite as 
\begin{multline*}
\mathscr{P}^{\boldsymbol{\sigma}}_{2s}(\mathfrak{E}^*,D_r)=\sigma_{hk} P_{2s}(E^*_h,D_r)+\sum_{j\not\in\{h,k\}} (\sigma_{hj}-\sigma_{hk})\,\mathcal{I}_{2s}(E^*_h\cap D_r,E^*_j)\\
+ \sum_{j\not\in\{h,k\}} \sigma_{kj}\,\mathcal{I}_{2s}(E^*_j,(E^*_h)^c\cap D_r)\,.
\end{multline*}
Then we notice that 
$$ \mathcal{I}_{2s}(E^*_j,(E^*_h)^c\cap D_r)=\mathcal{I}_{2s}(E^*_j,D_r)-\mathcal{I}_{2s}(E^*_j,E^*_h\cap D_r)\,, $$
which leads to 
\begin{equation}\label{rewritpernoinfil1}
\mathscr{P}^{\boldsymbol{\sigma}}_{2s}(\mathfrak{E}^*,D_r)=\sigma_{hk}\Big(P_{2s}(E^*_h,D_r)-\int_{E^*_h\cap D_r(x_0)} f_{hk}(x)\,\de x\Big)
+ \sum_{j\not\in\{h,k\}} \sigma_{kj}\,\mathcal{I}_{2s}(E^*_j,D_r)\,.
\end{equation}
Now we consider an arbitrary set $\widetilde E\subset\R^n$ such that $\widetilde E\triangle E^*_h\subset D_r$.  We define the competitor $\mathfrak{F}=(F_1,\ldots,F_m)$  by 
$$F_j:=\begin{cases} 
E^*_j & \text{if $j\not\in\{h,k\}$}\,,\\
\widetilde E & \text{if $j=h$}\,,\\
(E^*_k\setminus D_r)\cup (\widetilde E^c\cap D_r) & \text{if $j=k$}\,.
\end{cases}$$
By construction, we have $F_j\triangle E^*_j=\emptyset$ for $j\not\in\{h,k\}$, and $F_j\triangle E^*_j\subset D_r$ for $j\in\{h,k\}$. Hence Lemma~\ref{minimlaityinset} applies and we deduce that
\begin{equation}\label{miniminBinfil}
\mathscr{P}^{\boldsymbol{\sigma}}_{2s}(\mathfrak{E}^*,D_r)\leq \mathscr{P}^{\boldsymbol{\sigma}}_{2s}(\mathfrak{F},D_r)\,.
\end{equation}
On the other hand, $F_j\cap D_r=\emptyset$ and $F_j=E^*_j$ for $j\not\in\{h,k\}$ so that the computations above apply and leads to 
\begin{equation}\label{rewritpernoinfil2}
\mathscr{P}^{\boldsymbol{\sigma}}_{2s}(\mathfrak{F},D_r)=\sigma_{hk}\Big(P_{2s}(\widetilde E,D_r)-\int_{\widetilde E\cap D_r(x_0)} f_{hk}(x)\,\de x\Big)
+ \sum_{j\not\in\{h,k\}} \sigma_{kj}\,\mathcal{I}_{2s}(E^*_j,D_r)\,.
\end{equation}
Combining \eqref{rewritpernoinfil1}, \eqref{miniminBinfil}, and \eqref{rewritpernoinfil2} yields the announced result. 
\end{proof}

\begin{remark}
The argument of the previous proof also applies to the stationary case. More precisely, 
if $\mathfrak{E}=(E_1,\ldots,E_m)\in\mathscr{A}_m(\Omega)$ is a stationary  point in  $\Omega$ of $\mathscr{P}^{\boldsymbol{\sigma}}_{2s}$ (i.e., a solution to equation \eqref{SectParteq1}) satisfying $E_j\cap D_{2r}(x_0)=\emptyset$ for each $j\in\{1,\ldots,m\}\setminus\{h,k\}$ and some pair  $(h,k)$ of distinct indices, then 
\begin{equation}\label{firstvar2phases}
\delta P_{2s}\big(E_h,D_r(x_0)\big)[X]=\int_{E_h\cap D_r(x_0)}{\rm div}(f_{hk}X)\,\de x \qquad\forall X\in C_c^1(D_r(x_0);\R^n)\,.
\end{equation}
Indeed, given $X\in C_c^1(D_r(x_0);\R^n)$, its integral flow $\{\phi_t\}_{t\in\R}$ satisfies ${\rm spt}(\phi_t-{\rm id})\subset D_r(x_0)$. Therefore, $\phi_t(E_j)=E_j$ for $j\not\in\{h,k\}$, and $\phi_t(E_j)\triangle E_j\subset D_r(x_0)$ for $j\in\{h,k\}$.  
Computing as in the proof of Lemma \ref{minimlaityinset} and \eqref{rewritpernoinfil1}-\eqref{rewritpernoinfil2}, we realize that 
$\mathfrak{E}_t:=(\phi_t(E_1),\ldots,\phi_t(E_m))$ satisfies for every $t\in\R$, 
\begin{align*}
\mathscr{P}^{\boldsymbol{\sigma}}_{2s}(\mathfrak{E}_t,\Omega) -\mathscr{P}^{\boldsymbol{\sigma}}_{2s}(\mathfrak{E},\Omega)=\,&\mathscr{P}^{\boldsymbol{\sigma}}_{2s}\big(\mathfrak{E}_t,D_r(x_0)\big) -\mathscr{P}^{\boldsymbol{\sigma}}_{2s}\big(\mathfrak{E},D_r(x_0)\big)\\
=\,&\sigma_{hk}\Big(P_{2s}\big(\phi_t(E_h),D_r(x_0)\big))-P_{2s}\big(E_h,D_r(x_0)\big)\\
&\qquad -\int_{\phi_t(E_h)\cap D_r(x_0)} f_{hk}(x)\,\de x+\int_{E_h\cap D_r(x_0)} f_{hk}(x)\,\de x\Big)\,. 
\end{align*}
Taking the derivative at $t=0$ yields 
$$0=\delta \mathscr{P}^{\boldsymbol{\sigma}}_{2s}(\mathfrak{E},\Omega)[X]=\sigma_{hk}\Big(\delta P_{2s}\big(E_h,D_r(x_0)\big)[X] - \int_{E_h\cap D_r(x_0)} {\rm div}(f_{hk}X)\,\de x\Big)\,, $$
and \eqref{firstvar2phases} follows. 
\end{remark}
\vskip5pt
\begin{remark}\label{RemELeqpartit}
Relation \eqref{firstvar2phases} corresponds to the weak (or distributional) formulation of the Euler-Lagrange equation 
\begin{equation}\label{meancurveqrem}
{\rm H}^{(2s)}_{\partial E_h}=f_{hk} \quad \text{on $\partial E_h\cap D_r(x_0)$}\,,
\end{equation}
where ${\rm H}^{(2s)}_{\partial E_h}$ is the {\sl nonlocal mean curvature} of  $\partial E_h$ defined by 
\begin{equation}\label{defH2Srem}
{\rm H}^{(2s)}_{\partial E_h}(x)={\rm p.v.}\;\int_{\R^n}\frac{\chi_{E_h^c}(y)-\chi_{E_h}(y)}{|x-y|^{n+2s}}\,\de y\,, \quad x\in \partial E_h\,,
\end{equation}
and the notation p.v. means that integral is taken in the principal value sense. 

Indeed, if $\partial E_h\cap D_r(x_0)$ is a smooth hypersurface, then \cite[Theorem 6.1]{FFMMM} shows that 
$$ \delta P_{2s}\big(E_h,D_r(x_0)\big)[X]=\int_{\partial E_h\cap D_r(x_0)}{\rm H}^{(2s)}_{\partial E_h} \nu_{\rm ext}\cdot X\,\de\mathcal{H}^{n-1}\,,$$
where  $\nu_{\rm ext}$ denotes the unit exterior normal field on $\partial E_h\cap D_r(x_0)$. Hence, choosing $X$ of the form $X=w \,\nu_{ext}$  with $w\in C^\infty_c(D_r(x_0))$ in \eqref{firstvar2phases} and integrating by parts the right hand side yields 
$$\int_{\partial E_h\cap D_r(x_0)}{\rm H}^{(2s)}_{\partial E_h}w\,\de\mathcal{H}^{n-1} = \int_{\partial E_h\cap D_r(x_0)}f_{hk}w\,\de\mathcal{H}^{n-1}\,,$$
and \eqref{meancurveqrem} follows. 
\end{remark}

\vskip5pt

\begin{proof}[Proof of Theorem \ref{mainthm3}] 
Applying the results from Subsection \ref{complrestprescrNMMC}, we have the decomposition $\partial\mathfrak{E}^*\cap\Omega=(\Sigma_{\rm reg}(u_*^\e)\cap\Omega)\cup \Sigma_{\rm sing}$, 
where $\Sigma_{\rm reg}(u_*^\e)$ is defined by \eqref{defsigmareg} and $\Sigma_{\rm sing}:={\rm Sing}^{n-2}(u_*^\e)$.  By Theorem~\ref{dimMink}, the set $\Sigma_{\rm sing}$ has a Hausdorff dimension less than $(n-2)$, and it is countable for $n=2$. We aim to prove that, for each point $x_0\in \Sigma_{\rm reg}(u_*^\e)\cap\Omega$, 
$\partial\mathfrak{E}^*\cap\Omega$ is a smooth hypersurface in a neighborhood of~$x_0$. Obviously, this will imply that $\Sigma_{\rm sing}$ is relatively closed subset of $\partial\mathfrak{E}^*\cap\Omega$, which must be locally finite in $\Omega$ for $n=2$. 
\vskip5pt

We fix $x_0\in \Sigma_{\rm reg}(u_*^\e)\cap\Omega$, and we first apply Proposition \ref{regblowupset} to obtain a sequence of radii $\rho_l\to 0$ as $l\to\infty$, two distinct indices $h,k\in\{1,\ldots,m\}$ and a half space $H\subset \R^n$ with $0\in\partial H$ such that 
\begin{equation}\label{blowuptoplaneprf}
(u_*)_{x_0,\rho_l}\to  \varphi_0:=\chi_{{\rm int}(H)}{\bf a}_{h}+\chi_{{\rm int}(H^c)}{\bf a}_{k}\quad \text{in $L^1(D_4)$}\,,
\end{equation}
and 
\begin{equation}\label{blowuptoplaneprf2}
\Sigma\big((u^\e_*)_{x_0,\rho_l}\big)\to \partial H\quad\text{locally uniformly in $\R^n$}\,. 
\end{equation}
Without loss of generality, we may assume that $H=\{x_n<0\}$, $\rho_0\leq 1$, that the sequence $\{\rho_l\}$ is non increasing, and that $D_{5\rho_l}(x_0)\subset\Omega$ for each $l$. 
\vskip3pt

Let us write 
$$\mathfrak{E}^*_{x_0,\rho_l}:=\big(({E}^*_1)_{x_0,\rho_l},\ldots,({E}^*_m)_{x_0,\rho_l}\big)\text{ with }({E}^*_j)_{x_0,\rho_l}:=(E^*_j-x_0)/\rho_l\,.$$ 
For an arbitrary index $j\in\{1,\ldots,m\}\setminus\{h,k\}$, we consider the Lipschitz map ${\rm p}_j$ defined in \eqref{projtrickEj}. Noticing that  ${\rm p}_j(\varphi_0)=0$,  we deduce from \eqref{blowuptoplaneprf} that 
$$\big|({E}^*_j)_{x_0,\rho_l}\cap D_4\big|=\big\|{\rm p}_j\big((u_*)_{x_0,\rho_l}\big)-{\rm p}_j(\varphi_0)\big\|_{L^1(D_4)}\leq \frac{1}{\boldsymbol{\sigma}_{\rm min}^{1/2}} 
\|(u_*)_{x_0,\rho_l}-\varphi_0\|_{L^1(D_4)}\mathop{\longrightarrow}\limits_{l\to\infty}0\,,$$
Therefore, we can find $l_0=l_0(x_0)\in\mathbb{N}$ such that for $l\geq l_0$, 
$$\big|({E}^*_j)_{x_0,\rho_l}\cap D_4\big|\leq \mathfrak{p}_14^n\qquad\forall j\in\{1,\ldots,m\}\setminus\{h,k\}\,,$$
 where the constant $ \mathfrak{p}_1>0$  is given by Proposition \ref{NIPnearlyhom}. Scaling back, we infer that for $l\geq l_0$, 
 $$\big|{E}^*_j\cap D_{4\rho_l}(x_0)\big|\leq \mathfrak{p}_1(4\rho_l)^n\qquad\forall j\in\{1,\ldots,m\}\setminus\{h,k\}\,.$$
 By Proposition \ref{NIPnearlyhom}, it implies that $\big|{E}^*_j\cap D_{2\rho_l}(x_0)\big|=0$ for every $j\in\{1,\ldots,m\}\setminus\{h,k\}$ and $l\geq l_0$. Since each ${E}^*_j\cap\Omega$ is an open set, we conclude that for $l\geq l_0$,
 \begin{equation}\label{cleaninfilproof}
 E^*_j\cap D_{2\rho_l}(x_0)=\emptyset  \qquad\forall j\in\{1,\ldots,m\}\setminus\{h,k\}\,.
 \end{equation}
 Applying Proposition \ref{NIPnearlyhom}, we deduce that ${E}^*_h$ satisfies \eqref{minprescribcurvreduc} with $r=\rho_l$ for $l\geq l_0$. Rescaling variables again, we infer that, for $l\geq l_0$, the set 
  $({E}^*_h)_{x_0,\rho_l}$ satisfies 
  \begin{equation}\label{minimblowupprf}
  P_{2s}\big(({E}^*_h)_{x_0,\rho_l}, D_1\big)-\int_{({E}^*_h)_{x_0,\rho_l}\cap D_1}\widehat f_{l}(x)\,\de x\leq  P_{2s}\big(\widetilde E, D_1\big)-\int_{\widetilde E\cap D_1}\widehat f_{l}(x)\,\de x
  \end{equation}
 for every $\widetilde E\subset \R^n$ such that $\widetilde E\triangle ({E}^*_h)_{x_0,\rho_l}\subset D_1$, with 
 $$\widehat f_{l}(x):=\rho_l^{2s}f_{hk}(x_0+\rho_l x)\quad\text{and $f_{hk}$ given by \eqref{deffcthk}}\,.$$
 Since $\rho_l\leq 1$ and $\{\rho_l\}$ is non increasing, we infer from \eqref{cleaninfilproof} that for $l\geq l_0$, 
 \begin{equation}\label{linftyestprescrcurvblowup}
 \| \widehat f_{l}\|_{L^\infty(D_1)}\leq \|f_{hk}\|_{L^\infty(D_{\rho_{l_0}}(x_0))} \leq C(\rho_{l_0})^{-2s}\,,
 \end{equation}
 for a constant $C$ depending only on $n$, $s$, and $\boldsymbol{\sigma}$. 
 
 Moreover, 
 as a further consequence of \eqref{cleaninfilproof}, we have for $l\geq l_0$,  
 $$\Sigma(u_*^\e)\cap D_{2\rho_l}(x_0)=\partial\mathfrak{E}^*\cap D_{2\rho_l}(x_0)=\partial E^*_h\cap D_{2\rho_l}(x_0)\,.$$ 
 In view of \eqref{blowuptoplaneprf2}, it implies that $\partial ({E}^*_h)_{x_0,\rho_l}\to \partial H=\{x_n=0\}$ locally uniformly in $D_2$. In particular, we can find $l_1\geq l_0$ such that for $\rho_*:=\rho_{l_1}$,  
 $$\text{$({E}^*_h)_{x_0,\rho_{*}}$ satisfies \eqref{minimblowupprf} with $l=l_1$, and }\,   \partial({E}^*_h)_{x_0,\rho_{*}} \cap \overline D_1\subset  \big\{|x_n|\leq \boldsymbol{\varepsilon}_0\big \}\,,$$
 where the constant $\boldsymbol{\varepsilon}_0>0$ is given by \cite[Theorem 1.16]{DVV} and it only depends on $n$, $s$, $\boldsymbol{\sigma}$, and $\rho_{l_0}$ thanks to \eqref{linftyestprescrcurvblowup}.  According to \cite[Theorem 1.16]{DVV}, $\partial ({E}^*_h)_{x_0,\rho_{*}}\cap D_{1/2}$ is $C^{1,\alpha}$-hypersurface for every $\alpha\in(0,2s)$.  Scaling back, we conclude that 
 \begin{equation}\label{firstregapriori}
 \text{$\partial {E}^*_h\cap D_{\rho_*/2}(x_0)$ is a $C^{1,\alpha}$-hypersurface for every $\alpha\in(0,2s)$.} 
 \end{equation}
Since ${E}^*_h$ satisfies \eqref{minprescribcurvreduc} for $r=\rho_*/2$,  according to \cite{CapGui,DVV}, the following equation 
$${\rm H}^{(2s)}_{\partial E^*_h}= f_{hk} \quad\text{on $\partial {E}^*_h\cap D_{\rho_*/2}(x_0)$}\,,$$
holds in a certain viscosity sense (defined in \cite{CapGui,DVV}, see also the original article \cite{CRS} in the homogeneous case), where ${\rm H}^{(2s)}_{\partial E^*_h}$ is the nonlocal mean curvature of $\partial E^*_h$ defined in \eqref{defH2Srem}. By the a priori regularity \eqref{firstregapriori} and the fact that $f_{hk}$ is smooth in $ D_{\rho_1}(x_0)$, one can apply the higher order regularity result  in \cite[Theorem 5]{BFV} to conclude that 
$\partial\mathfrak{E}^*\cap D_{r}(x_0)=\partial {E}^*_h\cap D_{r}(x_0)$ is a $C^{\infty}$-hypersurface for some $r\in(0,\rho_*/2)$. 
\end{proof}

\begin{proof}[Proof of Theorem \ref{mainthm4}] 
The proof  follows exactly the proof of Theorem \ref{mainthm3} replacing Proposition  \ref{NIPnearlyhom} by Proposition \ref{NIP3STI}. 
\end{proof}

To handle the last case where $m=3$ and {\rm ({\bf SITI})} holds, we need the following general fact about tangent maps in the minimizing case. 

\begin{lemma}\label{lemmeminimtangmap}
If $u_*\in \widehat H^s(\Omega;\mathcal{Z}_{\boldsymbol{\sigma}})$ is  a $\mathcal{E}_s$-minimizing map into $\mathcal{Z}_{\boldsymbol{\sigma}}$ in $\Omega$, then for every $x_0\in \Omega$, (the restriction to $\R^n$ of) any tangent map $\varphi\in T_{x_0}(u_*^\e)$ is a $\mathcal{E}_s$-minimizing map into $\mathcal{Z}_{\boldsymbol{\sigma}}$ in $D_1$. 
\end{lemma}

\begin{proof}
Let $x_0\in \Omega$ and $\varphi\in T_{x_0}(u_*^\e)$ an arbitrary tangent map.  By the very definition of tangent map and Lemma \ref{tangmapstat}, there exists a sequence of radii $\rho_k\to 0$ such that the rescaled maps $u_k:=(u_*^\e)_{x_0,\rho_k}$ (as defined in \eqref{defrescmap}) strongly converges in $H^1(B_r^+,|z|^a\de {\bf x})$ for every $r>0$. By \cite[Lemma 2.9]{MilPegSch}, it implies that $u_k\to \varphi$ strongly in $H^s(D_r)$ for every $r>0$. By the embedding $H^s(D_r)\hookrightarrow L^1(D_r)$, we have $u_k\to \varphi$ strongly in $L^1_{\rm loc}(\R^n)$, and thus $u_k\to \varphi$ a.e. in $\R^n$ extracting a subsequence if necessary. 

We now fix $r\in(0,1)$ and $v\in  \widehat H^s(D_1;\mathcal{Z}_{\boldsymbol{\sigma}})$ such that ${\rm spt}(v-\varphi_0)\subset D_r$. We define a competitor $v_k\in  \widehat H^s(D_1;\mathcal{Z}_{\boldsymbol{\sigma}})$ by setting $v_k(x):=v(x)$ for $|x|<r$, and $v_k(x):=u_k(x)$ for $|x|>r$ (here we use that $P_{2s}(D_r,\R^n)<\infty$). Since ${\rm spt}(v_k -u_k)\subset D_1$, the minimality of $u_*$ and a rescaling of variables yield
\begin{equation}\label{mintangmapineq}
\mathcal{E}_s(v_k,D_r)- \mathcal{E}_s(u_k,D_r)=  \mathcal{E}_s(v_k,D_1)- \mathcal{E}_s(u_k,D_1)\geq 0\,.
\end{equation}
We have 
$$ \mathcal{E}_s(v_k,D_r)=\frac{1}{2}[v]^2_{H^s(D_r)} +\frac{\gamma_{n,s}}{2}\iint_{D_r\times D_r^c}\frac{|v(x)-u_k(y)|^2}{|x-y|^{n+2s}}\,\de x\de y\,,$$
and 
$$ \mathcal{E}_s(u_k,D_r)=\frac{1}{2}[u_k]^2_{H^s(D_r)} +\frac{\gamma_{n,s}}{2}\iint_{D_r\times D_r^c}\frac{|u_k(x)-u_k(y)|^2}{|x-y|^{n+2s}}\,\de x\de y\,.$$
Since 
$$ \frac{|v(x)-u_k(y)|^2}{|x-y|^{n+2s}} +\frac{|u_k(x)-u_k(y)|^2}{|x-y|^{n+2s}}\leq \frac{2\boldsymbol{\sigma}_{\rm max}}{|x-y|^{n+2s}} \in L^1(D_r\times D_r^c)\,,$$
by dominated convergence and the fact that  $u_k\to \varphi$ strongly in $H^s(D_r)$, we have
$$ \mathcal{E}_s(v_k,D_r)\to  \mathcal{E}_s(v,D_r) \quad\text{and}\quad \mathcal{E}_s(u_k,D_r)\to  \mathcal{E}_s(\varphi,D_r)\,. $$
Letting $k\to\infty$ in \eqref{mintangmapineq} yields
$$0\leq  \mathcal{E}_s(v,D_r)- \mathcal{E}_s(\varphi,D_r)=  \mathcal{E}_s(v,D_1)- \mathcal{E}_s(\varphi,D_1)\,, $$
and the minimality of $\varphi$ is proven. 
\end{proof}

From the previous lemma, we may now prove that certain $\mathcal{Z}_{\boldsymbol{\sigma}}$-cones do not appear as tangent maps in the minimizing case if {\rm ({\bf SITI})} holds, 
at least for $s$ close enough to $1/2$. 

\begin{lemma}\label{mintangmapclassm3}
Assume that $m=3$ and that condition {\rm ({\bf SITI})} holds: $\sigma_{i_0j_0}>\sigma_{i_0k_0}+\sigma_{k_0j_0}$  for some indices $\{i_0,j_0,k_0\}=\{1,2,3\}$. Then, there exists $s_*=s_*(\boldsymbol{\sigma},n)\in(0,1/2)$ such that if $s\in(s_*,1/2)$, at  every $x_0\in\Sigma_{\rm reg}(u_*^\e)$, any tangent map 
 $\varphi\in T_{x_0}(u_*^\e)$ is of the form  
 $$\varphi=\chi_{{\rm int}(H)} {\bf a}_{h}+\chi_{{\rm int}(H^c)}{\bf a}_{k_0}\quad\text{on $\R^n$}\,,$$
 for some $h\in\{i_0,j_0\}$ and  some half space $H\subset \R^n$ with $0\in\partial H$. 
\end{lemma}

\begin{proof}
By Remark \ref{remhalfspa}, since $x_0\in\Sigma_{\rm reg}(u_*^\e)$, it is enough to prove  that for $s$ close enough to $1/2$, the map
\begin{equation}\label{imposstgtmap}
\varphi_0=\chi_{{\rm int}(H)} {\bf a}_{i_0}+\chi_{{\rm int}(H^c)}{\bf a}_{j_0}
\end{equation}
is not a tangent map to $u_*^\e$ at $x_0$. Since $u_*$ is $\mathcal{E}_s$-minimizing in $\Omega$, by Lemma \ref{lemmeminimtangmap} it is enough to prove that 
 $\varphi_0$ in \eqref{imposstgtmap} is not $\mathcal{E}_s$-minimizing in $D_1$ for any half space $H\subset \R^n$ with $0\in\partial H$. By invariance under rotations, it is obviously enough to consider the case $H=\{x_n<0\}$. 

To prove the non minimality of $\varphi_0$, we argue by contradiction assuming that $\varphi_0$ is a $\mathcal{E}_s$-minimizing map into $\mathcal{Z}_{\boldsymbol{\sigma}}$ in $D_1$ 
for some (not labeled) sequence $s\to 1/2$. We shall obtain a contradiction construction a suitable competitor. To this purpose, we set 
$$\varepsilon:=\frac{1}{3}\min\Big\{1,\frac{\omega_{n-1}}{\omega_{n-2}}\Big(\frac{\sigma_{i_0j_0}}{\sigma_{i_0k_0}+\sigma_{j_0k_0}}-1\Big)\Big\}\in(0,1/3]\,, $$
and
$$T:=\big\{x=(x^\prime,x_n)\in\R^n: |x^\prime|< 1/2\,,\;|x_n|< \varepsilon\big\} \subset D_1,.$$
Then we define the map $v\in\widehat H(D_1;\mathcal{Z}_{\boldsymbol{\sigma}})$ by 
$$v:= \chi_{H \setminus T} {\bf a}_{i_0}+\chi_{H^c\setminus T}{\bf a}_{j_0}+\chi_{T}{\bf a}_{k_0}\,, $$
and the associated partitions 
$$\mathfrak{E}^0:=(E_1,E_2,E_3)\text{ with } E_{i_0}= H\,,\;E_{j_0}:=H^c\,,\;E_{k_0}=\emptyset\,,$$
and 
$$\mathfrak{F}:=(F_1,F_2,F_3)\text{ with } F_{i_0}= H\setminus T\,,\;E_{j_0}:=H^c\setminus T\,,\;F_{k_0}=T\,.$$
Since ${\rm spt}(v-\varphi_0)\subset D_1$, the minimality of $\varphi_0$ implies that 
$$\mathscr{P}^{\boldsymbol{\sigma}}_{2s}(\mathfrak{F},D_1)=\frac{2}{\gamma_{n,s}}\mathcal{E}_s(v,D_1)
\geq  \frac{2}{\gamma_{n,s}}\mathcal{E}_s(\varphi_0,D_1)=\mathscr{P}^{\boldsymbol{\sigma}}_{2s}(\mathfrak{E}^0,D_1)\,.$$ 
Since $F_j\triangle E^0_j\subset T$ for each $j$, we infer from Lemma \ref{minimlaityinset} that 
\begin{equation}\label{condmintangentsigmainv}
0\leq \mathscr{P}^{\boldsymbol{\sigma}}_{2s}(\mathfrak{F},D_1) - \mathscr{P}^{\boldsymbol{\sigma}}_{2s}(\mathfrak{E}^0,D_1)= 
 \mathscr{P}^{\boldsymbol{\sigma}}_{2s}(\mathfrak{F},T) - \mathscr{P}^{\boldsymbol{\sigma}}_{2s}(\mathfrak{E}^0,T)\,.
 \end{equation}
By the computations in \eqref{coeffalpham=3}, we have $\sigma_{ij}=\alpha_i+\alpha_j$ for every $i,j\in\{1,\ldots,m\}$ with 
$$\alpha_{i_0}>0\,,\quad \alpha_{j_0}>0\,,\quad \alpha_{k_0}<0\,. $$
As a consequence,
\begin{multline}\label{condmintangentsigmainv1}
 \mathscr{P}^{\boldsymbol{\sigma}}_{2s}(\mathfrak{E}^0,T)=\alpha_{i_0}P_{2s}(H,T)+ \alpha_{j_0}P_{2s}(H^c,T)\\
 \geq (\alpha_{i_0}+ \alpha_{j_0})\mathcal{I}_{2s}(H\cap T,H^c\cap T)
=\frac{\sigma_{i_0j_0}}{2}\iint_{T\times T}\frac{|\chi_H(x)-\chi_H(y)|}{|x-y|^{n+2s}}\,\de x\de y\,,
\end{multline}
and  by symmetry of $H$ and $T$ with respect to $\{x_n=0\}$, 
\begin{align}
\nonumber \mathscr{P}^{\boldsymbol{\sigma}}_{2s}(\mathfrak{F},T)&=\alpha_{i_0}P_{2s}(H\setminus T,T)+\alpha_{j_0}P_{2s}(H^c\setminus T,T)+\alpha_{k_0}P_{2s}(T,T)\\
\nonumber& = \alpha_{i_0}\mathcal{I}_{2s}(H\cap T^c,T)+ \alpha_{j_0}\mathcal{I}_{2s}(H^c\cap T^c,T)+ \alpha_{k_0}\mathcal{I}_{2s}(T,T^c)\\
\nonumber &= \Big(\frac{\alpha_{i_0}+\alpha_{j_0}}{2}+ \alpha_{k_0}\Big)\mathcal{I}_{2s}(T,T^c)\\
& = \frac{\sigma_{i_0k_0}+\sigma_{k_0j_0}}{2} \,\mathcal{I}_{2s}(T,T^c)\,. \label{condmintangentsigmainv2}
\end{align}
Then, we observe that  
\begin{multline}\label{condmintangentsigmainv3}
\mathcal{I}_{2s}(T,T^c)= \mathcal{I}_{2s}(T\cap D_2,T^c\cap D_2) + \mathcal{I}_{2s}(T\cap D_1, D^c_2)\\
=\frac{1}{2}\iint_{D_2\times D_2}\frac{|\chi_T(x)-\chi_T(y)|}{|x-y|^{n+2s}}\,\de x\de y + \mathcal{I}_{2s}(T\cap D_1, D^c_2)\,,
\end{multline}
and 
\begin{equation}\label{condmintangentsigmainv4}
 \mathcal{I}_{2s}(T\cap D_1, D^c_2)\leq  \omega_n\int_{D_2^c}\frac{1}{(|y|-1)^{n+2s}}\,\de y\leq 2^{n+2s}n\omega_n^2\int_2^{+\infty}\frac{dr}{r^{1+2s}}=\frac{2^n\omega_n^2}{2s}\,.
 \end{equation}
According to \cite[Theorem 1]{Dav}, we have 
\begin{equation}\label{condmintangentsigmainv5}
\lim_{s\to 1/2}(1-2s)\iint_{T\times T}\frac{|\chi_H(x)-\chi_H(y)|}{|x-y|^{n+2s}}\,\de x\de y=K_n {\rm Per}(H,T)= \frac{K_n}{2^{n-1}}\omega_{n-1}\,,
\end{equation}
and 
\begin{multline}\label{condmintangentsigmainv6}
 \lim_{s\to 1/2} (1-2s)\iint_{D_2\times D_2}\frac{|\chi_T(x)-\chi_T(y)|}{|x-y|^{n+2s}}\,\de x\de y = K_n{\rm Per}(T,D_2)\\
 =(\omega_{n-1}+2\varepsilon \omega_{n-2})\frac{K_n}{2^{n-2}}\,,
 \end{multline}
where ${\rm Per}(\cdot,A)$ denotes the standard relative perimeter in an open set $A$,  and $K_n$ is a constant depending only on $n$. 

Multiplying \eqref{condmintangentsigmainv} by $(1-2s)$ and letting $s\to 1/2$, we deduce from \eqref{condmintangentsigmainv1}-\eqref{condmintangentsigmainv2}-\eqref{condmintangentsigmainv3}-\eqref{condmintangentsigmainv4}-\eqref{condmintangentsigmainv5}-\eqref{condmintangentsigmainv6} and our choice of $\varepsilon$ that 
$$0\leq \frac{K_n}{2^n}\Big( (\omega_{n-1}+2\varepsilon\omega_{n-2})(\sigma_{i_0k_0}+\sigma_{k_0j_0}) -\omega_{n-1}\sigma_{i_0j_0} \Big)<0\,,$$
a contradiction. 
\end{proof}

\begin{proof}[Proof of Theorem \ref{mainthm4}] 
We choose $s\in(s_*,1/2)$ with $s_*$ given by Lemma \ref{mintangmapclassm3}.  Then we reproduce the proof of Theorem  \ref{mainthm3} using that \eqref{blowuptoplaneprf} holds with $h\in\{i_0,j_0\}$ and $k=k_0$ by Lemma \ref{mintangmapclassm3}. In turns, this allows to use Proposition \ref{NIP3ISTI} instead of Proposition  \ref{NIPnearlyhom}, and it proves that $(\partial\mathfrak{E}^*\cap\Omega)\setminus \Sigma_{\rm sing}=(\Sigma_{\rm reg}(u_*^\e)\cap\Omega)\setminus  \Sigma_{\rm sing}$ is locally a smooth hypersurface. 

We finally claim that $(\partial E^*_{i_0}\cap \partial E^*_{j_0})\cap \Omega\subset \Sigma_{\rm sing}$. Indeed, assume by contradiction that there exists a point $x_0\in (\partial E^*_{i_0}\cap \partial E^*_{j_0}\cap\Omega)\setminus  \Sigma_{\rm sing}$. Since $\partial\mathfrak{E}^*$ is smooth in a neighborhood of $x_0$, we have $\partial\mathfrak{E}^*\cap D_r(x_0)=\partial E^*_{i_0}\cap \partial E^*_{j_0}\cap D_r(x_0)$ and $|(E^*_{i_0}\cup E^*_{j_0})\triangle D_r(x_0)|=0$ for $r$ small enough. As a consequence, there is a unique tangent map to $u_*^\e$ at $x_0$ of the form  \eqref{imposstgtmap} on $\R^n$, which contradicts Lemma \ref{mintangmapclassm3}. 
\end{proof}



\end{document}